\documentclass{amsart}

\title{Logarithmic adic spaces: some foundational results}

\author{Hansheng Diao, Kai-Wen Lan, Ruochuan Liu, and Xinwen Zhu}
\address{Yau Mathematical Sciences Center, Tsinghua University, Beijing 100084, China}
\email{hdiao@mail.tsinghua.edu.cn}
\address{University of Minnesota, 127 Vincent Hall, 206 Church Street SE, Minneapolis, MN 55455, USA}
\email{kwlan@math.umn.edu}
\address{Beijing International Center for Mathematical Research, Peking University, 5 Yi He Yuan Road, Beijing 100871, China}
\email{liuruochuan@math.pku.edu.cn}
\address{California Institute of Technology, 1200 East California Boulevard, Pasadena, CA 91125, USA}
\email{xzhu@caltech.edu}
\thanks{K.-W.\@\xspace Lan was partially supported by the National Science Foundation under agreement No.\@\xspace DMS-1352216, by an Alfred P.\@\xspace Sloan Research Fellowship, and by a Simons Fellowship in Mathematics.  R.\@\xspace Liu was partially supported by the National Natural Science Foundation of China under agreement Nos.\@\xspace NSFC-11571017 and NSFC-11725101, and by the Tencent Foundation.  X.\@\xspace Zhu was partially supported by the National Science Foundation under agreement Nos.\@\xspace DMS-1602092 and DMS-1902239, by an Alfred P.\@\xspace Sloan Research Fellowship, and by a Simons Fellowship in Mathematics.  Any opinions, findings, and conclusions or recommendations expressed in this writing are those of the authors, and do not necessarily reflect the views of the funding organizations.}

\usepackage{amsfonts, latexsym, amssymb, mathrsfs}

\usepackage[driverfallback=hypertex, unicode]{hyperref}
\usepackage{color}

\usepackage[neveradjust]{paralist}

\usepackage{upref}
\usepackage[shortlabels]{enumitem}
\usepackage{verbatim}
\usepackage{xspace}

\usepackage[cmtip, all]{xy}
\begingroup\expandafter\expandafter\expandafter\endgroup
\expandafter\ifx\csname IncludeInRelease\endcsname\relax
\usepackage{fixltx2e}
\fi

\newtheorem{thm}[equation]{{Theorem}}
\newtheorem{cor}[equation]{{Corollary}}
\newtheorem{lem}[equation]{{Lemma}}
\newtheorem{prop}[equation]{{Proposition}}
\newtheorem{exam}[equation]{{Example}}
\newtheorem{constr}[equation]{{Construction}}

\theoremstyle{definition}
\newtheorem{defn}[equation]{{Definition}}
\newtheorem{rk}[equation]{{Remark}}
\newtheorem{conv}[equation]{{Convention}}

\newcommand{\quash}[1]{}

\newcommand{\bD}{\mathbb{D}}
\newcommand{\bE}{\mathbb{E}}
\newcommand{\bF}{\mathbb{F}}
\newcommand{\bG}{\mathbb{G}}

\newcommand{\bJ}{\mathbb{J}}
\newcommand{\bK}{\mathbb{K}}
\newcommand{\bL}{\mathbb{L}}
\newcommand{\bM}{\mathbb{M}}

\newcommand{\bP}{\mathbb{P}}
\newcommand{\bQ}{\mathbb{Q}}
\newcommand{\bR}{\mathbb{R}}

\newcommand{\bT}{\mathbb{T}}

\newcommand{\bZ}{\mathbb{Z}}

\newcommand{\cB}{\mathcal{B}}
\newcommand{\cC}{\mathcal{C}}

\newcommand{\cE}{\mathcal{E}}
\newcommand{\cF}{\mathcal{F}}
\newcommand{\cG}{\mathcal{G}}
\newcommand{\cH}{\mathcal{H}}
\newcommand{\cI}{\mathcal{I}}

\newcommand{\cM}{\mathcal{M}}
\newcommand{\cN}{\mathcal{N}}
\newcommand{\cO}{\mathcal{O}}

\newcommand{\ho}{\widehat{\otimes}}

\newcommand{\OP}[1]{\operatorname{#1}}

\newcommand{\Em}{\hookrightarrow}                   % embedding
\newcommand{\Surj}{\twoheadrightarrow}              % surjection
\newcommand{\Mi}{\stackrel{\sim}{\to}}              % mapping isomorphically
\newcommand{\Mapn}[1]{\stackrel{#1}{\to}}           % mapping with name
\newcommand{\can}{\Utext{can.}}                     % canonical morphism

\newcommand{\Grpmu}{\boldsymbol{\mu}}               % group scheme symbol for \mu
\newcommand{\chr}{\OP{char}}                        % characteristic
\newcommand{\coker}{\OP{coker}}                     % cokernel
\newcommand{\pr}{\OP{pr}}                           % projection
\newcommand{\rank}{\OP{rk}}                         % rank

\newcommand{\Aut}{\OP{Aut}}
\newcommand{\Der}{\OP{Der}}                         % derivations
\newcommand{\Frac}{\OP{Frac}}                       % fractions
\newcommand{\Gal}{\OP{Gal}}                         % Galois groups
\newcommand{\Hom}{\OP{Hom}}
\newcommand{\Id}{\OP{Id}}                           % identity
\newcommand{\im}{\OP{im}}                           % image
\newcommand{\Mor}{\OP{Mor}}
\newcommand{\Spec}{\OP{Spec}}
\newcommand{\Spa}{\OP{Spa}}
\newcommand{\Cont}{\OP{Cont}}

\newcommand{\Spf}{\OP{Spf}}
\newcommand{\GL}{\OP{GL}}
\newcommand{\M}{\mathrm{M}}                         % matrices

\newcommand{\cont}{\Utext{cont}}                    % continuous
\newcommand{\gp}{\Utext{gp}}                        % group
\newcommand{\inv}{\Utext{inv}}                      % invertible elements
\newcommand{\pro}{\Utext{pro-}}                     % pro-
\newcommand{\red}{\Utext{red}}                      % reduced
\newcommand{\tor}{\Utext{tor}}                      % torsion
\newcommand{\tr}{\Utext{tr}}                        % tamely ramified

\newcommand{\fine}{\Utext{fine}}                    % fine monoids
\newcommand{\fs}{\Utext{fs}}                        % fs monoids
\newcommand{\Int}{\Utext{int}}                      % integral monoids
\newcommand{\Sat}{\Utext{sat}}                      % saturated monoids
\newcommand{\Talg}[1]{\langle{#1}\rangle}           % notation for Tate algebra etc

\newcommand{\mono}[1]{e^{#1}}                       % notation for monoid algebra

\newcommand{\an}{\Utext{an}}                        % complex analytification

\newcommand{\Ex}{\wedge}                            % exterior power

\newcommand{\Ab}{\mathrm{Ab}}                       % abelian groups
\newcommand{\Loc}{\mathrm{Loc}}                     % local constant sheaves
\newcommand{\Sh}{\mathrm{Sh}}                       % sheaves

\newcommand{\et}{\Utext{\'et}}                      % etale
\newcommand{\ket}{\Utext{k\'et}}                    % Kummer etale
\newcommand{\fket}{\Utext{fk\'et}}                  % finite Kummer etale
\newcommand{\Fket}{\Utext{Fk\'et}}                  % finite Kummer etale
\newcommand{\Sets}{\Utext{Sets}}                    % sets
\newcommand{\FSets}{\Utext{FSets}}                  % finite sets
\newcommand{\proet}{\Utext{pro\'et}}                % pro-etale
\newcommand{\PFSets}{\Utext{PFSets}}                % profinite sets
\newcommand{\proket}{\Utext{prok\'et}}              % pro-Kummer etale
\newcommand{\profket}{\Utext{profk\'et}}            % pro-finite Kummer etale

\newcommand{\AC}[1]{\overline{#1}}                  % algebraic closure or geometric point
\newcommand{\sep}{\Utext{sep}}                      % separable closure

\newcommand{\unip}{\Utext{\textit{u}}}              % unipotent
\newcommand{\qunip}{\Utext{\textit{qu}}}            % quasi-unipotent

\newcommand{\Utext}[1]{\text{\rm #1}}               % upright math text

\newcommand{\Refenum}[1]{\Pth{\textrm{#1}}}
\newcommand{\Refeq}[1]{\Pth{#1}}

\newcommand{\Pth}[1]{{\rm (}#1{\rm )}}              % parenthesis ( )
\newcommand{\Qtn}[1]{``#1''}                        % quotation `` ''

\newcommand{\parenthesis}[1]{\Pth{#1}}              % parenthesis (old)

\newcommand{\apages}{pp.\@\xspace}                  % page

\newcommand{\aCh}{Ch.\@\xspace}                     % chapter
\newcommand{\aSec}{Sec.\@\xspace}                   % section
\newcommand{\aSecs}{Sec.\@\xspace}                  % sections
\newcommand{\aLec}{Lec.\@\xspace}                   % section
\newcommand{\aLecs}{Lec.\@\xspace}                  % sections

\newcommand{\aDef}{Def.\@\xspace}                   % definition
\newcommand{\aLem}{Lem.\@\xspace}                   % lemma
\newcommand{\aProp}{Prop.\@\xspace}                 % proposition
\newcommand{\aThm}{Thm.\@\xspace}                   % theorem
\newcommand{\aThms}{Thm.\@\xspace}                  % theorems
\newcommand{\aCor}{Cor.\@\xspace}                   % corollary
\newcommand{\aRem}{Rem.\@\xspace}                   % remark
\newcommand{\aEx}{Ex.\@\xspace}                     % example
\newcommand{\resp}{resp.\@\xspace}                  % resp.
\newcommand{\ie}{i.e.\@\xspace}                     % i.e.
\newcommand{\eg}{e.g.\@\xspace}                     % e.g.
\newcommand{\etc}{etc\xspace}                       % etc.

\newcommand{\Refcf}{cf.\@\xspace}                   % 'compare' for the purpose of reference

\numberwithin{equation}{subsection}

\begin{document}

\begin{abstract}
    We develop a theory of log adic spaces by combining the theories of adic spaces and log schemes, and study the Kummer \'etale and pro-Kummer \'etale topology for such spaces.  We also establish the primitive comparison theorem in this context, and deduce from it some related cohomological finiteness or vanishing results.
\end{abstract}

\maketitle

\tableofcontents

\section{Introduction}

There are two main goals of this paper.  Firstly, we would like to adapt many fundamental notions and features of the theory of log geometry for schemes, as in \cite{Kato:1989-lsfi, Kato:2021-lsfi-2, Kato:1989-lddt, Illusie:2002-fknle, Ogus:2018-LLG}, to the theory of adic spaces, as in \cite{Huber:1994-gfsra, Huber:1996-ERA}.  For example, we would like to introduce the notion of log adic spaces, which allow us to study the de Rham and \'etale cohomology of nonproper adic spaces by introducing the log de Rham and Kummer \'etale cohomology of proper adic spaces equipped with suitable log structures.  Secondly, we would like to adapt many foundational techniques in recent developments of $p$-adic geometry, as in \cite{Scholze:2012-ps, Kedlaya/Liu:2015-RPH, Scholze:2013-phtra, Scholze:2016-phtra-corr, Scholze/Weinstein:2020-BLG}, to the context of log geometry.  For example, we would like to introduce the pro-Kummer \'etale site, and show that log affinoid perfectoid objects form a basis for such a site, under suitable assumptions.  In particular, we would like to establish the \emph{primitive comparison theorem} and some related cohomological finiteness or vanishing results in this context.

Although a general formalism of log topoi has been introduced in \cite[\aSec 12.1]{Gabber/Ramero:2019-fart}, there are nevertheless several special features \Pth{such as the integral structure sheaves} or pathological issues \Pth{such as the lack of fiber products in general, or the necessary lack of noetherian property when working with perfectoid spaces} in the theory of adic spaces, which resulted in some complications in our adaption of many \Qtn{well-known arguments}; and we have chosen to spell out the modifications of such arguments in some detail, for the sake of clarity.  Moreover, this paper is intended to serve as the foundation for our development of a $p$-adic analogue of the Riemann--Hilbert correspondence in \cite{Diao/Lan/Liu/Zhu:lrhrv} \Pth{and forthcoming works such as \cite{Lan/Liu/Zhu:dcpdr}}.  Therefore, in addition to the above-mentioned goals, we have also included some foundational treatment of quasi-unipotent nearby cycles, following \Pth{and reformulating} Beilinson's ideas in \cite{Beilinson:1987-hgps}.

Here is an outline of this paper.

In Section \ref{sec-log-adic}, we introduce log adic spaces and study their basic properties.  In Section \ref{sec-monoid}, we review some basic terminologies of monoids.  In Section \ref{sec-log-adic-sp}, we introduce the definition and some basic notions of log adic spaces, and study some important examples.  In Section \ref{sec-chart}, we study the important notion of \emph{charts} in the context of log adic spaces, which are useful for defining the categories of coherent, fine, and fs log adic spaces, and for constructing fiber products in them.

In Section \ref{sec-log-sm-log-diff}, we study log smooth morphisms of log adic spaces, and their associated sheaves of log differentials.  In Section \ref{sec-log-sm}, we introduce the notion of log smooth and log \'etale morphisms, and show the existence of smooth toric charts for smooth fs log adic spaces.  In Sections \ref{sec-log-diff} and \ref{sec-log-diff-sh}, we develop a theory of log differentials for homomorphisms of log Huber rings and morphisms of coherent log adic spaces, and compare it with the theory in Section \ref{sec-log-sm}.

In Section \ref{sec-ket}, we study the Kummer \'etale topology of locally noetherian fs log adic spaces.  In Section \ref{sec-ket-site}, we introduce the Kummer \'etale site and study its basic properties.  In Section \ref{sec-Abhyankar}, we establish an analogue of \emph{Abhyankar's lemma} for rigid analytic varieties, and record some related general facts.  In Section \ref{sec-ket-coh-descent}, we study the structure sheaves and analytic coherent sheaves on the Kummer \'etale site, and show that their higher cohomology vanishes on affinoids.  In Section \ref{sec-ket-cov-descent}, we show that Kummer \'etale surjective morphisms satisfy effective descent in the category of finite Kummer \'etale covers, and define Kummer \'etale fundamental groups with desired properties.  In Section \ref{sec-ket-bc}, we study certain direct and inverse images of abelian sheaves on Kummer \'etale sites.  In Section \ref{sec-purity}, we establish some purity results for torsion Kummer \'etale local systems.

In Section \ref{sec-proket}, we study the pro-Kummer \'etale topology of locally noetherian fs log adic spaces.  In Section \ref{sec-proket-site}, based on the theory in Section \ref{sec-ket}, we introduce the pro-Kummer \'etale site, by following Scholze's ideas in \cite{Scholze:2013-phtra} and \cite{Scholze:2016-phtra-corr}.  In Section \ref{sec-proket-bc}, we study certain direct and inverse images of abelian sheaves on pro-Kummer \'etale sites.  In Section \ref{sec-log-aff-perf}, we introduce the log affinoid perfectoid objects, and show that they form a basis for the pro-Kummer \'etale topology, for locally noetherian fs log adic spaces over $\Spa(\bQ_p, \bZ_p)$.  In Section \ref{sec-proket-sheaves}, we introduce the completed structure sheaves and their integral and tilted variants on the pro-Kummer \'etale site, and prove various almost vanishing results for them.

In Section \ref{sec-loc-syst}, we study the Kummer \'etale cohomology of fs log adic spaces log smooth over a nonarchimedean base field $k$.  In Section \ref{sec-toric-chart}, we start with some preparations using the log affinoid perfectoid objects defined by towers over some associated toric charts.  In Section \ref{sec-thm-prim-comp}, we establish the \emph{primitive comparison theorem}, generalizing the strategy in \cite[\aSec 5]{Scholze:2013-phtra}, and deduce from it some finiteness results for the cohomology of torsion Kummer \'etale local systems.  In Section \ref{sec-lisse}, we introduce the notions of $\bZ_p$-, $\bQ_p$-, $\widehat{\bZ}_p$-, and $\widehat{\bQ}_p$-local systems, and record some finiteness results.  In Section \ref{sec-nearby}, as an application of the theory thus developed, we reformulate Beilinson's ideas in \cite{Beilinson:1987-hgps} and define the unipotent and quasi-unipotent nearby cycles in the rigid analytic setting.

In Appendix \ref{app-Kiehl}, we state a version of Tate's sheaf property and Kiehl's gluing property for the analytic and \'etale sites of adic spaces that are either locally noetherian or analytic stably adic.  This includes, in particular, a proof of Kiehl's property for coherent sheaves on \Pth{possibly nonanalytic} noetherian adic spaces which \Pth{as far as we know} is not yet available in the literature.

\subsection*{Notation and conventions}\label{sec-notation-conv}

By default, all monoids are assumed to be commutative, and the monoid operations are written additively \Pth{rather than multiplicatively}, unless otherwise specified.  For a monoid $P$, let $P^\gp$ denote its group completion.  For any commutative ring $R$ with unit and any monoid $P$, we denote by $R[P]$ the monoid algebra over $R$ associated with $P$.  The image of $a \in P$ in $R[P]$ will often be denoted by $\mono{a}$.  Then we have $\mono{a + b} = \mono{a} \cdot \mono{b}$ in $R[P]$, for all $a, b \in P$.

Group cohomology will always mean continuous group cohomology.

For each site $\cC$, the category of sheaves \Pth{\resp abelian sheaves} on $\cC$ is denoted by $\Sh(\cC)$ \Pth{\resp $\Sh_\Ab(\cC)$}, although the associated topos is denoted by $\cC^\sim$.

We shall follow \cite[\aLecs 2--7]{Scholze/Weinstein:2020-BLG} for the general definitions and results of Huber rings and pairs, adic spaces, and perfectoid spaces.  Unless otherwise specified, all Huber rings and pairs will be assumed to be complete.

We say that an adic space is \emph{locally noetherian} if it is locally isomorphic to $\Spa(R, R^+)$, where either $R$ is analytic \Pth{see \cite[\aRem 2.2.7 and \aProp 4.3.1]{Scholze/Weinstein:2020-BLG}} and \emph{strongly noetherian}---\ie, the rings
\[
    R\Talg{T_1, \ldots, T_n} = \Big\{ \sum_{i_1, \ldots, i_n \geq 0} a_{i_1, \ldots, i_n} T_1^{i_1} \cdots T_n^{i_n} \in R[[T_1, \ldots, T_n]] : a_{i_1, \ldots, i_n} \to 0 \Big\}
\]
are noetherian, for all $n \geq 0$; or $R$ is \Pth{complete, by our convention on Huber pairs, and} finitely generated over a noetherian ring of definition.  We say that an adic space is \emph{noetherian} if it is locally noetherian and \emph{qcqs} \Pth{\ie, quasi-compact and quasi-separated}.

We shall follow \cite[\aDef 1.2.1]{Huber:1996-ERA} for the definition for morphisms of locally noetherian adic spaces to be \emph{locally of finite type} \Pth{\emph{lft} for short}.  A useful fact is that a fiber product $Y \times_X Z$ of locally noetherian adic spaces exists when the first morphism $Y \to X$ is lft, in which case its base change \Pth{\ie, the second projection} $Y \times_X Z \to Z$ is also lft \Pth{see \cite[(1.1.1), \aProp 1.2.2, and \aCor 1.2.3]{Huber:1996-ERA}}.

An \emph{affinoid field} $(k, k^+)$ is a Huber pair in which $k$ is a \Pth{possibly trivial} nonarchimedean local field \Pth{\ie, a field complete with respect to a nonarchimedean multiplicative norm $| \cdot |: k \to \bR_{\geq 0}$}, and $k^+$ is an open valuation subring of $\cO_k := \{ x \in k : |x| \leq 1 \}$ \Pth{see \cite[\aDef 4.2.4]{Scholze/Weinstein:2020-BLG}}.  When $k$ is a nontrivial nonarchimedean field \Pth{\ie, a field that is complete with respect to a nontrivial nonarchimedean multiplicative norm}, we shall regard rigid analytic varieties over $k$ as adic spaces over $(k, \cO_k)$, by virtue of \cite[(1.1.11)]{Huber:1996-ERA}.

We shall follow \cite[\aSecs 1.6 and 1.7]{Huber:1996-ERA} for the definition and basic properties of unramified, smooth, and \'etale morphisms of locally noetherian adic spaces.  More generally, without the locally noetherian hypothesis, we say that a homomorphism $(R, R^+) \to (S, S^+)$ of Huber pairs is \emph{finite \'etale} if $R \to S$ is finite \'etale as a ring homomorphism, and if $S^+$ is the integral closure of $R^+$ in $S$.  We say that a morphism $f: Y \to X$ of adic spaces is \emph{finite \'etale} if, for each $x \in X$, there exists an open affinoid neighborhood $U$ of $x$ in $X$ such that $V = f^{-1}(U)$ is affinoid, and if the induced homomorphism of Huber pairs $\bigl(\cO_X(U), \cO_X^+(U)\bigr) \to \bigl(\cO_Y(V), \cO_Y^+(V)\bigr)$ is finite \'etale.  We say that a morphism $f: Y \to X$ of adic spaces is \emph{\'etale} if, for each $y \in Y$, there exists open neighborhood $V$ of $y$ in $Y$ such that the restriction of $f$ to $V$ factors as the composition of an open immersion, a finite \'etale morphism, and another open immersion.

Given any adic space $X$, we denote by $X_\et$ the category of adic spaces \'etale over $X$.  If fiber products exist in $X_\et$, then $X_\et$ acquires a natural structure of a site.  We say that $X$ is \emph{\'etale sheafy} if $X_\et$ is a site and if the \'etale structure presheaf $\cO_{X_\et}: U \mapsto \cO_U(U)$ is a sheaf.  \'Etale sheafiness is known when $X$ is either locally noetherian or a perfectoid space---see Appendix \ref{app-Kiehl} for more information.

A \emph{geometric point} of an adic space $X$ is a morphism $\eta: \xi = \Spa(l, l^+) \to X$, where $l$ is a separably closed nonarchimedean field.  For simplicity, we shall write $\xi \to X$, or even $\xi$, when the context is clear.  The image of the unique closed point $\xi_0$ of $\xi$ under $\eta: \xi \to X$ is called the \emph{support} of $\xi$.  Given any $x \in X$, we have a geometric point $\AC{x} = \Spa(\AC{\kappa}(x), \AC{\kappa}(x)^+)$ above $x$ \Pth{\ie, $x$ is the support of $\AC{x}$}, as in \cite[(2.5.2)]{Huber:1996-ERA}, where $\AC{\kappa}(x)$ is the completion of a separable closure of the residue field $\kappa(x)$ of $\cO_{X, x}$.  An \emph{\'etale neighborhood} of $\eta$ is a lifting of $\eta$ to a composition $\xi \to U \Mapn{\phi} X$ in which $\phi$ is \'etale.  For any sheaf $\cF$ on $X_\et$, the \emph{stalk} of $\cF$ at $\eta$ is $\cF_\xi := \Gamma\bigl(\xi, \eta^{-1}(\cF)\bigr) \cong \varinjlim \cF(V)$, where the direct limit runs through all \'etale neighborhoods $V$ of $\xi$.  \Pth{Recall that, by \cite[\aProp 2.5.5]{Huber:1996-ERA}, when $X$ is locally noetherian, geometric points form a conservative family for $X_\et$.}

An adic space $X = \Spa(R, R^+)$ is \emph{strictly local} if $R$ is a strictly local ring and if $X$ contains a unique closed point $x$ such that the support of the valuation $| \cdot(x) |$ is the maximal ideal of $R$.  We shall denote by $X(\xi) = \Spa(\cO_{X, \xi}, \cO_{X, \xi}^+)$ the strict localization of a geometric point $\xi \to X$ of a locally noetherian adic space $X$, as in \cite[(2.5.9) and \aLem 2.5.10]{Huber:1996-ERA}.  By the explicit description of the completion of $(\cO_{X, \xi}, \cO_{X, \xi}^+)$ as in \cite[\aProp 2.5.13]{Huber:1996-ERA}, $X(\xi)$ is a noetherian adic space, which is canonically isomorphic to $\xi$ when the support of $\xi$ is analytic.

As for almost mathematics, we shall adopt the following notation and conventions.  We shall denote by $M^a$ the almost module associated with a usual module $M$, depending on the context.  For usual modules $M$ and $N$, we shall say \Qtn{there is an \emph{almost isomorphism} $M \Mi N$} when there is an isomorphism $M^a \Mi N^a$ between the associated almost modules.  We shall write interchangeably both \Qtn{$M^a = 0$} and \Qtn{$M$ is almost zero}, with exactly the same meaning.

\subsection*{Acknowledgements}

This paper was initially based on a paper written by the first author, and we would like to thank Christian Johansson, Kiran Kedlaya, Teruhisa Koshikawa, Martin Olsson, Fucheng Tan, and Jilong Tong for helpful correspondences and conversations during the preparation of that paper.  We would also like to thank David Sherman and an anonymous referee for their careful reading and many helpful corrections, questions, and suggestions.  Moreover, we would like to thank the Beijing International Center for Mathematical Research and the California Institute of Technology for their hospitality.

\section{Log adic spaces}\label{sec-log-adic}

\subsection{Recollection on monoids}\label{sec-monoid}

In this subsection, we recollect some basics in the theory of monoids.  This is mainly to introduce the terminologies and fix the notation.  For more details, we refer the readers to \cite{Ogus:2018-LLG}.

\begin{defn}\phantomsection\label{def-monoid}
    \begin{enumerate}
        \item\label{def-monoid-fg}  A monoid $P$ is called \emph{finitely generated} if there exists a surjective homomorphism $\bZ_{\geq 0}^n \Surj P$ for some $n$.

        \item\label{def-monoid-int}  A monoid $P$ is called \emph{integral} if the natural homomorphism $P \to P^\gp$ is injective.

        \item\label{def-monoid-fine}  A monoid $P$ is called \emph{fine} if it is integral and finitely generated.

        \item\label{def-monoid-sat}  A monoid $P$ is called \emph{saturated} if it is integral and, for every $a \in P^\gp$ such that $n a \in P$ for some integer $n \geq 1$, we have $a \in P$.  A monoid that is both fine and saturated is called an \emph{fs monoid}.

        \item\label{def-monoid-sharp}  For any monoid $P$, we denote by $P^\inv$ the subgroup of invertible elements in $P$, and write $\overline{P} := P / P^\inv$.  A monoid $P$ is called \emph{sharp} if $P^\inv = \{ 0 \}$.

        \item\label{def-monoid-tor}  An sharp fs monoid is called a \emph{toric monoid}.
    \end{enumerate}
\end{defn}

\begin{rk}\label{rem-monoid-cat}
    Arbitrary direct and inverse limits exist in the category of monoids \Pth{see \cite[\aSec I.1.1]{Ogus:2018-LLG}}.  In particular, for a homomorphism of monoids $u: P \to Q$, we have $\ker(u) = u^{-1}(0)$, and $\coker(u)$ is determined by the conditions that $Q \to \coker(u)$ is surjective and that two elements $q_1, q_2 \in Q$ have the same image in $\coker(u)$ if and only if there exist $p_1, p_2 \in P$ such that $u(p_1) + q_1 = u(p_2) + q_2$.  In general, the induced map $P / \ker(u) \to \im(u)$ is surjective, but not necessarily injective.  \Pth{For a typical example, consider the homomorphism $u: \bZ_{\geq 0}^2 \to \bZ_{\geq 0}: (x_1, x_2) \mapsto x_1 + x_2$.  Then $\ker(u) = 0$ but $u$ is not injective.}  Therefore, the category of monoids is not abelian.  Nevertheless, if $P$ is a submonoid of $Q$, and if $u: P \to Q$ is the canonical inclusion, then we shall denote $\coker(u)$ by $Q / P$.  Note that $Q / P$ can be zero even when $P \neq Q$.
\end{rk}

\begin{rk}\label{rem-fg}
    It is not hard to show that a monoid $P$ is finitely generated if and only if $P^\inv$ is finitely generated \Pth{as a group} and $\overline{P} = P / P^\inv$ is finitely generated \Pth{as a monoid}.  \Pth{See \cite[\aProp I.2.1.1]{Ogus:2018-LLG}.}  A deeper fact is that a finitely generated \Pth{commutative} monoid $P$ is always \emph{finitely presented}; \ie, it is the coequalizer of some homomorphisms $\bZ_{\geq 0}^m \rightrightarrows \bZ_{\geq 0}^n$, for some $m, n$.  \Pth{See \cite[\aThm I.2.1.7]{Ogus:2018-LLG}.}  As a result, if $Q = \varinjlim_{i \in I} Q_i$ is a filtered direct limit of monoids, then any injective map $P \to Q$ lifts to an injective $P \to Q_i$, for some $i \in I$.  \Pth{The opposite assertion, that any surjective map $Q \to P$ lifts to a surjective $Q_i \to P$, for some $i \in I$, only requires the finite generation of $P$.}
\end{rk}

\begin{defn}\label{def-amalg-sum}
    Given any two homomorphisms of monoids $u_1: P \to Q_1$ and $u_2: P \to Q_2$, the \emph{amalgamated sum} $Q_1 \oplus_P Q_2$ is the coequalizer of $P \rightrightarrows Q_1 \oplus Q_2$, with the two homomorphisms given by $(u_1, 0)$ and $(0, u_2)$, respectively.
\end{defn}

\begin{lem}\label{lem-amalg-sum}
    In Definition \ref{def-amalg-sum}, suppose moreover that any of $P$, $Q_1$, or $Q_2$ is a group.  Then the natural map $Q_1 / P \to (Q_1 \oplus_P Q_2) / Q_2$ is an isomorphism.
\end{lem}
\begin{proof}
    The surjectivity is clear.  As for the injectivity, by assumption and by \cite[\aProp I.1.1.5]{Ogus:2018-LLG}, two elements $(q_1, q_2), (q_1', q_2') \in Q_1 \oplus Q_2$ have the same image in $Q_1 \oplus_P Q_2$ if and only if there exist $a, b \in P$ such that $q_1 + u_1(a) = q_1' + u_1(b)$ and $q_2 + u_2(b) = q_2' + u_2(a)$.  Therefore, for $q_1, q_1' \in Q_1$, if they have the same image in $(Q_1 \oplus_P Q_2) / Q_2$---\ie, there exist $q_2, q_2' \in Q_2$ such that $(q_1, q_2)$ and $(q_1', q_2')$ have the same image in $Q_1 \oplus_P Q_2$---then there exist $a, b \in P$ such that $q_1 + u_1(a) = q_1' + u_1(b)$.  Thus, $q_1$ and $q_1'$ have the same image in $Q_1 / P$.
\end{proof}

\begin{defn}\label{def-int-sat}
    For any monoid $P$, let $P^\Int$ denote the image of the canonical homomorphism $P \to P^\gp$.  For any integral monoid $P$, let
    \[
        P^\Sat := \{ a \in P^\gp : \Utext{$n a \in P$, for some $n \geq 1$} \}.
    \]
    For a general monoid $P$ not necessarily integral, we write $P^\Sat$ for $(P^\Int)^\Sat$.
\end{defn}

\begin{rk}\label{rem-int-sat}
    The functor $P \mapsto P^\Int$ is the left adjoint of the inclusion from the category of integral monoids into the category of all monoids.  Similarly, $P \to P^\Sat$ is the left adjoint of the inclusion from the category of saturated monoids into the category of integral monoids.
\end{rk}

\begin{lem}\label{lem-ama-int}
    Let $P \to Q_1$ and $P \to Q_2$ be homomorphisms of monoids.  Then $(Q_1 \oplus_P Q_2)^\Int$ can be naturally identified with the image of $Q_1 \oplus_P Q_2$ in $Q^\gp_1 \oplus_{P^\gp} Q_2^\gp$.  Moreover, if $P$, $Q_1$, and $Q_2$ are integral and if any of these monoids is a group, then $Q_1 \oplus_P Q_2$ is also integral.
\end{lem}
\begin{proof}
    See \cite[\aProp I.1.3.4]{Ogus:2018-LLG}.
\end{proof}

\begin{lem}\label{lem-int-sat-quot}
    The quotient of an integral \Pth{\resp a saturated} monoid by a submonoid is also integral \Pth{\resp saturated}.  In particular, for any fs monoid $P$, the quotient $\overline{P} = P / P^\inv$ is a toric monoid.
\end{lem}
\begin{proof}
    Let $Q$ be any submonoid of an integral monoid $P$.  By \cite[\aProp I.1.3.3]{Ogus:2018-LLG}, $P / Q$ is also integral.  Suppose moreover that $P$ is saturated.  For any $\overline{a} \in (P / Q)^\Sat$, by definition, there exists some $n \geq 1$ such that $n \overline{a} \in P / Q$.  That is, there exist $b \in P$ and $q_1, q_2 \in Q$ such that $n a = b + (q_1 - q_2)$ in $P^\gp$.  Then $n (a + q_2) = b + q_1 + (n - 1) q_2$, and hence $a + q_2 \in P$ and $\overline{a} \in P / Q$.  Thus, $P / Q = (P / Q)^\Sat$ is also saturated.
\end{proof}

\begin{lem}\label{lem-monoid-split}
    Let $P$ be an integral monoid, and $u: P \to Q$ a surjective homomorphism onto a toric monoid $Q$.  Suppose that $\ker(u^\gp) \subset P$.  Then $u$ admits a \Pth{noncanonical} section.  In particular, for any fs monoid $P$, the canonical homomorphism $P \to \overline{P}$ admits a \Pth{noncanonical} section.
\end{lem}
\begin{proof}
    For $a \in Q^\gp$, if $n a = 0$ for some $n \geq 1$, then $a = 0$, as $Q$ is saturated and sharp.  Hence, $Q^\gp$ is torsion-free, $Q^\gp \cong \bZ^r$ for some $r$, and the projection $u^\gp: P^\gp \to Q^\gp$ admits a section $s: Q^\gp \to P^\gp$.  It remains to show that $s(Q) \subset P$.  For each $q \in Q$, choose any $\widetilde{q} \in P$ lifting $q$.  Then $s(q) - \widetilde{q} \in P^\gp$ lies in $\ker(u^\gp) \subset P$, and therefore $s(q) \in \widetilde{q} + P \subset P$, as desired.
\end{proof}

\begin{constr}\label{constr-monoid-loc}
    Let $P$ be a monoid, and $S$ a subset of $P$.  There exists a monoid $S^{-1} P$ together with a homomorphism $\lambda: P \to S^{-1} P$ sending elements of $S$ to invertible elements of $S^{-1} P$ satisfying the universal property that any homomorphism of monoids $u: P \to Q$ with the property that $u(S) \subset Q^\inv$ uniquely factors through $S^{-1} P$.  The monoid $S^{-1} P$ is called the \emph{localization of $P$ with respect to $S$}.  Concretely, let $T$ denote the submonoid of $P$ generated by $S$.  Then, as a set, $S^{-1} P$ consists of equivalence classes of pairs $(a, t) \in P \times T$, where two such pairs $(a, t)$ and $(a', t')$ are equivalent if there exists some $t'' \in T$ such that $a + t' + t'' = a' + t + t''$.  The monoid structure of this set is given by $(a, t) + (a', t') = (a + a', t + t')$.  The homomorphism $\lambda$ is given by $\lambda(a) = (a, 0)$.
\end{constr}

\begin{rk}\label{rem-int-sat-loc}
    The localization of an integral \Pth{\resp saturated} monoid is still integral \Pth{\resp saturated}.
\end{rk}

\begin{rk}\label{rem-ama-loc}
    Let $P \to Q_1$ and $P \to Q_2$ be homomorphisms of monoids, and let $S$ be a subset of $P$.  Let $S_1$, $S_2$, and $S_3$ denote the images of $S$ in $Q_1$, $Q_2$, and $Q_1 \oplus_P Q_2$, respectively.  Then the natural homomorphism
    \[
        Q_1 \oplus_P Q_2 \to (S_1^{-1} Q_1) \oplus_{S^{-1} P} (S_1^{-1} Q_2)
    \]
    factors through an isomorphism
    \[
        S_3^{-1}(Q_1 \oplus_P Q_2) \Mi (S_1^{-1} Q_1) \oplus_{S^{-1} P} (S_2^{-1} Q_2),
    \]
    by the universal properties of the objects.
\end{rk}

\begin{defn}\label{def-monoid-hom}
    Let $u: P \to Q$ be a homomorphism of monoids.
    \begin{enumerate}
        \item We say it is \emph{local} if $P^\inv = u^{-1}(Q^\inv)$.

        \item We say it is \emph{sharp} if the induced homomorphism $P^\inv \to Q^\inv$ is an isomorphism.

        \item We say it is \emph{strict} if the induced homomorphism $\overline{P} \to \overline{Q}$ is an isomorphism.

        \item We say it is \emph{exact} if the induced homomorphism $P \to P^\gp \times_{Q^\gp} Q$ is an isomorphism.  \Pth{When $P$ and $Q$ are integral and canonically identified as submonoids of $P^\gp$ and $Q^\gp$, respectively, we simply need $P = (u^\gp)^{-1}(Q)$.}
    \end{enumerate}
\end{defn}

\subsection{Log adic spaces}\label{sec-log-adic-sp}

In this subsection, we give the definition of log adic spaces, introduce some basic notions, and study some important examples.

\begin{conv}\label{conv-et-sheafy}
    From now on, we shall only work with adic spaces that are \emph{\'etale sheafy}.  \Pth{They include locally noetherian adic spaces and perfectoid spaces.}
\end{conv}

\begin{defn}\label{def-log-str}
    Let $X$ be an \Pth{\'etale sheafy} adic space.
    \begin{enumerate}
        \item A \emph{pre-log structure} on $X$ is a pair $(\cM_X, \alpha)$, where $\cM_X$ is a sheaf of monoids on $X_\et$ and $\alpha: \cM_X \to \cO_{X_\et}$ is a morphism of sheaves of monoids, called the structure morphism.  \Pth{Here $\cO_{X_\et}$ is equipped with the natural multiplicative monoid structure.}

        \item Let $(\cM, \alpha)$ and $(\cN, \beta)$ be pre-log structures on $X$.  A morphism from $(\cM, \alpha)$ to $(\cN, \beta)$ is a morphism $\cM \to \cN$ of sheaves of monoids that is compatible with the structure morphisms $\alpha$ and $\beta$.

        \item A pre-log structure $(\cM_X, \alpha)$ on $X$ is called a \emph{log structure} if the morphism $\alpha^{-1}(\cO_{X_\et}^\times) \to \cO_{X_\et}^\times$ induced by $\alpha$ is an isomorphism.  In this case, we call the triple $(X, \cM_X, \alpha)$ a \emph{log adic space}.  We shall simply write $(X, \cM_X)$ or $X$ when the context is clear.

        \item We say that a sheaf of monoids $\cM$ on $X_\et$ is \emph{integral} \Pth{\resp \emph{saturated}} if it is a sheaf of integral \Pth{\resp saturated} monoids.  A pre-log structure $(\cM_X, \alpha)$ on $X$ is called \emph{integral} \Pth{\resp \emph{saturated}} if $\cM_X$ is.  We say that a log adic space $(X, \cM_X, \alpha)$ is \emph{integral} \Pth{\resp \emph{saturated}} if $\cM_X$ is.

        \item For a log structure $(\cM_X, \alpha)$ on $X$, we set $\overline{\cM}_X := \cM_X / \alpha^{-1}(\cO_{X_\et}^\times)$, called the \emph{characteristic} of the log structure.

        \item For a pre-log structure $(\cM_X, \alpha)$ on $X$, we have the \emph{associated log structure} $({^a}\cM_X, {^a}\alpha)$, where ${^a}\cM_X$ is the pushout of $\cO_{X_\et}^\times \leftarrow \alpha^{-1}(\cO_{X_\et}^\times) \to \cM_X$ in the category of sheaves of monoids on $X_\et$, and where ${^a}\alpha: {^a}\cM_X \to \cO_{X_\et}$ is canonically induced by the natural morphism $\cO_{X_\et}^\times \to \cO_{X_\et}$ and the structure morphism $\alpha: \cM_X \to \cO_{X_\et}$ \Pth{\Refcf{} \cite[\aSec 12.1.6]{Gabber/Ramero:2019-fart}}.  Again, we shall simply write ${^a}\cM_X$ when the context is clear.

        \item A morphism $f: (Y, \cM_Y, \alpha_Y) \to (X, \cM_X, \alpha_X)$ of log adic spaces is a morphism $f: Y \to X$ of adic spaces together with a morphism of sheaves of monoids $f^\sharp: f^{-1}(\cM_X) \to \cM_Y$ compatible with $f^\sharp: f^{-1}(\cO_{X_\et}) \to \cO_{Y_\et}$, $f^{-1}(\alpha_X): f^{-1}(\cM_X) \to f^{-1}(\cO_{X_\et})$, and $\alpha_Y: \cM_Y \to \cO_{Y_\et}$.  In this case, we have the log structure $f^*(\cM_X)$ on $Y$ associated with the pre-log structure $f^{-1}(\cM_X) \to f^{-1}(\cO_{X_\et}) \to \cO_{Y_\et}$.  The morphism $f$ is called \emph{strict} if the induced morphism $f^*(\cM_X) \to \cM_Y$ is an isomorphism.

        \item A morphism $f: (Y, \cM_Y) \to (X, \cM_X)$ is called \emph{exact} if, at each geometric $\AC{y}$ of $Y$, the induced homomorphism $\bigl(f^*(\cM_X)\bigr)_{\AC{y}} \to \cM_{Y, {\AC{y}}}$ is exact.

        \item A log adic space is called \emph{noetherian} \Pth{\resp \emph{locally noetherian}, \resp \emph{quasi-compact}, \resp \emph{quasi-separated}, \resp \emph{affinoid}, \resp \emph{perfectoid}, \resp \emph{stably uniform}, \resp \emph{analytic}} if its underlying adic space is.

        \item A morphism of log adic spaces is called \emph{lft} \Pth{\resp \emph{quasi-compact}, \resp \emph{quasi-separated}, \resp \emph{separated}, \resp \emph{proper}, \resp \emph{finite}, \resp \emph{surjective}} if the underlying morphism of adic spaces is.  As usual, a separated \Pth{\resp proper} log adic space over $\Spa(k, k^+)$ is a locally noetherian log adic space with a separated \Pth{\resp proper} structure morphism to $\Spa(k, k^+)$.
    \end{enumerate}
\end{defn}

\begin{rk}\label{rem-log-str-assoc}
    As explained in \cite[\aSec 12.1.6]{Gabber/Ramero:2019-fart}, the functor of taking associated log structures from the category of pre-log structures to the category of log structures on $X$ is the left adjoint of the natural inclusion functor from the category of log structures to the category of pre-log structures on $X$.
\end{rk}

\begin{lem}\label{lem-int-sat-stalk}
    A sheaf of monoids $\cM$ on an adic space $X_\et$ is integral \Pth{\resp saturated} if and only if $\cM_{\AC{x}}$ is integral \Pth{\resp saturated} at each geometric point $\AC{x}$ of $X$.  In particular, a log adic space $(X, \cM_X, \alpha)$ is integral \Pth{\resp saturated} if and only if $\cM_{X, \AC{x}}$ is integral \Pth{\resp saturated} at each geometric point $\AC{x}$ of $X$.
\end{lem}
\begin{proof}
    This follows from the proof of \cite[\aLem 12.1.18(ii)]{Gabber/Ramero:2019-fart}.
\end{proof}

\begin{rk}\label{rem-stalk-sharp}
    For a log adic space $(X, \cM_X, \alpha)$ and a geometric point $\AC{x}$ of $X$, it follows that $\cM_{X, \AC{x}}^\inv = \alpha^{-1}(\cO_{X_\et, \AC{x}}^\times) \Mi \cO_{X_\et, \AC{x}}^\times$ \Pth{\ie, the homomorphism $\cM_{X, \AC{x}} \to \cO_{X, \AC{x}}$ is local and sharp}.  Hence, $\overline{\cM}_{X, \AC{x}} \cong \cM_{X, \AC{x}} / \alpha^{-1}(\cO_{X_\et, \AC{x}}^\times)$ is a sharp monoid.  If $\cM_{X, \AC{x}}$ is integral \Pth{which is the case when $(X, \cM_X, \alpha)$ is integral, by Lemma \ref{lem-int-sat-stalk}}, then $\overline{\cM}_{X,\AC{x}}^\gp \cong \cM_{X, \AC{x}}^\gp / \alpha^{-1}(\cO_{X_\et, \AC{x}}^\times)$, and $\cM_{X, \AC{x}} \to \overline{\cM}_{X, \AC{x}}$ is exact.
\end{rk}

\begin{rk}\label{rem-stalk-strict}
    Let $f: (Y, \cM_Y, \alpha_Y) \to (X, \cM_X, \alpha_X)$ be a morphism of log adic spaces.  At each geometric point $\AC{y}$ of $Y$, since $f^\sharp_{\AC{y}}: \cO_{X_\et, f(\AC{y})} \to \cO_{Y_\et, \AC{y}}$ is local, and since $f^\sharp_{\AC{y}}: \cM_{X, f(\AC{y})} \cong f^{-1}(\cM_{X})_{\AC{y}} \to \cM_{Y, \AC{y}}$ is by definition compatible with $f^\sharp_{\AC{y}}: \cO_{X_\et, f(\AC{y})} \cong f^{-1}(\cO_{X_\et})_{\AC{y}} \to \cO_{Y_\et, \AC{y}}$, we see that $f^\sharp_{\AC{y}}: \cM_{X, f(\AC{y})} \to \cM_{Y, \AC{y}}$ is local as in Definition \ref{def-monoid-hom}.  By Lemma \ref{lem-amalg-sum}, $\bigl(\overline{f^*(\cM_X)}\bigr)_{\AC{y}} \cong \overline{\cM}_{X, f(\AC{y})}$.  Therefore, $f$ is strict if and only if $\overline{\cM}_{X, f(\AC{y})} \Mi \overline{\cM}_{Y, \AC{y}}$, \ie, $f^\sharp_{\AC{y}}: \cM_{X, f(\AC{y})} \to \cM_{Y, \AC{y}}$ is strict, at each geometric point $\AC{y}$ of $Y$.
\end{rk}

Here are some basic examples of log adic spaces.

\begin{exam}\label{ex-log-adic-sp-triv}
    Every \Pth{\'etale sheafy} adic space $X$ has a natural log structure given by $\alpha: \cM_X = \cO_{X_\et}^\times \Mapn{\can} \cO_{X_\et}$.  We call it the \emph{trivial log structure} on $X$.
\end{exam}

\begin{exam}\label{ex-log-adic-sp-pt}
    A \emph{log point} is a log adic space whose underlying adic space is $\Spa(l, l^+)$, where $l$ is a nonarchimedean local field.  We remark that the underlying topological space may not be a single point.
\end{exam}

\begin{exam}\label{ex-log-adic-sp-pt-sep-cl}
    In Example \ref{ex-log-adic-sp-pt}, if $l$ is \emph{separably closed}, then the \'etale topos of $\Spa(l, l^+)$ is equivalent to the category of sets \Pth{see \cite[\aCor 1.7.3 and \aProp 2.3.10, and the paragraph after (2.5.2)]{Huber:1996-ERA}}.  In this case, a log structure of $\Spa(l, l^+)$ is given by a homomorphism of monoids $\alpha: M \to l$ inducing an isomorphism $\alpha^{-1}(l^\times) \Mi l^\times$.  For simplicity, by abuse of notation, we shall sometimes introduce a log point by writing $s = (\Spa(l, l^+), M)$.  Also, we shall simply denote by $s$ the underlying adic space $\Spa(l, l^+)$, when the context is clear.
\end{exam}

\begin{exam}\label{ex-log-adic-sp-tilt}
    Let $(X, \cM_X, \alpha_X)$ be a \emph{perfectoid log adic space}; \ie, a log adic space whose underlying adic space $X$ is a perfectoid space \Pth{see Definition \ref{def-log-str}}.  Let $\cM_{X^\flat} := \varprojlim \cM_X$, where the transition maps are given by sending a section to its $p$-th multiple.  Let $X^\flat$ be the tilt of $X$.  Then there is a natural morphism of sheaves of monoids $\alpha_{X^\flat}: \cM_{X^\flat} \to \cO_{X^\flat_\et}$ making $(X^\flat, \cM_{X^\flat}, \alpha_{X^\flat})$ a perfectoid log adic space, called the \emph{tilt} of $(X, \cM_X, \alpha_X)$.  Note that the isomorphism $\alpha_X^{-1}(\cO_{X_\et}^\times) \Mi \cO_{X_\et}^\times$ induces an isomorphism $\alpha_{X^\flat}^{-1}(\cO_{X^\flat_\et}^\times) \to \cO_{X^\flat_\et}^\times$ by taking inverse limit.
\end{exam}

We would like to study log adic spaces of the form $\Spa(R[P], R^+[P])$, whenever $(R, R^+)$ is a Huber pair.

\begin{lem}\label{lem-monoid-alg-Huber}
    Suppose that $(R, R^+)$ is a Huber pair with a ring of definition $R_0 \subset R$, which is adic with respect to a finitely generated ideal $I$.  Let us equip $R[P]$ with the topology determined by the ring of definition $R_0[P]$ such that $\{ I^m R_0[P] \}_{m \geq 0}$ forms a basis of open neighborhoods of $0$.  Then $(R[P], R^+[P])$ is also a Huber pair.
\end{lem}
\begin{proof}
    Note that $R^+[P]$ is open in $R[P]$ because $R^+$ is open in $R$.  We only need to check that $R^+[P]$ is integrally closed in $R[P]$.  By writing $P$ as the direct limit of its finitely generated submonoids, we may assume that $P$ is finitely generated.  But this case is standard \Pth{see, for example, \cite[\aThm 4.42]{Bruns/Gubeladze:2009-PRK}}.
\end{proof}

\begin{rk}\label{rem-monoid-alg-Huber-compl}
    Let $(R\Talg{P}, R^+\Talg{P})$ denote the completion of $(R[P], R^+[P])$.  Since taking completions of Huber pairs does not alter the associated adic spectra, we can identify $\Spa(R[P], R^+[P])$ with $\Spa(R\Talg{P}, R^+\Talg{P})$ \Pth{not just as adic spaces, but also as log adic spaces} whenever it is convenient to do so.
\end{rk}

\begin{lem}\label{lem-str-noe}
    Let $P$ be a finitely generated monoid.  Suppose that $R$ is either
    \begin{enumerate}[label=(\arabic*), ref=\arabic*]
        \item analytic and strongly noetherian; or

        \item \Pth{complete, by our convention, and} finitely generated over a noetherian ring of definition.
    \end{enumerate}
    Then so is $R\Talg{P}$ \Pth{which is complete by definition}.  As a result, $\Spa(R\Talg{P}, R^+\Talg{P})$ is a noetherian adic space when $\Spa(R, R^+)$ is, and $\Spa(R\Talg{P}, R^+\Talg{P})$ is \'etale sheafy \Pth{see Corollary \ref{cor-et-sheafy}}.  Moreover, the formation of the canonical morphism $\Spa(R\Talg{P}, R^+\Talg{P}) \to \Spa(R, R^+)$ is compatible with rational localizations on the target $\Spa(R, R^+)$.
\end{lem}
\begin{proof}
    Suppose that $R$ is analytic and strongly noetherian.  Since $P$ is finitely generated, there is some surjection $\bZ_{\geq 0}^r \Surj P$, which induces a continuous surjection $R\Talg{T_1, \ldots, T_r} \cong R\Talg{\bZ_{\geq 0}^r} \Surj R\Talg{P}$.  In this case, $R\Talg{P}\Talg{T_1, \ldots, T_n}$ is a quotient of $R\Talg{\bZ_{\geq 0}^r}\Talg{T_1, \ldots, T_n} \cong R\Talg{T_1, \ldots, T_{r + n}}$, for each $n \geq 0$, which is noetherian as $R$ is strongly noetherian.  Hence, $R\Talg{P}$ is also analytic and strongly noetherian.

    Alternatively, suppose that $R$ is generated by some $u_1, \ldots, u_n$ over a noetherian ring of definition $R_0$, with an ideal of definition $I \subset R_0$.  Since $R_0[P]$ is noetherian as $P$ is finitely generated, its $I R_0[P]$-adic completion $R_0\Talg{P}$ is also noetherian.  Then the image of $R_0\Talg{P}$ is a noetherian ring of definition of $R\Talg{P}$, over which $R\Talg{P}$ is generated by the images of $u_1, \ldots, u_n$, as desired.

    In both cases, the formation of $\Spa(R\Talg{P}, R^+\Talg{P}) \to \Spa(R, R^+)$ is clearly compatible with rational localizations on $\Spa(R, R^+)$.
\end{proof}

In a different direction, we would like to show that, under certain condition on $P$, if $(R, R^+)$ is a perfectoid affinoid algebra, then $(R\Talg{P}, R^+\Talg{P})$ also is.  In this case, $(R\Talg{P}, R^+\Talg{P})$ is \'etale sheafy \Pth{see Corollary \ref{cor-et-sheafy}, again}.
\begin{defn}\label{def-monoid-n-div}
    For each integer $n \geq 1$, a monoid $P$ is called \emph{$n$-divisible} \Pth{\resp \emph{uniquely $n$-divisible}} if the $n$-th multiple map $[n]: P \to P$ is surjective \Pth{\resp bijective}.
\end{defn}

\begin{lem}\label{lem-monoid-alg-perf}
    Suppose that $(R, R^+)$ is a perfectoid Huber pair.  If a monoid $P$ is uniquely $p$-divisible, then $(R\Talg{P}, R^+\Talg{P})$ is also a perfectoid Huber pair.  Moreover, the formation of the canonical morphism $\Spa(R\Talg{P}, R^+\Talg{P}) \to \Spa(R, R^+)$ is compatible with rational localizations on the target $\Spa(R, R^+)$.
\end{lem}
\begin{proof}
    If $p R = 0$, then the unique $p$-divisibility of $P$ implies that $R^+[P]$ is perfect, and so is its completion $R^+\Talg{P}$.  Also, it is clear that $(R\Talg{P})^\circ = R^\circ\Talg{P}$ in $R\Talg{P}$.  Hence, $R\Talg{P}$ is uniform, and $(R\Talg{P}, R^+\Talg{P})$ is a perfectoid Huber pair.

    In general, let $(R^{\flat}, R^{\flat+})$ be the tilt of $(R, R^+)$.  Let $\varpi \in R$ be a pseudo-uniformizer of $R$ satisfying $\varpi^p | p$ in $R^\circ$ and admitting a sequence of $p$-th power roots $\varpi^{\frac{1}{p^n}}$, so that $\varpi^\flat = (\varpi, \varpi^{\frac{1}{p}}, \ldots) \in R^{\flat\circ}$ is a pseudo-uniformizer of $R^\flat$, as in \cite[\aLem 6.2.2]{Scholze/Weinstein:2020-BLG}.  Let $\xi$ be a generator of $\ker(\theta: W(R^{\flat+}) \to R^+)$, which can be written as $\xi = p + [\varpi^\flat] a$ for some $a \in W(R^{\flat+})$, by \cite[\aLem 6.2.8]{Scholze/Weinstein:2020-BLG}.  By the first paragraph above and the tilting equivalence \Pth{see \cite[\aThm 6.2.11]{Scholze/Weinstein:2020-BLG}}, it suffices to show that
    \[
        R^+\Talg{P} \cong W(R^{\flat+}\Talg{P}) / (\xi).
    \]
    For this purpose, note that there is a natural homomorphism
    \[
        \theta': W(R^{\flat+}\Talg{P}) \to R^+\Talg{P}
    \]
    induced by the surjective homomorphism
    \[
        R^{\flat+}\Talg{P} \to (R^{\flat+} / \varpi^\flat)[P] \cong (R^+ / \varpi)[P]
    \]
    and the universal property of Witt vectors, and $\theta'$ is surjective because both its source and target are complete.  Since $\xi = p + [\varpi^\flat] a$, we have
    \[
        W(R^{\flat+}\Talg{P}) / (\xi, [\varpi^\flat]) = W(R^{\flat+}\Talg{P}) / (p, [\varpi^\flat]) \cong (R^{\flat+} / \varpi^\flat)[P] \cong (R^+ / \varpi)[P].
    \]
    Since $\theta'([\varpi^\flat]) = \varpi$, by induction, we see that the homomorphism
    \[
        W(R^{\flat+}\Talg{P}) / (\xi, [\varpi^\flat]^n) \to (R^+ / \varpi^n)[P]
    \]
    induced by $\theta'$ is an isomorphism, for each $n \geq 1$.  Thus, since $W(R^{\flat+}\Talg{P}) / (\xi)$ is $[\varpi^\flat]$-adically complete and separated, $\ker(\theta')$ is generated by $\xi$, as desired.

    Finally, the formation of $\Spa(R\Talg{P}, R^+\Talg{P}) \to \Spa(R, R^+)$ is clearly compatible with rational localizations on $\Spa(R, R^+)$, as in Lemma \ref{lem-str-noe}.
\end{proof}

\begin{rk}\label{rk-lem-monoid-alg-perf}
    In Lemma \ref{lem-monoid-alg-perf}, perfectoid Huber pairs are Tate by our convention following \cite[\aLec 6]{Scholze/Weinstein:2020-BLG}, but the statement of the lemma remains true for more general analytic perfectoid Huber pairs as in \cite{Kedlaya:2019-AWS}, by using \cite[\aLem 2.7.9]{Kedlaya:2019-AWS} instead of \cite[\aThm 6.2.11]{Scholze/Weinstein:2020-BLG}.
\end{rk}

\begin{defn}\label{def-P-log}
    When $X = \Spa(R[P], R^+[P]) \cong \Spa(R\Talg{P}, R^+\Talg{P})$ is \'etale sheafy, we denote by $P_X$ the constant sheaf on $X_\et$ defined by $P$.  Then the natural homomorphism $P \to R\Talg{P}$ of monoids induces a pre-log structure $P_X \to \cO_{X_\et}$ on $X$, whose associated log structure we simply denote by $P^{\log}$.
\end{defn}

\begin{conv}\label{conv-P-log}
    From now on, when $\Spa(R\Talg{P}, R^+\Talg{P})$ is \'etale sheafy and regarded as a log adic space, we shall endow it with the log structure $P^{\log}$ as in Definition \ref{def-P-log}, unless otherwise specified.
\end{conv}

Let us continue with some more examples of log adic spaces.

\begin{exam}\label{ex-log-adic-sp-monoid}
    Given any locally noetherian adic space $Y$ with trivial log structure as in Example \ref{ex-log-adic-sp-triv}, and given any finitely generated monoid $P$, by gluing the morphisms $\Spa(R\Talg{P}, R^+\Talg{P}) \to \Spa(R, R^+)$ as in Lemma \ref{lem-str-noe} over the noetherian affinoid open $\Spa(R, R^+)$ in $Y$, where each $\Spa(R\Talg{P}, R^+\Talg{P})$ is equipped with the structure of a log adic space as in Definition \ref{def-P-log} and Convention \ref{conv-P-log}, we obtain a morphism $X \to Y$ of log adic spaces, which we shall denote by $Y\Talg{P} \to Y$.  In this case, we shall also denote the log structure of $X = Y\Talg{P}$ by $P^{\log}$.
\end{exam}

\begin{exam}\label{ex-log-adic-sp-toric}
    If $P$ is a toric monoid as in Definition \ref{def-monoid}\Refenum{\ref{def-monoid-tor}}, then we say that $X = \Spa(k\Talg{P}, k^+\Talg{P})$ is an \emph{affinoid toric log adic space}.  This is a special case of Example \ref{ex-log-adic-sp-monoid} with $Y = \Spa(k, k^+)$, and is closely related to the theory of toroidal embeddings and toric varieties \Pth{see, \eg, \cite{Kempf/Knudsen/Mumford/Saint-Donat:1973-TE-1} and \cite{Fulton:1993-ITV}}.  Roughly speaking, such affinoid toric log adic spaces provide affinoid open subspaces of the rigid analytification of toric varieties, which are then also useful for studying local properties of more general varieties or rigid analytic varieties which are locally modeled on toric varieties.  Note that the underlying spaces of affinoid toric log adic spaces are always normal, by \cite[\aSec 7.3.2, \aProp 8]{Bosch/Guntzer/Remmert:1984-NAA}, \cite[IV-2, 7.8.3.1]{EGA}, and \cite[\aThm 1]{Hochster:1972-itcmp} \Pth{\Refcf{} \cite[\aThm 4.1]{Kato:1994-ts}}.
\end{exam}

\begin{exam}\label{ex-log-adic-sp-disc}
    A special case of Example \ref{ex-log-adic-sp-toric} is when $P \cong \bZ_{\geq 0}^n$ for some integer $n \geq 0$.  In this case, we obtain
    \[
        X = \Spa(k\Talg{P}, k^+\Talg{P}) \cong \bD^n := \Spa(k\Talg{T_1, \ldots, T_n}, k^+\Talg{T_1, \ldots, T_n}),
    \]
    the \emph{$n$-dimensional unit disc}, with the log structure of $\bD^n$ associated with the pre-log structure given by $\bZ_{\geq 0}^n \to k\Talg{T_1, \ldots, T_n}: (a_1, \ldots ,a_n) \mapsto T_1^{a_1} \cdots T_n^{a_n}$.
\end{exam}

The following proposition provides many more examples of log adic spaces coming from locally noetherian log formal schemes.
\begin{prop}\label{prop-log-adic-sp-fs}
    The canonical fully faithful functor from the category of locally noetherian formal schemes to the category of locally noetherian adic spaces defined locally by $\Spf(A) \mapsto \Spa(A, A)$ \Pth{as in \cite[\aSec 4.1]{Huber:1994-gfsra}} canonically extends to a fully faithful functor from the category of locally noetherian log formal schemes \Pth{as in \cite[\aSec 12.1]{Gabber/Ramero:2019-fart}} to the category of locally noetherian log adic spaces \Pth{introduced in this paper}.
\end{prop}
\begin{proof}
    Given any locally noetherian log formal scheme $(\mathfrak{X}, \cM_{\mathfrak{X}})$, let $X$ denote the adic spaces associated with the formal scheme $\mathfrak{X}$, with a canonical morphism of sites $\lambda: X_\et \to \mathfrak{X}_\et$, as in \cite[\aLem 3.5.1]{Huber:1996-ERA}.  By construction, we have a canonical morphism $\lambda^{-1}(\cO_{\mathfrak{X}_\et}) \to \cO_{X_\et}^+$.  Let $\cM_X$ be the log structure of $X$ associated with the pre-log structure $\lambda^{-1}(\cM_{\mathfrak{X}}) \to \lambda^{-1}(\cO_{\mathfrak{X}_\et}) \to \cO_{X_\et}^+ \to \cO_{X_\et}$.  Then the assignment $(\mathfrak{X}, \cM_{\mathfrak{X}}) \mapsto (X, \cM_X)$ gives the desired functor, which is fully faithful by adjunction.
\end{proof}

\begin{defn}[{\Refcf{} \cite[\aDef III.2.3.1]{Ogus:2018-LLG}}]\label{def-imm}
    We say that a morphism $f : Y \to X$ of log adic spaces is an \emph{open immersion} \Pth{\resp a \emph{closed immersion}} if the underlying morphism of adic spaces is an open immersion \Pth{\resp a closed immersion} and if the morphism $f^\sharp: f^{-1}(\cM_X) \to \cM_Y$ is an isomorphism \Pth{\resp a surjection}.  We say that $f$ is an immersion if it is a composition of a closed immersion of log adic spaces followed by an open immersion of log adic spaces.  We say that $f$ is \emph{strict} if it is a strict morphism of log adic spaces.
\end{defn}

\begin{exam}\label{ex-imm}
    Let $(X, \cM_X, \alpha_X)$ be a log adic space and $\imath: Z \to X$ an immersion of adic spaces.  Let $(Z, \cM_Z, \alpha_Z)$ be the log adic space associated with the pre-log structure $\imath_\et^{-1}(\cM_X) \to \imath_\et^{-1}(\cO_{X_\et}) \to \cO_{Z_\et}$.  Then the induced morphism $(Z, \cM_Z, \alpha_Z) \to (X, \cM_X, \alpha_X)$ of log adic spaces is a \emph{strict} immersion.  It is an open \Pth{\resp a closed} immersion exactly when the immersion $\imath$ of adic spaces is.
\end{exam}

\begin{rk}\label{rem-imm-ex}
    More generally, by the same argument as in \cite[the paragraph after \aDef III.2.3.1]{Ogus:2018-LLG}, a closed immersion is strict when it is exact.
\end{rk}

\subsection{Charts and fiber products}\label{sec-chart}

In this subsection, we introduce the notion of \emph{charts} for log adic spaces.  Compared with the corresponding notion for log schemes, a notable difference is that the definition of charts for a log adic space $X$ involves not just $\cO_{X_\et}$ but also $\cO_{X_\et}^+$.  Based on this notion, we also introduce the category of coherent \Pth{\resp fine, \resp fs} log adic spaces and study the fiber products in it.

\begin{defn}\label{def-chart}
    Let $(X, \cM_X, \alpha)$ be a log adic space.  Let $P$ be a monoid, and let $P_X$ denote the associated constant sheaf of monoids on $X_\et$.  A \emph{\Pth{global} chart of $X$ modeled on $P$} is a morphism of sheaves of monoids $\theta: P_X \to \cM_X$ such that $\alpha\bigl(\theta(P_X)\bigr) \subset \cO_{X_\et}^+$ and such that $\theta$ canonically induces \Pth{by the universal property of pushouts} an isomorphism ${^a}P_X \Mi \cM_X$ from the log structure ${^a}P_X$ associated with the pre-log structure $\alpha\circ\theta: P_X \to \cO_{X_\et}$.  We call the chart \emph{finitely generated} \Pth{\resp \emph{fine}, \resp \emph{fs}} if $P$ is finitely generated \Pth{\resp fine, \resp fs}.
\end{defn}

\begin{rk}\label{rem-def-chart}
    Giving a morphism $\theta: P_X \to \cM_X$ such that $\alpha\bigl(\theta(P_X)\bigr) \subset \cO_{X_\et}^+$ as in Definition \ref{def-chart} is equivalent to giving a homomorphism $P \to \cM_X(X)$ of monoids whose composition with $\alpha(X): \cM_X(X) \to \cO_{X_\et}(X)$ factors through $\cO_{X_\et}^+(X)$.  If the monoid $P$ is finitely generated, and if the underlying adic space $X$ is over some affinoid adic space $\Spa(R, R^+)$, then giving such a homomorphism $P \to \cM_X(X)$ whose composition with $\alpha(X)$ factors through $\cO_{X_\et}^+(X)$ is equivalent to giving a morphism $f: (X, \cM_X) \to (\Spa(R\Talg{P}, R^+\Talg{P}), P^{\log})$ of log adic spaces, whenever $\Spa(R\Talg{P}, R^+\Talg{P})$ is an \'etale sheafy adic space.  In this case, $\theta$ is a chart if and only if the morphism $f$ is strict.  We imposed the condition $\alpha\bigl(\theta(P_X)\bigr) \subset \cO_{X_\et}^+$ in Definition \ref{def-chart} because we will make crucial use of morphisms $f: (X, \cM_X) \to \bigl(\Spa(R\Talg{P}, R^+\Talg{P}), P^{\log}\bigr)$ as above in this paper.
\end{rk}

\begin{rk}\label{rem-def-chart-rel}
    In Remark \ref{rem-def-chart}, if the underlying adic space $X$ is over some locally noetherian adic space $Y$, then giving a morphism $\theta: P_X \to \cM_X$ such that $\alpha\bigl(\theta(P_X)\bigr) \subset \cO_{X_\et}^+$ is also equivalent to giving a morphism $g: X \to Y\Talg{P}$ as in Example \ref{ex-log-adic-sp-monoid}, in which case $\theta$ is a chart if and only if the morphism $g$ is strict.  Moreover, if $X$ is itself locally noetherian, then we can take $Y = X$, and obtain a closed immersion $h: X \to X\Talg{P}$, in which case $\theta$ is a chart if and only if $h$ is a strict closed immersion.
\end{rk}

\begin{rk}\label{rem-chart-stalk}
    Let $\theta: P_X \to \cM_X$ be a chart of a log adic space $(X, \cM_X, \alpha)$.  By Lemma \ref{lem-amalg-sum} and Remark \ref{rem-stalk-sharp}, for each geometric point $\AC{x}$ of $X$, we obtain a canonical isomorphism $P / (\alpha \circ \theta)^{-1}(\cO_{X_\et, \AC{x}}^\times) \Mi \cM_{X, \AC{x}}/\alpha^{-1}(\cO_{X_\et, \AC{x}}^\times) \cong \overline{\cM}_{X, \AC{x}}$.  In particular, the composition $P_X \Mapn{\theta} \cM_X \to \overline{\cM}_X$ is surjective.
\end{rk}

\begin{defn}\label{def-log-adic-sp-fs}
    A \emph{quasi-coherent} \Pth{\resp \emph{coherent}, \resp \emph{fine}, \resp \emph{fs}} \emph{log adic space} is a log adic space $X$ that \'etale locally admits some charts modeled on some monoids \Pth{\resp finitely generated monoids, \resp fine monoids, \resp fs monoids}.  \Pth{Quasi-coherent log adic spaces will not play any important role in this paper.}
\end{defn}

\begin{lem}\label{lem-chart-fg-Ogus}
    Let $(X, \cM_X, \alpha)$ be a log adic space, and $\theta: P_X \to \cM_X$ a chart modeled on some monoid $P$.  Suppose that there is a finitely generated monoid $P'$ such that $\theta$ factors as $P_X \to P'_X \Mapn{\theta'} \cM_X$ and such that $\alpha \circ \theta': P'_X \to \cO_{X_\et}$ factors through $\cO_{X_\et}^+$.  Then, \'etale locally on $X$, there exists a chart $\theta'': P''_X \to \cM_X$ modeled on some finitely generated monoid $P''$ such that $\theta'$ factors through $\theta''$.
\end{lem}
\begin{proof}
    The proof is similar to \cite[\aProp II.2.2.1]{Ogus:2018-LLG}, except that, when compared with charts on log schemes, charts $\theta$ on log adic spaces $(X, \cM_X, \alpha)$ are subject to the additional requirement that $\alpha \circ \theta$ factors through $\cO_{X_\et}^+$.

    Let $\{ a'_i \}_{i \in I}$ be a finite set of generators of $P'$.  Since $P_X \to \overline{\cM}_X$ is surjective \Pth{by Remark \ref{rem-chart-stalk}}, \'etale locally on $X$, there exist some $a_i \in P$ and $f_i \in \cO_X^\times(X)$ such that $\theta'(a'_i) = \theta(a_i) \, f_i$, for all $i \in I$.  By \cite[(1) in the proof of \aProp 2.5.13]{Huber:1996-ERA}, for each geometric point $\AC{x}$ of $X$, we have
    \[
        \cO_{X_\et, \AC{x}}^+ = \{ f \in \cO_{X_\et, \AC{x}} : |f(\AC{x})| \leq 1 \}
    \]
    in $\cO_{X_\et, \AC{x}}$.  By Remark \ref{rem-fg}, up to further \'etale localization on $X$, we may assume that, for each $i \in I$, at least one of $f_i$ and $f_i^{-1}$ is in $\cO_X^+(X)$.  Consider the homomorphism $P' \oplus \bZ_{\geq 0}^I \to \cM_X(X)$ sending $(a'_i, 0) \mapsto \theta(a_i)$ and sending
    \[
        \begin{cases}
            (0, e_i) \mapsto f_i, & \Utext{if $f_i \in \cO_X^+(X)$;} \\
            (0, e_i) \mapsto f_i^{-1}, & \Utext{if $f_i \not \in \cO_X^+(X)$ but $f_i^{-1} \in \cO_X^+(X)$,}
        \end{cases}
    \]
    where $e_i$ denotes the $i$-th standard basis element of $\bZ_{\geq 0}^I$.  Let $\beta$ denote the homomorphism $P \to P'$, and let $P''$ be the quotient of $P' \oplus \bZ_{\geq 0}^I$ modulo the relations
    \[
        \begin{cases}
            (a'_i, 0) \sim (\beta(a_i), e_i), & \Utext{if $f_i \in \cO_X^+(X)$;} \\
            (a'_i, e_i) \sim (\beta(a_i), 0), & \Utext{if $f_i \not \in \cO_X^+(X)$ but $f_i^{-1} \in \cO_X^+(X)$.}
        \end{cases}
    \]
    By construction, $P' \oplus \bZ_{\geq 0}^I \to \cM_X(X)$ factors through an induced homomorphism $\theta'': P'' \to \cM_X(X)$, and $\alpha \circ \theta''$ factors through $\cO_X^+$, as desired.

    It remains to check that the log structure associated with the pre-log structure $\alpha \circ \theta'': P''_X \to \cM_X \to \cO_{X_\et}$ coincides with $\cM_X$; \ie, the natural morphism ${^a}P_X \to {^a}P''_X$ induced by $P \to P''$ is an isomorphism.  It is injective because the composition ${^a}P_X \to {^a}P''_X \to \cM_X$ is an isomorphism.  It is also surjective, because the induced morphism $P_X / (\alpha \circ \theta)^{-1}(\cO_{X_\et}^\times) \to P_X'' / (\alpha \circ \theta'')^{-1}(\cO_{X_\et}^\times)$ is surjective, since the target is generated by the images of $a'_i$ which lift to the images of $a_i$ in the source.
\end{proof}

\begin{lem}\label{lem-int-sat-chart}
    Let $(X, \cM_X, \alpha)$ be a locally noetherian coherent log adic space, $P$ a monoid, and $P_X \to \cM_X$ a chart.  Suppose that $(X, \cM_X, \alpha)$ is integral \Pth{\resp saturated}, in which case $P_X \to \cM_X$ factors through $P^\Int_X \to \cM_X$ \Pth{\resp $P^\Sat_X \to \cM_X$}.  Then $P^\Int_X \to \cM_X$ \Pth{\resp $P^\Sat_X \to \cM_X$} is also a chart.
\end{lem}
\begin{proof}
    Suppose that $(X, \cM_X, \alpha)$ is integral.  Since $P_X \to \cM_X$ is a chart, the composition ${^a}P_X \to {^a}P_X^\Int \to \cM_X$ is an isomorphism, and hence the induced morphism ${^a}P_X \to {^a}P_X^\Int$ is injective.  Since $P \to P^\Int$ is surjective, the composition $P_X^\Int \to \cM_X \to \cO_{X_\et}$ factors through $\cO_{X_\et}^+$, and the induced map ${^a}P_X \to {^a}P_X^\Int$ is surjective and hence is an isomorphism.

    Suppose that $(X, \cM_X, \alpha)$ is saturated.  Then the chart $P_X \to \cM_X$ factors as a composition of $P_X \to P^\Sat_X \to \cM_X$.  Since $\cO_{X_\et}^+$ is integrally closed in $\cO_{X_\et}$, the composition $P^\Sat_X \to \cM_X \to \cO_{X_\et}$ factors through $\cO_{X_\et}^+$.  It remains to show that the induced morphism ${^a}P_X \to {^a}P_X^\Sat$ is an isomorphism.  By Remark \ref{rem-chart-stalk}, it suffices to show that, at each geometric point $\AC{x}$ of $X$, if we denote by $\beta: P \to \cM_{X, \AC{x}}$ and $\beta': P^\Sat \to \cM_{X, \AC{x}}$ the induced homomorphisms, then the canonical homomorphism
    \begin{equation}\label{eq-lem-int-sat-chart-1}
        P / (\alpha \circ \beta)^{-1}(\cO_{X_\et, \AC{x}}^\times) \to P^\Sat / (\alpha \circ \beta')^{-1}(\cO_{X_\et, \AC{x}}^\times)
    \end{equation}
    is an isomorphism.  Let $\overline{\beta}$ denote the composition of $P \to \cM_{X, \AC{x}} \to \overline{\cM}_{X, \AC{x}}$.  Since $\ker(\cM_{X, \AC{x}}^\gp \to \overline{\cM}_{X, \AC{x}}^\gp) = \cM_{X, \AC{x}}^\inv = \alpha^{-1}(\cO_{X_\et, \AC{x}}^\times)$, because $\cM_{X, \AC{x}}$ is integral \Pth{see Remark \ref{lem-int-sat-stalk}}, we obtain
    \[
        \ker(\overline{\beta}^\gp) = (\beta^\gp)^{-1}\bigl(\alpha^{-1}(\cO_{X_\et, \AC{x}}^\times)\bigr).
    \]
    Since $P^\gp / \ker(\overline{\beta}^\gp) \cong \overline{\cM}^\gp_{X, \AC{x}} \cong P^\gp / \bigl((\alpha \circ \beta)^{-1}(\cO_{X_\et, \AC{x}}^\times)\bigr)^\gp$, we obtain
    \[
        \bigl((\alpha \circ \beta)^{-1}(\cO_{X_\et, \AC{x}}^\times)\bigr)^\gp = (\beta^\gp)^{-1}\bigl(\alpha^{-1}(\cO_{X_\et, \AC{x}}^\times)\bigr).
    \]
    Since $(\alpha \circ \beta)^{-1}(\cO_{X_\et, \AC{x}}^\times) \subset (\alpha \circ \beta')^{-1}(\cO_{X_\et, \AC{x}}^\times) \subset (\beta^\gp)^{-1}\bigl(\alpha^{-1}(\cO_{X_\et, \AC{x}}^\times)\bigr)$, we obtain
    \[
        \bigl((\alpha \circ \beta')^{-1}(\cO_{X_\et, \AC{x}}^\times)\bigr)^\gp = (\beta^\gp)^{-1}\bigl(\alpha^{-1}(\cO_{X_\et, \AC{x}}^\times)\bigr).
    \]
    By Lemma \ref{lem-int-sat-quot} and the above, we see that the natural homomorphism
    \begin{equation}\label{eq-lem-int-sat-chart-2}
        P^\Sat / (\alpha \circ \beta')^{-1}(\cO_{X_\et, \AC{x}}^\times) \to P^\gp / \ker(\overline{\beta}^\gp)
    \end{equation}
    is injective, whose image is contained in $\bigl(P / (\alpha\circ\beta)^{-1}(\cO_{X_\et, \AC{x}}^\times)\bigr)^\Sat$.  Moreover, the composition of \Refeq{\ref{eq-lem-int-sat-chart-1}} and \Refeq{\ref{eq-lem-int-sat-chart-2}} induces the canonical homomorphism
    \begin{equation}\label{eq-lem-int-sat-chart-3}
        P / (\alpha \circ \beta)^{-1}(\cO_{X_\et, \AC{x}}^\times) \to \bigl(P / (\alpha \circ \beta)^{-1}(\cO_{X_\et, \AC{x}}^\times)\bigr)^\Sat.
    \end{equation}
    By Lemmas \ref{lem-int-sat-quot} and \ref{lem-int-sat-stalk}, $P / (\alpha \circ \beta)^{-1}(\cO_{X_\et, \AC{x}}^\times) \cong \overline{\cM}_{X, \AC{x}}$ is saturated.  Thus, \Refeq{\ref{eq-lem-int-sat-chart-3}} is an isomorphism, and so is \Refeq{\ref{eq-lem-int-sat-chart-1}}, as desired.
\end{proof}

\begin{prop}\label{prop-int-fine}
    Let $(X, \cM_X, \alpha)$ be a locally noetherian coherent log adic space.  Then it is fine \Pth{\resp fs} if and only if it is integral \Pth{\resp saturated}.
\end{prop}
\begin{proof}
    If $(X, \cM_X, \alpha)$ is integral \Pth{\resp saturated}, then it is fine \Pth{\resp fs} by Lemma \ref{lem-int-sat-chart}.  Conversely, if $(X, \cM_X, \alpha)$ is fine, then it is integral by Lemma \ref{lem-ama-int}.  Suppose that $(X, \cM_X, \alpha)$ admits an fs chart $\theta: P_X \to \cM_X$.  By Remark \ref{rem-chart-stalk}, we have $P / (\alpha \circ \theta)^{-1}(\cO_{X_\et, \AC{x}}^\times) \Mi \cM_{X, \AC{x}} / \alpha^{-1}(\cO_{X_\et, \AC{x}}^\times) \cong \overline{\cM}_{X, \AC{x}}$ at each geometric point $\AC{x}$ of $X$.  By Lemma \ref{lem-int-sat-quot}, $\overline{\cM}_{X, \AC{x}}$ is saturated, because $P$ is.  By \cite[\aProp I.1.3.5]{Ogus:2018-LLG}, $\cM_{X, \AC{x}}$ is also saturated, because $\cM_{X, \AC{x}}$ is integral and $\overline{\cM}_{X, \AC{x}}$ is saturated.  Thus, $(X, \cM_X, \alpha)$ is saturated, by Lemma \ref{lem-int-sat-stalk}.
\end{proof}

\begin{lem}\label{lem-fs-split-int}
    Let $(X, \cM_X, \alpha)$ be a fine \Pth{\resp fs} log adic space.  For any geometric point $\AC{x}$ of $X$, the monoid $\overline{\cM}_{X, \AC{x}}$ is sharp fine \Pth{\resp toric---\ie, sharp fs}, and the canonical homomorphism $\cM_{X, \AC{x}} \to \overline{\cM}_{X, \AC{x}}$ admits a splitting $s$ that factors through the preimage of $\cO_{X_\et, \AC{x}}^+$ in $\cM_{X, \AC{x}}$.
\end{lem}
\begin{proof}
    Let $P := \overline{\cM}_{X, \AC{x}}$, which is finitely generated because $X$ is fine.  Under the assumption that $X$ is fine \Pth{\resp fs}, by Proposition \ref{prop-int-fine} and Lemma \ref{lem-int-sat-stalk}, $\cM_{X, \AC{x}}$ is integral \Pth{\resp saturated}, and so its sharp quotient $P$ is sharp fine \Pth{\resp toric}.  By Lemma \ref{lem-monoid-split}, the surjective homomorphism $f: \cM_{X, \AC{x}} \to P$ admits a section $s_0$.  We need to modify this into a section $s: P \to \cM_{X, \AC{x}}$ such that $(\alpha \circ s)(P)\subset \cO_{X_\et, \AC{x}}^+$.

    By \cite[(1) in the proof of \aProp 2.5.13]{Huber:1996-ERA}, we have
    \[
        \cO_{X_\et, \AC{x}}^+ = \{ f \in \cO_{X_\et, \AC{x}} : | f(\AC{x}) | \leq 1 \}
    \]
    and
    \[
        \{ f \in \cO_{X_\et, \AC{x}} : | f(\AC{x}) | > 1 \} \subset \cO_{X_\et, \AC{x}}^\times
    \]
    in $\cO_{X_\et, \AC{x}}$.  Let $\{ a_1, \ldots, a_r \}$ be a finite set of generators of $P$.  For each $i$, let $\gamma_i := | \alpha(s_0(a_i))(\AC{x}) |$.  If $\gamma_i \leq 1$ for all $i$, then we set $f_0 := 1$ in $\cO_{X_\et, \AC{x}}$.  Otherwise, there exists some $i_0$ such that $\gamma_{i_0} > 1$ and $\gamma_{i_0} \geq \gamma_i$, for all $i$.  Then $\alpha(s_0(a_{i_0}))$ admits an inverse $f_0$ in $\cO_{X_\et, \AC{x}}$, so that $| f_0(\AC{x}) | = \gamma_{i_0}^{-1}$.  By \cite[\aCor I.2.2.7]{Ogus:2018-LLG}, we can identify $P$ with a submonoid of $\bZ_{\geq 0}^{r'}$, for some $r' \geq 0$, so that we can describe elements of $P$ by $r'$-tuples of integers.  Then the homomorphism
    \[
        s: P \to \cM_{X, \AC{x}}: (n_1, \ldots, n_{r'}) \mapsto f_0^{n_1 + \cdots + n_{r'}} s_0\bigl((n_1, \ldots, n_{r'})\bigr)
    \]
    satisfies $(\alpha \circ s)(P)\subset \cO_{X_\et, \AC{x}}^+$, as desired.
\end{proof}

\begin{prop}\label{prop-chart-stalk-chart-fs}
    Let $(X, \cM_X, \alpha)$ be an fine log adic space, and $\AC{x}$ any geometric point of $X$.  Then $X$ admits, \'etale locally at $\AC{x}$, a chart modeled on $\overline{\cM}_{X, \AC{x}}$.
\end{prop}
\begin{proof}
    By Lemma \ref{lem-fs-split-int}, we have a splitting $s: P := \overline{\cM}_{X, \AC{x}} \to \cM_{X, \AC{x}}$ such that $(\alpha_{\AC{x}} \circ s)(P) \subset \cO_{X_\et, \AC{x}}^+$.  Since $P$ is fine because $X$ is fine, by Remark \ref{rem-fg}, up to \'etale localization on $X$, the splitting $s$ lifts to a morphism $\widetilde{s}: P_X \to \cM_X$ such that $(\alpha \circ \widetilde{s})(P_X) \subset \cO_{X_\et}^+$ \Pth{see Remark \ref{rem-fg}}.  Then the composition of $P_X \to \cM_X$ with $\alpha$ is a pre-log structure, whose associated log structure ${^a}P_X \to \cO_{X_\et}$ factors through ${^a}\widetilde{s}: {^a}P_X \to \cM_X$ \Pth{and $\alpha$}.  The induced ${^a}\widetilde{s}_{\AC{x}}: {^a}P_{X, \AC{x}} \to \cM_{X, \AC{x}}$ is an isomorphism, because the quotients of both sides by the isomorphic preimages of $\cO_{X_\et, \AC{x}}^\times$ induce the canonical isomorphism $P \Mi \overline{\cM}_{X, \AC{x}}$, by construction.  Hence, up to further \'etale localization on $X$, we may assume that ${^a}\widetilde{s}: {^a}P_X \to \cM_X$ is an isomorphism, because the quotients of both sides by the isomorphic preimages of $\cO_{X_\et}^\times$ induce the canonical morphism $P_X \to \overline{\cM}_X$ \Pth{again, see Remark \ref{rem-fg}}.  As a result, $\widetilde{s}: P_X \to \cM_X$ is a chart modeled on $P = \overline{\cM}_{X, \AC{x}}$, as desired.
\end{proof}

\begin{exam}\label{ex-log-adic-sp-pt-fs}
    An \emph{fs log point} is a log point \Pth{as in Example \ref{ex-log-adic-sp-pt}} that is an fs log adic space.  In the setting of Example \ref{ex-log-adic-sp-pt-sep-cl}, by Remark \ref{rem-stalk-sharp}, a log point $s = (\Spa(l, l^+), M)$ with $l$ separably closed is an fs log point exactly when $M / l^\times$ is toric \Pth{\ie, sharp fs}.  In this case, by Lemmas  \ref{lem-monoid-split} and \ref{lem-fs-split-int}, there always exists a homomorphism of monoids $M / l^\times \to M$ splitting the canonical homomorphism $M \to M / l^\times$ and defining a chart of $s$ modeled on $M$.
\end{exam}

\begin{exam}\label{ex-log-adic-sp-pt-fs-split}
    A special case of Example \ref{ex-log-adic-sp-pt-fs} is
    a \emph{split fs log point} \ie, a log point of the form $s = (X, \cM_X) \cong (\Spa(l, l^+), \cO_{X_\et}^\times \oplus P_X)$ for some \Pth{necessarily} toric monoid $P$.  This is equivalent to a log point $(\Spa(L, L^+), M)$, where $L$ is the completion of a separable closure $l^\sep$ of $l$, with a $\Gal(l^\sep / l)$-equivariant splitting of the homomorphism $M \to M / L^\times$.  We also remark that this is the same as a $\Gal(l^\sep / l)$-equivariant splitting of the homomorphism $M^\gp \to M^\gp / L^\times$.
\end{exam}

\begin{exam}\label{ex-log-adic-sp-div}
    Let $D$ be an effective Cartier divisor on a normal rigid analytic variety $X$ over a nonarchimedean field $k$, and let $\imath: D \Em X$ denote the associated closed immersion.  By viewing $X$ as a noetherian adic space, we equip $X$ with the log structure $\alpha: \cM_X \to \cO_{X_\et}$ defined by setting
    \[
        \cM_X(V) = \{ f \in \cO_{X_\et}(V) : \Utext{$f$ is invertible on the preimage of $X - D$} \},
    \]
    for each object $V \to X$ in $X_\et$, with $\alpha(V): \cM_X(V) \to \cO_{X_\et}(V)$ given by the natural inclusion.  This makes $X$ a locally noetherian fs log adic space.  \Pth{The normality of $X$ is necessary for showing that the log structure $\cM_X$ is indeed saturated.}  Then $X - D$ is the maximal open subspace of $X$ over which $\cM_X$ is trivial.  Note that, in Example \ref{ex-log-adic-sp-disc}, the log structure of $X \cong \bD^n$ can be defined alternatively as above by the closed immersion $\imath: D := \{ T_1 \cdots T_n = 0 \} \Em \bD^n$.
\end{exam}

The following special case is useful in many applications:
\begin{exam}\label{ex-log-adic-sp-ncd}
    Let $X$, $D$, and $k$ be as in Example \ref{ex-log-adic-sp-div}.  Suppose moreover that $X$ is smooth.  We say that $D$ is a \emph{\Pth{reduced} normal crossings divisor} of $X$ if, \'etale locally on $X$---or equivalently \Pth{by \cite[\aLem 3.1.5]{deJong/vanderPut:1996-ecras}}, analytic locally on $X$, up to replacing the base field $k$ with a finite separable extension---$X$ and $D$ are of the form $S \times \bD^m$ and $S \times \{ T_1 \cdots T_m = 0 \}$, where $S$ is a smooth connected rigid analytic variety over $k$, and $\iota: D \Em X$ is the pullback of the canonical closed immersion $\{ T_1 \cdots T_m = 0 \} \Em \bD^m$.  \Pth{This definition is justified by \cite[\aThm 1.18]{Kiehl:1967-drkab}.}  Then we equip $X$ with the fs log structure defined as in Example \ref{ex-log-adic-sp-div}, which is compatible with the one of $\bD^m$ as in Example \ref{ex-log-adic-sp-disc} via pullback.
\end{exam}

The following example will be useful when studying the geometric monodromy and nearby cycles of \'etale local systems \Qtn{\emph{along the boundary}}:
\begin{exam}\label{ex-log-adic-sp-ncd-strict-cl-imm}
    Let $X$, $D$, and $k$ be as in Example \ref{ex-log-adic-sp-ncd}.  Suppose that $\{ D_j \}_{j \in I}$ is the set of irreducible components of $D$ \Pth{see \cite{Conrad:1999-icrs}}.  For each $J \subset I$, as locally closed subspaces of $X$, consider $X_J := X \cap \bigl(\cap_{j \in J} \, D_j\bigr)$, $D_J := \cup_{J \subsetneq J' \subset I} \, X_{J'}$, and $U_J := X_J - D_J$.  By pulling back the log structure from $X$ to $X_J$ and $U_J$, respectively, we obtain log adic spaces $(X_J^\partial, \cM_{X_J^\partial})$ and $(U_J^\partial, \cM_{U_J^\partial})$ \Pth{with strict immersions to $X$}.  When $X_J$ is also smooth and so $D_J$ is a normal crossings divisor, we equip $X_J$ with the fs log structure defined by $D_J$ as in Example \ref{ex-log-adic-sp-div}, whose restriction to $U_J$ is then the trivial log structure.  If we also consider $D^J := \cup_{j \in I - J} \, D_j$, and let $X^J$ denote the same adic space $X$ but equipped with the fs log structure defined by $D^J$ as in Example \ref{ex-log-adic-sp-div}, then $\cM_{X_J}$ and $\cM_{U_J} = \cO_{U_J, \et}^\times$ are nothing but the log structures pulled back from $X^J$.  Moreover, since $D^J \subset D$, there is a canonical morphism of log adic spaces $X \to X^J$; and since $D_J = D^J \cap X_J$, this morphism induces a canonical morphism of log adic spaces $X_J^\partial \to X_J$, whose underlying morphism of adic spaces is an isomorphism.  Since $X$ and $D$ is \'etale locally of the form $S \times \bD^m$ and $S \times \{ T_1 \cdots T_m = 0 \}$ for some smooth $S$ over $k$, it follows that $X_J$ is \'etale locally of the form $S \times \bD^{m - |J|}$, in which case the log structures $\cM_{X_J^\partial}$ and $\cM_{X_J}$ are associated with the pre-log structures $\bZ_{\geq 0}^m \to \cO_{X_J^\partial, \et}$ and $\bZ_{\geq 0}^{m - |J|} \to \cO_{X_J, \et}$, and we have a direct sum $\cM_{X_J^\partial} \cong \cM_{X_J} \oplus (\bZ_{\geq 0}^J)_{X_J}$.
\end{exam}

\begin{defn}\label{def-chart-mor}
    Let $f: (Y, \cM_Y, \alpha_Y) \to (X, \cM_X, \alpha_X)$ be a morphism of log adic spaces.  A \emph{chart} of $f$ consists of charts $\theta_X: P_X \to \cM_X$ and $\theta_Y: Q_Y \to \cM_Y$ and a homomorphism $u: P \to Q$ of monoids such that the diagram
    \[
        \xymatrix{ {P_Y} \ar^-u[r] \ar^-{\theta_X}[d] & {Q_Y} \ar^-{\theta_Y}[d] \\
        {f^{-1}(\cM_X)} \ar^-{f^\sharp}[r] & {\cM_Y} }
    \]
    commutes.  We say that the chart is \emph{finitely generated} \Pth{\resp \emph{fine}, \resp \emph{fs}} if both $P$ and $Q$ are finitely generated \Pth{\resp fine, \resp fs}.  When the context is clear, we shall simply say that $u: P \to Q$ is the chart of $f$.
\end{defn}

\begin{exam}\label{ex-log-adic-sp-ncd-cov}
    Let $P := \bZ_{\geq 0}^n$ and let $Q$ be a toric submonoid of $\frac{1}{m} \bZ_{\geq 0}^n$ containing $P$, for some $m \geq 1$.  Then the canonical homomorphism $u: P \to Q$ induces a morphism $f: Y := \Spa(k\Talg{Q}, k^+\Talg{Q}) \to X := \Spa(k\Talg{P}, k^+\Talg{P}) \cong \bD^n$ of normal adic spaces, whose source and target are equipped with canonical log structures as in Examples \ref{ex-log-adic-sp-toric} and \ref{ex-log-adic-sp-disc}, making $f: Y \to X$ a morphism of fs log adic spaces.  Moreover, these log structures on $X$ and $Y$ coincide with those on $X$ and $Y$ defined by $D = \{ T_1 \cdots T_n = 0 \} \Em X$ and its pullback to $Y$, respectively, as in Example \ref{ex-log-adic-sp-div}.  A chart of $f: Y \to X$ is given by the canonical charts $P \to \cM_X(X)$ and $Q \to \cM_Y(Y)$ and the above $u: P \to Q$.
\end{exam}

\begin{prop}\label{prop-chart-mor-exist}
    Let $f: Y \to X$ be a morphism of coherent log adic spaces, and let $P \to \cM_X(X)$ be a chart modeled on a finitely generated monoid $P$.  Then, \'etale locally on $Y$, there exist a chart $Q \to \cM_Y(Y)$ modeled on a finitely generated monoid $Q$ and a homomorphism $P \to Q$, which together provide a chart of $f$.
\end{prop}
\begin{proof}
    Up to \'etale localization on $Y$ and $X$, we may assume that $(X, \cM_X)$ and $(Y, \cM_Y)$ are modeled on some finitely generated monoids $P$ and $Q'$, respectively.  Then the composition of $P_Y \cong f^{-1}(P_X) \to f^{-1}(\cM_X) \to \cM_Y$ induces a morphism $P_Y \to (P \oplus Q')_Y \to \cM_Y$.  Note that $P \oplus Q'$ is finitely generated, and that the composition $(P \oplus Q')_Y \to \cM_Y \to \cO_{Y_\et}$ factors through $\cO_{Y_\et}^+$.  By applying Lemma \ref{lem-chart-fg-Ogus} to $Q'_Y \to (P \oplus Q')_Y \to \cM_Y$, we see that, \'etale locally, $(P \oplus Q')_Y \to \cM_Y$ factors as $(P \oplus Q')_Y \to Q_Y \to \cM_Y$, where $Q_Y \to \cM_Y$ is a chart modeled on a finitely generated monoid $Q$.  Thus, the composition $P \to P \oplus Q' \to Q$ gives a chart of $f$, as desired.
\end{proof}

\begin{prop}\label{prop-chart-mor-exist-fine}
    Any morphism between fine \Pth{\resp fs} log adic spaces \'etale locally admits fine \Pth{\resp fs} charts.
\end{prop}
\begin{proof}
    By Proposition \ref{prop-chart-stalk-chart-fs}, \'etale locally, $X$ admits a chart modeled on a fine \Pth{\resp fs} monoid $P$.  By Proposition \ref{prop-chart-mor-exist}, $f$ admits, \'etale locally on $Y$, a chart $P \to Q$ with finitely generated $Q$.  By Lemma \ref{lem-int-sat-chart}, the induced $Q^\Int_Y \to \cM_Y$ \Pth{\resp $Q^\Sat_Y \to \cM_Y$} is also a chart of $Y$.  Hence, the composition of $P \to Q \to Q^\Int$ \Pth{\resp $P \to Q \to Q^\Sat$} is a fine \Pth{\resp fs} chart of $f$.
\end{proof}

\begin{prop}\phantomsection\label{prop-adj-int-sat}
    \begin{enumerate}
        \item\label{prop-adj-int-sat-1} The inclusion from the category of noetherian \Pth{\resp locally noetherian} fine log adic spaces to the category of noetherian \Pth{\resp locally noetherian} coherent log adic spaces admits a right adjoint $X \mapsto X^\Int$, and the corresponding morphism of underlying adic spaces is a closed immersion.

        \item\label{prop-adj-int-sat-2} The inclusion from the category of noetherian \Pth{\resp locally noetherian} fs log adic spaces to the category of noetherian \Pth{\resp locally noetherian} fine log adic spaces admits a right adjoint $X \mapsto X^\Sat$, and the corresponding morphism of underlying adic spaces is finite and surjective.
    \end{enumerate}
\end{prop}
\begin{proof}
    In case \Refenum{\ref{prop-adj-int-sat-1}} \Pth{\resp \Refenum{\ref{prop-adj-int-sat-2}}}, let $? = \Int$ \Pth{\resp $\Sat$} in the following.

    Suppose that $X = \Spa(R, R^+)$ is noetherian affinoid and admits a global chart modeled on a finitely generated \Pth{\resp fine} monoid $P$, so that we have a homomorphism $P \to R$ of monoids, inducing a homomorphism $\bZ[P] \to R$ of rings.  Let $R^? := R \otimes_{\bZ[P]} \bZ[P^?]$, and let $R^{?+}$ denote the integral closure of $R^+ \otimes_{\bZ[P]} \bZ[P^?]$ in $R^?$.  Since $P$ is finitely generated, $\bZ[P^?]$ is a finite $\bZ[P]$-algebra, and $(R^?, R^{?+})$ is equipped with a unique topology extending that of $(R, R^+)$, which is not necessarily complete.  Let $X^? := \Spa(R^?, R^{?+})$ \Pth{which, as usual, depends only on the completion of $(R^?, R^{?+})$}, with the log structure induced by $P^? \to \cO_{X^?}(X^?) = R \otimes_{\bZ[P]} \bZ[P^?]: a \mapsto 1 \otimes \mono{a}$ \Pth{where $\mono{a}$ denotes the image of $a \in P^?$ in $\bZ[P^?]$, by our convention}.  Clearly, the natural projection $X^? \to X$ is a closed immersion \Pth{\resp finite and surjective morphism} of log adic spaces.  We claim that, if $(Y, \cM_Y)$ is a fine \Pth{\resp fs} log adic space, then each morphism $f: (Y, \cM_Y) \to (X, \cM_X)$ of log adic spaces factors through $X^?$, yielding $\Mor(Y, X) \cong \Mor(Y, X^?)$.  Indeed, by Proposition \ref{prop-int-fine}, the induced morphism $P_Y \cong f^{-1}(P_X) \to f^{-1}(\cM_X) \to \cM_Y$ factors through $P^?_Y$, and hence $Y \to X$ factors through $Y \to X^?$, as desired.

    In general, there exists an \'etale covering of $X$ by affinoids $X_i = \Spa(R_i, R^+_i)$ such that each $X_i$ admits a global chart modeled on a finitely generated \Pth{\resp fine} monoid \Pth{see Definition \ref{def-log-adic-sp-fs}}.  Consider $\widetilde{X} = \coprod_i \, X_i$.  By the affinoid case treated in the last paragraph, we obtain a finite morphism $\widetilde{X}^? \to \widetilde{X}$, which is equipped with a descent datum.  By \'etale descent of coherent sheaves \Pth{see Proposition \ref{prop-et-coh-descent}}, $\widetilde{X}^? \to \widetilde{X}$ descends to a locally noetherian adic space $X^? \to X$.  Also, the \'etale sheaf of monoids descends \Pth{essentially by definition}.  Finally, by Proposition \ref{prop-chart-mor-exist-fine} and the local construction in the previous paragraph, the formation $X \mapsto X^?$ is functorial, as desired.
\end{proof}

\begin{rk}\label{rem-adj-sat}
    For a noetherian \Pth{\resp locally noetherian} coherent log adic space $X$, we shall simply denote by $X^\Sat$ the fs log adic space $(X^\Int)^\Sat$.  By combining the two cases in Proposition \ref{prop-adj-int-sat}, the functor $X \mapsto X^\Sat$ from the category of noetherian \Pth{\resp locally noetherian} coherent log adic spaces to the category of noetherian \Pth{\resp locally noetherian} fs log adic spaces is the right adjoint of the inclusion from the category of noetherian \Pth{\resp locally noetherian} fs log adic spaces to the category of noetherian \Pth{\resp locally noetherian} coherent log adic spaces.
\end{rk}

\begin{rk}\label{rem-adj-int-sat-f-et}
    By construction, both the functors $X \mapsto X^\Int$ and $X \mapsto X^\Sat$ send strict and finite \Pth{\resp \'etale} morphisms to strict and finite \Pth{\resp \'etale} morphisms.
\end{rk}

\begin{rk}\label{rem-adj-int-sat-rel}
    Again by construction, when $X$ is a locally noetherian log adic space over a locally noetherian fs log adic space $Y$ and admits a global chart modeled on a finitely generated \Pth{\resp fine} monoid $P$, for $? = \Int$ \Pth{\resp $\Sat$}, we have $X^? \cong X \times_{Y\Talg{P}} Y\Talg{P^?}$ as adic spaces, where $Y\Talg{P}$ and $Y\Talg{P^?}$ are as in Example \ref{ex-log-adic-sp-monoid}.  \Pth{Note that the fiber product $X \times_{Y\Talg{P}} Y\Talg{P^?}$ exists because the morphism $Y\Talg{P^?} \to Y\Talg{P}$ is lft when $P$ is finitely generated.}
\end{rk}

Now, let us study fiber products in the category of locally noetherian coherent \Pth{\resp fine, \resp fs} log adic spaces:
\begin{prop}\phantomsection\label{prop-fiber-prod-log-adic}
    \begin{enumerate}
        \item\label{prop-fiber-prod-log-adic-1}  Finite fiber products exist in the category of locally noetherian log adic spaces when the corresponding fiber products of the underlying adic spaces exist.  Moveover, finite fiber products of locally noetherian coherent log adic spaces over locally noetherian coherent log adic spaces are coherent \Pth{when defined}.  The forgetful functor from the category of locally noetherian log adic spaces to the category of locally noetherian adic spaces respects finite fiber products \Pth{when defined}.

        \item\label{prop-fiber-prod-log-adic-2}  Finite fiber products exist in the category of locally noetherian fine \Pth{\resp fs} log adic spaces when the corresponding fiber products of the underlying adic spaces exist.
    \end{enumerate}
\end{prop}
\begin{proof}
   As for \Refenum{\ref{prop-fiber-prod-log-adic-1}}, let $Y \to X$ and $Z \to X$ be morphisms of locally noetherian log adic spaces such that the fiber product $W := Y \times_X Z$ of the underlying adic spaces is defined.  Let $\pr_Y$, $\pr_Z$, and $\pr_X$ denote the natural projections from $W$ to $Y$, $Z$, and $X$, respectively, and equip $W$ with the log structure associated with the pre-log structure $\pr^{-1}_Y(\cM_Y) \oplus_{\pr_X^{-1}(\cM_X)} \pr_Z^{-1}(\cM_Z) \to \cO_{W_\et}$.  Then the log adic space thus obtained clearly satisfies the desired universal property.  Suppose moreover that $X$, $Y$, and $Z$ are all coherent.  By Proposition \ref{prop-chart-mor-exist}, \'etale locally, $Y \to X$ and $Z \to X$ admit charts $P \to Q$ and $P \to R$, respectively, where $P$, $Q$, and $R$ are all finitely generated monoids, in which case $W$ is \Pth{by construction} modeled on the finitely generated monoid $S := Q \oplus_P R$, and hence is coherent.

   As for \Refenum{\ref{prop-fiber-prod-log-adic-2}}, let $Y \to X$ and $Z \to X$ be morphisms of locally noetherian fine \Pth{\resp fs} log adic spaces such that the fiber product $Y \times_X Z$ of the underlying adic spaces is defined, in which case we equip it with the structure of a coherent log adic space as in \Refenum{\ref{prop-fiber-prod-log-adic-1}}.  Then, by Proposition \ref{prop-adj-int-sat}, $Y \times^\fine_X Z := (Y \times_X Z)^\Int$ \Pth{\resp $Y \times^\fs_X Z := (Y \times_X Z)^\Sat$} satisfies the desired universal property.
\end{proof}

\begin{rk}\label{rem-fiber-prod-chart}
    Let $P \to Q$ and $P \to R$ be homomorphisms of finitely generated \Pth{\resp fine, \resp fs} monoids, and let $S^? := (Q \oplus_P R)^?$, where $? = \emptyset$ \Pth{\resp $\Int$, \resp $\Sat$}.  Let $Y$ be a locally noetherian fs log adic space.  By Remark \ref{rem-def-chart} and Proposition \ref{prop-fiber-prod-log-adic} \Pth{and the construction in its proof}, $Y\Talg{S^?}$ is canonically isomorphic to the fiber product of $Y\Talg{Q}$ and $Y\Talg{R}$ over $Y\Talg{P}$ in the category of noetherian coherent \Pth{\resp fine, \resp fs} log adic spaces.
\end{rk}

\begin{rk}\label{rem-fiber-prod-chart-int-sat}
    Let $P \to Q$ and $P \to R$ be fine \Pth{\resp fs} charts of morphisms $Y \to X$ and $Z \to X$, respectively, of locally noetherian fine \Pth{\resp fs} log adic spaces such that $Y \times_X Z$ is defined.  Then $Y \times^\fine_X Z$ \Pth{\resp $Y \times^\fs_X Z$} is modeled on $(Q \oplus_P R)^\Int$ \Pth{\resp $(Q \oplus_P R)^\Sat$}.
\end{rk}

\begin{rk}\label{rem-fiber-prod-forget}
    The forgetful functor from the category of locally noetherian fine \Pth{\resp fs} log adic spaces to the category of locally noetherian adic spaces does not respect fiber products \Pth{when defined}, because the underlying adic spaces may change under the functor $X \mapsto X^\Int$ \Pth{\resp $X \mapsto X^\Sat$}.
\end{rk}

\begin{conv}\label{conv-fib-prod}
    From now on, all fiber products of locally noetherian fs log adic spaces are taken in the category of fs ones unless otherwise specified.  For simplicity, we shall omit the superscript \Qtn{$\fs$} from \Qtn{$\times$}.
\end{conv}

We will need the following analogue of Nakayama's \emph{Four Point Lemma} \cite[\aProp 2.2.2]{Nakayama:1997-lec}:
\begin{prop}\label{prop-lem-four-pt}
    Let $f: Y \to X$ and $g: Z \to X$ be two lft morphisms of locally noetherian fs log adic spaces, and assume that $f$ is exact.  Then, given any two points $y \in Y$ and $z \in Z$ that are mapped to the same point $x \in X$, there exists some point $w \in W := Y \times_X Z$ that is mapped to $y \in Y$ and to $z \in Z$.
\end{prop}
In order to prove Proposition \ref{prop-lem-four-pt}, it suffices to treat the case where $X$, $Y$, and $Z$ are geometric points, and where $x$, $y$ and $z$ are the respective unique closed points.  By \cite[\aLem 1.1.10]{Huber:1996-ERA}, it suffices to prove the following:
\begin{lem}\label{lem-four-pt-geom}
    Let $f: Y \to X$ and $g: Z \to X$ be morphisms of fs log adic spaces such that the underlying adic spaces of $X$, $Y$, and $Z$ are $\Spa(l, l^+)$ for the same complete separably closed nonarchimedean field $l$, and such that the underlying morphisms of adic spaces of $f$ and $g$ are the identity morphism.  Assume that $f$ is exact.  Then $W = Y \times_X Z$ is nonempty.
\end{lem}
\begin{proof}
    Let $\AC{x}$, $\AC{y}$, and $\AC{z}$ be the unique closed points of $X$, $Y$, and $Z$, respectively, and let $P = \overline{\cM}_{X, \AC{x}}$, $Q = \overline{\cM}_{Y, \AC{y}}$, and $R = \overline{\cM}_{Z, \AC{z}}$.  Let $u: P \to Q$ and $v: P \to R$ be the corresponding maps of monoids.  By \cite[\aProp I.4.2.1]{Ogus:2018-LLG}, $u$ is exact.

    Consider the homomorphism $\phi: P^\gp \to Q^\gp \oplus R^\gp: a \mapsto \bigl(u^\gp(a), -v^\gp(a)\bigr)$.  Since $u$ is exact and $R$ is sharp, $\phi^{-1}(Q \oplus R)$ is trivial.  By \cite[\aLem 2.2.6]{Nakayama:1997-lec}, the sharp monoid $S := Q \oplus_P R$ is \emph{quasi-integral} \Pth{\ie, if $a + b = a$, then $b = 0$}, and the natural homomorphism $P \to Q \oplus_P R$ is injective.

    Note that $f$ and $g$ admit charts modeled on $u: P \to Q$ and $v: P \to R$, respectively.  This is because, by the proof of \cite[\aLem 2.2.3]{Nakayama:1997-lec}, there exist compatible homomorphisms $(\cM_X(X))^\gp \to l^\times$, $(\cM_Y(Y))^\gp \to l^\times$, and $(\cM_Z(Z))^\gp \to l^\times$ such that the compositions $l^\times \to (\cM_X(X))^\gp \to l^\times$, $l^\times \to (\cM_Y(Y))^\gp \to l^\times$, and $l^\times \to (\cM_Z(Z))^\gp \to l^\times$ are the identity homomorphisms.  Therefore, the morphisms $f^*(\cM_X) \to \cM_Y$ and $g^*(\cM_X) \to \cM_Z$ can be \Pth{noncanonically} identified with $\Id \oplus u: l^\times \oplus P \to l^\times \oplus Q$ and $\Id \oplus v: l^\times \oplus P \to l^\times \oplus R$, respectively.

    Consequently, $W \cong \Spa(l, l^+) \times_{\Spa(l\Talg{S}, l^+\Talg{S})} \Spa(l\Talg{S^\Sat}, l^+\Talg{S^\Sat})$.  The image of $\Spa(l, l^+) \to \Spa(l\Talg{S}, l^+\Talg{S})$ consists of equivalence classes of valuations on $l\Talg{S}$ \Pth{bounded by $1$ on $l^+$} whose support contains the ideal $I$ of $l\Talg{S}$ generated by $\{ \mono{a}: a \in S, \, a \neq 0 \}$.  On the other hand, the kernel of $l\Talg{S} \to l\Talg{S^\Sat}$, which is generated by $\{ \mono{a} - \mono{b} : \Utext{$a, b \in S$, $a = b$ in $S^\Int$}\}$, is contained in $I$ because $S$ is quasi-integral.  Thus, $W$ is nonempty, as desired.
\end{proof}

\section{Log smoothness and log differentials}\label{sec-log-sm-log-diff}

\subsection{Log smooth morphisms}\label{sec-log-sm}

\begin{defn}\label{def-log-sm}
    Let $f: Y \to X$ be a morphism between locally noetherian fs log adic spaces.  We say that $f$ is \emph{log smooth} \Pth{\resp \emph{log \'etale}} if, \'etale locally on $Y$ and $X$, the morphism $f$ admits an fs chart $u: P \to Q$ such that
    \begin{enumerate}
        \item\label{def-log-sm-1} the kernel and the torsion part of the cokernel \Pth{\resp the kernel and cokernel} of $u^\gp: P^\gp \to Q^\gp$ are finite groups of order invertible in $\cO_X$; and

        \item\label{def-log-sm-2} $f$ and $u$ induce a morphism $Y \to X \times_{X\Talg{P}} X\Talg{Q}$ of log adic spaces \Pth{\Refcf{} Remark \ref{rem-def-chart-rel}} whose underlying morphism of adic spaces is \'etale.
    \end{enumerate}
\end{defn}

\begin{rk}\label{rem-def-log-sm}
    In Definition \ref{def-log-sm}, the fiber product in \Refenum{\ref{def-log-sm-2}} exists and the morphism $f: Y \to X$ is lft, because $X\Talg{Q} \to X\Talg{P}$ and hence the first projection $X \times_{X\Talg{P}} X\Talg{Q} \to X$ is lft when $Q$ is finitely generated.  Hence, fiber products involving log smooth or log \'etale morphisms always exist.
\end{rk}

\begin{prop}\label{prop-log-sm-bc}
    Base changes of log smooth \Pth{\resp log \'etale} morphisms \Pth{by arbitrary morphisms between locally noetherian fs log adic spaces, which are justified by Remark \ref{rem-def-log-sm}} are still log smooth \Pth{\resp log \'etale}.
\end{prop}
\begin{proof}
    Suppose that $Y \to X$ is a log smooth \Pth{\resp log \'etale} morphism of locally noetherian fs log adic spaces, with a chart $P \to Q$ satisfying the conditions in Definition \ref{def-log-sm}.  Let $Z \to X$ be any morphism of locally noetherian fs log adic spaces.  By Proposition \ref{prop-chart-mor-exist-fine}, up to \'etale localization, we may assume that $Z \to X$ admits an fs chart $P \to R$.  By Remark \ref{rem-fiber-prod-chart-int-sat}, $Z \times_X Y$ is modeled on $S := (R \oplus_P Q)^\Sat$.  By Remark \ref{rem-adj-int-sat-rel}, $Z \times_{X\Talg{P}} X\Talg{Q} \cong Z \times_{Z\Talg{R}} Z\Talg{S}$.  By Remark \ref{rem-adj-int-sat-f-et}, the morphism $Z \times_X Y \to Z \times_{Z\Talg{R}} Z\Talg{S}$ induces an \'etale morphism of underlying adic spaces.  It remains to note that $R^\gp \to \big((Q \oplus_P R)^\Sat\big)^\gp$ satisfies the analogue of Definition \ref{def-log-sm}\Refenum{\ref{def-log-sm-1}}, by the assumption on $P^\gp \to Q^\gp$ and the fact that $\big((Q\oplus_PR)^\Sat\big)^\gp \cong (Q \oplus_P R)^\gp \cong Q^\gp \oplus_{P^\gp} R^\gp$.
\end{proof}

\begin{prop}\label{prop-log-sm-chart}
    Let $f: Y \to X$ be a log smooth \Pth{\resp log \'etale} morphism of locally noetherian fs log adic spaces.  Suppose that $X$ is modeled on a global fs chart $P$.  Then, \'etale locally on $Y$ and $X$, there exists an injective fs chart $u: P \to Q$ of $f$ satisfying the conditions in Definition \ref{def-log-sm}.  Moreover, if $P$ is torsion-free, we can choose $Q$ to be torsion-free as well.
\end{prop}
\begin{proof}
    This is an analogue of the smooth and \'etale cases of \cite[\aLem 3.1.6]{Kato:1989-lddt}.

    Suppose that, \'etale locally, $f$ admits a chart $P_1 \to Q_1$ satisfying the conditions in Definition \ref{def-log-sm}.  We may assume that $X = \Spa(R, R^+)$ is a noetherian affinoid log adic space.  Let us begin with some preliminary reductions.

    Firstly, we may assume that $P_X \to \cM_X$ factors through $(P_1)_X \to \cM_X$.  Indeed, by Lemma \ref{lem-chart-fg-Ogus}, \'etale locally, $X$ admits an fs chart $P_2$ such that the canonically induced morphism $(P \oplus P_1)_X \to \cM_X$ factors through $(P_2)_X \to \cM_X$.  Let $Q_2$ be $(P_2 \oplus_{P_1} Q_1)^\Sat$.  Then $Q_2^\gp \cong P_2^\gp \oplus_{P_1^\gp} Q_1^\gp$ \Pth{\Refcf{} the proof of Proposition \ref{prop-log-sm-bc}} and hence $P_2 \to Q_2$ is also an fs chart of $f$ satisfying the conditions in Definition \ref{def-log-sm}.

    Secondly, we may assume that $P_1 \to Q_1$ is injective.  Indeed, since $P_1^\gp$ and $Q_1^\gp$ are finitely generated, and since $K := \ker(P_1^\gp \to Q_1^\gp)$ is finite, there exists some finitely generated abelian group $H_1$ fitting into a cartesian diagram
    \[
        \xymatrix{ {P_1^\gp} \ar@{^(->}[r] \ar@{->>}[d] & {H_1} \ar@{->>}[d] \\
        {P_1^\gp / K} \ar@{^(->}[r] & {Q_1^\gp} }
    \]
    such that
    \[
        \coker(P_1^\gp \Em H_1) \cong \coker(P_1^\gp / K \Em Q_1^\gp),
    \]
    and so that
    \[
        K \cong \ker(P_1^\gp \Surj P_1^\gp / K) \cong \ker(H_1 \Surj Q_1^\gp).
    \]
    For any geometric point $\AC{y}$ of $Y$, let $Q_2$ be the preimage of $Q_1$ under $H_1 \to Q_1^\gp$.  Note that $Q_2$ is fs, $Q_2^\gp = H_1$, and $P_1 \to Q_2$ is injective.  We claim that $P_1 \to Q_2$ is an fs chart of $f$, \'etale locally at $\AC{y}$ and $f(\AC{y})$, satisfying the conditions in Definition \ref{def-log-sm}.  By Remark \ref{rem-chart-stalk}, $Q_1 \to \overline{\cM}_{Y, \AC{y}}$ is surjective with kernel given by the preimage of $\cO_{Y_\et, \AC{y}}^\times$.  Since $Q_2 / K \cong Q_1$, the induced homomorphism $Q_2 \to \overline{\cM}_{Y, \AC{y}}$ satisfies the same properties, and $P_1 \to Q_2$ is an fs chart of $f$, \'etale locally at $\AC{y}$ and $f(\AC{y})$.  By construction, it satisfies the condition \Refenum{\ref{def-log-sm-1}} in Definition \ref{def-log-sm}.  It also satisfies the condition \Refenum{\ref{def-log-sm-2}} in Definition \ref{def-log-sm}, because $Y \to X \times_{X\Talg{P_1}} X\Talg{Q_2}$ is the composition of $Y \to X \times_{X\Talg{P_1}} X\Talg{Q_1} \to Y \to X \times_{X\Talg{P_1}} X\Talg{Q_2}$, and $X\Talg{Q_1} \to X\Talg{Q_2}$ is \'etale, by \cite[\aProp 1.7.1]{Huber:1996-ERA}, as $|K|$ is invertible in $\cO_X$.  Thus, the claim follows.

    Thirdly, we claim that, up to further modifying $P_1 \to Q_1$, we can find $H$ fitting into a cartesian diagram of finitely generated abelian groups:
    \begin{equation}\label{eq-prop-log-sm-chart-Niziol}
        \xymatrix{ {P^\gp} \ar@{^(->}[r] \ar@{->}[d] & {H} \ar@{->}[d] \\
        {P_1^\gp} \ar@{^(->}[r] & {Q_1^\gp.} }
    \end{equation}
    Given such an $H$, let $Q$ be the preimage of $Q_1$ via $H \to Q_1^\gp$.  Since $P$ and $P_1$ are both fs charts of $X$, the homomorphism $P^\gp \to P_1^\gp$ induces an isomorphism after passing to quotients of the source and target by the preimages of $\cO_{X_\et, f(\AC{y})}^\times$.  Since \Refeq{\ref{eq-prop-log-sm-chart-Niziol}} is Cartesian, and since $Q_1$ is an fs chart of $Y$, the analogous statement for $Q$, $Q_1$, and $\cO_{Y_\et, \AC{y}}^\times$ is also true.  Thus, $u: P \to Q$ is an injective fs chart of $f$, \'etale locally at $\AC{y}$ and $f(\AC{y})$, satisfying the conditions of Definition \ref{def-log-sm}.

    Let us verify the claim by modifying the arguments in the proofs of \cite[\aLem 3.1.6]{Kato:1989-lddt} and \cite[\aLem 2.8]{Niziol:2008-ktls-1}.  \Pth{We need to modify the arguments because our requirement that charts induce morphisms to $\cO_{X_\et}^+$ and $\cO_{Y_\et}^+$ can be affected by localizations of monoids.}  Let $\overline{G} := \im(P^\gp \to P_1^\gp)$, $G_1 := P_1^\gp / \overline{G}$, and $W := \coker(P_1^\gp \to Q_1^\gp)$, and consider the pushout $0 \to G_1 \to T_1 \to W \to 0$ of the extension $0 \to P_1^\gp \to Q_1^\gp \to W \to 0$ via $P_1^\gp \to G$.  By assumption, there exists some integer $n \geq 1$ invertible in $\cO_X$ which annihilates the torsion part of $W$.  Since $K_1 := \ker(P_1^\gp \to \overline{\cM}_{X, f(\AC{y})}^\gp)$ is finitely generated, there exists some finitely generated abelian group $K_2$ such that $K_1 \cong n K_2 = \{ n k : k \in K_2 \}$.  Let $H_2$ denote the pushout of $Q_1^\gp \leftarrow K_1 \to K_2$, which contains $Q_1^\gp$ as a finite index subgroup.  Let $P_2 := \{ a \in H_2 : n a \in P_1 \}$ and $Q_2 := \{ b \in H_2: n b \in Q_1 \}$, which are fs monoids because $P_1$ and $Q_1$ are.  Note that $P_1$, $P_2$, $Q_1$, and $Q_2$ are all submonoids of $H_2$.  Let $G_2 := P_2^\gp / \overline{G}$, and let $0 \to G_2 \to T_2 \to W \to 0$ be defined by pushout as before.  Since $n$ is invertible in $\cO_X$, the induced homomorphism $K_1 \to \cO_{X_\et, f(\AC{y})}^\times$ lifts to some homomorphism $K_2 \to \cO_{X_\et, f(\AC{y})}^\times$.  Hence, up to further \'etale localization, we may assume that $P_1 \to Q_1$ lifts to an fs chart $P_2 \to Q_2$ of $f$ at $\AC{y}$ and $f(\AC{y})$, which still satisfies the conditions in Definition \ref{def-log-sm}.  Given any torsion element $w$ of $W$ of order $m$ \Pth{which necessarily divides $n$}, let $t_1 \in T_1$ be any lifting of $w$.  Then $g_1 := m t_1 \in G_1 = \coker(P^\gp \to P_1^\gp)$.  Since $P^\gp \to P_1^\gp / K_1 \Mi \overline{\cM}_{X, f(\AC{y})}^\gp$ is surjective \Pth{by Remark \ref{rem-chart-stalk} again}, $g_1$ lifts to some $k_1 \in K_1$, which is the $m$-th multiple of some $k_2 \in K_2$ with image $g_2$ in $G_2$.  Then $t_2 := t_1 - g_2 \in T_2$ is a lifting of $w$ which satisfies $m t_2 = g_1 - m g_2 = 0$.  Hence, $\bZ t_2 \subset T_2$ defines a lifting of $\bZ w \subset W$.  Since $w$ is arbitrary, the homomorphism $T_2 \to W$ of finitely generated abelian groups splits, and the preimage $\overline{H}$ of the split image of $W$ in $Q_2^\gp$ is an extension $0 \to \overline{G} \to \overline{H} \to W \to 0$ whose pushout via $\overline{G} \to P_1^\gp$ recovers $0 \to P_1^\gp \to Q_1^\gp \to W \to 0$.  Since $\overline{H}$ is finitely generated, there is some surjection $F \Surj \overline{H}$ from a finitely generated free abelian group, and the preimage $E$ of $\overline{G}$ is also finitely generated free and lifts to some $E \to P^\gp$.  Then the claim follows by taking $H$ to be the pushout of $P^\gp \leftarrow E \to F$.

    Finally, if $P$ is torsion-free, let us show that we can take $Q$ to be torsion-free as well.  We learned the following argument from \cite[\aProp A.2]{Nakayama:1998-nclsf}.  Consider the torsion submonoid $Q_\tor$ of $Q$, which is necessarily contained in $Q^\inv$; and choose any splitting $s$ of $\pi: Q \Surj Q' := Q / Q_\tor$.  Let $n$ be any integer invertible in $\cO_X$ which annihilates the torsion in $\coker(u^\gp)$.  Since $P$ is torsion-free, the composition $u' = \pi \circ u: P \to Q'$ is injective, and $Q_\tor$ is also annihilated by $n$.  Let $S$ be the finite \'etale $R$-algebra obtained from $R\Talg{Q_\tor}$ by formally joining the $n$-th roots of $\mono{a}$, for all $a \in Q_\tor$; and let $S^+$ be the integral closure of $R^+\Talg{Q_\tor}$ in $S$.  Then the morphism $Z := \Spa(S, S^+) \to X\Talg{Q_\tor} = \Spa(R\Talg{Q_\tor}, R^+\Talg{Q_\tor})$ over $X$ is finite \'etale and surjective, with base change $Z\Talg{Q'} \to X\Talg{Q}$.  Consider the composition $v = s \circ \pi \circ u: P \to Q$.  Then $u - v: P \to Q$ factors through $P \to Q_\tor$, which extends to some $\phi: Q' \to S^\times_\tor$; and $a \mapsto \phi(a) a$, for $a \in Q'$, induces an isomorphism between the two compositions $g, h: Z\Talg{Q'} \to X\Talg{Q} \to X\Talg{P}$ induced by $u, v$, respectively.  Since $u: P \to Q$ is a chart of $f$, the induced morphism $Y \to X \times_{X\Talg{P}} X\Talg{Q}$ is \'etale, whose pullback is an \'etale morphism $Y \times_{X\Talg{Q_\tor}} Z \to X \times_{X\Talg{P}, g} Z\Talg{Q'}$.  The target is isomorphic to $X \times_{X\Talg{P}, h} Z\Talg{Q'}$, and hence is \'etale over $X \times_{X\Talg{P}} X\Talg{Q'}$, by the above explanation.  Consequently, the morphism $Y \to X \times_{X\Talg{P}} X\Talg{Q'}$ induced by $f$ and $u': P \to Q'$ is \'etale, and so $u'$ is also an injective fs chart of $f$, as desired.
\end{proof}

\begin{prop}\label{prop-log-sm-compos}
    Compositions of log smooth \Pth{\resp log \'etale} morphisms are still log smooth \Pth{\resp log \'etale}.
\end{prop}
\begin{proof}
    This follows from Definition \ref{def-log-sm} and Proposition \ref{prop-log-sm-chart}.
\end{proof}

\begin{prop}\label{prop-log-sm-strict}
    If $f: Y \to X$ is log smooth \Pth{\resp log \'etale} and strict, then the underlying morphism of adic spaces is smooth \Pth{\resp \'etale}.
\end{prop}
\begin{proof}
    \'Etale locally at geometric points $\AC{y}$ of $Y$ and $f(\AC{y})$ of $X$, by Propositions \ref{prop-chart-stalk-chart-fs} and \ref{prop-log-sm-chart}, we may assume that $f: Y \to X$ admits an injective fs chart $u: P = \overline{\cM}_{X, f(\AC{y})} \to Q$ as in Definition \ref{def-log-sm}, where the torsion part $K_\tor$ of $K := \coker(u^\gp: P^\gp \to Q^\gp)$ is a finite group of order invertible in $\cO_X$, and where $K$ itself is finite when $f$ is log \'etale.  Since $f$ is strict, by Remark \ref{rem-stalk-strict}, $P = \overline{\cM}_{X, f(\AC{y})} \cong \overline{\cM}_{Y, \AC{y}}$.  Hence, we can identify $K$ with $\ker(Q^\gp \to \overline{\cM}_{Y, \AC{y}}^\gp)$, so that $u^\gp(P^\gp) \cap K = 0$ in $Q^\gp$.  Since $K$ is a finitely generated abelian group, we have a decomposition $K_\tor \oplus \bigl(\oplus_{i = 1}^r \bZ a_i\bigr) \Mi K$, for some elements $a_i \in K$ which are necessarily mapped to $\cO_{Y_\et, \AC{y}}^\times$.  Up to replacing $a_i$ with $- a_i$, for each $1 \leq i \leq r$, we may assume that $a_i$ is mapped to $\cO_{Y_\et, \AC{y}}^+$.  Let $Q' := u(P) \oplus K_\tor \oplus \bigl(\oplus_{i = 1}^r \bZ_{\geq 0} a_i\bigr)$ in $Q^\gp$.  Then $Q^\gp \to \cM_{Y, \AC{y}}^\gp$ maps $Q'$ to $\cM_{Y, \AC{y}}$, the induced map $Q' \to \cO_{Y_\et, \AC{y}}$ factors through $\cO_{Y_\et, \AC{y}}^+$, and the induced map $Q' \to \overline{\cM}_{Y, \AC{y}}$ is surjective.  In this case, up to further \'etale localization, $u': P \to Q'$ is also an injective fs chart of $f$.  Thus, it suffices to show that $X\Talg{Q'} \to X\Talg{P}$ is smooth \Pth{\resp \'etale} at the image of $\AC{y}$.  Since $X\Talg{Q'} \cong X\Talg{P} \times_X X\Talg{K_\tor} \times_X X\Talg{\bZ_{\geq 0}^r}$ over $X\Talg{P}$, it remains to note that, by \cite[\aCor 1.6.10 and \aProp 1.7.1]{Huber:1996-ERA}, $X\Talg{K_\tor} \times_X X\Talg{\bZ_{\geq 0}^r}$ is smooth \Pth{\resp \'etale} over $X$, because $K_\tor$ is a finite groups of order invertible in $\cO_X$, and because $r = 0$ when $K$ itself is finite \Pth{\ie, when $f$ is log \'etale}.
\end{proof}

\begin{defn}\label{def-str-smooth}
    If $f$ satisfies the condition in Proposition \ref{prop-log-sm-strict}, we say that $f$ is \emph{strictly smooth} \Pth{\resp \emph{strictly \'etale}}, or simply \emph{smooth} \Pth{\resp \emph{\'etale}}, when the context is clear.
\end{defn}

\begin{defn}\label{def-log-sm-base-field}
    Let $(k, k^+)$ be an affinoid field.  A locally noetherian fs log adic space $X$ is called \emph{log smooth over} $\Spa(k, k^+)$ if there is a log smooth morphism $X \to \Spa(k, k^+)$, where $\Spa(k, k^+)$ is endowed with the trivial log structure.  When $X$ is log smooth \Pth{\resp smooth} over $\Spa(k, \cO_k)$, we simply say that $X$ is \emph{log smooth} \Pth{\resp \emph{smooth}} \emph{over $k$}.
\end{defn}

Local structures of log smooth log adic spaces can be described by \emph{toric charts}, by the following proposition:
\begin{prop}\label{prop-toric-chart}
    Let $X$ be an fs log adic space log smooth over $\Spa(k, k^+)$, where $(k, k^+)$ is an affinoid field.  Then, \'etale locally on $X$, there exist a sharp fs monoid $P$ and a strictly \'etale morphism $X \to \Spa(k\Talg{P}, k^+\Talg{P})$ that is a composition of rational localizations and finite \'etale morphisms.
\end{prop}
\begin{proof}
    By Proposition \ref{prop-log-sm-chart}, \'etale locally on $X$, there exists a torsion-free fs monoid $Q$ and a strictly \'etale morphism $X \to \Spa(k\Talg{Q}, k^+\Talg{Q})$.  We may further assume that $X \to \Spa(k\Talg{Q}, k^+\Talg{Q})$ is a composition of rational localizations and finite \'etale morphisms.  By Lemma \ref{lem-monoid-split}, $Q \to \overline{Q}$ splits, and hence there is a decomposition $Q \Mi \overline{Q} \oplus Q^\inv \Mi \overline{Q} \oplus \bZ^r$, for some $r$.  Let $P := \overline{Q} \oplus \bZ_{\geq 0}^r$, which is a sharp fs monoid.  Since $P \to Q$ is a localization of monoids, as in Construction \ref{constr-monoid-loc}, $\Spa(k\Talg{Q}, k^+\Talg{Q}) \to \Spa(k\Talg{P}, k^+\Talg{P})$ is a rational localization, whose pre-composition with $X \to \Spa(k\Talg{Q}, k^+\Talg{Q})$ gives the desired morphism.
\end{proof}

\begin{cor}\label{cor-sm-toric-chart}
    Let $X$ and $(k, k^+)$ be as in Proposition \ref{prop-toric-chart}.  Suppose moreover that underlying adic space of $X$ is smooth over $\Spa(k, k^+)$.  Then, \'etale locally on $X$, there exists a strictly \'etale morphism $X \to \bD^n$ \Pth{see Example \ref{ex-log-adic-sp-disc}} that is a composition of rational localizations and finite \'etale morphisms.
\end{cor}
\begin{proof}
    As in the proof of Proposition \ref{prop-toric-chart}, \'etale locally on $X$, there is a torsionfree fs monoid $Q$ and a strictly \'etale morphism $X \to Y := \Spa(k\Talg{Q}, k^+\Talg{Q})$ that is a composition of rational localizations and finite \'etale morphisms.  Consider the canonical morphism $Y \to Z := \Spec(k[Q])$.  Note that $Z = \Spec(k[Q])$ admits a stratification by locally closed subschemes of the form $Z_F := \Spec(k[F^\gp])$, where $F$ are the faces of $Q$ \Pth{see \cite[\aSecs I.1.4 and I.3.4]{Ogus:2018-LLG}}, which is a closed subscheme of the open subscheme $Z_{(F)} := \Spec(k[Q_F])$ of $Z$, where $Q_F$ denotes the localization of $Q$ with respect to $F$ \Pth{as in Construction \ref{constr-monoid-loc}}.  By \cite[\aSec I.3.6]{Ogus:2018-LLG}, given any closed point $z$ of $Z_F$ with residue field $\kappa(z)$, the completion $\cO_{Z, z}^\wedge$ of the local ring $\cO_{Z, z}$ is isomorphic to $\kappa(z)[[Q_F]]$.

    For each $F$, let $Y_F := Y \times_Z Z_F$ and $Y_{(F)} := Y \times_Z Z_{(F)}$.  Then we have a canonical open immersion of adic spaces $Y_{(F)} \to \Spa(k, k^+) \times_{\Spec(k)} Z_{(F)}$, by comparing the construction of both sides using Lemma \ref{lem-monoid-alg-Huber}.  \Pth{For the construction of such fiber products of adic spaces with schemes, see \cite[\aProp 3.8 and its proof]{Huber:1994-gfsra}.}  For any $F$ such that $Y_F$ meets the image of $X \to Y$, we claim that $Q_F \cong \bZ_{\geq 0}^s \oplus \bZ^t$, for some $s$ and $t$.  Assuming this claim, then $Z_{(F)}$ admits an open immersion into $(\bP^1_k)^n$, for $n := s + t$, and we obtain an open immersion $Y_{(F)} \to \Spa(k, k^+) \times_{\Spec(k)} (\bP^1_k)^n$ of log adic spaces, where the log structure on the target is defined \Pth{via fiber product} by naturally covering each factor $\Spa(k, k^+) \times_{\Spec(k)} \bP^1_k$ with two log adic spaces $\Spa(k\Talg{T}, k^+\Talg{T})$ and $\Spa(k\Talg{T^{-1}}, k^+\Talg{T^{-1}})$ isomorphic to $\bD$ \Pth{see Example \ref{ex-log-adic-sp-disc}}.  Therefore, up to further localization on $X$, we may assume that $X \to Y$ extends to a strictly \'etale morphism $X \to \bD^n$, which is still a composition of rational localizations and finite \'etale morphisms.  Thus, the corollary follows from the claim.

    It remains to verify the claim.  Since it only concerns monoids $Q_F$ as above, we may base change to $\Spa(k, \cO_k)$, and assume that $k^+ = \cO_k$, so that $X$ is a smooth rigid analytic variety over $k$.  For any $F$ such that $Y_F$ meets the image of the \'etale morphism $X \to Y$, since $X$ is smooth over $k$, and since the open immersion $Y_{(F)} \to \Spa(k, k^+) \times_{\Spec(k)} Z_{(F)} \cong Z_{(F)}^\an$ maps $Y_F$ to $Z_F^\an$, we see that $Z$ is smooth over $k$ at some closed point $z$ of $Z_F$, so that $\cO_{Z, z}^\wedge \cong \kappa(z)[[Q_F]]$ \Pth{as explained above} is regular.  Since the localization $Q_F$ is torsionfree fs as $Q$ is, by decomposing $Q_F \Mi \overline{Q_F} \oplus Q_F^\inv \Mi \overline{Q_F} \oplus \bZ^t$, for some $t$, as in the proof of Proposition \ref{prop-toric-chart}, we obtain $\kappa(z)[[Q_F]] \Mi \kappa(z)[[\overline{Q_F}]][[T_1, \ldots, T_t]]$.  Hence, the regularity of $\kappa(z)[[Q_F]]$ implies that of $\kappa(z)[[\overline{Q_F}]]$.  Since $\overline{Q_F}$ is fine and sharp, by \cite[\aLem I.1.11.7]{Ogus:2018-LLG}, we have $\overline{Q_F} \cong \bZ_{\geq 0}^s$, for some $s$, and the claim follows.
\end{proof}

\begin{defn}\label{def-toric-chart}
    A strictly \'etale morphism $X \to \Spa(k\Talg{P}, k^+\Talg{P})$ as in Proposition \ref{prop-toric-chart} is called a \emph{toric chart}.  A strictly \'etale morphism $X \to \bD^n$ as in Corollary \ref{cor-sm-toric-chart} is called a \emph{smooth toric chart}.
\end{defn}

\begin{exam}\label{ex-log-adic-sp-ncd-chart}
   Let $X$, $D$, and $k$ be as in Example \ref{ex-log-adic-sp-ncd}.  We claim that, \'etale locally, $X$ admits a smooth toric chart $X \to \bD^n$, where $n = \dim(X)$.  In order to see this, we may assume that, up \'etale localization, $X$ is $S \times \bD^m$ as in Example \ref{ex-log-adic-sp-ncd}, and that there is a morphism \Pth{of adic spaces with trivial log structures} $S \to \bT^{n - m} = \Spa(k\Talg{T_1^\pm, \ldots, T_{n - m}^\pm}, \cO_k\Talg{T_1^\pm, \ldots, T_{n - m}^\pm})$ that is a composition of finite \'etale morphisms and rational localizations.  Then the composition of $X = S \times \bD^m \to \bT^{n - m} \times \bD^m \Em \bD^{n - m} \times \bD^m \cong \bD^n$ is a desired smooth toric chart.  In particular, $X$ is log smooth over $k$.
\end{exam}

\subsection{Log differentials}\label{sec-log-diff}

In this subsection, we develop a theory of log differentials from scratch.  We first introduce log structures and log differentials for Huber rings.

\begin{defn}\phantomsection\label{def-log-Huber}
    \begin{enumerate}
        \item A \emph{pre-log Huber ring} is a triple $(A, M, \alpha)$ consisting of a \Pth{not necessarily complete} Huber ring $A$, a monoid $M$, and a homomorphism $\alpha: M \to A$ of multiplicative monoids called a \emph{pre-log structure}.  We sometimes denote a pre-log Huber ring just by $(A, M)$, when the pre-log structure $\alpha$ is clear from the context.

        \item A \emph{log Huber ring} is a pre-log Huber ring $(A, M, \alpha)$ where $A$ is \emph{complete} and where the induced homomorphism $\alpha^{-1}(A^\times) \to A^\times$ is an isomorphism.  In this case, $\alpha$ is called a \emph{log structure}.

        \item Given a pre-log Huber ring $(A, M, \alpha)$, let us still denote by $\alpha$ the composition of $M \Mapn{\alpha} A \Mapn{\can} \widehat{A}$, where $\widehat{A}$ denotes the completion of $A$.  Then we define the \emph{associated log Huber ring} to be $(\widehat{A}, {^a}M, \widehat{\alpha})$, where ${^a}M$ is the pushout of $\widehat{A}^\times \leftarrow \alpha^{-1}(\widehat{A}^\times) \rightarrow M$ in the category of monoids, which is equipped with the canonical homomorphism $\widehat{\alpha}: {^a}M \to \widehat{A}$.  In this case, $\widehat{\alpha}$ is called the \emph{associated log structure}.

        \item A homomorphism $f: (A, M, \alpha) \to (B, N, \beta)$ of pre-log Huber rings consists of a continuous homomorphism $f: A \to B$ of Huber rings and a homomorphism of monoids $f_\sharp: M \to N$ such that $\beta \circ f_\sharp = f \circ \alpha$.  In this case, we have a canonically induced morphism $(B, M, \beta \circ f_\sharp) \to (B, N, \beta)$ of pre-log Huber rings, and we say that $f$ is \emph{strict} if the associated morphism of log Huber rings is an isomorphism.  In general, any homomorphism $f: (A, M) \to (B, N)$ of log Huber rings factors as $(A, M) \to \bigl(B, f_*(M)\bigr) \to (B, N)$.
    \end{enumerate}
\end{defn}

\begin{defn}\label{def-der-mod}
    Let $f: (A, M, \alpha) \to (B, N, \beta)$ be a homomorphism of pre-log Huber rings.  Given any complete topological $B$-module $L$, a \emph{derivation} from $(B, N, \beta)$ to $L$ over $(A, M, \alpha)$ \Pth{or an \emph{$(A, M, \alpha)$-derivation} of $(B, N, \beta)$ to $L$} consists of a continuous $A$-linear derivation $d: B \to L$ and a homomorphism of monoids $\delta: N \to L$ such that $\delta\bigl(f_\sharp(m)\bigr) = 0$ and $d\bigl(\beta(n)\bigr) = \beta(n) \, \delta(n)$, for all $m \in M$ and $n \in N$.  We denote the set of all $(A, M, \alpha)$-derivations from $(B, N, \beta)$ to $L$ by $\Der^{\log}_A(B, L)$.  It has a natural $B$-module structure induced by that of $L$.  If $M = \alpha^{-1}(A^\times)$ and $N = \beta^{-1}(B^\times)$, then $\Der^{\log}_A(B, L)$ is simply $\Der_A(B, L)$, the usual $B$-module of continuous $A$-derivations from $B$ to $L$, and we shall omit the superscript \Qtn{$\log$} from the notation.
\end{defn}

\begin{rk}\label{rem-def-der-mod}
    In Definition \ref{def-der-mod}, $(d, \delta)$ naturally extends to a log derivation on $(B, {^a}N, \beta)$, and the $B$-module $\Der^{\log}_A(B, L)$ remains unchanged if we replace $(B, N, \beta)$ with $(B, {^a}N, \beta)$.  In addition, $\delta$ naturally extends to a group homomorphism $\delta^\gp: ({^a}N)^\gp \to L$.
\end{rk}

\begin{defn}\label{def-log-Huber-tft}
    A homomorphism $f: (A, M, \alpha) \to (B, N, \beta)$ of pre-log Huber rings is called \emph{topologically of finite type} \Pth{or \emph{tft} for short} if $A$ and $B$ are complete, $f: A \to B$ is topologically of finite type \Pth{as in \cite[\aSec 3]{Huber:1994-gfsra}}, and $N^\gp / \bigl((f_*(M))^\gp \beta^{-1}(B^\times)\bigr)$ is a finitely generated abelian group.
\end{defn}

Now, let $f: (A, M, \alpha) \to (B, N, \beta)$ be tft, as in Definition \ref{def-log-Huber-tft}.  Consider the monoid algebra $(B \ho_A B)[N]$ over $B \ho_A B$ associated with the monoid $N$, and for each $n \in N$, its element $\mono{n}$ corresponding to $n$ \Pth{by our convention}.  Let $I$ be its ideal generated by $\{ \mono{f_\sharp(m)} - 1 \}_{m \in M}$ and $\{ (\beta(n) \otimes 1) - (1 \otimes \beta(n)) \, \mono{n} \}_{n \in N}$.  Note that, if $n \in \beta^{-1}(B^\times)$, then $\mono{n} = \beta(n) \otimes \beta(n)^{-1}$ in $ \bigl( (B \ho_A B)[N] \bigr)/I$.  Let $J$ be the kernel of the homomorphism
\begin{equation}\label{eq-def-log-diff-aff-diag}
    \Delta_{\log}: \bigl( (B \ho_A B)[N] \bigr) / I \to B
\end{equation}
sending $b_1 \otimes b_2$ to $b_1 b_2$ and all $\mono{n}$ to 1.  We set
\begin{equation}\label{eq-def-log-diff-aff}
    \Omega^{\log}_{B / A} := J / J^2,
\end{equation}
and define $d_{B / A}: B \to \Omega^{\log}_{B / A}$ and $\delta_{B / A}: N \to \Omega^{\log}_{B / A}$ by setting
\[
    d_{B / A}(b) = \overline{(b \otimes 1) - (1 \otimes b)}
\]
and
\[
    \delta_{B / A}(n) = \overline{\mono{n} - 1}.
\]
A short computation shows that $d_{B / A}$ is an $A$-linear derivation, and that $\delta_{B / A}$ is a homomorphism of monoids satisfying the required properties in Definition \ref{def-der-mod}.

As observed in Remark \ref{rem-def-der-mod}, $\delta_{B / A}$ naturally extends to a group homomorphism $\delta^\gp_{B / A}: N^\gp \to \Omega^{\log}_{B / A}$ such that $\delta_{B / A}^\gp\bigl(f_\sharp^\gp(M^\gp)\bigr) = 0$.  Then $\Omega^{\log}_{B / A}$ is generated as a $B$-module by $\ker(B \ho_A B \to B)$ and $\{ \delta^\gp_{B/A}(n) \}$, where $n$ runs through a set of representatives of generators of $N^\gp / \bigl((f_*(M))^\gp \beta^{-1}(B^\times)\bigr)$.  More precisely,
\begin{equation}\label{eq-def-log-diff-aff-quot}
    \Omega^{\log}_{B / A} \cong \bigl(\Omega_{B / A} \oplus ( B \otimes_\bZ N^\gp ) \bigr) \big/ R,
\end{equation}
where $\Omega_{B / A}$ is the usual $B$-module of continuous differentials \Pth{see \cite[\aDef 1.6.1 and (1.6.2)]{Huber:1996-ERA}}, and where $R$ is the $B$-module generated by
\begin{equation}\label{eq-def-log-diff-aff-quot-mod}
    \{ (d\beta(n), -\beta(n) \otimes n) : n \in N \} \cup \{ (0, 1 \otimes f_\sharp(m)) : m \in M \}.
\end{equation}
In particular, $\Omega^{\log}_{B / A}$ is a finite $B$-module.  Therefore, $\Omega^{\log}_{B / A}$ is complete with respect to its natural $B$-module topology, and $d_{B / A}$ is continuous.

\begin{prop}\label{prop-log-diff-univ}
    Under the above assumption, $(\Omega^{\log}_{B / A}, d_{B / A}, \delta_{B / A})$ is a universal object among all $(A, M, \alpha)$-derivations of $(B, N, \beta)$.
\end{prop}
\begin{proof}
    Let $(d, \delta)$ be a derivation from $(B, N, \beta)$ to some complete topological $B$-module $L$ over $(A, M, \alpha)$.  We turn the $B$-module $B \oplus L$ into a complete topological $B$-algebra, which we denote by $B * L$, with the multiplicative structure defined by $(b_1, x_1) \, (b_2, x_2) = (b_1 b_2, b_1 x_2 + b_2 x_1)$.  Note that the $A$-linear derivation $d$ gives rise to a continuous homomorphism of topological $B$-algebras $B \ho_A B \to B * L$ sending $b_1 \otimes b_2$ to $\bigl(b_1 b_2, b_1 d(b_2)\bigr)$.  This extends to a homomorphism $(B \ho_A B)[N] \to B * L$ by sending $\mono{n}$ to $\bigl(1, \delta(n)\bigr)$, for each $n \in N$.  By the conditions in Definition \ref{def-der-mod}, this homomorphism factors through $\bigl((B \ho_A B)[N]\bigr)/I$, inducing a homomorphism $\bigl((B \ho_A B)[N]\bigr) / I \to B * L$ which we denote by $\varphi$.  By construction, the composition of $\varphi$ with the natural projection $B * L \to B$ recovers \Refeq{\ref{eq-def-log-diff-aff-diag}}.  Therefore, $\varphi$ induces a continuous morphism of $B$-modules $\overline{\varphi}: \Omega^{\log}_{B / A} = J / J^2 \to L$.  Now, a careful chasing of definitions verifies that $\overline{\varphi} \circ d_{B / A} = d$ and $\overline{\varphi} \circ \delta_{B / A} = \delta$, as desired.
\end{proof}

Given any complete topological $B$-module $L$, there is a natural forgetful functor $\Der_A^{\log}(B, L) \to \Der_A(B, L)$ defined by $(d, \delta) \mapsto d$.  The following lemma is obvious:
\begin{lem}\label{lem-log-diff-strict}
    If $f: (A, M, \alpha) \to (B, N, \beta)$ is a strict homomorphism of log Huber rings, then the canonical morphism $\Der_A^{\log}(B, L) \to \Der_A(B, L)$ is an isomorphism, for every complete topological $B$-module $L$.  Consequently, the canonical morphism $\Omega_{B / A} \to \Omega^{\log}_{B / A}$ is an isomorphism.
\end{lem}

\begin{defn}\label{def-log-Huber-thick}
    A homomorphism $(D, T, \mu) \to (D', T', \mu')$ of log Huber rings is called a \emph{log thickening of first order} if it satisfies the following conditions:
    \begin{enumerate}
        \item\label{def-log-Huber-thick-1} The underlying homomorphism $D \to D'$ of Huber rings is surjective, whose kernel $H$ is a closed ideal satisfying $H^2 = 0$.

        \item\label{def-log-Huber-thick-2} The log structure $\mu': T' \to D'$ is canonically induced by $\mu: T \to D$.

        \item\label{def-log-Huber-thick-3} The subgroup $1 + H$ of $D^\times \cong T^\inv$ \Pth{via the log structure $\mu$} acts freely on $T$, and induces an isomorphism $T / (1 + H) \Mi T'$ of monoids.
    \end{enumerate}
\end{defn}

\begin{rk}\label{rem-def-log-Huber-thick}
    The condition \Refenum{\ref{def-log-Huber-thick-3}} in Definition \ref{def-log-Huber-thick} is automatic when $T^\inv$ acts freely on $T$; or, equivalently, when $T^\inv \to T^\gp$ is injective.  \Pth{In this case, $T$ is \emph{$u$-integral}, as in \cite[\aDef I.1.3.1]{Ogus:2018-LLG}.}  In particular, the condition \Refenum{\ref{def-log-Huber-thick-3}} is satisfied when $T$ is \emph{integral} \Pth{\ie, $T \to T^\gp$ is injective}.
\end{rk}

Consider the following commutative diagram
\begin{equation}\label{eq-log-Huber-thick}
    \xymatrix{ {(A, M, \alpha)} \ar_-f[d] \ar[r] & {(D, T, \mu)} \ar[d] \\
    {(B, N, \beta)} \ar^-g[r] \ar@{.>}[ru]^(.4){\widetilde{g}} & {(D', T', \mu')} }
\end{equation}
of solid arrows given by homomorphisms of log Huber rings, in which the arrow $(D, T, \mu) \to (D', T', \mu')$ is a log thickening of first order as in Definition \ref{def-log-Huber-thick}.

\begin{defn}\label{def-log-Huber-formal-log-sm-unram-et}
    A homomorphism $f: (A, M, \alpha) \to (B, N, \beta)$ of log Huber rings is called \emph{formally log smooth} \Pth{\resp \emph{formally log unramified}, \resp \emph{formally log \'etale}} if, for any diagram as in \Refeq{\ref{eq-log-Huber-thick}}, there exists at least one \Pth{\resp at most one, \resp exactly one} lifting $\widetilde{g}: (B, N, \beta) \to (D, T, \mu)$ of $g$, as the dotted arrow in \Refeq{\ref{eq-log-Huber-thick}}, making the whole diagram commute.  If $M = \alpha^{-1}(A^\times)$ and $N = \beta^{-1}(B^\times)$, then we simply say that $f: A \to B$ \Pth{the underlying ring homomorphism of Huber rings} is \emph{formally smooth} \Pth{\resp \emph{formally unramified}, \resp \emph{formally \'etale}} \Pth{\Refcf{} the formal lifting conditions in \cite[\aDef 1.6.5]{Huber:1996-ERA}}.
\end{defn}

\begin{rk}\label{rem-rigid-formal-sm-unram-et}
    Let $k$ be a nontrivial nonarchimedean field.  By \cite[\aProp 1.7.11]{Huber:1996-ERA}, a tft homomorphism $f: A \to B$ of Tate $k$-algebras is formally smooth \Pth{\resp formally unramified, \resp formally \'etale} if and only if the induced morphism $\Spa(B, B^\circ) \to \Spa(A, A^\circ)$ is smooth \Pth{\resp unramified, \resp \'etale} in the sense of classical rigid analytic geometry.
\end{rk}

\begin{rk}\label{rem-def-log-Huber-formal-log-sm-unram-et-compos-bc}
    It follows easily from the definition that we have the following:
    \begin{enumerate}
        \item\label{rem-def-log-Huber-formal-log-sm-unram-et-compos-bc-1} Formally log smooth \Pth{\resp formally unramified, \resp formally \'etale} homomorphisms are stable under compositions and completed base changes.

        \item\label{rem-def-log-Huber-formal-log-sm-unram-et-compos-bc-2} If a homomorphism $f: (A, M, \alpha) \to (B, N, \beta)$ is formally log \'etale, then a homomorphism $g: (B, N, \beta) \to (C, O, \gamma)$ is formally log smooth \Pth{\resp formally log \'etale} if and only if $g \circ f: (A, M, \alpha) \to (C, O, \gamma)$ is.
    \end{enumerate}
\end{rk}

\begin{lem}\label{lem-def-log-Huber-formal-log-sm-unram-et-strict}
    If the homomorphism $f: (A, M, \alpha) \to (B, N, \beta)$ is a strict homomorphism of log Huber rings, then it is \emph{formally log smooth} \Pth{\resp \emph{formally log unramified}, \resp \emph{formally log \'etale}} if the underlying homomorphism of Huber rings $f: A \to B$ is \emph{formally smooth} \Pth{\resp \emph{formally unramified}, \resp \emph{formally \'etale}}.  If $A^\times \cong M^\inv \to M^\gp$ is injective, then the converse is true.
\end{lem}
\begin{proof}
    Suppose we are given any diagram as in \Refeq{\ref{eq-log-Huber-thick}}, with the top horizontal row denoted by $h: (A, M, \alpha) \to (D, T, \mu)$ in this proof.  Since the homomorphism $f: (A, M, \alpha) \to (B, N, \beta)$ is strict, each $n \in N$ is of the form $n = b + f_\sharp(m)$ for some $b \in B^\times$ and $m \in M$, where $m$ is uniquely determined by $n$ up to an element of $A^\times$.  Hence, any homomorphism $\widetilde{g}: B \to D$ of Huber rings lifting $g: B \to D'$ uniquely extends to a homomorphism $\widetilde{g}: (B, N, \beta) \to (D, T, \mu)$ of log Huber rings lifting $g: (B, N, \beta) \to (D', T', \mu')$, by setting $\widetilde{g}_\sharp(n) = \widetilde{g}(b) + h_\sharp(m)$, for each $n = b + f_\sharp(m) \in N$ as above.  Hence, the formal lifting properties without log structures imply those with log structures.

    Conversely, suppose we are given a diagram as in \Refeq{\ref{eq-log-Huber-thick}}, but without any log structures.  Nevertheless, we can \emph{define} $\mu: T \to D$ and $\mu': T' \to D'$ to be the log structures associated with $M \Mapn{\alpha} A \Mapn{f} D$ and $N \Mapn{\beta} B \Mapn{g} D'$, respectively.  Since $A^\times \cong M^\inv \to M^\gp$ is injective, $A^\times$ acts freely on $M$.  Therefore, by choosing any set-theoretic section of $M \to \overline{M}$, we obtain a bijection $A^\times \times \overline{M} \to M$ compatible with the actions of $A^\times$ \Pth{on $A^\times$ and $M$}, which induces a bijection $D^\times \times \overline{M} \to T$ compatible with the actions of $D^\times$ \Pth{on $D^\times$ and $T$}.  As a result, $1 + H \subset D^\times \cong T^\inv$ \Pth{via $\mu$} acts freely on $T$, and $(D, T, \mu) \to (D', T', \mu')$ is a log thickening of first order, as in Definition \ref{def-log-Huber-thick}.  Thus, we obtain a full diagram as in \Refeq{\ref{eq-log-Huber-thick}}, with log structures, and the formal lifting properties with log structures imply those without, as desired.
\end{proof}

We have the \emph{first fundamental exact sequence} for log differentials, as follows:
\begin{thm}\phantomsection\label{thm-log-diff-fund}
    \begin{enumerate}
        \item\label{thm-log-diff-fund-1}  A composition $(A, M, \alpha) \Mapn{f} (B, N, \beta) \Mapn{g} (C, O, \gamma)$ of tft homomorphisms of log Huber rings leads to an exact sequence
            \[
                C \otimes_B \Omega^{\log}_{B / A} \to \Omega^{\log}_{C / A} \to \Omega^{\log}_{C / B} \to 0
            \]
            of finite topological $C$-modules \Pth{\Refcf{} \cite[\aProp 1.6.3]{Huber:1996-ERA}}, where the first map sends $c \otimes d_{B / A}(b)$ and $c \otimes \delta_{B / A}(n)$ to $c \, d_{C / A}\bigl(g(b)\bigr)$ and $c \, \delta_{C / A}\bigl(g_\sharp(n)\bigr)$, respectively, and the second map sends $d_{C / A}(c)$ and $\delta_{C / A}(l)$ to $d_{C / B}(c)$ and $\delta_{C / B}(l)$, respectively.

        \item\label{thm-log-diff-fund-2}  If the homomorphism $g: (B, N, \alpha) \to (C, O, \gamma)$ is formally log smooth, then $C \otimes_B \Omega^{\log}_{B / A} \to \Omega^{\log}_{C / A}$ is injective, and the short exact sequence
            \[
                0 \to C \otimes_B \Omega^{\log}_{B / A} \to \Omega^{\log}_{C / A} \to \Omega^{\log}_{C / B} \to 0
            \]
            is split in the category of topological $C$-modules.

        \item\label{thm-log-diff-fund-3}  If $g$ is formally log unramified, then $\Omega^{\log}_{C / B} = 0$.

        \item\label{thm-log-diff-fund-4}  If $g$ is formally log \'etale, then $\Omega^{\log}_{C / A} \cong C \otimes_B \Omega^{\log}_{B / A}$.

        \item\label{thm-log-diff-fund-5}  If $g \circ f$ is formally log smooth, then the converses of \Refenum{\ref{thm-log-diff-fund-2}}, \Refenum{\ref{thm-log-diff-fund-3}}, and \Refenum{\ref{thm-log-diff-fund-4}} hold.
    \end{enumerate}
\end{thm}
\begin{proof}
    Since the homomorphisms of log Huber rings are all tft, $C \otimes_B \Omega^{\log}_{B / A}$, $\Omega^{\log}_{C / A}$, and $\Omega^{\log}_{C / B}$ are finite $C$-modules.  Thus, to prove the exactness in \Refenum{\ref{thm-log-diff-fund-1}}, it suffices to show that, for any complete topological $C$-module $H$, the induced sequence
    \begin{equation}\label{eq-thm-log-diff-fund-dual}
        0 \to \Hom_C(\Omega^{\log}_{C / B}, H) \to \Hom_C(\Omega^{\log}_{C / A}, H) \to \Hom_C(C \otimes_B \Omega^{\log}_{B / A}, H)
    \end{equation}
    is exact.  By Proposition \ref{prop-log-diff-univ}, this sequence is nothing but
    \[
        0 \to \Der^{\log}_B(C, H) \to \Der^{\log}_A(C, H) \to \Der^{\log}_A(B, H),
    \]
    whose exactness is obvious.  Therefore, \Refenum{\ref{thm-log-diff-fund-1}} follows.

    In the rest of the proof, let $H$ be a complete topological $C$-module, and let $(d, \delta): (B, N, \beta) \to H$ be an $(A, M, \alpha)$-derivation.  Let $C * H$ be the $C$-algebra defined as in the proof of Proposition \ref{prop-log-diff-univ}, equipped with the log structure $(\gamma * \Id): O \oplus H \to C * H: (a, b) \mapsto (\gamma(a), \gamma(a) \, b)$, and denote by $(C, O, \gamma) * H$ the log Huber ring thus obtained.  Note that $(C, O, \gamma) * H \to (C, O, \gamma)$ is a log thickening of first order, as in Definition \ref{def-log-Huber-thick}, because the action of $1 + H$ on $O \oplus H$ is free.

    We claim that there is a natural bijection between the set of extensions of $(d, \delta)$ to $(A, M, \alpha)$-derivations $(\widetilde{d}, \widetilde{\delta}): (C, O, \gamma) \to H$ and the set of homomorphisms of log Huber rings $h: (C, O, \gamma) \to (C, O, \gamma) * H$ making the diagram
    \begin{equation}\label{eq-thm-log-diff-fund-thick-special}
        \xymatrix{ {(B, N, \beta)} \ar[r] \ar[d]_-g & {(C, O, \gamma) * H} \ar[d] \\
        {(C, O, \gamma)} \ar@{.>}[ru]^(.4)h \ar[r]^-\Id & {(C, O, \gamma)} }
    \end{equation}
    commute.  Here the upper horizontal map is a homomorphism of log Huber rings sending $(b, n)$ to $\bigl((g(b), d(b)), (g_\sharp(n), \delta(n))\bigr)$, and the right vertical one is the natural projection.  To justify the claim, for each map $h': (C, O) \to (C * H, O \oplus H)$ lifting the projection $(C * H, O \oplus H) \to (C, O)$, let us write $h' = \bigl((\Id, \widetilde{d}), (\Id, \widetilde{\delta})\bigr)$.  Then a short computation shows that $h'$ is a homomorphism of log Huber rings if and only if $\widetilde{d}$ is a derivation and $\widetilde{\delta}$ is a homomorphism of monoids such that $\widetilde{d}(\gamma(x)) = \gamma(x) \, \widetilde{\delta}(x)$ for all $x \in O$, and the claim follows.

    Thus, if $(C, O, \gamma)$ is formally log unramified over $(B, N, \beta)$, then the natural map $\Der^{\log}_A(C, H) \to \Der^{\log}_A(B, H)$ is surjective for each finite $C$-module $H$.  \Pth{Since $C \otimes_B \Omega^{\log}_{B / A}$, $\Omega^{\log}_{C / A}$, and $\Omega^{\log}_{C / B}$ are finite $C$-modules, it suffices to consider finite $C$-modules $H$ in this paragraph.}  In other words, the natural map $\Hom_C(\Omega^{\log}_{C / A}, H) \to \Hom_C(C \otimes_B \Omega^{\log}_{B / A}, H)$ is surjective, and therefore $C \otimes_B \Omega^{\log}_{B / A} \to \Omega^{\log}_{C / A}$ is injective, yielding \Refenum{\ref{thm-log-diff-fund-3}}.  Similarly, if $(C, O, \gamma)$ is formally log smooth over $(B, N, \beta)$, then $\Hom_C(\Omega^{\log}_{C / A}, H) \to \Hom_C(C \otimes_B \Omega^{\log}_{B / A}, H)$ is surjective, and we can justify \Refenum{\ref{thm-log-diff-fund-2}} by taking $H = C \otimes_B \Omega^{\log}_{B / A}$, which shows that the natural map $\Hom_C(\Omega^{\log}_{C / B}, H) \to \Hom_C(\Omega^{\log}_{C / A}, H)$ is injective and admits a left inverse splitting \Refeq{\ref{eq-thm-log-diff-fund-dual}}.  By combining \Refenum{\ref{thm-log-diff-fund-2}} and \Refenum{\ref{thm-log-diff-fund-3}}, we obtain \Refenum{\ref{thm-log-diff-fund-4}}.

    Finally, let us prove \Refenum{\ref{thm-log-diff-fund-5}}.  Suppose we are given a commutative diagram
    \begin{equation}\label{eq-thm-log-diff-fund-thick}
        \xymatrix{ {(B, N, \beta)} \ar_-g[d] \ar[r]^-u & {(D, T, \mu)} \ar[d]^-i \\
        {(C, O, \gamma)} \ar^-v[r] \ar@{.>}[ru]^(.4){\widetilde{v}} & {(D', T', \mu')} }
    \end{equation}
    of solid arrows, with $i$ a log thickening of first order.  If $g \circ f$ is formally log smooth, then there exists some lifting $w: (C, O, \gamma) \to (D, T, \mu)$ of $v$ making the diagram
    \[
        \xymatrix{ {(A, M, \alpha)} \ar_-{g \circ f}[d] \ar[r]^-{u \circ f} & {(D, T, \mu)} \ar[d]^-i \\
        {(C, O, \gamma)} \ar^-v[r] \ar@{.>}[ru]^(.4)w & {(D', T', \mu')} }
    \]
    commute.  Note that this $w$ might not satisfy $u = g \circ w$, but $u - g \circ w$ defines a derivation $(d, \delta): (B, N, \beta) \to H = \ker(D \to D')$.  Hence, we obtain a homomorphism $\widetilde{g}: (B, N, \beta) \to (C, O, \gamma) * H$ extending $g: (B, N, \beta) \to (C, O, \gamma)$, as before.  Since $H \subset D$ by definition, $w$ canonically extends to a homomorphism $\widetilde{w}: (C, O, \gamma) * H \to (D, T, \mu)$ sending $H$ canonically into $D$.  By combining these, we can extend \Refeq{\ref{eq-thm-log-diff-fund-thick}} to a commutative diagram
    \begin{equation}\label{eq-thm-log-diff-fund-thick-ext}
        \xymatrix{ {(A, M, \alpha)} \ar_-f[d] \ar@/^1pc/[rrd]^-{u \circ f} \\
        {(B, N, \beta)} \ar_-g[d] \ar[r]_-{\widetilde{g}} \ar@/^1.25pc/[rr]^(.35)u & {(C, O, \gamma) * H} \ar[d] \ar[r]_-{\widetilde{w}} & {(D, T, \mu)} \ar[d]^-i \\
        {(C, O, \gamma)} \ar^-\Id[r] \ar@{.>}[ru]^(.4)h & {(C, O, \gamma)} \ar^-v[r] \ar@{.>}[ru]^(.4){\widetilde{v}} & {(D', T', \mu')} }
    \end{equation}
    of solid arrows, and any $h: (C, O, \gamma) \to (C, O, \gamma) * H$ making the diagram \Refeq{\ref{eq-thm-log-diff-fund-thick-ext}} commute canonically induces $\widetilde{v} := \widetilde{w} \circ h: (C, O, \gamma) \to (D, T, \mu)$ making the diagrams \Refeq{\ref{eq-thm-log-diff-fund-thick}} and \Refeq{\ref{eq-thm-log-diff-fund-thick-ext}} commute.  Moreover, $h$ is uniquely determined by $\widetilde{v} = \widetilde{w} \circ h$, because if $h'$ is another such map such that $\widetilde{v} = \widetilde{w} \circ h = \widetilde{w} \circ h'$, then $\widetilde{w} \circ (h - h') = 0$, but $(h - h')(C) \subset H$ and $\widetilde{w}|_H: H \to D$ is \Pth{by definition} the canonical injection.  Thus, if the conclusion in \Refenum{\ref{thm-log-diff-fund-2}} \Pth{\resp \Refenum{\ref{thm-log-diff-fund-3}}, \resp \Refenum{\ref{thm-log-diff-fund-4}}} holds, then $C \otimes_B \Omega^{\log}_{B / A} \to \Omega^{\log}_{C / A}$ is injective and splits \Pth{\resp is surjective, \resp is bijective}, and so $\Hom_C(\Omega^{\log}_{C / A}, H) \to \Hom_C(C \otimes_B \Omega^{\log}_{B / A}, H)$ is surjective \Pth{\resp injective, \resp bijective}.  By the first three paragraphs of this proof, and by the relation between $h$ and $\widetilde{v}$ explained above, there exists at least one \Pth{\resp at most one, \resp exactly one} $\widetilde{v}$ making the diagrams \Refeq{\ref{eq-thm-log-diff-fund-thick}} and \Refeq{\ref{eq-thm-log-diff-fund-thick-ext}} commute.  Since \Refeq{\ref{eq-thm-log-diff-fund-thick}} is arbitrary, $g$ is formally log smooth \Pth{\resp formally log unramified, \resp formally log \'etale}, as desired.
\end{proof}

\begin{lem}\label{lem-lft-formal}
    In order to verify that a tft homomorphism of log Huber pairs $f: (A, M, \alpha) \to (B, N, \beta)$ is formally smooth \Pth{\resp formally unramified, \resp formally \'etale}, it suffices to verify the corresponding lifting condition in Definition \ref{def-log-Huber-formal-log-sm-unram-et} only for all commutative diagrams \Refeq{\ref{eq-log-Huber-thick}} in which the underlying homomorphism of Huber rings $A \to D'$ is tft and in which $H = \ker(D \to D')$ is a finite $D'$-module.
\end{lem}
\begin{proof}
    Given any diagram \Refeq{\ref{eq-log-Huber-thick}}, since $f: (A, M, \alpha) \to (B, N, \beta)$ is tft, the homomorphism $B \to D'$ factors through a complete topological $B$-subalgebra $\breve{D}'$ of $D'$ that is tft over $A$.  By lifting the topological generators of $D'$ over $A$, there exists some complete topological $A$-subalgebra $\breve{D}$ of $D$ such that the composition of the homomorphisms $\breve{D} \to D \to D'$ factors through a surjection $\breve{D} \to \breve{D}'$ such that $\breve{H} := \ker(\breve{D} \to \breve{D}') \subset H$ is a finite $\breve{D}'$-module.  Let $\breve{\mu}: \breve{T} \to \breve{D}$ and $\breve{\mu}': \breve{T}' \to \breve{D}'$ denote the pullbacks of $\mu: T \to D$ and $\mu': T' \to D'$, respectively.  Note that $\breve{H}^2 = 0$ because $H^2 = 0$, and $1 + \breve{H}$ acts freely on $\breve{T}$ because $1 + H$ acts freely on $T$.  Then \Refeq{\ref{eq-log-Huber-thick}} extends to a commutative diagram
    \[
        \xymatrix{ {(A, M, \alpha)} \ar_-f[d] \ar[r] & {(\breve{D}, \breve{T}, \breve{\mu})} \ar[d] \ar[r] & {(D, T, \mu)} \ar[d] \\
        {(B, N, \beta)} \ar^-{\breve{g}}[r] \ar@{.>}[ru]^(.4){\widetilde{\breve{g}}} \ar@/_1.25pc/[rr]^-g & {(\breve{D}', \breve{T}', \breve{\mu}')} \ar[r] & {(D', T', \mu')} }
    \]
    of solid arrows, in which $(\breve{D}, \breve{T}, \breve{\mu}) \to (\breve{D}', \breve{T}', \breve{\mu}')$ is a log thickening of first order.  Since any dotted arrow $\widetilde{\breve{g}}$ lifting $\breve{g}$ in the above diagram induces a dotted arrow $\widetilde{g}$ lifting $g$ in \Refeq{\ref{eq-log-Huber-thick}}, we obtain the formally log smooth case of this lemma.

    It remains to establish the formally log unramified case of this lemma.  Given any liftings $\widetilde{g}$ and $\widetilde{g}'$ of $g$ in \Refeq{\ref{eq-log-Huber-thick}}, their difference $\widetilde{g} - \widetilde{g}'$ defines a derivation $(d, \delta): (B, N, \beta) \to H = \ker(D \to D')$ over $A$, which corresponds to a morphism $\Omega^{\log}_{B / A} \to H$ of $B$-modules, by Proposition \ref{prop-log-diff-univ}.  Since $f$ is tft, $\Omega^{\log}_{B / A}$ is a finite $B$-module.  Thus, in order to show that $f$ is formally log unramified, it suffices to show that, when $H$ is a finite $B$-module, all morphisms $\Omega^{\log}_{B / A} \to H$ as above are zero.  As in the proof of Theorem \ref{thm-log-diff-fund}, this can be verified using only diagrams \Refeq{\ref{eq-log-Huber-thick}} in which $(D, T, \mu) \to (D', T', \mu')$ is $(B, N, \beta) * H \to (B, N, \beta)$, where the underlying homomorphism $A \to B$ is tft and $H$ is a finite $B$-module.
\end{proof}

\begin{defn}\label{def-log-Huber-R-alg}
    Let $u: P \to Q$ be a homomorphism of fine monoids, and let $R$ be a Huber ring.  Then we have the pre-log Huber ring $P \to R\Talg{P}: \, a \mapsto \mono{a} $ \Pth{\resp $Q \to R\Talg{Q}: \, a \mapsto \mono{a}$}, with the topology given in Lemma \ref{lem-monoid-alg-Huber}.  In this case, we say that $R\Talg{P}$ is a pre-log Huber $R$-algebra.  By abuse of notation, we shall still denote by $R\Talg{P}$ \Pth{\resp $R\Talg{Q}$} the log Huber $R$-algebras thus obtained.
\end{defn}

\begin{prop}\label{prop-log-diff-monoid}
    Let $u: P \to Q$ and $R$ be as in Definition \ref{def-log-Huber-R-alg}.  If the kernel and the torsion part of the cokernel \Pth{\resp the kernel and the cokernel} of $u^\gp: P^\gp \to Q^\gp$ are finite groups of orders invertible in $R$, then $R\Talg{Q}$ is formally log smooth \Pth{\resp formally log \'etale} over $R\Talg{P}$.  In this case, the map $\delta: Q^\gp \to \Omega^{\log}_{R\Talg{Q} / R\Talg{P}}$ induces an isomorphism of finite free $R\Talg{Q}$-modules
    \[
        \Omega^{\log}_{R\Talg{Q} / R\Talg{P}} \cong R\Talg{Q} \otimes_\bZ \bigl(Q^\gp / u^\gp(P^\gp)\bigr).
    \]
\end{prop}
\begin{proof}
    Consider the commutative diagram
    \begin{equation}\label{eq-prop-log-diff-monoid-1}
        \xymatrix{ {R\Talg{P}} \ar[d] \ar[r] & {(D, T, \mu)} \ar[d] \\
        {R\Talg{Q}} \ar[r] & {(D', T', \mu')} }
    \end{equation}
    as in \Refeq{\ref{eq-log-Huber-thick}}, in which $(D, T, \mu) \to (D', T', \mu')$ is a log thickening of first order as in Definition \ref{def-log-Huber-thick}.  This gives rise to a commutative diagram of monoids
    \begin{equation}\label{eq-prop-log-diff-monoid-2}
        \xymatrix{ {P} \ar[d] \ar[r] & {T} \ar[d] \\
        {Q} \ar[r] & {T',} }
    \end{equation}
    which in turn induces a commutative diagram of abelian groups
    \begin{equation}\label{eq-prop-log-diff-monoid-3}
        \xymatrix{ {P^\gp} \ar[d] \ar[r] & {T^\gp} \ar[d] \\
        {Q^\gp} \ar[r] & {(T')^\gp.} }
    \end{equation}
    Note that there is a natural bijection between the set of homomorphisms of log Huber $R$-algebras $R\Talg{Q} \to (D, T, \mu)$ extending \Refeq{\ref{eq-prop-log-diff-monoid-1}} and the set of homomorphisms of monoids $Q \to T$ extending \Refeq{\ref{eq-prop-log-diff-monoid-2}}.  By using the cartesian diagram
    \[
        \xymatrix{ {T} \ar[d] \ar[r] & {T^\gp} \ar[d] \\
        {T'} \ar[r] & {(T')^\gp,} }
    \]
    and the fact that $P$ and $Q$ are fine monoids, we see that there is also a bijection between the set of desired homomorphisms $R\Talg{Q} \to (D, T, \mu)$ and the set of group homomorphisms $Q^\gp \to T^\gp$ extending \Refeq{\ref{eq-prop-log-diff-monoid-3}}.

    Since $(D, T, \mu) \to (D', T', \mu')$ is a log thickening of first order as in Definition \ref{def-log-Huber-thick}, we have $\ker(T^\gp \to (T')^\gp) = \mu^{-1}(1 + H) \cong H$.  Let $G = Q^\gp / u^\gp(P^\gp)$.  Since the kernel and the torsion part of the cokernel \Pth{\resp the kernel and the cokernel} of $u^\gp$ are finite groups of orders invertible in $R$, the set of desired homomorphisms $Q^\gp \to T^\gp$ is a torsor under $\Hom(G, H) \cong \Hom(G / G_\tor, H)$, where $G_\tor$ is the torsion subgroup of $G$.  This proves the first statement of the proposition.

    On the other hand, for any finite $R\Talg{Q}$-module $L$, by the same argument as in the proof of Theorem \ref{thm-log-diff-fund}, there is a bijection between $\Der^{\log}_{R\Talg{P}}(R\Talg{Q}, L)$ and the set of $h: R\Talg{Q} \to R\Talg{Q} * L$ extending the following commutative diagram
    \[
        \xymatrix{ {R\Talg{P}} \ar[d] \ar[r] & {R\Talg{Q} * L} \ar[d] \\
        {R\Talg{Q}} \ar@{.>}[ru]^-h \ar[r]^-\Id & {R\Talg{Q}.} }
    \]
    Then $\Hom_{R\Talg{Q}}(\Omega^{\log}_{R\Talg{Q} / R\Talg{P}}, L) \cong \Der^{\log}_{R\Talg{P}}(R\Talg{Q}, L) \cong \Hom(G / G_\tor, L)$, by the previous paragraph.  The second statement of the proposition follows.
\end{proof}

\begin{cor}\label{cor-log-diff-monoid-strict-cl-imm}
    Let $u: P \to Q$ and $R$ as in Definition \ref{def-log-Huber-R-alg} such that the kernel and the torsion part of the cokernel of $u^\gp: P^\gp \to Q^\gp$ are finite groups of orders invertible in $R$.  Let $Q'$ be any fine monoid, and let $S$ denote the log Huber ring associated with the pre-log Huber ring $\widetilde{Q} := Q \oplus Q' \to R\Talg{Q}$ \Pth{mapping $Q' - \{ 0 \}$ to $0$}, so that the surjective homomorphism $R\Talg{\widetilde{Q}} \to S$ of log Huber rings is strict.  Then the map $\delta: \widetilde{Q}^\gp \to \Omega^{\log}_{S / R\Talg{P}}$ induces an isomorphism
    \[
        \Omega^{\log}_{S / R\Talg{P}} \cong R\Talg{Q} \otimes_\bZ \bigl(\widetilde{Q}^\gp / u^\gp(P^\gp)\bigr).
    \]
    If, in addition, the torsion part of $(Q')^\gp$ is a finite group whose order is invertible in $R$, then we also have
    \[
        \Omega^{\log}_{S / R\Talg{P}} \cong R\Talg{Q} \otimes_{R\Talg{\widetilde{Q}}} \Omega^{\log}_{R\Talg{\widetilde{Q}} / R\Talg{P}}.
    \]
\end{cor}
\begin{proof}
    By comparing the definitions of $\Omega^{\log}_{S / R\Talg{P}}$ and $\Omega^{\log}_{R\Talg{Q} / R\Talg{P}}$ as in \Refeq{\ref{eq-def-log-diff-aff}}, we obtain $\Omega^{\log}_{S / R\Talg{P}} \cong \Omega^{\log}_{R\Talg{Q} / R\Talg{P}} \oplus (R\Talg{Q} \otimes_\bZ (Q')^\gp)$, because $\widetilde{Q} \to S$ maps $Q' - \{ 0 \}$ to zero.  Since this isomorphism is compatible with the canonical maps $Q^\gp \to \Omega^{\log}_{R\Talg{Q} / R\Talg{P}}$, $\widetilde{Q}^\gp \to \Omega^{\log}_{S / R\Talg{P}}$, and $\widetilde{Q}^\gp \to \Omega^{\log}_{R\Talg{\widetilde{Q}} / R\Talg{P}}$ denoted by $\delta$, we can finish the proof by applying Proposition \ref{prop-log-diff-monoid} to $P \to Q$ and $P \to \widetilde{Q}$.
\end{proof}

\subsection{Sheaves of log differentials}\label{sec-log-diff-sh}

Our next step is to define sheaves of log differentials for locally noetherian coherent log adic spaces, and show that their formation is compatible with fiber products in the category of locally noetherian coherent, fine, and fs log adic spaces.  Then we shall globalize several results in Section \ref{sec-log-diff} and relate them to the definitions we made in Section \ref{sec-log-sm}.
\begin{defn}\label{def-der-sheaf}
    Let $f: (Y, \cM_Y, \alpha_Y) \to (X, \cM_X, \alpha_X)$ be a morphism of locally noetherian log adic spaces, and let $\cF$ be a sheaf of complete topological $\cO_{Y_\et}$-modules.  By a derivation of $(Y, \cM_Y, \alpha_Y)$ over $(X, \cM_X, \alpha_X)$ valued in $\cF$, we mean a pair $(d, \delta)$, where $d: \cO_{Y_\et} \to \cF$ is a continuous $\cO_{X_\et}$-linear derivation and $\delta: \cM_Y \to \cF$ is a morphism of sheaves of monoids such that $\delta\bigl(f^{-1}(\cM_X)\bigr) = 0$ and $d\bigl(\alpha_Y(m)\bigr) = \alpha_Y(m) \, \delta(m)$, for all sections $m$ of $\cM_Y$.
\end{defn}

\begin{constr}\label{constr-log-diff-sheaf-univ-aff}
    Let $f: (Y, \cM_Y, \alpha_Y) \to (X, \cM_X, \alpha_X)$ be a morphism of noetherian coherent log adic spaces, where $X = \Spa(A, A^+)$ and $Y = \Spa(B, B^+)$ are affinoid.  Suppose that $f$ induces a tft homomorphism of log Huber rings $(A, M, \alpha) \to (B, N, \beta)$, where $M := \cM_X(X)$ and $N := \cM_Y(Y)$, and where $\alpha := \alpha_X(X): M \to A$ and $\beta := \alpha_Y(Y): N \to B$.  Let $\Omega^{\log}_{Y / X}$ denote the coherent sheaf on $Y_\et$ associated with $\Omega^{\log}_{B / A}$ \Pth{see \Refeq{\ref{eq-def-log-diff-aff}}}.  For each $\Spa(C, C^+) \in Y_\et$, by Theorem \ref{thm-log-diff-fund}, the log differential $\Omega^{\log}_{C / A}$ for $(A, M, \alpha) \to (C, N, (B \to C) \circ \beta)$ is naturally isomorphic to $C \otimes_B \Omega^{\log}_{B / A} = \Omega^{\log}_{Y / X}\bigl(\Spa(C, C^+)\bigr)$; and the maps $d_{C / A}: C \to \Omega^{\log}_{C / A}$ and $\delta_{C / A}: N \to \Omega^{\log}_{C / A}$ naturally and compatibly extend to a continuous $\cO_{X_\et}$-linear derivation $d_{Y / X}: \cO_{Y_\et} \to \Omega^{\log}_{Y / X}$ and a morphism $\delta_{Y / X}: N_Y \to \Omega^{\log}_{Y / X}$ of sheaves of monoids satisfying $\delta_{Y / X}\bigl(f^{-1}(M_X)\bigr) = 0$ and $d_{Y / X}\bigl(\alpha_Y(n)\bigr) = \alpha_Y(n) \, \delta_{Y / X}(n)$, for sections $n$ of $N_Y$ over objects of $Y_\et$.  We may further extend $\delta_{Y / X}$ to a morphism $\delta_{Y / X}: \cM_Y \to \Omega^{\log}_{Y / X}$ of sheaves of monoids satisfying $\delta_{Y / X}\bigl(f^{-1}(\cM_X)\bigr) = 0$ and $d_{Y / X}\bigl(\alpha_Y(m)\bigr) = \alpha_Y(m) \, \delta_{Y / X}(m)$, for sections $m$ of $\cM_Y$ over objects of $Y_\et$.
\end{constr}

\begin{lem}\label{lem-log-diff-sheaf-univ-aff}
    In Construction \ref{constr-log-diff-sheaf-univ-aff}, the triple $(\Omega^{\log}_{Y / X}, d_{Y / X}, \delta_{Y / X})$ is a universal object among all derivations of $(Y, \cM_Y, \alpha_Y)$ over $(X, \cM_X, \alpha_X)$.  Moreover, for any affinoid objects $V \in Y_\et$ and $U \in X_\et$ fitting into a commutative diagram
    \[
        \xymatrix{ {V} \ar[d] \ar[r] & {Y} \ar[d] \\
        {U} \ar[r] & {X}, }
    \]
    we have a canonical isomorphism $(\Omega^{\log}_{Y / X}, d_{Y / X}, \delta_{Y / X})|_V \cong (\Omega^{\log}_{V / U}, d_{V / U}, \delta_{V / U})$.
\end{lem}
\begin{proof}
    Let $(d, \delta)$ be a derivation of $(Y, \cM_Y, \alpha_Y)$ over $(X, \cM_X, \alpha_X)$ valued in some complete topological $\cO_{Y_\et}$-module $\cF$.  At each $\Spa(C, C^+) \in Y_\et$, the evaluation of the derivation $(d, \delta)$ defines a derivation
    \[
        \bigl(C, \cM_{Y}(\Spa(C, C^+)), \alpha_Y(\Spa(C, C^+))\bigr) \to \cF\bigl(\Spa(C, C^+)\bigr)
    \]
    over $\bigl(A, \cM_X(X), \alpha_X(X)\bigr)$, which restricts to a derivation
    \[
        \bigl(C, N, (B \to C) \circ \beta\bigr) \to \cF\bigl(\Spa(C, C^+)\bigr)
    \]
    over $(A, M, \alpha)$.  By the universal property of log differentials, it factors through a continuous $C$-linear morphism $\Omega^{\log}_{C / A} \to \cF\bigl(\Spa(C, C^+)\bigr)$.  Moreover, we deduce from the universal property of log differentials a commutative diagram
    \[
        \xymatrix{ {\Omega^{\log}_{C_2 / A}} \ar[d] \ar[r] & {\cF\bigl(\Spa(C_2, C_2^+)\bigr)} \ar[d] \\
        {\Omega^{\log}_{C_1 / A}} \ar[r] & {\cF\bigl(\Spa(C_1, C_1^+)\bigr),} }
    \]
    for any morphism $\Spa(C_1, C_1^+) \to \Spa(C_2, C_2^+)$ in $Y_\et$.  As a result, the morphisms $\Omega^{\log}_{C / A} \to \cF(\Spa(C, C^+))$, for $\Spa(C, C^+) \in Y_\et$, are compatible with each other and define a continuous $\cO_{X_\et}$-linear morphism $\Omega^{\log}_{Y / X} \to \cF$, whose compositions with $d_{Y / X}$ and $\delta_{Y / X}$ are equal to $d$ and $\delta$, respectively.  This proves the first assertion of the lemma.  The second then follows from Theorem \ref{thm-log-diff-fund}.
\end{proof}

\begin{constr}\label{constr-log-diff-sheaf-univ}
    Given any lft morphism $f: Y \to X$ of noetherian coherent log adic spaces, by Proposition \ref{prop-chart-mor-exist}, there exist a finite index set $I$ and \'etale coverings $\{ X_i \to X \}_{i \in I}$ and $\{ Y_i \to Y \}_{i \in I}$, respectively, by affinoid log adic spaces such that $f$ induces a morphism $Y_i \to X_i$ which fits into the setting of Construction \ref{constr-log-diff-sheaf-univ-aff}, for each $i \in I$.  By Lemma \ref{lem-log-diff-sheaf-univ-aff}, the pullbacks of the triples $(\Omega^{\log}_{Y_i / X_i}, d_{Y_i / X_i}, \delta_{Y_i / X_i})$ are canonically isomorphic to each other over the fiber products of $Y_i$ over $Y$.  Thus, by Proposition \ref{prop-et-coh-descent}, we obtain a triple $(\Omega^{\log}_{Y / X}, d_{Y / X}, \delta_{Y / X})$ on $Y_\et$, where $(d_{Y / X}, \delta_{Y / X})$ gives a derivation of $Y$ over $X$ valued in $\Omega^{\log}_{Y / X}$.
\end{constr}

By Lemma \ref{lem-log-diff-sheaf-univ-aff}, we immediately obtain the following:
\begin{lem}\label{lem-log-diff-sheaf-univ}
    In Construction \ref{constr-log-diff-sheaf-univ}, $(\Omega^{\log}_{Y / X}, d_{Y / X}, \delta_{Y / X})$ is a universal object among all derivations of $Y$ over $X$.  As a result, $(\Omega^{\log}_{Y / X}, d_{Y / X}, \delta_{Y / X})$ is well defined, \ie, independent of the choice of \'etale coverings; and its definition extends to all lft morphisms $f: Y \to X$ of locally noetherian coherent log adic spaces.
\end{lem}

\begin{defn}\label{def-log-diff-sheaf}
    We call the $\Omega^{\log}_{Y / X}$ in Lemma \ref{lem-log-diff-sheaf-univ} the \emph{sheaf of log differentials} of $f$ and $(d_{Y / X}, \delta_{Y / X})$ the associated \emph{universal log derivations}.  By abuse of notation, we shall denote the pushforward of $\Omega^{\log}_{Y / X}$ to $Y_\an$ by the same symbols.  \Pth{When there is any risk of confusion, we shall denote the sheaves of log differentials on $Y_\et$ and $Y_\an$ more precisely by $\Omega^{\log}_{Y / X, \et}$ and $\Omega^{\log}_{Y / X, \an}$, respectively.}  If $X = \Spa(k, \cO_k)$, for simplicity, we shall write $\Omega^{\log}_Y$ instead of $\Omega^{\log}_{Y / X}$.
\end{defn}

\begin{prop}\label{prop-log-diff-sheaf-bc}
    Let
    \[
        \xymatrix{ {Y'} \ar[d]_-f \ar[r] & {X'} \ar[d] \\
        {Y} \ar[r] & {X} }
    \]
    be a cartesian diagram in the category of locally noetherian coherent \Pth{\resp fine, \resp fs} log adic spaces in which $Y \to X$ is lft.  Then $f^*(\Omega^{\log}_{Y / X}) \cong \Omega^{\log}_{Y' / X'}$.
\end{prop}
\begin{proof}
    By the \'etale local construction of sheaves of log differentials, we may assume that $Y = \Spa(B, B^+)$, $X = \Spa(A, A^+)$, $X' = \Spa(A', A'^+)$, and $Y' = \Spa(B', B'^+)$ are affinoid, and that $Y \to X$ and $X' \to X$ admit charts $P \to Q$ and $P \to P'$, respectively, given by finitely generated \Pth{\resp fine, \resp fs} monoids.  Let $Q' := Q \oplus_P P'$.  By the proofs of Propositions \ref{prop-adj-int-sat} and \ref{prop-fiber-prod-log-adic}, $B'$ is isomorphic to $B \ho_A A'$ \Pth{\resp $(B \ho_A A') \ho_{\bZ[Q']} \bZ[(Q')^\Int]$, \resp $(B \ho_A A') \ho_{\bZ[Q']} \bZ[(Q')^\Sat]$}, and $Y'$ is modeled on $Q'$ \Pth{\resp $(Q')^\Int$, \resp $(Q')^\Sat$}.  We need to show that
    \[
        \Omega^{\log}_{B' / A'} \cong \Omega^{\log}_{B / A} \otimes_B B'.
    \]
    Since we have
    \[
        \Hom_{B'}(\Omega^{\log}_{B' / A'}, L) \cong \Der^{\log}_{A'}(B', L)
    \]
    and
    \[
        \Hom_{B'}(\Omega^{\log}_{B / A} \otimes_B B', L) \cong \Hom_B(\Omega^{\log}_{B / A}, L) \otimes_B B' \cong \Der^{\log}_A(B, L) \otimes_B B',
    \]
    for each complete topological $B'$-module $L$, it suffices to show that
    \[
        \Der^{\log}_{A'}(B', L) \cong \Der^{\log}_A(B, L) \otimes_B B'.
    \]
    In the case of coherent log adic spaces, by Remark \ref{rem-def-der-mod}, this follows from essentially the same argument as in the proof of \cite[\aProp IV.1.1.3]{Ogus:2018-LLG} \Pth{in the log scheme case}.  Since $(Q')^\gp \cong ((Q')^\Int)^\gp \cong ((Q')^\Sat)^\gp$, by essentially the same argument as in the proof of \cite[\aProp IV.1.1.9]{Ogus:2018-LLG}, we also have
    \[
        \Der^{\log}_{A'}(B', L) \cong \Der^{\log}_{A'}(B \ho_A A', L) \otimes_{B \ho_A A'} B' \cong \Der^{\log}_A(B, L) \otimes_B B',
    \]
    yielding the desired isomorphism in the cases of fine and fs log adic spaces.
\end{proof}

\begin{defn}[{\Refcf{} Definition \ref{def-log-Huber-thick}}]\label{def-log-adic-sp-thick}
    A morphism $i: Z' \to Z$ of log adic spaces is called a \emph{log thickening of first order} if it satisfies the following conditions:
    \begin{enumerate}
        \item\label{def-log-adic-sp-thick-1} It is a strict closed immersion \Pth{see Definition \ref{def-imm} and Example \ref{ex-imm}} defined by an $\cO_Z$-ideal $\cI$ satisfying $\cI^2 = 0$.

        \item\label{def-log-adic-sp-thick-2} The subsheaf $1 + \cI_\et$ of $\cO_{Z_\et}^\times$, where $\cI_\et$ denotes the natural pullback of $\cI$ to $Z_\et$ as a coherent ideal, acts freely on $\cM_Z$, and induces an isomorphism $i^{-1}\bigl(\cM_Z / (1 + \cI_\et)\bigr) \Mi \cM_{Z'}$ \Pth{of sheaves of monoids} over $Z'_\et$.
    \end{enumerate}
\end{defn}

\begin{rk}[{\Refcf{} Remark \ref{rem-def-log-Huber-thick}}]\label{rem-def-log-adic-sp-thick}
    The condition \Refenum{\ref{def-log-adic-sp-thick-2}} in Definition \ref{def-log-adic-sp-thick} is automatic when $\cO_{Z_\et}^\times$ acts freely on $\cM_Z$; or, equivalently, when the canonically induced morphism $\alpha_Z^{-1}(\cO_{Z_\et}^\times) \to \cM_Z^\gp$ is injective.  Hence, the condition is satisfied when $Z$ is \emph{integral}.
\end{rk}

\begin{defn}[{\Refcf{} Definition \ref{def-log-Huber-formal-log-sm-unram-et}}]\label{def-formal-log-sm-unram-et}
    A morphism $f: Y \to X$ of log adic spaces is called \emph{formally log smooth} \Pth{\resp \emph{formally unramified}, \resp \emph{formally log \'etale}} if, for each commutative diagram
    \begin{equation}\label{eq-def-formal-log-sm-unram-et}
        \xymatrix{ {Z'} \ar@{^(->}[d]_-i \ar[r] & {Y} \ar[d]^-f \\
        {Z} \ar[r]^-g \ar@{.>}[ru]^(.4){\widetilde{g}} & {X} }
    \end{equation}
    of solid arrows in which $i$ is a log thickening of first order as in Definition \ref{def-log-adic-sp-thick}, there exists, up to \Pth{strictly} \'etale localization on $Z$, at least one \Pth{\resp at most one, \resp exactly one} lifting $\widetilde{g}: Z \to Y$ of $g$, as the dotted arrow in \Refeq{\ref{eq-log-Huber-thick}}, making the whole diagram commute.
\end{defn}

\begin{rk}\label{rem-def-formal-log-sm-unram-et-compos-bc}
    It follows easily from the definition that we have the following:
    \begin{enumerate}
        \item\label{rem-def-formal-log-sm-unram-et-compos-bc-1} Formally log smooth \Pth{\resp formally log unramified, \resp formally log \'etale} morphisms are stable under compositions and base changes \Pth{when defined}.

        \item\label{rem-def-formal-log-sm-unram-et-compos-bc-2} If a morphism $g: X \to S$ of log adic spaces is formally log \'etale, then a morphism $f: Y \to X$ of log adic spaces is formally log smooth \Pth{\resp formally log \'etale} if and only if $g \circ f: Y \to S$ is.
    \end{enumerate}
\end{rk}

\begin{rk}\label{rem-def-formal-log-sm-unram-et-loc}
    By Lemmas \ref{lem-def-log-Huber-formal-log-sm-unram-et-strict} and \ref{lem-lft-formal}, and by \cite[\aProp 1.7.1]{Huber:1996-ERA}, in order to show that an lft morphism $f: Y \to X$ of locally noetherian fine log adic spaces is formally log smooth \Pth{\resp formally log \'etale}, it suffices to take locally finite \Pth{strictly} \'etale coverings $\{ X_i \to X \}_{i \in I}$ and $\{ Y_i \to Y \}_{i \in I}$ by affinoid log adic spaces such that $f$ induces lft morphisms $f_i: Y_i \to X_i$, and verify for each such $f_i$ the corresponding formal lifting condition in Definition \ref{def-formal-log-sm-unram-et} only for all commutative diagrams \Refeq{\ref{eq-def-formal-log-sm-unram-et}} with affinoid log adic spaces $Z$.
\end{rk}

\begin{lem}\label{lem-def-formal-log-sm-unram-et-strict}
    If $f: Y \to X$ is a strict lft morphism of locally noetherian log adic spaces, then it is \emph{formally log smooth} \Pth{\resp \emph{formally log unramified}, \resp \emph{formally log \'etale}} if the underlying morphism $f: Y \to X$ of adic spaces is \emph{formally smooth} \Pth{\resp \emph{formally unramified}, \resp \emph{formally \'etale}} in the sense that it satisfies the formal lifting conditions in \cite[\aDef 1.6.5]{Huber:1996-ERA}.  If the canonical morphism $\cO_{X_\et}^\times \to \cM_X^\gp$ is injective, then the converse is true.
\end{lem}
\begin{proof}
    In the notation of \Refeq{\ref{eq-def-formal-log-sm-unram-et}}, when $f$ satisfies the assumptions of this lemma, and when $Z$ is noetherian and affinoid, the formal lifting conditions in \cite[\aDef 1.6.5]{Huber:1996-ERA} of the underlying adic spaces can be verified \'etale locally on $Z$, as in the theory for schemes in \cite[III]{SGA:1}, because any liftings over an \'etale covering of $Z$ defines a cohomology class in $H^1\bigl(Z'_\et, \mathcal{H}\mathit{om}_{\cO_{Z'}}((Z' \Em Y)^*(\Omega_{Y / X}), \cI)\bigr)$ \Pth{by working locally as in the proof of Theorem \ref{thm-log-diff-fund} and in Construction \ref{constr-log-diff-sheaf-univ-aff}, ignoring all log structures}, which vanishes exactly when the liftings can be modified \Pth{up to further \'etale localization} to descend to a global lifting on $Z'$; and because $H^1\bigl(Z'_\et, \mathcal{H}\mathit{om}_{\cO_{Z'}}((Z' \Em Y)^*(\Omega_{Y / X}), \cI)\bigr) = 0$, by Proposition \ref{prop-et-coh-descent}, since $Z$ and hence $Z'$ are noetherian and affinoid, since the $\cO_Z$-ideal $\cI$ can be identified with a coherent $\cO_{Z'}$-module in this case, and since $\Omega_{Y / X}$ is a coherent $\cO_Y$-module when $f$ is lft.  Thus, this lemma follows from essentially the same arguments as in the proof of Lemma \ref{lem-def-log-Huber-formal-log-sm-unram-et-strict} \Pth{by working with sheaves of monoids and their stalks instead}.
\end{proof}

\begin{lem}\label{lem-formal-log-sm-log-diff}
    If $f: Y \to X$ is a formally log smooth lft morphism of locally noetherian fs log adic spaces, then $\Omega^{\log}_{Y / X}$ is a locally free $\cO_Y$-module of finite rank.
\end{lem}
\begin{proof}
    Since $X$ and $Y$ are fs, up to \'etale localization, we may assume that $X = \Spa(A, A^+)$ and $Y = \Spa(B, B^+)$ are affinoid, with log structures induced by some homomorphisms $P \to A^+ \to A$ and $Q \to B^+ \to B$ from fs monoids $P$ and $Q$, respectively, and that there exists a surjection $A\Talg{T_1, \ldots, T_n} \Surj B$, for some $n \geq 0$.  Moreover, since $Q$ is fs, it contains the torsion part $Q_\tor$ of $Q^\gp$, which we may assume to be embedded into $B^\times$.  Hence, we may assume that $Q_\tor$ is a finite subgroup of $B^\times$ annihilated by an integer invertible in $B$, so that $B \otimes_\bZ Q^\gp$ is a finite free $B$-module.  It suffices to show that $\Omega^{\log}_{B / A}$ is a finite projective $B$-module.

    Let us equip the Huber ring $D := A\Talg{T_1, \ldots, T_n}\Talg{Q}$ with the log structure induced by the pre-hog structure $P \oplus Q \to D$ given by $P \to A^+ \to A$ and $Q \to A\Talg{T_1, \ldots, T_n}\Talg{Q}: q \mapsto \mono{q}$, for all $q \in Q$.  Then we obtain a strict surjection $D \Surj B$ of log Huber rings over $A$, whose kernel we denote by $H$, which factors as a composition of strict surjections $D \Surj D' := D / H^2 \Surj B \cong D / H$.  Note that $D' \to B$ is a log thickening, as in Definition \ref{def-log-Huber-thick}, because the log structure of $D$ is integral \Pth{see Remark \ref{eq-log-Huber-thick}}, as $P$ and $Q$ are fine.  Since $A \to B$ is formally log smooth, there exists some splitting $B \to D'$ of log Huber rings over $A$.

    By Theorem \ref{thm-log-diff-fund} \Pth{applied to $A \to A\Talg{T_1, \ldots, T_n} \to D$}, Lemma \ref{lem-log-diff-strict} \Pth{applied to the strict homomorphism $A \to A\Talg{T_1, \ldots, T_n}$}, Proposition \ref{prop-log-diff-monoid} \Pth{applied to $A\Talg{T_1, \ldots, T_n} \to D$ with $u: P \to P \oplus Q$}, and Proposition \ref{prop-log-diff-sheaf-bc}, we obtain a \emph{split} short exact sequence $0 \to D^n \to \Omega^{\log}_{D / A} \to D \otimes_\bZ Q^\gp \to 0$ of finite $D$-modules, which remains exact after base change to $B$.  Therefore, since $B \otimes_\bZ Q^\gp$ is a finite free $B$-module, so is $B \otimes_D \Omega^{\log}_{D / A}$.  By construction \Pth{see \Refeq{\ref{eq-def-log-diff-aff}}}, we have canonical surjections $B \otimes_D \Omega^{\log}_{D / A} \to B \otimes_{D'} \Omega^{\log}_{D' / A} \to \Omega^{\log}_{B / A}$ of finite $B$-modules.  \Pth{This assertion can be interpreted as a comparison between the \emph{second fundamental exact sequences} associated with the strict surjections $D \Surj B$ and $D'' \Surj B$ via the strict surjection $D \Surj D''$.}  The first morphism $B \otimes_D \Omega^{\log}_{D / A} \to B \otimes_{D'} \Omega^{\log}_{D' / A}$ is an isomorphism, because its kernel is generated over $B \cong D / H$ by $1 \otimes d(x y) = x \otimes dy + y \otimes dx = 0$ in $B \otimes_D \Omega^{\log}_{D / A}$, for all $x, y \in H$, by \Refeq{\ref{eq-def-log-diff-aff-quot}} and \Refeq{\ref{eq-def-log-diff-aff-quot-mod}}.  On the other hand, any splitting $B \to D'$ above induces a splitting of the second morphism $B \otimes_{D'} \Omega^{\log}_{D' / A} \to \Omega^{\log}_{B / A}$, which embeds $\Omega^{\log}_{B / A}$ as a direct summand of $B \otimes_D \Omega^{\log}_{D' / A}$.  Thus, $\Omega^{\log}_{B / A}$ is a finite projective $B$-module as $B \otimes_D \Omega^{\log}_{D / A}$ is, as desired.
\end{proof}

\begin{prop}\label{prop-formal-log-sm-imply-log-sm}
    Let $f: Y \to X$ be a lft morphism of locally noetherian fs log adic spaces.  Then $f$ is formally log smooth \Pth{\resp formally log \'etale} as in Definition \ref{def-formal-log-sm-unram-et} if and only if it is log smooth \Pth{\resp log \'etale} as in Definition \ref{def-log-sm}.
\end{prop}
\begin{proof}
    Suppose $f$ is log smooth \Pth{\resp log \'etale}.  Then $f$ admits \'etale locally an fs chart $u: P \to Q$ as in Definition \ref{def-log-sm}, and hence is formally log smooth \Pth{\resp log \'etale} by Remarks \ref{rem-def-formal-log-sm-unram-et-compos-bc}\Refenum{\ref{rem-def-formal-log-sm-unram-et-compos-bc-1}} and \ref{rem-def-formal-log-sm-unram-et-loc}, and Proposition \ref{prop-log-diff-monoid}.

    Conversely, suppose $f$ is formally log smooth \Pth{\resp formally log \'etale}.  Let $\AC{y} = \Spa(l, l^+)$ be a geometric point of $Y$, which is mapped to a geometric point $\AC{x} = f(\AC{y})$ of $X$.  By Proposition \ref{prop-chart-stalk-chart-fs}, up to \'etale localization at $\AC{x}$, we may assume that $X$ admits an fs chart $\theta_X: P_X \to \cM_X$, with $P = \overline{\cM}_{X, \AC{x}}$.  We need to show that, up to further \'etale localization at $\AC{x}$ and $\AC{y}$, there exists some fs chart $u: P \to Q$ satisfying the conditions in Definition \ref{def-log-sm}.  \Pth{Note that $f$ remains formally log smooth \Pth{\resp formally log \'etale}, by Remarks \ref{rem-def-formal-log-sm-unram-et-compos-bc} and \ref{rem-def-formal-log-sm-unram-et-loc}.}

    Consider $\delta := \delta_{Y / X, \AC{y}}: \cM_{Y, \AC{y}} \to \Omega^{\log}_{Y / X, \AC{y}}$ \Pth{see Constructions \ref{constr-log-diff-sheaf-univ-aff} and \ref{constr-log-diff-sheaf-univ}}.  Since $\delta(t) = t^{-1} dt$ for every $t \in \cM_{Y, \AC{y}}$ that is mapped to $\cO^\times_{Y_\et, \AC{y}}$, by \Refeq{\ref{eq-def-log-diff-aff-quot}} and \Refeq{\ref{eq-def-log-diff-aff-quot-mod}}, $\delta$ induces a surjection $\cO_{Y_\et, \AC{y}} \otimes_\bZ \cM_{Y, \AC{y}}^\gp \to \Omega^{\log}_{Y / X, \AC{y}}$.  Since $f$ is formally log smooth, by Lemma \ref{lem-formal-log-sm-log-diff}, $\Omega^{\log}_{Y / X}$ is locally free of finite rank.  Take $t_1, \ldots, t_r$ in $\cM_{Y, \AC{y}}$ whose images in $\Omega^{\log}_{Y / X, \AC{y}}$ form a basis over $\cO_{Y_\et, \AC{y}}$.  Consider the homomorphism of monoids $\bZ_{\geq 0}^r \oplus P \to \cM_{Y, \AC{y}}$ induced by sending the $i$-th basis element of $\bZ_{\geq 0}^r$ to $\delta(t_i)$, and by the composition of $P \to \cM_{X, \AC{x}} \cong \bigl(f^{-1}(\cM_X)\bigr)_{\AC{y}} \to \cM_{Y, \AC{y}}$.  By assumption, $\overline{\cM}_{Y, \AC{y}} \cong \cM_{Y, \AC{y}} / \alpha_{Y, \AC{y}}^{-1}(\cO_{Y_\et, \AC{y}}^\times)$ is a sharp fs monoid.  Also, recall that $\AC{y} = \Spa(l, l^+)$.  By \Refeq{\ref{eq-def-log-diff-aff-quot}} and \Refeq{\ref{eq-def-log-diff-aff-quot-mod}} again, the canonical homomorphism $\cM_{Y, \AC{y}} \to \overline{\cM}_{Y, \AC{y}}^\gp / \im(P^\gp)$ induces a surjection $\Omega^{\log}_{Y / X, \AC{y}} \Surj l \otimes_\bZ \bigl(\overline{\cM}_{Y, \AC{y}}^\gp / \im(P^\gp)\bigr)$.  Consequently, $\bZ_{\geq 0}^r \oplus P \to \cM_{Y, \AC{y}}$ induces a surjection $l \otimes_\bZ (\bZ^r \oplus P^\gp) \Surj l \otimes_\bZ \overline{\cM}_{Y, \AC{y}}^\gp$.  Since $\overline{\cM}_{Y, \AC{y}}^\gp$ is a free abelian group of finite rank, the cokernel of $\bZ^r \oplus P^\gp \to \overline{\cM}_{Y, \AC{y}}^\gp$ is a finite group annihilated by some integer $n$ invertible in $l$, and hence in $\cO_{Y_\et, \AC{y}}$.  Since $\cO_{Y_\et, \AC{y}}^\times$ is $n$-divisible, we can \Pth{noncanonically} extend $\bZ^r \oplus P^\gp \to \overline{\cM}_{Y, \AC{y}}^\gp$ to some $h: G \to \cM_{Y, \AC{y}}^\gp$, where $G$ is some free abelian group of finite rank containing $\bZ^r \oplus P^\gp$ such that $G / (\bZ^r \oplus P^\gp)$ is annihilated by $n$, and such that the induced map $G \to \overline{\cM}_{Y, \AC{y}}^\gp$ is surjective.

    Let $Q_1 := h^{-1}(\cM_{Y, \AC{y}})$.  By construction, $P$ is a submonoid of $Q_1$, the induced map $Q_1 \to \cM_{Y, \AC{y}}$ is strict, and the torsion part of $Q_1^\gp / P^\gp$ is annihilated by $n$.  By the same argument as in the proof of Lemma \ref{lem-fs-split-int}, there is a finitely generated submonoid $Q_2$ of $Q_1$ such that $Q_2^\gp = Q_1^\gp$, the induced map $Q_2 \to \cM_{Y, \AC{y}}$ is still strict, and the composition of $Q_2 \to \cM_{Y, \AC{y}} \to \cO_{Y_\et, \AC{y}}$ factors through $\cO_{Y_\et, \AC{y}}^+$.  Let $Q$ be the saturation of the submonoid of $Q_1$ generated by $Q_2$ and $P$, which is an fs submonoid of $Q_1$ with the same properties.  Let $u: P \to Q$ denote the induced map.  By construction, $u$ is injective and compatible with $P \to \cM_{X, \AC{x}} \to \cM_{Y, \AC{y}}$ and $Q \to \cM_{Y, \AC{y}}$, and the torsion part of $Q^\gp / u(P^\gp)$ is annihilated by $n$.  Since these are all finitely generated monoids, because of the explanation in Remark \ref{rem-fg}, up to \'etale localization at $\AC{y}$ we may assume that $Q \to \cM_{Y, \AC{y}}$ extends to a chart $\theta_Y: Q_Y \to \cM_Y$; that $\theta_X$, $\theta_Y$, and $u: P \to Q$ form an fs chart of $f$, as in Definition \ref{def-chart-mor}; that $n$ is invertible in $\cO_Y$; and that $\cO_{Y, \AC{y}} \otimes_\bZ \bigl(Q^\gp / u^\gp(P^\gp)\bigr) \cong \Omega^{\log}_{Y / X, \AC{y}}$.  It remains to show that $u: P \to Q$ satisfies the conditions in Definition \ref{def-log-sm} after these \'etale localizations.

    We already know that $\ker(u^\gp) = 0$ and the torsion part of $Q^\gp / u^\gp(P^\gp)$ is annihilated by $n$.  If $f$ is formally log \'etale \Pth{and hence formally log unramified}, by Theorem \ref{thm-log-diff-fund}\Refenum{\ref{thm-log-diff-fund-3}} \Pth{and the construction of $\Omega^{\log}_{Y / X}$ over affinoid coverings}, we have $\Omega^{\log}_{Y / X, \AC{y}} = 0$, in which case the whole $Q^\gp / u^\gp(P^\gp)$ is torsion and therefore annihilated by $n$.  Thus, $u$ satisfies the condition \Refenum{\ref{def-log-sm-1}} of Definition \ref{def-log-sm}.

    Let $g: Y \to Y' := X \times_{X\Talg{P}} X\Talg{Q} \cong X \times_{\Spa(k\Talg{P}, k^+\Talg{P})} \Spa(k\Talg{Q}, k^+\Talg{Q})$ be the morphism induced by the chart $\theta_X$, $\theta_Y$, and $u$.  Note that $g$ is \emph{strict}.  Since $\Omega^{\log}_{Y / X, \AC{y}}$ is locally free of finite rank, up to further localization at $\AC{y}$, we may assume that $\theta_Y$ and $\delta_{Y / X}$ induce $\cO_Y \otimes_\bZ \bigl(Q^\gp / u^\gp(P^\gp)\bigr) \cong \Omega^{\log}_{Y / X}$.  By Proposition \ref{prop-log-diff-sheaf-bc} \Pth{and the construction of sheaves of log differentials over affinoid coverings}, the canonical morphism $g^*(\Omega^{\log}_{Y' / X}) \to \Omega^{\log}_{Y / X}$ is an isomorphism.  Since $f: Y \to X$ is formally smooth, by Theorem \ref{thm-log-diff-fund}\Refenum{\ref{thm-log-diff-fund-5}} and Remark \ref{rem-def-formal-log-sm-unram-et-loc}, $g: Y \to Y'$ is formally log \'etale.  Since $Y'$ is integral, by Lemma \ref{lem-def-formal-log-sm-unram-et-strict}, the underlying lft morphism of $g$ is formally \'etale, and hence \'etale \Pth{see the definition and the equivalent formulations in \cite[\aSecs 1.6 and 1.7]{Huber:1996-ERA}}.  Thus, $u$ also satisfies the condition \Refenum{\ref{def-log-sm-2}} of Definition \ref{def-log-sm}.
\end{proof}

\begin{thm}\phantomsection\label{thm-log-diff-sheaf-fund}
    \begin{enumerate}
        \item\label{thm-log-diff-sheaf-fund-1}  A composition of lft morphisms $Y \Mapn{f} X \Mapn{g} S$ of locally noetherian coherent log adic spaces naturally induces an exact sequence $f^*(\Omega^{\log}_{X / S}) \to \Omega^{\log}_{Y / S} \to \Omega^{\log}_{Y / X} \to 0$ of coherent $\cO_Y$-modules.

        \item\label{thm-log-diff-sheaf-fund-2}  If $f$ is a log smooth morphism of locally noetherian fs log adic spaces, then $f^*(\Omega^{\log}_{X / S}) \to \Omega^{\log}_{Y / S}$ is injective, and $\Omega^{\log}_{Y / X}$ is a locally free $\cO_Y$-module of finite rank.  \Pth{\'Etale locally on $X$ and $Y$, when $f$ admits an fs chart $u: P \to Q$ as in Definition \ref{def-log-sm}, the rank of $\Omega^{\log}_{Y / X}$ as a locally free $\cO_Y$-module is equal to the rank of $Q^\gp / u^\gp(P^\gp)$ as a finitely generated abelian group.}

        \item\label{thm-log-diff-sheaf-fund-3}  If $f$ is a log \'etale morphism of locally noetherian fs log adic spaces, then $f^*(\Omega^{\log}_{X / S}) \cong \Omega^{\log}_{Y / S}$ and $\Omega^{\log}_{Y / X} = 0$.

        \item\label{thm-log-diff-sheaf-fund-4}  If $X$, $Y$, and $S$ are locally noetherian fs log adic spaces, and if $g \circ f$ is log smooth, then the converses of \Refenum{\ref{thm-log-diff-sheaf-fund-2}} and \Refenum{\ref{thm-log-diff-sheaf-fund-3}} hold.

        \item\label{thm-log-diff-sheaf-fund-5}  If $X$, $Y$, and $S$ are locally noetherian fs log adic spaces, and if $g$ is log \'etale, then $\Omega^{\log}_{Y / S} \Mi \Omega^{\log}_{Y / X}$, and $f$ is log smooth \Pth{\resp log \'etale} if and only if $g \circ f$ is log smooth \Pth{\resp log \'etale}.
    \end{enumerate}
\end{thm}
\begin{proof}
    By the construction of sheaves of log differentials, the assertion \Refenum{\ref{thm-log-diff-sheaf-fund-1}} follows from Theorem \ref{thm-log-diff-fund}\Refenum{\ref{thm-log-diff-fund-1}}.  The assertions \Refenum{\ref{thm-log-diff-sheaf-fund-2}} and \Refenum{\ref{thm-log-diff-sheaf-fund-3}} follow from Theorem \ref{thm-log-diff-fund} \Refenum{\ref{thm-log-diff-fund-2}} and \Refenum{\ref{thm-log-diff-fund-4}}, and Propositions \ref{prop-log-diff-monoid} and \ref{prop-log-diff-sheaf-bc}.  The assertion \Refenum{\ref{thm-log-diff-sheaf-fund-4}} follows from Theorem \ref{thm-log-diff-fund}\Refenum{\ref{thm-log-diff-fund-5}}, Remark \ref{rem-def-formal-log-sm-unram-et-loc}, and Proposition \ref{prop-formal-log-sm-imply-log-sm}.  The assertion \Refenum{\ref{thm-log-diff-sheaf-fund-5}} follows from the assertion \Refenum{\ref{thm-log-diff-sheaf-fund-1}}, Remark \ref{rem-def-formal-log-sm-unram-et-compos-bc}\Refenum{\ref{rem-def-formal-log-sm-unram-et-compos-bc-2}}, and Proposition \ref{prop-formal-log-sm-imply-log-sm}.
\end{proof}

\begin{cor}\label{cor-log-diff-sheaf-strict-cl-imm}
    Suppose that $\widetilde{X} \Mapn{f} X \Mapn{g} S$ are morphisms of locally noetherian fs log adic spaces such that $g$ is log smooth; such that the underlying morphism of adic spaces of $f$ is an isomorphism; and such that the canonical homomorphism $\cM_{X, \AC{x}} \to \cM_{\widetilde{X}, \AC{x}}$ of fs monoids splits as a direct summand, with $\alpha_{\widetilde{X}, \AC{x}}$ mapping the split image of $(\cM_{\widetilde{X}, \AC{x}} / \cM_{X, \AC{x}}) - \{ 0 \}$ to $0$ in $\cO_{\widetilde{X}_\et, \AC{x}} \cong \cO_{X_\et, \AC{x}}$, at each geometric point $\AC{x}$ of $X$.  Then $\Omega^{\log}_{\widetilde{X} / S, \AC{x}} \cong \Omega^{\log}_{X / S, \AC{x}} \oplus \bigl(\cO_{X_\et, \AC{x}} \otimes_\bZ (\cM^{\gp}_{\widetilde{X}, \AC{x}} / \cM^{\gp}_{X, \AC{x}})\bigr)$ at each $\AC{x}$.  Moreover, if there is a strict closed immersion $\imath: \widetilde{X} \to Y$ to a log adic space $Y$ log smooth over $S$, then we also have $\Omega^{\log}_{\widetilde{X} / S} \cong \imath^*(\Omega^{\log}_{Y / S})$.
\end{cor}
\begin{proof}
    This follows from Theorem \ref{thm-log-diff-sheaf-fund} and Corollary \ref{cor-log-diff-monoid-strict-cl-imm}.
\end{proof}

\begin{defn}\label{def-log-diff-sheaf-ex}
    Let $X \to S$ be a log smooth morphism of locally noetherian fs log adic spaces.  Then $\Omega^{\log}_{X / S}$ is a locally free $\cO_X$-module of finite rank, by Theorem \ref{thm-log-diff-sheaf-fund}\Refenum{\ref{thm-log-diff-sheaf-fund-2}}, and we set $\Omega^{\log, \bullet}_{X / S} := \Ex^a \, \Omega^{\log}_{X / S}$, for each integer $a \geq 0$.  More generally, for any $\widetilde{X}$ over $X$ such that $\widetilde{X} \to X \to S$ is as in Corollary \ref{cor-log-diff-sheaf-strict-cl-imm}, and such that $\widetilde{X}$ admits a strict closed immersion to a log adic space $Y$ log smooth over $S$, we also set $\Omega^{\log, a}_{\widetilde{X} / S} := \Ex^a \, \Omega^{\log}_{\widetilde{X} / S}$, which is canonically isomorphic to the pullback of $\Omega^{\log, a}_{Y / S}$, for each integer $a \geq 0$.  When $S = \Spa(k, k^+)$, for some nonarchimedean field $k$ with $k^+ = \cO_k$, and when there is no risk of confusion in the context, we shall often omit $S$ and $k$ from the notation, for the sake of simplicity.  In particular, when $X$ is log smooth over $k$ as in Definition \ref{def-log-sm-base-field}, we shall simply write $\Omega^{\log}_X$ and $\Omega^{\log, \bullet}_X$.
\end{defn}

\begin{exam}\label{ex-log-diff-sheaf-ncd}
    In Example \ref{ex-log-adic-sp-ncd-strict-cl-imm}, the morphisms $X \leftarrow X_J^\partial \to X_J$ satisfy the requirements of the morphisms $Y \leftarrow \widetilde{X} \to X$ in the second half of Corollary \ref{cor-log-diff-sheaf-strict-cl-imm}, and hence we have a canonical isomorphism $\Omega^{\log, \bullet}_{X_J^\partial} \cong (X_J \to X)^*(\Omega^{\log, \bullet}_X)$ and \'etale locally on $X^\partial_{J, \an} \cong X_{J, \an}$ \Pth{depending on the choices of coordinates} some isomorphisms $\Omega^{\log}_{X_J^\partial} \cong \Omega^{\log}_{X_J} \oplus \cO_X^J$ of vector bundles.
\end{exam}

\section{Kummer \'etale topology}\label{sec-ket}

\subsection{The Kummer \'etale site}\label{sec-ket-site}

\begin{defn}\label{def-Kummer}
    A homomorphism $u: P \to Q$ of saturated monoids is called \emph{Kummer} if it is injective and if the following conditions hold:
    \begin{enumerate}
        \item For any $a \in Q$, there exists some integer $n \geq 1$ such that $n a \in u(P)$.

        \item The quotient $Q^\gp / u^\gp(P^\gp)$ is a finite group.
    \end{enumerate}
\end{defn}

\begin{defn}\phantomsection\label{def-ket-mor}
    \begin{enumerate}
        \item\label{def-ket-mor-1} A morphism \Pth{\resp finite morphism} $f: Y \to X$ of locally noetherian fs log adic spaces is called \emph{Kummer} \Pth{\resp \emph{finite Kummer}} if it admits, \'etale locally on $X$ and $Y$ \Pth{\resp \'etale locally on $X$}, an fs chart $u: P \to Q$ that is Kummer as in Definition \ref{def-Kummer}.  Such a chart $u$ is called a \emph{Kummer chart} of $f$.

        \item\label{def-ket-mor-2} An $f$ as above is called \emph{Kummer \'etale} \Pth{\resp \emph{finite Kummer \'etale}} if the Kummer chart $u$ above can be chosen such that $|Q^\gp / u^\gp(P^\gp)|$ is invertible in $\cO_Y$, and such that $f$ and $u$ induce a morphism $Y \to X \times_{X\Talg{P}} X\Talg{Q}$ of log adic spaces \Pth{\Refcf{} Remark \ref{rem-def-chart-rel}} whose underlying morphism of adic spaces is \'etale \Pth{\resp finite \'etale}.

        \item\label{def-ket-mor-3} A Kummer morphism is called a \emph{Kummer cover} if it is surjective.
    \end{enumerate}
\end{defn}

\begin{rk}\label{rem-def-ket-mor}
    Definition \ref{def-ket-mor} can be extended beyond the case of locally noetherian fs log adic spaces, with suitable $P$ and $Q$, when all adic spaces involved are \'etale sheafy.  However, we will not pursue this generality in this paper.
\end{rk}

\begin{rk}\label{rem-Kummer-exact}
    Any Kummer homomorphism $u: P \to Q$ as in Definition \ref{def-Kummer} is exact.  Accordingly, as we shall see in Lemma \ref{lem-ket-stalk}, any Kummer morphism $f: Y \to X$ as in Definition \ref{def-ket-mor} is exact.  In particular, Proposition \ref{prop-lem-four-pt} is applicable to Kummer morphisms.  \Pth{See also Lemma \ref{lem-ket-log-et-ex}.}
\end{rk}

\begin{defn}\phantomsection\label{def-Kummer-n-th}
    \begin{enumerate}
        \item For any saturated torsion-free monoid $P$ and any positive integer $n$, let $\frac{1}{n} P$ be the saturated torsion-free monoid such that the inclusion $P \Em \frac{1}{n} P$ is isomorphic to the $n$-th multiple map $[n]: P \to P$.

        \item Let $X$ be a locally noetherian log adic space modeled on a torsion-free fs monoid $P$, and $n$ any positive integer.  Then we have the log adic space $X^{\frac{1}{n}} := X \times_{X\Talg{P}} X\Talg{\frac{1}{n} P}$, with a natural chart modeled on $\frac{1}{n} P$.
    \end{enumerate}
\end{defn}

The structure morphism $X^{\frac{1}{n}} \to X$ is a finite Kummer cover with a Kummer chart given by the natural inclusion $P \Em \frac{1}{n} P$, which is finite Kummer \'etale when $n$ is invertible in $\cO_X$.  Such morphisms will play an important role in Sections \ref{sec-ket-coh-descent} and \ref{sec-ket-cov-descent}.  More generally, we have the following:
\begin{prop}\label{prop-ket-std}
    Let $X$ be a locally noetherian log adic space with a chart modeled on an fs monoid $P$.  Let $u: P \to Q$ be a Kummer homomorphism of fs monoids such that $G := Q^\gp / u^\gp(P^\gp)$ is a finite group.  Consider
    \[
        Y := X \times_{X\Talg{P}} X\Talg{Q},
    \]
    which is equipped with a canonical action of the group object
    \[
        G^D_X := X\Talg{G}
    \]
    over $X$, which is an analogue of a diagonalizable group scheme that is Cartier dual to the constant group scheme $G$.  Then the following hold:
    \begin{enumerate}
        \item\label{prop-ket-std-1}  The natural projection $f: Y \to X$ is a finite Kummer cover, which is finite \Pth{see \cite[(1.4.4)]{Huber:1996-ERA}} and surjective.

        \item\label{prop-ket-std-2}  When $X$ and therefore $Y$ are affinoid, we have a canonical exact sequence $0 \to \cO_X(X) \to \cO_Y(Y) \to \cO_{Y \times_X Y}(Y \times_X Y)$.

        \item\label{prop-ket-std-3}  The morphism $G^D_X \times_X Y \to Y \times_X Y$ induced by the action and the second projection is an isomorphism.

        \item\label{prop-ket-std-4}  When $G$ is annihilated by an integer $m \geq 1$ invertible in $\cO_X$, the group $G^D_X$ is \'etale over $X$ \Pth{which is simply the constant group $\Hom(G, \cO_X(X)^\times)_X$ when $\cO_X(X)$ contains all the $m$-th roots of unity}; and the cover $f: Y \to X$ is a Galois finite Kummer \'etale cover with Galois group $G^D_X$, which is open.
    \end{enumerate}
\end{prop}

For the proof of Proposition \ref{prop-ket-std}, we need the following general construction, which will also be useful later in Section \ref{sec-ket-cov-descent}.
\begin{lem}\label{lem-quot}
    Let $Y = \Spa(S, S^+) \to X = \Spa(R, R^+)$ be a finite morphism of noetherian adic spaces, and let $\Gamma$ be a finite group which acts on $Y$ by morphisms over $X$.  Then $(T, T^+) := (S^\Gamma, (S^+)^\Gamma)$ is a Huber pair, and $Z := \Spa(T, T^+)$ is a noetherian adic space finite over $X$.  Moreover, the canonical morphism $Y \to X$ factors through a finite, open, and surjective morphism $Y \to Z$, which induces a homeomorphism $Y / \Gamma \Mi Z$ of underlying topological spaces and identifies $Z$ as the categorical quotient $Y / \Gamma$ in the category of adic spaces.
\end{lem}
\begin{proof}
    For analytic adic spaces, and for any finite group $\Gamma$ such that $|\Gamma|$ is invertible in $S$, this essentially follows from \cite[\aThm 1.2]{Hansen:2016-qasfg} without the noetherian hypothesis. Nevertheless, we have the noetherian hypothesis, but not the analytic or invertibility assumptions here.  Moreover, we have a base space $X$ over which $Y$ is finite.  Hence, we can resort to the following more direct arguments.

    Since $R$ is noetherian, and since $S$ is a finite $R$-module, $T = S^\Gamma$ is also a finite $R$-module, and $T^+ = (S^+)^\Gamma$ is the integral closure of $R^+$ in $T$.  Therefore, $(T, T^+)$ has a canonical structure of a Huber pair such that $Z := \Spa(T, T^+)$ is a noetherian adic space finite over $X = \Spa(R, R^+)$.  Moreover, $Y \to Z$ is also finite.  By \cite[(1.4.2) and (1.4.4)]{Huber:1996-ERA} and \cite[\aSec 2]{Huber:1994-gfsra}, if $\{ s_1, \ldots, s_r \}$ is any set of generators of $S$ as an $T$-module, then the topology of $S$ is generated by $\sum_{i = 1}^r \, U_i \, s_i$, where $U_i$ runs through a basis of the topology of $T$, for all $i$.  Suppose that $w: T \to \Gamma_w$ is any continuous valuation, and that $v: S \to \Gamma_v$ is any valuation extending $w$.  Note that $v$ and $w$ factor through the domains $\overline{S} := S / v^{-1}(0)$ and $\overline{T} := T / w^{-1}(0)$, respectively, and $\Frac(\overline{S})$ is a finite extension of $\Frac(\overline{T})$.  Therefore, we may assume that $\Gamma_v$ and $\Gamma_w$ are generated by $v(S)$ and $w(T)$, respectively, and that $\Gamma_w$ is a finite index subgroup of $\Gamma_v$.  For each $\gamma \in \Gamma_v$, the subgroup $\{ s \in S : v(s) < \gamma \}$ of $S$ is open because it contains $\sum_{i = 1}^r \, U_i \, s_i$, where $U_i := \{ t \in T : w(t) < \gamma - v(s_i) \}$ is open by the continuity of $w$.  Consequently, $\Cont(S) \to \Cont(T)$ is surjective.  This replaces the main argument in Step 1 of the proof of \cite[\aThm 3.1]{Hansen:2016-qasfg} where the Tate assumption is used.  After this step, the remaining arguments in the proof of \cite[\aThm 3.1]{Hansen:2016-qasfg} work verbatim and show that $\Spa(S, S^+) \to \Spa(T, T^+)$ induces a homeomorphism $\Spa(S, S^+) / \Gamma \to \Spa(T, T^+)$.

    Since $Y$ and $Z$ are finite over $X$, and since $T = S^\Gamma$, by \cite[(1.4.4)]{Huber:1996-ERA} and Proposition \ref{prop-coh-descent}, the canonical morphism $\cO_Z \to (Y \to Z)_*(\cO_Y)$ factors through an isomorphism $\cO_Z \Mi \bigr((Y \to Z)_*(\cO_Y)\bigr)^\Gamma$.  \Pth{This provides a replacement of \cite[\aThm 3.2]{Hansen:2016-qasfg}.}  Thus, the canonical morphism $Y \to Z$ factors through an isomorphism $Y / \Gamma \Mi Z$ of adic spaces, as in \cite[\aThm 1.2]{Hansen:2016-qasfg}, as desired.
\end{proof}

Now we are ready for the following:
\begin{proof}[Proof of Proposition \ref{prop-ket-std}]
    Let us identify $P$ as a submonoid of $Q$ via the injection $u: P \to Q$.  Since the assertions are local in nature on $X$, we may assume that $X = \Spa(R, R^+)$ and hence $Y$ is affinoid.  Then $\cO_X(X) \to \cO_Y(Y)$ is injective, because it is the base change of $\bZ[P] \to \bZ[Q]$ from $\bZ[P]$ to $R$, and $\bZ[P]$ is a direct summand of $\bZ[Q]$ as $\bZ[P]$-modules, as explained in the proof of \cite[\aLem 2.1]{Illusie/Nakayama/Tsuji:2013-olfd}.  Moreover, since $\bZ[P] \to \bZ[Q]$ is finite because $Q$ is finitely generated and $u$ is Kummer, its base change $\cO_X(X) \to \cO_Y(Y)$ is also finite, and therefore \Refenum{\ref{prop-ket-std-1}} follows.  Since the canonical sequence $\cO_X(X) \to \cO_Y(Y) \to \cO_{Y \times_X Y}(Y \times_X Y)$ is exact by \cite[\aLem 3.28]{Niziol:2008-ktls-1}, \Refenum{\ref{prop-ket-std-2}} also follows.

    By \cite[\aLem 3.3]{Illusie:2002-fknle}, $(Q \oplus_P Q)^\Sat \cong Q \oplus G$.  Accordingly, by Remark \ref{rem-fiber-prod-chart}, $X\Talg{G} \times_X X\Talg{Q} \cong X\Talg{Q} \times_{X\Talg{P}} X\Talg{Q}$, and the action of $G^D_X = X\Talg{G}$ on $Y$ induces $X\Talg{G} \times_X Y \Mi Y \times_X Y$.  This verifies \Refenum{\ref{prop-ket-std-3}}.

    As for \Refenum{\ref{prop-ket-std-4}}, since it can be verified \'etale locally on $X$, we may assume that $\cO_X(X) = R$ contains all $|G|$-th roots of unity.  In this case, $X\Talg{G} \cong \Gamma_X$, where $\Gamma := \Hom(G, \cO_X(X)^\times)$, and the action of $G^D_X$ is induced by the canonical actions of $\Gamma$ on $X\Talg{Q}$ and $X\Talg{Q^\gp}$, by sending $q$ to $\gamma(q) q$, for each $q \in Q^\gp$ and $\gamma \in \Gamma$.  Note that $(R\Talg{Q})^\Gamma = (R\Talg{Q^\gp})^\Gamma \cap R\Talg{Q} = R\Talg{P^\gp} \cap R\Talg{Q} = R\Talg{P}$, where the last one follows from the assumptions that $u: P \to Q$ is Kummer and that $P$ is saturated; and the formation of $\Gamma$-invariants commutes with the base change from $R\Talg{P}$ to $R$, because $|\Gamma|$ is invertible in $R$.  Thus, if $Y = \Spa(S, S^+)$, then the morphism $Y \to X = \Spa(R, R^+) \cong \Spa(S^\Gamma, (S^+)^\Gamma)$ is open and induces an isomorphism $Y / \Gamma \Mi X$, by Lemma \ref{lem-quot}.  Moreover, for any subgroup $\Gamma'$ of $\Gamma = \Hom(G, R^\times)$, which is of the form $\Hom(G', R^\times)$ for some quotient $G'$ of $G = Q^\gp / u^\gp(P^\gp)$, we have $(R\Talg{Q})^{\Gamma'} \cong R\Talg{Q'}$ and $(R^+\Talg{Q})^{\Gamma'} \cong R^+\Talg{Q'}$, where $Q'$ is the preimage of $\ker(G \to G')$ under the canonical homomorphism $Q \to G = Q^\gp / u^\gp(P^\gp)$, so that $Y / \Gamma' \cong X \times_{X\Talg{P}} X\Talg{Q'} \to X$ is a finite Kummer \'etale.  Consequently, $f: Y \to X$ is a Galois finite Kummer \'etale cover with Galois group $\Gamma$, as desired.
\end{proof}

\begin{defn}\label{def-ket-std}
    Kummer \Pth{\resp Kummer \'etale} covers $f: Y \to X$ as in Proposition \ref{prop-ket-std} are called \emph{standard Kummer} \Pth{\resp \emph{standard Kummer \'etale}} \emph{covers}.
\end{defn}

\begin{cor}\label{cor-ket-op}
    Kummer \'etale morphisms are open.
\end{cor}
\begin{proof}
    This is because, by definition, Kummer \'etale morphisms are, \'etale locally on the source and target, compositions of standard Kummer \'etale covers and strictly \'etale morphisms, both of which are open.
\end{proof}

In the remainder of this subsection, let us study some general properties of Kummer \'etale morphisms.  Our goal is to introduce the \emph{Kummer \'etale site}.

\begin{lem}\label{lem-ket-chart}
    Let $f: Y \to X$ be a Kummer \'etale morphism of locally noetherian fs log adic spaces.  Suppose that $X$ is modeled on an fs monoid $P$.  Then $f$ admits, \'etale locally on $Y$ and $X$, a Kummer chart $P \to Q$ as in Definition \ref{def-ket-mor}\Refenum{\ref{def-ket-mor-2}}, with the same prescribed $P$ as above.  Moreover, if $P$ is torsion-free \Pth{\resp sharp}, then we can choose $Q$ to be torsion-free \Pth{\resp sharp}.
\end{lem}
\begin{proof}
    \'Etale locally on $Y$ and $X$, let $u_1: P_1 \to Q_1$ be a Kummer chart of $f$ as in Definition \ref{def-ket-mor}.  \Pth{A priori, $P_1$ might be different from $P$.}  As in the proof of Proposition \ref{prop-log-sm-chart}, up to further \'etale localization and modifying $P_1 \to Q_1$, we can find some group $H$ fitting into the cartesian diagram \Refeq{\ref{eq-prop-log-sm-chart-Niziol}}.  Let
    \[
        Q := \{ a \in H : \Utext{$n a \in P$, for some $n \geq 1$ invertible in $\cO_Y$} \},
    \]
    so that $u: P \to Q$ is Kummer, as in Definition \ref{def-Kummer}.  Let $Q'$ be the preimage of $Q_1$ via $H \to Q_1^\gp$, so that $u': P \to Q'$ is an fs chart of $f$ satisfying the conditions of Definition \ref{def-log-sm}, as explained in the paragraph after \Refeq{\ref{eq-prop-log-sm-chart-Niziol}}.  Since
    \[
        Q_1 = \{ a \in Q_1^\gp : \Utext{$n a \in u_1(P_1)$, for some $n \geq 1$ invertible in $\cO_Y$} \},
    \]
    by the assumption on $u_1$, and since \Refeq{\ref{eq-prop-log-sm-chart-Niziol}} is cartesian, we can identify $Q'$ with the localization of $Q$ with respect to $\ker(Q \to Q_1)$.  Therefore, $u: P \to Q$ is an fs chart of $f$ as $u'$ is, which also satisfies the conditions of Definition \ref{def-log-sm}, or rather of Definition \ref{def-ket-mor}\Refenum{\ref{def-ket-mor-2}}.  By the proof of the last assertion of Proposition \ref{prop-log-sm-chart}, if $P$ is torsion-free, then we may assume that $Q$ is torsion-free as well.  Finally, suppose that $P$ is sharp and $Q$ is torsion-free.  For any $q \in Q^\inv$, there is some $n \geq 1$ such that $n q$ and $-n q$ are both in $u(P)$ and hence in $u(P^\inv) = \{ 0 \}$.  Since $Q$ is torsion-free, this forces $q = 0$.  Thus, $Q$ is also sharp, as desired.
\end{proof}

\begin{lem}\label{lem-ket-stalk}
    Let $f: Y \to X$ be a Kummer morphism of locally noetherian fs log adic spaces.  Then:
    \begin{enumerate}[label=(\arabic*), ref=\arabic*]
        \item\label{lem-ket-stalk-exact}  The morphism $f$ is exact.

        \item\label{lem-ket-stalk-Kummer}  For any geometric point $\AC{y}$ of $Y$, the induced homomorphism of sharp fs monoids $f^\sharp_{\AC{y}}: \overline{\cM}_{X, f(\AC{y})} \to \overline{\cM}_{Y, \AC{y}}$ is Kummer.  Moreover, if $f$ is Kummer \'etale, then $\bigl|\coker\bigl((f^\sharp_{\AC{y}})^\gp\bigr)\bigr|$ is invertible in $\cO_{Y, \AC{y}}$.

        \item\label{lem-ket-stalk-ker}  Suppose that $f$ admits a Kummer chart $u: P \to Q$.  For any geometric point $\AC{y}$ of $Y$, if $K_{f(\AC{y})} := \ker(P \to \overline{\cM}_{X, f(\AC{y})})$ and $K_{\AC{y}} := \ker(Q \to \overline{\cM}_{Y, \AC{y}})$, then $K_{f(\AC{y})} = u^{-1}(K_{\AC{y}})$, and the induced homomorphism $w: K_{f(\AC{y})} \to K_{\AC{y}}$ is Kummer.  Thus, if $K_{f(\AC{y})} = 0$, then $K_{\AC{y}}$ is torsion, and is zero if $Q$ is sharp.
    \end{enumerate}
\end{lem}
\begin{proof}
    Let us start with \Refenum{\ref{lem-ket-stalk-ker}}.  By Remark \ref{rem-chart-stalk}, $u$ and $v := f^\sharp_{\AC{y}}$ are compatible with surjective homomorphisms $\theta_{f(\AC{y})}: P \to \overline{\cM}_{X, f(\AC{y})}$ and $\theta_{\AC{y}}: Q \to \overline{\cM}_{Y, \AC{y}}$, with kernels $K_{f(\AC{y})}$ and $K_{\AC{y}}$ given by preimages of $\cO_{X_\et, f(\AC{y})}^\times$ and $\cO_{Y_\et, \AC{y}}^\times$, respectively.  Since $f^\sharp_{\AC{y}}: \cO_{X_\et, f(\AC{y})} \to \cO_{Y_\et, \AC{y}}$ is local, $K_{f(\AC{y})} = u^{-1}(K_{\AC{y}})$.  If $\overline{a} \in \overline{\cM}_{X, f(\AC{y})}$ and $v(\overline{a}) = 0$, then $\overline{a} = \theta_{f(\AC{y})}(a)$, for some $a \in P$ such that $u(a) \in K_{\AC{y}}$.  Hence, $a \in K_{f(\AC{y})}$, and $\overline{a} = \theta_{f(\AC{y})}(a) = 0$.  This shows that $v$ is injective.  Since $u: P \to Q$ is Kummer, if $b \in K_{\AC{y}} \subset Q$, then $n b = u(a)$ for some integer $n \geq 1$ and $a \in P$, and $v$ maps $\theta_{f(\AC{y})}(a)$ to $\theta_{\AC{y}}(n b) = 0$, and so $\theta_{f(\AC{y})}(a) = 0$ by the injectivity of $v$.  Therefore, $a \in K_{f(\AC{y})}$.  It follows that the induced homomorphism $w: K_{f(\AC{y})} \to K_{\AC{y}}$ is Kummer, with $\coker(w^\gp)$ given by a subgroup of $\coker(u^\gp)$, and \Refenum{\ref{lem-ket-stalk-ker}} follows.

    Next, let us verify \Refenum{\ref{lem-ket-stalk-Kummer}}.  By assumption, up to \'etale localization on $Y$ and $X$, the morphism $f$ admits a Kummer chart $u: P \to Q$.  If $\overline{b} \in \overline{\cM}_{Y, \AC{y}}$, then $\overline{b} = \theta_{\AC{y}}(b)$, for some $b \in Q$.  Since $u$ is Kummer, $n b = u(a)$, for some integer $n \geq 1$ and $a \in P$.  Then $v$ maps $\overline{a} := \theta_{f(\AC{y})}(a)$ to $\theta_{\AC{y}}(n b) = n \overline{b}$.  Furthermore, $\coker(v^\gp)$ is a finite group, because it is a quotient of $\coker(u^\gp)$.  Thus, $v$ is Kummer.  By Lemma \ref{lem-ket-chart}, if $f$ is Kummer \'etale, then we may assume that the order of $\coker(u^\gp)$ is invertible in $\cO_{Y, \AC{y}}$, and the same is true for its quotient $\coker(v^\gp)$.

    Finally, since any Kummer homomorphism of monoids is exact \Pth{see Remark \ref{rem-Kummer-exact}}, \Refenum{\ref{lem-ket-stalk-exact}} follows from \Refenum{\ref{lem-ket-stalk-Kummer}}, \cite[\aProp I.4.2.1]{Ogus:2018-LLG}, and Remark \ref{rem-stalk-sharp}.
\end{proof}

\begin{defn}\label{def-ram-index}
    In Lemma \ref{lem-ket-stalk}, the \emph{ramification index} of $f$ at $\AC{y}$ is defined to be the smallest positive integer $n$ that annihilates $\coker(\overline{\cM}_{X, f(\AC{y})}^\gp \to \overline{\cM}_{Y, \AC{y}}^\gp)$.  The \emph{ramification index} of a Kummer \'etale morphism $f$ is the least common multiple, when defined, of the ramification indices among the geometric points $\AC{y}$ of $Y$.  \Pth{The ramification index is not always defined.}  The ramification index of a Kummer \'etale morphism is 1 if and only if $f$ is strictly \'etale.
\end{defn}

\begin{lem}\label{lem-ket-log-et-ex}
    A morphism $f: Y \to X$ of locally noetherian fs log adic spaces is Kummer \'etale if and only if it is log \'etale and Kummer, and if and only if it is log \'etale and exact.  It is finite Kummer \'etale if and only it is log \'etale and finite Kummer.
\end{lem}
\begin{proof}
    If $f$ is Kummer \'etale \Pth{\resp finite Kummer \'etale}, then it is log \'etale and Kummer \Pth{\resp finite Kummer} by definition, and it is exact by Lemma \ref{lem-ket-stalk}.  Conversely, assume that $f$ is log \'etale and exact.  By Propositions \ref{prop-chart-stalk-chart-fs} and \ref{prop-log-sm-chart}, $f$ admits, \'etale locally at geometric points $\AC{y}$ of $Y$ and $\AC{x} = f(\AC{y})$ of $X$, an injective fs chart $u: P = \overline{\cM}_{X, \AC{x}} \to Q$ satisfy the conditions in Definition \ref{def-log-sm}\Refenum{\ref{def-log-sm-2}}, in which case $\coker(u^\gp)$ is a finite group whose order is invertible in $\cO_Y$.  Since $f$ is exact, by \cite[\aProp I.4.2.1]{Ogus:2018-LLG} and Remark \ref{rem-stalk-sharp}, $\overline{f}^\sharp_{\AC{y}}: \overline{\cM}_{X, \AC{x}} \to \overline{\cM}_{Y, \AC{y}}$ is exact.  Given any $b \in Q$, since $\coker(u^\gp)$ is annihilated by $n$, we have $n b = u^{gp}(a)$ for some $a \in P^\gp$.  Since $a$ is mapped to the image of $n b$ in $\overline{\cM}_{Y, \AC{y}}$, by the exactness of $\overline{f}^\sharp_{\AC{y}}$, we have $a \in P = \overline{\cM}_{X, \AC{x}}$.  Hence, $u: P \to Q$ is a Kummer chart as in Definition \ref{def-ket-mor}\Refenum{\ref{def-ket-mor-2}}.  Since $\AC{y}$ is arbitrary, $f$ is Kummer \'etale, as desired.

    Alternatively, assume that $f$ is log \'etale and finite Kummer.  Up to \'etale localization on $X$, we may assume that it admits a Kummer chart $u: P \to Q$, and that $X$ has at most one positive residue characteristic $\ell$.  When no such $\ell$ exists, we set $\ell = 0$, for simplicity.  Then we have a Kummer homomorphism
    \[
        u': P \to Q' := \{ b \in Q : \Utext{$n b \in u(P)$, for some $n \geq 1$ s.t.~$\ell \nmid n$} \}
    \]
    such that $\ell \nmid \bigl|\coker\bigl((u')^\gp\bigr)\bigr|$.  We claim that the morphism of adic spaces
    \[
        Y \to X \times_{X\Talg{P}} X\Talg{Q'}
    \]
    induced by $u'$ is \'etale.  Note that this can be verified up to \'etale localization on $Y$.  By the previous paragraph, $f$ is Kummer \'etale.  By Lemmas \ref{lem-ket-chart} and \ref{lem-ket-stalk}, $f$ admits, \'etale locally at geometric points $\AC{y}$ of $Y$ and $f(\AC{y})$ of $X$, another Kummer chart $u_1: P_1 \to Q_1$, with $P_1 \Mi \overline{\cM}_{X, f(\AC{y})}$, $Q_1 \Mi \overline{\cM}_{Y, \AC{y}}$, and $\ell \nmid |\coker(u_1^\gp)|$, such that the morphism
    \[
        Y \to X \times_{X\Talg{P_1}} X\Talg{Q_1}
    \]
    induced by $u_1$ is \'etale.  Note that $\ell \nmid |\coker(u_1^\gp)|$ implies that $Q' \to \overline{\cM}_{Y, \AC{y}}$ is surjective as $Q \to \overline{\cM}_{Y, \AC{y}}$ is, and so $u'$ is an fs chart as $u$ is.  Then $u': P \to Q'$ and $u_1: P_1 \to Q_1$ compatibly extend to a Kummer homomorphism $u_2: P_2 \to Q_2$ of fs monoids, where $P_2$ \Pth{\resp $Q_2$} is the localization of $P \oplus P_1$ \Pth{\resp $Q' \oplus Q_1$} with respect to the kernel of $P \oplus P_1 \to \cM_{X, f(\AC{y})}$ \Pth{\resp $Q' \oplus Q_1 \to \cM_{Y, \AC{y}}$}.  Since $P_1 \Mi \overline{\cM}_{X, f(\AC{y})}$ and $Q_1 \Mi \overline{\cM}_{Y, \AC{y}}$, we have $P_2 = P_1 \oplus P_2^\inv$, $Q_2 = Q_1 \oplus Q_2^\inv$, and $u_2 = u_1 \oplus u_2^\inv$, for some Kummer homomorphism $u_2^\inv: P_1^\inv \to H_1^\inv$ such that $\ell \nmid |\coker(u_2^\inv)|$.  In this case, the above morphism $Y \to X \times_{X\Talg{P}} X\Talg{Q'}$ induced by $u'$ is the composition of the morphisms
    \[
        Y \to X \times_{X\Talg{P_2}} X\Talg{Q_2} \to X \times_{X\Talg{P_2}} (X\Talg{P_2} \times_{X\Talg{P}} X\Talg{Q'}) \cong X \times_{X\Talg{P}} X\Talg{Q'},
    \]
    induced by $u_2$ and the canonical homomorphisms among $P$, $Q'$, $P_2$, and $Q_2$.  Note that the second morphism is the pullback of the canonical morphism
    \[
        X\Talg{Q_2} \to X\Talg{P_2} \times_{X\Talg{P}} X\Talg{Q'}.
    \]
    Let $P_2 + Q'$ be the submonoid of $Q_2^\gp$ generated by the images of $P_2$ and $Q'$, and $G' := \coker\bigl((u')^\gp\bigr)$.  Then we have an isomorphism of monoids
    \[
        (P_2 \oplus_P Q')^\Sat \Mi (P_2 + Q') \oplus G': (a, b) \mapsto (a + b, \overline{b}),
    \]
    where $a \in P_2$ and $b \in Q'$, and $\overline{b} \in G'$ denotes the image of $b$, as in the proof of \cite[\aLem 3.3]{Illusie:2002-fknle}.  Moreover, we have an induced isomorphism of adic spaces
    \[
        X\Talg{P_2} \times_{X\Talg{P}} X\Talg{Q'} \Mi X\Talg{P_2 + Q'} \times_X X\Talg{G'},
    \]
    where $X\Talg{G'} \to X$ is \'etale, as in Proposition \ref{prop-ket-std}.  Since $Q' \to \overline{\cM}_{Y, \AC{y}}$ is surjective and $Q_1 \Mi \cM_{Y, \AC{y}}$, we have $P_2 + Q' = P_2 + Q' + P_2^\inv \subset P_2 + Q' + Q_2^\inv = Q_2$ in $Q_2^\gp$, and the monoid $Q_2$ is generated by $P_2 + Q'$ and some finitely many invertible elements of $Q_2$ whose $|\coker(u_2^\inv)|$-th multiples are in $P_2 + Q'$.  Since $\ell \nmid |\coker(u_2^\inv)|$, the induced morphism $X\Talg{Q_2} \to X\Talg{P_2 + Q'}$ is \'etale, by \cite[\aProp 1.7.1]{Huber:1996-ERA}, and so is the above $X\Talg{Q_2} \to X\Talg{P_2} \times_{X\Talg{P}} X\Talg{Q'}$.  Therefore, in order to verify the above claim, it suffices to show that the morphism
    \[
        Y \to X \times_{X\Talg{P_2}} X\Talg{Q_2}
    \]
    induced by $u_2 = u_1 \oplus u_2^\inv$ is \'etale.  Again since $\ell \nmid |\coker(u_2^\inv)|$, this follows from the known exactness of the morphism of adic spaces $Y \to X \times_{X\Talg{P_1}} X\Talg{Q_1}$ induced by $u_1$.  Thus, $f$ is finite Kummer \'etale because it admits, \'etale locally on $X$, an fs chart $u': P \to Q'$ satisfying the conditions in Definition \ref{def-ket-mor}\Refenum{\ref{def-ket-mor-2}}.
\end{proof}

\begin{prop}\label{prop-ket-compos-bc}
    Kummer \'etale \Pth{\resp finite Kummer \'etale} morphisms as in Definition \ref{def-ket-mor} are stable under compositions and base changes under arbitrary morphisms between locally noetherian fs log adic spaces \Pth{which are justified by Remark \ref{rem-def-log-sm} and Proposition \ref{prop-log-sm-bc}}.
\end{prop}
\begin{proof}
    The stability under compositions follows from Proposition \ref{prop-log-sm-compos}, Lemma \ref{lem-ket-log-et-ex}, and the stability of exactness under compositions \Pth{by definition}.  As for the stability under base changes, it suffices to note that, if $P \to Q$ is a Kummer homomorphism \Pth{of fs monoids}, and if $P \to R$ is any homomorphism of fs monoids, then the induced homomorphism $R \to (R \oplus_P Q)^\Sat$ is also Kummer, because it is injective as the composition $R \to (R \oplus_P Q)^\Sat \to R^\gp \oplus_{P^\gp} Q^\gp$ is, and because it satisfies the conditions in Definition \ref{def-Kummer} as $P \to Q$ does.
\end{proof}

\begin{prop}\label{prop-ket-mor}
    Suppose that $f: Y \to X$ and $g: Z \to X$ are Kummer \'etale morphisms of locally noetherian fs log adic spaces.  Then any morphism $h: Y \to Z$ such that $f = g \circ h$ is also Kummer \'etale.
\end{prop}
\begin{proof}
    By Lemma \ref{lem-ket-log-et-ex}, it suffices to show that $h$ is log \'etale and exact.  By Theorem \ref{thm-log-diff-sheaf-fund}\Refenum{\ref{thm-log-diff-sheaf-fund-5}}, $h$ is log \'etale because $f = g \circ h$ and $g$ are.  By Lemma \ref{lem-ket-stalk}, \'etale locally at each geometric point $\AC{y}$ of $Y$, with $\AC{x} = f(\AC{y})$ and $\AC{z} = h(\AC{y})$, the homomorphisms $f^\sharp_{\AC{y}}: \overline{\cM}_{X, \AC{x}} \to \overline{\cM}_{Y, \AC{y}}$ and $g^\sharp_{\AC{z}}: \overline{\cM}_{X, \AC{x}} \to \overline{\cM}_{Z, \AC{z}}$ are both Kummer.  Consequently, the homomorphism $h^\sharp_{\AC{y}}: \overline{\cM}_{Z, \AC{z}} \to \overline{\cM}_{Y, \AC{y}}$ is also Kummer, and hence exact.  Thus, $h$ is exact, by \cite[\aProp I.4.2.1]{Ogus:2018-LLG} and Remark \ref{rem-stalk-sharp}.
\end{proof}

By Proposition \ref{prop-lem-four-pt} and Remark \ref{rem-Kummer-exact}, and by Propositions \ref{prop-ket-compos-bc} and \ref{prop-ket-mor}, we are now ready for the following:
\begin{defn}\label{def-ket-site}
    Let $X$ be a locally noetherian fs log adic space.  The \emph{Kummer \'etale site} $X_\ket$ has as underlying category the full subcategory of the category of locally noetherian fs log adic spaces consisting of objects that are Kummer \'etale over $X$, and has coverings given by the topological coverings.
\end{defn}

\begin{rk}\label{rem-ket-site-mor}
    Let $X$ be as in Definition \ref{def-ket-site}.
    \begin{enumerate}
        \item For each $U \in X_\et$, we can view $U$ as a log adic space by restricting the log structure $\alpha: \cM_X \to \cO_{X_\et}$ to $U_\et$.  This gives rise to a strictly \'etale morphism $U \to X$ of log adic spaces, which is Kummer \'etale by definition.  Therefore, we obtain a natural projection of sites $\varepsilon_\et: X_\ket \to X_\et$, which is an isomorphism when the log structure of $X$ is trivial.

        \item For any morphism $f: Y \to X$ of locally noetherian fs log adic spaces, we have a natural morphism of sites $f_\ket: Y_\ket \to X_\ket$, because base changes of Kummer \'etale morphisms are still Kummer \'etale, by Proposition \ref{prop-ket-compos-bc}.
    \end{enumerate}
\end{rk}

\begin{rk}\label{rem-ket-site-gen}
    By definition, the Kummer \'etale topology on $X$ is generated by surjective \Pth{strictly} \'etale morphisms and standard Kummer \'etale covers.
\end{rk}

\subsection{Abhyankar's lemma}\label{sec-Abhyankar}

An important class of finite Kummer \'etale covers arise in the following way:
\begin{prop}[rigid Abhyankar's lemma]\label{prop-Abhyankar}
    Let $X$ be a smooth rigid analytic variety over a nonarchimedean field $k$ of characteristic zero, and let $D$ be a normal crossings divisor of $X$.  We equip $X$ with the fs log structure induced by $D$ as in Example \ref{ex-log-adic-sp-ncd}.  Suppose that $h: V \to U := X - D$ is a finite \'etale surjective morphism of rigid analytic varieties over $k$.  Then it extends to a finite surjective and Kummer \'etale morphism of log adic spaces $f: Y \to X$, where $Y$ is a normal rigid analytic variety with its log structures defined by the preimage of $D$.  Consequently, $Y_\an$ has a basis consisting of affinoid $W$ satisfying $\pi_0\bigl(W \cap f^{-1}(U)\bigr) = \pi_0(W)$.
\end{prop}
\begin{proof}
    By \cite[\aThm 1.6]{Hansen:2020-vcrag} \Pth{which was based on \cite[\aThm 3.1 and its proof]{Lutkebohmert:1993-repaf}}, $h: V \to U$ extends to a finite ramified cover $f: Y \to X$, for some \emph{normal} rigid analytic variety $Y$ \Pth{viewed as a noetherian adic space}.  Then $Y_\an$ has a basis consisting of affinoid open subspaces $W$ satisfying $\pi_0\bigl(W \cap f^{-1}(U)\bigr) = \pi_0(W)$, by the unique existence of extensions of bounded functions \Pth{which include locally constant functions, in particular} from $W \cap f^{-1}(U)$ to $W$, for any affinoid open subspaces $W$ of $Y$, by \cite[\aSec 3]{Bartenwerfer:1976-erhnf} \Pth{see also \cite[\aThm 2.6]{Hansen:2020-vcrag}}.  Just as $X$ is equipped with the log structure defined by $D$, we equip $Y$ with the log structure defined by the preimage of $D$.  The question is whether the map $f$ is Kummer \'etale \Pth{with respect to the log structures on $X$ and $Y$}, and such a question can be answered analytic locally on $X$, up to replacing $k$ with a finite extension.  As in Example \ref{ex-log-adic-sp-ncd}, we may assume that there is an affinoid smooth rigid analytic variety $S$ over $k$ such that $X = S \times \bD^r \cong S\Talg{\bZ_{\geq 0}^r}$ \Pth{see Example \ref{ex-log-adic-sp-monoid}} for some $r \in \bZ_{\geq 0}$, with $D = S \times \{ T_1 \cdots T_r = 0 \}$.  Thus, we can finish the proof of this proposition by combining the following Lemmas \ref{lem-Abhyankar-ref} and \ref{lem-Abhyankar-ess}.
\end{proof}

For simplicity, let us introduce some notation for the following two lemmas.  We write $P := \bZ_{\geq 0}^r$ and identify $\bD^r$ with $\Spa(k\Talg{P}, k^+\Talg{P})$ as in Example \ref{ex-log-adic-sp-disc}.  For each $m \in \bZ_{\geq 1}$, we also write $\frac{1}{m} P = \frac{1}{m} \bZ_{\geq 0}^r$.  For each power $\rho$ of $p$, we denote by $\bD_\rho$ the \Pth{one-dimensional} disc of radius $\rho$, so that $\bD = \bD_\rho$ when $\rho = 1$.  We also denote by $\bD_\rho^\times$ the punctured disc of radius $\rho$, and by $\bD^\times$ the punctured unit disc.  For any rigid analytic variety with a canonical morphism to $\bD^r$, we denote with a subscript \Qtn{$\rho$} \Pth{\resp superscript \Qtn{$\times$}} its pullback under $\bD_\rho^r \to \bD^r$ \Pth{\resp $(\bD^\times)^r \to \bD^r$}.

\begin{lem}\label{lem-Abhyankar-ref}
    Suppose that $X = S \times \bD^r \cong S\Talg{P}$, $D$, and $U = X - D \cong S\Talg{P}^\times$ are as in the proof of Proposition \ref{prop-Abhyankar}.  Assume there is some $\rho \leq 1$ such that, for each connected component $Y'$ of $Y_\rho$, there exist $d_1, \ldots, d_r \in \bZ_{\geq 1}$ such that induced cover $Y' \to X' := X_\rho$ is refined \Pth{\ie, admits a further cover} by some finite ramified cover $Z := S\Talg{P'}_\rho \to X'$, where $P' = \oplus_{1 \leq i \leq r} \bigl(\frac{1}{d_i} \bZ_{\geq 0}\bigr)$.  Then, up to replacing $k$ with a finite extension, we have $Y' \cong S\Talg{Q}_\rho$, for some sharp fs monoid $Q$ such that $P \subset Q \subset P'$.  Consequently, $Y_\rho \to X' = X_\rho$ is \emph{finite Kummer \'etale}.  Moreover, if $m Q \subset P$ for some $m \in \bZ_{\geq 1}$, and if $X^{\frac{1}{m}} := S\Talg{\frac{1}{m} P}$, then the finite \Pth{a priori ramified} cover $Y \times_X X^{\frac{1}{m}}_\rho \to X^{\frac{1}{m}}_\rho$ \emph{splits completely} \Pth{\ie, the source is a disjoint union of sections} and is therefore \emph{strictly \'etale}.
\end{lem}
\begin{proof}
    Let $V'$ \Pth{\resp $W$} be the preimage of $U' := U_\rho \cong S\Talg{P}_\rho^\times$ in $Y'$ \Pth{\resp $Z$}.  Up to replacing $k$ with a finite extension containing all $d_j$-th roots of unity for all $j$, by Proposition \ref{prop-ket-std}, the finite \'etale cover $W \to U'$ is Galois with Galois group $G' := \Hom\bigl((P')^\gp/P^\gp, k^\times\bigr)$, and $V'$ is \Pth{by the usual arguments, as in \cite[V]{SGA:1}} the quotient of $W$ by some subgroup $G$ of $G'$ \Pth{as in Lemma \ref{lem-quot}}, which is isomorphic to $S\Talg{Q}_\rho^\times$ for some monoid $Q$ such that $P \subset Q \subset P'$ and $Q = Q^\gp \cap P'$.  These conditions imply that $Q$ is toric, and hence $S\Talg{Q}$ is normal because $\Spa(k\Talg{Q}, k^+\Talg{Q})$ is \Pth{see Example \ref{ex-log-adic-sp-toric} and the references given there}.  Since $Y'$ and $S\Talg{Q}_\rho$ are both normal and are both finite ramified covers of $X'$ extending the same finite \'etale cover $V'$ of $U'$, they are canonically isomorphic by \cite[\aThm 1.6]{Hansen:2020-vcrag}, as desired.  Finally, for the last assertion of the lemma, it suffices to note that, for any $Q$ as above satisfying $m Q \subset P$, up to replacing $k$ with a finite extension containing all $m$-th roots of unity, the connected components of $S\Talg{Q} \times_{S\Talg{P}} S\Talg{\frac{1}{m} P}$ are all of the form $S\Talg{\frac{1}{m} P}$, because $\bigl(Q \oplus_P (\frac{1}{m} P)\bigr)^\Sat$ is the product of $\frac{1}{m} P$ with a finite group annihilated by $m$.
\end{proof}

\begin{lem}\label{lem-Abhyankar-ess}
    The \Pth{cover-refinement} assumption in Lemma \ref{lem-Abhyankar-ref} holds up to replacing $k$ with a finite extension and $S$ with a strictly finite \'etale cover; and we may assume that the positive integers $d_1, \ldots, d_r$ there \Pth{for various $Y'$} are no greater than the degree $d$ of $f: Y \to X$.  Moreover, we can take $\rho = p^{-b(d, p)}$, where $b(d, p)$ is defined as in \cite[\aThm 2.2]{Lutkebohmert:1993-repaf}, which depends on $d$ and $p$ but not on $r$; and we can take $m = d!$ in the last assertion of Lemma \ref{lem-Abhyankar-ref}.
\end{lem}
\begin{proof}
    We shall proceed by induction on $r$.  When $r = 0$, the assumption in Lemma \ref{lem-Abhyankar-ref} means, for each connected component $Y'$ of $Y$, the strictly \'etale cover $Y' \to X = S$ splits completely.  This can always be achieved up to replacing $S$ with a Galois strictly finite \'etale cover refining $Y' \to S$ for all $Y'$.

    In the remainder of this proof, suppose that $r \geq 1$, and that the lemma has been proved for all strictly smaller $r$.  Let $\rho = p^{-b(d, p)}$ be as above.  Fix some $a \in k$ such that $|a| = \rho$.  We shall denote normalizations of fiber products by $\breve{\times}$ \Pth{instead of $\times$}.

    Let $P_1$ be the submonoid $\bZ_{\geq 0}^{r - 1} \oplus \{ 0 \}$ of $P = \bZ_{\geq 0}^r$.  Let $X_1 := S \times \bD^{r - 1} \cong S\Talg{P_1}$, which we identify with the subspace $S \times \bD^{r - 1} \times \{ a \}$ of $X = S \times \bD^r \cong S\Talg{P}$.  Let $Y_1 := Y \breve{\times}_X X_1$.  Note that the degree of $Y_1 \to X_1$ is also $d$.    By induction, up to replacing $k$ with a finite extension and $S$ with a strictly finite \'etale cover, for each connected component $Y_1'$ of $(Y_1)_\rho$, there exist $1 \leq d_1, \ldots, d_{r - 1} \leq d$ such that the induced finite ramified cover $Y_1' \to (X_1)_\rho$ is refined by $S\Talg{P_1'}_\rho \to (X_1)_\rho$, where $P_1' := \oplus_{1 \leq i \leq r - 1} (\frac{1}{d_i} \bZ_{\geq 0})$.  Let $X_1^{\frac{1}{d!}} := S\Talg{\frac{1}{d!} P_1}$, with $\frac{1}{d!} P_1 = \frac{1}{d!} \bZ_{\geq 0}^{r - 1}$.

    Let $\widetilde{X} := X_1^{\frac{1}{d!}} \times_{X_1} X \cong X_1^{\frac{1}{d!}} \times \bD$ and $\widetilde{Y} := X_1^{\frac{1}{d!}} \breve{\times}_{X_1} Y$.  Let $\widetilde{f}: \widetilde{Y} \to \widetilde{X}$ denote the induced finite ramified cover, which is also of degree $d$.  Then the \Pth{strictly finite \'etale} pullback of $\widetilde{f}$ to $(X_1^{\frac{1}{d!}})_\rho^\times \times \{ a \}$ can be identified with the pullback of $Y_1 \to X_1$ to $(X_1^{\frac{1}{d!}})_\rho^\times$, which \emph{splits completely}, by the induction hypothesis and the last assertion in Lemma \ref{lem-Abhyankar-ref}.  Hence, since $\rho = p^{-b(d, p)}$, for each connected component $\widetilde{Y}'$ of $\widetilde{Y}_\rho$, by applying \cite[\aLem 3.2]{Lutkebohmert:1993-repaf} to the morphism $(\widetilde{Y}')^\times \to \widetilde{X}_\rho^\times$ induced by $\widetilde{f}$, we obtain a rigid analytic function $\widetilde{T}$ on $(\widetilde{Y}')^\times$ such that $\widetilde{T}^{d_r} = T_r$, where $T_r$ is the coordinate on the $r$-th factor of $\bD^r$.  By \cite[\aSec 3]{Bartenwerfer:1976-erhnf} \Pth{see also \cite[\aThm 2.6]{Hansen:2020-vcrag}}, $\widetilde{T}$ extends to a rigid analytic function on the normal $\widetilde{Y}'$, which still satisfies $\widetilde{T}^{d_r} = T_r$.  Hence, we can view $\widetilde{T}$ as $T_r^{\frac{1}{d_r}}$, and $\widetilde{Y}' \to \widetilde{X}_\rho \cong (X_1^{\frac{1}{d!}} \times \bD)_\rho$ factors through $S\Talg{\widetilde{P}}_\rho \to (X_1^{\frac{1}{d!}} \times \bD)_\rho$, where $\widetilde{P} := (\frac{1}{d!} P_1) \oplus (\frac{1}{d_r} \bZ_{\geq 0})$.  Since these are finite ramified covers of the \emph{same} degree $d_r$ from connected and normal rigid analytic varieties, we obtain an induced isomorphism $\widetilde{Y}' \cong S\Talg{\widetilde{P}}_\rho$.

    Since each connected component $Y'$ of $Y_\rho$ is covered by some connected component $\widetilde{Y}'$ of $\widetilde{Y}_\rho$, by Lemma \ref{lem-Abhyankar-ref}, up to replacing $k$ with a finite extension, $Y' \cong S\Talg{Q}_\rho$ for some monoid $Q$ satisfying $P \subset Q \subset \widetilde{P} = (\frac{1}{d!} P_1) \oplus (\frac{1}{d_r} \bZ_{\geq 0})$, for some $1 \leq d_r \leq d$ and $\widetilde{P}$ determined by $\widetilde{Y}'$ as above.  By the construction of $\widetilde{Y}'$, the monoid $\bigl((\frac{1}{d!} P_1) \oplus_{P_1} Q\bigr)^\Sat$ is the product of $\widetilde{P}$ with a finite group, which forces $\{ 0 \}^{r - 1} \oplus (\frac{1}{d_r} \bZ_{\geq 0}) \subset Q$.  By the construction of $Y_1'$ and the induction hypothesis, the projection $\widetilde{P} \to \frac{1}{d!} P_1$ maps $Q$ into $\oplus_{1 \leq i \leq r - 1} (\frac{1}{d_i} \bZ_{\geq 0})$ for some $1\leq d_1, \ldots, d_{r - 1} \leq d$.  Thus, $P \subset Q \subset P' := \oplus_{1 \leq i \leq r} (\frac{1}{d_i} \bZ_{\geq 0})$, as desired.
\end{proof}

\begin{rk}\label{rem-prop-Abhyankar}
    Proposition \ref{prop-Abhyankar} can be regarded as the \emph{Abhyankar's lemma} \Pth{\Refcf{} \cite[XIII, 5.2]{SGA:1}} in the rigid analytic setting, because of the last assertion in Lemma \ref{lem-Abhyankar-ref}.
\end{rk}

More generally, we have the following basic but useful facts:
\begin{lem}\label{lem-Abhyankar-basic}
    Let $X$ be a noetherian fs log adic space modeled on a sharp fs monoid $P$, and let $f: Y \to X$ be a Kummer \'etale \Pth{\resp finite Kummer \'etale} morphism.  Then $Y \times_X X^{\frac{1}{n}} \to X^{\frac{1}{n}}$ is \'etale \Pth{\resp finite \'etale} for some positive integer $n$.  If $X$ has at most one positive residue characteristic, then we can take $n$ to be invertible on all of $X$.
\end{lem}
\begin{proof}
    Since $X$ is noetherian, by taking the least common multiple of the positive integers obtained on finitely many members in an \'etale covering, it suffices to work \'etale locally on $X$.  By Lemma \ref{lem-ket-chart}, up to \'etale localization on $X$, there exists an \'etale covering $\{ Y_i \to Y \}_{i \in I}$ indexed by a finite set $I$ such that each induced Kummer \'etale morphism $Y_i \to X$ admits a Kummer chart $P \to Q_i$ with a sharp $Q_i$.  Then there exists some positive integer $n$, which we may assume to be invertible on all of $X$ when $X$ has at most one positive residue characteristic, such that $P \to \frac{1}{n} P$ factors as $P \to Q_i \Mapn{u_i} \frac{1}{n} P$ for some injective $u_i$, for all $i \in I$.  The induced morphism $Y_i \times_X X^{\frac{1}{n}} \to X^{\frac{1}{n}}$ is \'etale, for each $i \in I$, because it admits a Kummer chart $\frac{1}{n} P \to \bigl(Q_i \oplus_P (\frac{1}{n} P)\bigr)^\Sat \cong \bigl((Q_i \oplus_P Q_i) \oplus_{Q_i} (\frac{1}{n} P)\bigr)^\Sat \cong G_i \oplus (\frac{1}{n} P)$, where $G_i := (Q_i)^\gp / u_i^\gp(P^\gp)$ has order invertible in $\cO_{Y_i}$ by assumption.  \Pth{Since $G_i$ is a group, $\frac{1}{n} P \to G_i \oplus (\frac{1}{n} P)$ is strict, by definition.  See also Remark \ref{rem-stalk-strict}.}  By Proposition \ref{prop-lem-four-pt}, the \'etale map $\coprod_{i \in I} Y_i \times_X X^{\frac{1}{n}} \to Y \times_X X^{\frac{1}{n}}$ is surjective as $\coprod_{i \in I} Y_i \to Y$ is.  Hence, by \'etale descent, $Y \times_X X^{\frac{1}{n}} \to X^{\frac{1}{n}}$ is also \'etale.  Finally, by \cite[\aLem 1.4.5 i)]{Huber:1996-ERA}, $Y \times_X X^{\frac{1}{n}} \to X^{\frac{1}{n}}$ is finite when $Y \to X$ is.
\end{proof}

\begin{lem}\label{lem-Abhyankar-basic-cov}
    Let $X$ be a noetherian fs log adic space modeled on a sharp fs monoid $P$.  Let $\{ U_i \to X \}_{i \in I}$ be a Kummer \'etale covering indexed by a finite set $I$.  Then there exists a Kummer \'etale covering $\{ V_j \to X \}_{j \in J}$ indexed by a finite set $J$ refining $\{ U_i \to X \}_{i \in I}$ such that each $V_j \to X$ admits a chart $P \to \frac{1}{n_j} P$ for some integer $n_j$ invertible in $\cO_{V_j}$, and such that $\{ V_j \times_X X^{\frac{1}{n}} \to X^{\frac{1}{n}} \}_{j \in J}$ is an \'etale covering of $X^{\frac{1}{n}}$ for some integer $n$ divided by all $n_j$.  If $X$ has at most one positive residue characteristic, then we may take $n$ to be invertible in $\cO_X$.
\end{lem}
\begin{proof}
    Since $X$ is noetherian, by Lemma \ref{lem-ket-chart}, we may replace $\{ U_i \to X \}_{i \in I}$ with a finite refinement $\{ U_j \to X \}_{j \in J}$ such that each $U_j \to X$ admits a Kummer chart $P \to Q_j$ with a sharp $Q_j$.  Then there exists some positive integer $n_j$ invertible in $\cO_{U_j}$ such that $P \to \frac{1}{n_j} P$ factors as $P \to Q_j \Mapn{u_j} \frac{1}{n_j} P$ for some injective $u_j$, for each $j \in J$.  Therefore, each $V_j := U_j \times_{U_j\Talg{Q_j}} U_j\Talg{\frac{1}{n_j} P} \to X$ is Kummer \'etale with a Kummer chart $P \to \frac{1}{n_j} P$, and the induced morphism $V_j \to X^{\frac{1}{n_j}} = X \times_{X\Talg{P}} X\Talg{\frac{1}{n_j} P}$ is \'etale \Pth{as in Definition \ref{def-ket-mor}}.  In this case, if $n$ is divisible by all $n_j$, then $V_j \times_X X^{\frac{1}{n}} \to X^{\frac{1}{n}} = X \times_{X\Talg{P}} X\Talg{\frac{1}{n} P}$ is also \'etale, and we can take $n$ to be invertible on $X$ when $X$ has at most one positive residue characteristic \Pth{\Refcf{} the proof of Lemma \ref{lem-Abhyankar-basic}}.  Since $\coprod_{j \in J} \, U_j \to X$ is surjective by assumption, by Proposition \ref{prop-lem-four-pt}, $\{ V_j \times_X X^{\frac{1}{n}} \}_{j \in J} \to X^{\frac{1}{n}}$ is an \'etale covering of $X^{\frac{1}{n}}$, as desired.
\end{proof}

The following two propositions show that the properties of morphisms being Kummer \'etale, log smooth, and log \'etale can be verified up to Kummer \'etale localization on either the source or the target:
\begin{prop}\label{prop-log-sm-et-ket-cov}
    Let $Y \Mapn{f} X \Mapn{g} S$ be lft morphisms of locally noetherian fs log adic spaces such that $f$ is Kummer \'etale and surjective.  Then $g$ is log smooth \Pth{\resp log \'etale, \resp Kummer \'etale} if and only if $g \circ f$ is.
\end{prop}
\begin{proof}
    Since $f$ is Kummer \'etale \Pth{and hence log \'etale}, by Propositions \ref{prop-log-sm-compos} and \ref{prop-ket-compos-bc}, if $g$ is log smooth \Pth{\resp log \'etale, \resp Kummer \'etale}, then so is $g \circ f$.  \Pth{The surjectivity of $f$ is not needed in this direction of implication.}

    Conversely, suppose that $g \circ f$ is log smooth \Pth{\resp log \'etale, \resp Kummer \'etale}.  It suffices to show that $X \to S$ is log smooth \Pth{\resp log \'etale, \resp Kummer \'etale}, \'etale locally at geometric points $\AC{x}$ of $X$ and $\AC{s} = g(\AC{x})$ of $S$.

    Up to \'etale localization at $\AC{x}$, we may assume that $X$ has at most one positive residue characteristic.  By Proposition \ref{prop-chart-mor-exist-fine}, up to \'etale localization at $\AC{x}$ and $\AC{s}$, we may assume that $X \to S$ admits an fs chart $u: L := \overline{\cM}_{S, \AC{s}} \to P'$, inducing a strict morphism $g': X \to X' := S \times_{S\Talg{L}} S\Talg{P'}$.  Let $\overline{u} := g^\sharp_{\AC{x}}: L \to P := \overline{\cM}_{X, \AC{x}}$.  By Remark \ref{rem-chart-stalk}, $P' \to P$ is surjective with kernel given by the preimage of $\cO_{X_\et, \AC{x}}^\times$.  Since $L$ and $P$ are sharp fs, we may assume that the order of the torsion part of $\ker\bigl((P')^\gp \to P^\gp\bigr)$ is invertible in $\cO_{X_\et, \AC{x}}$.  Since $f$ is Kummer \'etale, by Definition \ref{def-ket-mor} and Lemma \ref{lem-Abhyankar-basic-cov}, and by the first paragraph above, we are reduced to the case where $f: Y \to X$ is of the form $X^{\frac{1}{m}} \to X$, for some integer $m$ invertible in $\cO_X$, which admits a global fs chart $v: P \to \frac{1}{m}P$.

    Let $\AC{y}$ be any geometric point of $Y$ such that $f(\AC{y}) = \AC{x}$, which exists because $f$ is surjective.  By Lemma \ref{lem-ket-stalk}, $\overline{v} := f^\sharp_{\AC{y}}: \overline{\cM}_{X, \AC{x}} \to \overline{\cM}_{Y, \AC{y}}$ is also given by $v: P \to \frac{1}{m}P$.  By the same argument as above, up to \'etale localization at $\AC{y}$, we may assume that $Y \to S$ admits an fs global chart $w: L \to Q$, and we have a surjection $Q \to \overline{\cM}_{Y, \AC{y}} \cong \frac{1}{m}P$ such that the order of the torsion part of $\ker\bigl(Q^\gp \to \frac{1}{m}P^\gp\bigr)$ is invertible in $\cO_{X_\et, \AC{x}}$.  By definition, $\overline{w} := \overline{(g \circ f)}^\sharp_{\AC{y}} = \overline{v} \circ \overline{u}$ as homomorphisms from $L$ to $\frac{1}{m}P$.  If $g \circ f$ is log smooth \Pth{\resp log \'etale}, then the kernel and the torsion part of the cokernel \Pth{\resp the kernel and the cokernel} of $\overline{w}^\gp$ have orders invertible in $\cO_{X_\et, \AC{x}}$, and so are $\overline{u}^\gp$ and $u^\gp$ \Pth{\Refcf{} Definition \ref{def-log-sm}}.  If $g \circ f$ is Kummer \'etale, then $\overline{w}$ is Kummer, and so are $\overline{u}$ and $u$ \Pth{see Definitions \ref{def-Kummer} and \ref{def-ket-mor}, and Lemmas \ref{lem-ket-stalk} and \ref{lem-ket-log-et-ex}}.  Therefore, by the first paragraph above, $X' \to S$ is log smooth \Pth{\resp log \'etale, \resp Kummer \'etale} when $g \circ f$ is.  Thus, up to replacing $g$ with $g': X \to X'$, we are reduced to the case where $g$ is strict, and it remains to show that $g: X \to S$ is \Pth{strictly} \'etale when $g \circ f: Y = X^{\frac{1}{m}} \to S$ is log \'etale.  Up to \'etale localization on $S$, we may assume that $\cO_S(S)^\times$ and hence $\cO_X(X)^\times$ contain all $m$-th roots of unity.  Then $Y = X^{\frac{1}{m}} \to X$ is a Galois finite Kummer \'etale cover with Galois group $\Gamma := \Hom\bigl((\frac{1}{m}P^\gp) / P^\gp, \cO_S(S)^\times\bigr)$, by Proposition \ref{prop-ket-std}.

    Let us write $\bigl(\cO_S(S), \cM_S(S)\bigr) = (A, M)$, $\bigl(\cO_X(X), \cM_X(X)\bigr) = (B, N)$, and $\bigl(\cO_Y(Y), \cM_Y(Y)\bigr) = (C, O)$, for simplicity.  Since $X \to S$ is strict, up to further \'etale localization on $X$ and $S$, we may assume that $(A, M) \to (B, N)$ is also strict.  By Lemma \ref{lem-def-formal-log-sm-unram-et-strict}, Remark \ref{rem-def-formal-log-sm-unram-et-loc}, and Proposition \ref{prop-formal-log-sm-imply-log-sm}, it suffices to show that $(A, M) \to (B, N)$ is formally log \'etale \Pth{as in Definition \ref{def-log-Huber-formal-log-sm-unram-et}} when $(A, M) \to (C, O)$ is.  Since $Y \cong X \times_{X\Talg{P}} X\Talg{\frac{1}{m}P}$, we have $O \cong \bigl(N \oplus_P (\frac{1}{m}P)\bigr)^\Sat$ and $C \cong (B \ho_{\bZ[P]} \bZ[\frac{1}{m}P]) \ho_{\bZ[N \oplus_P (\frac{1}{m}P)]} \bZ[\bigl(N \oplus_P (\frac{1}{m}P)\bigr)^\Sat] \cong B \ho_{\bZ[N]} \bZ[O]$.  Consider any commutative diagram as in \Refeq{\ref{eq-log-Huber-thick}} for $(A, M) \to (B, N)$ \Pth{with $\alpha$, $\beta$, \etc omitted here}.  For $? = \emptyset$ and $\prime$, consider $\widetilde{T}^? := \bigl(T^? \oplus_P (\frac{1}{m}P)\bigr)^\Sat$ and $\widetilde{D}^? := D^? \ho_{\bZ[T^?]} \bZ[\widetilde{T}^?]$, so that $(D', T') \to (\widetilde{D}', \widetilde{T}')$ is the completion of the common pullback of $(B, N) \to (C, O)$ and $(D, T) \to (\widetilde{D}, \widetilde{T})$ in the category of log Huber rings with fs log structures.  Moreover, we have $(D^?, T^?) \Mi \bigl((\widetilde{D}^?)^\Gamma, (\widetilde{T}^?)^\Gamma\bigr)$, for $? = \emptyset$ and $\prime$, because the formation of $\Gamma$-invariants is compatible with arbitrary base changes and completions when $|\Gamma|$ is invertible \Pth{as $m$ is}, and because of Remark \ref{rem-int-sat}.  Thus, we have obtained an extended commutative diagram for $(A, M) \to (C, O)$ and the base change $(\widetilde{D}, \widetilde{T}) \to (\widetilde{D}', \widetilde{T}')$ of $(D, T) \to (D', T')$.  Since $(A, M) \to (C, O)$ is formally log \'etale, $(C, O) \to (\widetilde{D}', \widetilde{T}')$ uniquely lifts to $(C, O) \to (\widetilde{D}, \widetilde{T})$, whose pre-composition with $(B, N) \to (C, O)$ is $\Gamma$-invariant and hence factors through $(D, T)$.  This shows that $(A, M) \to (B, N)$ is also formally log \'etale, as desired.
\end{proof}

\begin{prop}\label{prop-log-sm-et-ket-descent}
    Let $f: Y \to X$ and $g: X' \to X$ be lft morphisms of locally noetherian fs log adic spaces such that $g$ is Kummer \'etale and surjective.  Then $f$ is log smooth \Pth{\resp log \'etale, \resp Kummer \'etale} if and only if its pullback $f': Y' := Y \times_X X' \to X'$ under $g$ is.
\end{prop}
\begin{proof}
    By definition, we have the following commutative diagram
    \[
        \xymatrix{ {Y'} \ar[r]^-{f'} \ar[d]_-{g'} & {X'} \ar[d]^-g \\
        {Y} \ar[r]^-f & {X} }
    \]
    in which $g'$ is the pullback of $g$ under $f$.  If $f$ is log smooth \Pth{\resp log \'etale, \resp Kummer \'etale}, then so is $f'$, by Proposition \ref{prop-log-sm-bc} and \ref{prop-ket-compos-bc}.  Conversely, suppose $f'$ is log smooth \Pth{\resp log \'etale, \resp Kummer \'etale}.  Since $g$ is Kummer \'etale \Pth{and hence log \'etale}, by Propositions \ref{prop-log-sm-compos} and \ref{prop-ket-compos-bc}, $g \circ f' = f \circ g'$ is also log smooth \Pth{\resp log \'etale, \resp Kummer \'etale}.  By Propositions \ref{prop-ket-compos-bc} and \ref{prop-lem-four-pt}, $g'$ is Kummer \'etale and surjective as $g$ is.  Thus, by Propositions \ref{prop-log-sm-et-ket-cov}, $f$ is log smooth \Pth{\resp log \'etale, \resp Kummer \'etale}, as desired.
\end{proof}

\subsection{Coherent sheaves}\label{sec-ket-coh-descent}

In this subsection, we show that, when $X$ is a locally noetherian fs log adic space, the presheaf $\cO_{X_\ket}$ \Pth{\resp $\cO_{X_\ket}^+$} on $X_\ket$ defined by $U \mapsto \cO_U(U)$ \Pth{\resp $U \mapsto \cO_U^+(U)$} is indeed a sheaf, generalizing a well-known result of Kato's \cite{Kato:2021-lsfi-2} for log schemes.  We also study some problems related to the Kummer \'etale descent of coherent sheaves.

\begin{thm}\label{thm-O-ket-sheaf}
    Let $X$ be a locally noetherian fs log adic space.
    \begin{enumerate}
        \item The presheaves $\cO_{X_\ket}$ and $\cO_{X_\ket}^+$ are sheaves.

        \item If $X$ is affinoid, then $H^i(X_\ket, \cO_{X_\ket}) = 0$, for all $i > 0$.
    \end{enumerate}
\end{thm}

A key input is the following:
\begin{lem}\label{lem-O-ket-Cech}
    Let $X$ be an affinoid noetherian fs log adic spaces, endowed with a chart modeled on a sharp fs monoid $P$.  Let $Y \to X$ be a standard Kummer cover \Pth{see Definition \ref{def-ket-std}}.  Then the \v{C}ech complex
    \[
        C^\bullet(Y / X): 0 \to \cO(X) \to \cO(Y) \to \cO(Y \times_X Y) \to \cO(Y \times_X Y \times_X Y) \to \cdots
    \]
    \Pth{where we omit the subscripts of the structure sheaf $\cO$ for simplicity} is exact.
\end{lem}
\begin{proof}
    This is essentially \cite[\aLem 3.28]{Niziol:2008-ktls-1}, based on the idea in \cite[\aLem 3.4.1]{Kato:2021-lsfi-2}.  Suppose that $Y \to X = \Spa(R, R^+)$ is associated with a Kummer homomorphism $u: P \to Q$ as in Proposition \ref{prop-ket-std}.  Then $C^\bullet(Y / X)$ is already known to be exact at the first three terms; and
    \[
        \cO(Y \times_X Y \times_X \cdots \times_X Y) \cong \cO(Y) \otimes_R R[G] \otimes_R R[G] \cdots \otimes_R R[G],
    \]
    where $G = Q^\gp / u^\gp(P^\gp)$, in which case we can write the differentials of $C^\bullet(Y / X)$ explicitly and construct a contracting homotopy for $C^\bullet(Y / X)$, by the same argument as in the proof of \cite[\aLem 3.28]{Niziol:2008-ktls-1}.
\end{proof}
We emphasize that Lemma \ref{lem-O-ket-Cech} also works for standard Kummer covers that are not necessarily Kummer \'etale.

\begin{proof}[Proof of Theorem \ref{thm-O-ket-sheaf}]
    \begin{enumerate}
        \item It suffices to prove that $\cO_{X_\ket}$ is a sheaf, in which case $\cO_{X_\ket}^+$ also is, because
            \[
                \cO_{X_\ket}^+(U) = \cO_U^+(U) = \{ f \in \cO_{X_\ket}(U) = \cO_U(U) : |f(x)| \leq 1, \Utext{for all $x \in U$} \},
            \]
            exactly as in \cite[\aProp 3.1.7]{Scholze/Weinstein:2020-BLG}.  Since the sheafiness for the \'etale topology is known for all locally noetherian adic spaces, by Lemma \ref{lem-Abhyankar-basic-cov}, the statement is reduced to Lemma \ref{lem-O-ket-Cech}.

        \item By Propositions \ref{prop-chart-stalk-chart-fs} and \ref{prop-et-coh-descent}, we may reduce to the case where $X$ is affinoid with a global sharp fs chart $P$.  By Lemma \ref{lem-Abhyankar-basic-cov}, any Kummer \'etale covering $\{ U_i \to X \}_{i \in I}$ of $X$ admits some refinement $\{ V_j \to X \}_{j \in J}$ as finite Kummer \'etale covering such that $\{ V_j \times_X X^{\frac{1}{m}} \to X^{\frac{1}{m}} \}_{j \in J}$ is an \'etale covering, for some $m$, and such that each $V_j \times_X X^{\frac{1}{m}} \to V_j$ is a composition of \'etale morphisms and standard Kummer \'etale covers.  Thus, by Lemma \ref{lem-O-ket-Cech}, the \v{C}ech complex
            \[
                \cO(X) \to \cO(X^{\frac{1}{m}}) \to \cO(X^{\frac{1}{m}} \times_X X^{\frac{1}{m}}) \to \cdots
            \]
            is exact.  As a result, by Proposition \ref{prop-et-coh-descent}, the \v{C}ech complex
            \[
                \cO(X) \to \oplus_j \, \cO(V_j) \to \oplus_{j, j'} \, \cO(V_j \times_X V_{j'}) \to \cdots
            \]
            is also exact, as desired.  \qedhere
    \end{enumerate}
\end{proof}

\begin{cor}\label{cor-O-ket-et-an}
    Let $X$ be a locally noetherian fs log adic space.  Consider the natural projections of sites $\varepsilon_\an: X_\ket \to X_\an$ and $\varepsilon_\et: X_\ket \to X_\et$.  Then we have canonical isomorphisms $\cO_{X_\an} \Mi R\varepsilon_{\an, *}(\cO_{X_\ket})$ and $\cO_{X_\et} \Mi R\varepsilon_{\et, *}(\cO_{X_\ket})$.  As a result, the pullback functor from the category of vector bundles on $X_\an$ \Pth{\resp $X_\et$} to the category of $\cO_{X_\ket}$-modules is fully faithful \Pth{\Refcf{} Proposition \ref{prop-et-coh-descent}}.
\end{cor}

\begin{prop}\label{prop-M-ket-sheaf}
    Let $X$ be a locally noetherian fs log adic space.  Then the presheaf $\cM_{X_\ket}$ assigning $U \mapsto \cM_U(U)$ is a sheaf on $X_\ket$.  In particular, we also have a canonical isomorphism $\varepsilon_{\et, *}(\cM_{X_\ket}) \Mi \cM_X$.
\end{prop}
\begin{proof}
    The proof is similar to \cite[\aLem 3.5.1]{Kato:2021-lsfi-2}.  Since $\cM_X$ is already a sheaf on the \'etale topology, by replacing $X$ with its strict localization at a geometric point $\AC{x}$, it suffices to show the exactness of
    \[
        0 \to \cM_X(X) \to \cM_Y(Y) \rightrightarrows \cM_{Y \times_X Y}(Y \times_X Y),
    \]
    where $X = \Spa(R, R^+)$ admits a chart modeled on a sharp fs monoid $P \cong \overline{\cM}_{X, \AC{x}}$, for some strictly local ring $R$ \Pth{see Proposition \ref{prop-chart-stalk-chart-fs}}; and where $Y \to X$ is a standard Kummer \'etale cover with a Kummer chart $u: P \to Q$ with a sharp $Q$ such that the order of $G := \coker(u^\gp)$ is invertible in $R$ \Pth{see Lemma \ref{lem-ket-chart}}.  Note that $P$ is also sharp, because $u$ is injective \Pth{see Definition \ref{def-Kummer}}.

    Let $R' := \cO_Y(Y)$ and $R'' := \cO_{Y \times_X Y}(Y \times_X Y)$.  By Definition \ref{def-ket-std} and Proposition \ref{prop-ket-std}, we have $R' \cong R \ho_{f_1, R\Talg{P}, f_2} R\Talg{Q}$, where $f_1: R\Talg{P} \to R$ and $f_2: R\Talg{P} \to R\Talg{Q}$ are induced by the charts.  By Proposition \ref{prop-ket-std}, we have
    \[
        R'' \cong R \ho_{R\Talg{P}} R\Talg{(Q \oplus_P Q)^\Sat} \cong R \ho_{R\Talg{P}} R\Talg{Q \oplus G} \cong R'[G].
    \]
    Let $I$ \Pth{\resp $I'$, \resp $I''$} be the ideal of $R$ \Pth{\resp $R'$, \resp $R''$} generated by the image of $P \setminus \{ 0 \}$ \Pth{\resp $Q \setminus \{ 0 \}$, \resp $Q \setminus \{ 0 \}$}, which is a proper ideal because $P$ and $Q$ are sharp.  Since $I$ is contained in the maximal ideal of the strictly local ring $R$, and since $u: P \to Q$ is Kummer, $I'$ must be contained in all maximal ideals of $R'$.  The canonical morphism $R / I \to R' / I'$ is an isomorphism, because it is induced by compatibly completing both sides of the canonical isomorphism $R \otimes_{f_1, R\Talg{P}, f_3} R \Mi (R \otimes_{f_1, R\Talg{P}, f_2} R\Talg{Q}) \otimes_{f_5, R\Talg{Q}, f_4} R$, where:
    \begin{itemize}
        \item $f_1: R\Talg{P} \to R$ and $f_2: R\Talg{P} \to R\Talg{Q}$ are given by the charts, as above;

        \item $f_3: R\Talg{P} \to R$ and $f_4: R\Talg{Q} \to R$ are $R$-algebra homomorphisms defined by sending nonzero elements of $P$ and $Q$ to $0$, respectively; and

        \item $f_5: R\Talg{Q} \to R \otimes_{f_1, R\Talg{P}, f_2} R\Talg{Q}$ is the pullback of $f_1$ under $f_2$.
    \end{itemize}
    Since $I'$ is contained in all maximal ideals of $R'$, this forces $R'$ to be local.  Since $R \to R'$ is finite \Pth{by Proposition \ref{prop-ket-std}\Refenum{\ref{prop-ket-std-1}}}, $R'$ is strictly local as $R$ is.

    Let $V$ \Pth{\resp $V'$, \resp $V''$} be the subgroup of elements in $R^\times$ \Pth{\resp $(R')^\times$, \resp $(R'')^\times$} congruent to 1 modulo $I$ \Pth{\resp $I'$, \resp $I''$}.  Since $R$ and $R'$ are strictly local, and since $R'' \cong R'[G]$, where the order of $G$ is invertible in $R$ and hence in $R'$, we have compatible canonical isomorphisms
    \[
        R^\times / V \Mi (R / I)^\times,
    \]
    \[
        (R')^\times / V' \Mi (R' / I')^\times \cong (R / I)^\times,
    \]
    and
    \[
        (R'')^\times / V'' \Mi (R'' / I'')^\times \cong ((R' / I')[G])^\times \cong ((R / I)[G])^\times.
    \]

    By Lemma \ref{lem-O-ket-Cech}, we know that
    \[
        0 \to R \to R' \rightrightarrows R''
    \]
    is exact.  Since the injection $R \to R'$ is finite, we can identify $R$ as a subring of $R'$ over which $R'$ is integral.  Hence, it is elementary that $R^\times = (R')^\times \cap R$, and
    \[
        0 \to R^\times \to (R')^\times \rightrightarrows (R'')^\times
    \]
    is exact.  Moreover, we have $I = I' \cap R$ and $V = V' \cap R^\times$, and
    \[
        0 \to V \to V' \rightrightarrows V''
    \]
    is also exact.  By some diagram chasing, it suffices to show the exactness of
    \[
        0 \to \cM_X(X) / V \to \cM_Y(Y) / V' \rightrightarrows \cM_{Y \times_X Y}(Y \times_X Y) / V''.
    \]

    Since $\cM_X(X) \to R$ and $\cM_Y(Y) \to R'$ are associated with the pre-log structures $P \to R$ and $Q \to R'$, since $u: P \to Q$ is a Kummer chart, and since $P$ and $Q$ are sharp, we have compatible isomorphisms
    \[
        \cM_X(X) / V \cong (R^\times / V) \oplus P \cong (R / I)^\times \oplus P
    \]
    and
    \[
        \cM_Y(Y) / V' \cong ((R')^\times / V') \oplus Q \cong (R / I)^\times \oplus Q.
    \]
    Since the log structure $\cM_{Y \times_X Y}(Y \times_X Y) \to R''$ is associated with the pre-log structure $Q \oplus G \to R'' \cong R'[G]$ induced by the same $Q \to R'$ as above and by the identity map $G \to G$, we have
    \[
        \cM_{Y \times_X Y}(Y \times_X Y) / V'' \cong ((R / I)[G])^\times \oplus Q.
    \]
    Accordingly, the above sequence can be identified with
    \[
        0 \to (R / I)^\times \oplus P \to (R / I)^\times \oplus Q \rightrightarrows ((R / I)[G])^\times \oplus Q,
    \]
    where the double arrows are $(x, q) \mapsto (x, q)$ and $(x, q) \mapsto (x \mono{\overline{q}}, q)$, with $\overline{q}$ denoting the image of $q$ in $G = Q^\gp / u^\gp(P^\gp)$.  Thus, it suffices to note that the sequence
    \[
        0 \to P \to Q \rightrightarrows Q \oplus G,
    \]
    where the double arrows are $q \mapsto (q, 0)$ and $q \mapsto (q, \overline{q})$, is exact.
\end{proof}

As a byproduct, let us show that representable presheaves are sheaves on $X_\ket$.  The log scheme version can be found in \cite[\aThm 2.6]{Illusie:2002-fknle}, which can be further traced back to \cite[\aThm 3.1]{Kato:2021-lsfi-2}.
\begin{prop}\label{prop-ket-repr}
    Let $Y \to X$ be a morphism of locally noetherian fs log adic spaces.  Then the presheaf $\Mor_X(\,\cdot\,, Y)$ on $X_\ket$ is a sheaf.
\end{prop}
\begin{proof}
    We follow the idea of \cite[\aThm 3.1]{Kato:2021-lsfi-2}.  It suffices to show that the presheaf $\Mor(\,\cdot\,, Y)$ on $X_\ket$ is a sheaf, because $\Mor_X(\,\cdot\,, Y)$ is just the sub-presheaf of sections of $\Mor(\,\cdot\,, Y)$ with compatible morphisms to $X$.  We may assume that $Y = \Spa(R, R^+)$ is affinoid with a chart modeled on a sharp fs monoid $P$.

    We claim that the presheaves $\cF: T \mapsto \Hom\big((R, R^+\big), (\cO_T(T), \cO_T^+(T))\big)$, $\cG: T \mapsto \Hom\big(P, \cM_T(T)\big)$, $\cH: T \mapsto \Hom\big(P, \cO_T(T)\big)$ on $X_\ket$, where the first $\Hom$ is in the category of Huber pairs, and where the latter two are in the category of monoids, are all sheaves.  As for the case of $\cF$, it suffices to show that $\cF': T \mapsto \Hom_\cont\bigl(R, \cO_T(T)\bigr)$ is a sheaf, where the homomorphisms are continuous ring homomorphisms; or that $\cF'': T \mapsto \Hom\bigl(R, \cO_T(T)\bigr)$ is a sheaf, where we consider all ring homomorphisms.  Consider any presentation $R \cong \bZ[T_i]_{i \in I} / (f_j)_{j \in J}$ of the ring $R$, so that $\cF''(T) \cong \ker(\cO_T(T)^I \to \cO_T(T)^J)$.  Then $\cF''$ is a sheaf on $X_\ket$ as $T \mapsto \cO_T(T)$ is \Pth{see Theorem \ref{thm-O-ket-sheaf}}.  As for the cases of $\cG$ and $\cH$, consider any presentation $\bZ_{\geq 0}^r \rightrightarrows \bZ_{\geq 0}^s \to P \to 0$ of the finitely generated monoid $P$, which exists by \cite[\aThm I.2.1.7]{Ogus:2018-LLG}.  Then $\cG(T)$ \Pth{\resp $\cH(T)$} is the equalizer of $\cM_T(T)^s \rightrightarrows \cM_T(T)^r$ \Pth{\resp $\cO_T(T)^s \rightrightarrows \cO_T(T)^r$}.  Hence, both presheaves are sheaves on $X_\ket$ as $T \mapsto \cM_T(T)$ and $T \mapsto \cO_T(T)$ are \Pth{see Theorem \ref{thm-O-ket-sheaf} and Proposition \ref{prop-M-ket-sheaf}}.

    By the claim just established, since $\Mor(\,\cdot\,, Y)$ \Pth{when $Y = \Spa(R, R^+)$ is modeled on $P$ as above} is the fiber product of the morphisms $\cF \to \cH$ and $\cG \to \cH$ induced by $P \to R$ and $\cM_T(T) \to \cO_T(T)$, respectively, it is also a sheaf, as desired.
\end{proof}

In the remainder of this subsection, we study coherent sheaves on the Kummer \'etale site.

\begin{defn}\label{def-ket-coh}
    Let $X$ be a locally noetherian fs log adic space.
    \begin{enumerate}
        \item An $\cO_{X_\ket}$-module $\cF$ is called an \emph{analytic coherent sheaf} if it is isomorphic to the inverse image of a coherent sheaf on the analytic site of $X$.

        \item An $\cO_{X_\ket}$-module $\cF$ is called a \emph{coherent sheaf} if there exists a Kummer \'etale covering $\{ U_i \to X \}_i$ such that each $\cF|_{U_i}$ is an analytic coherent sheaf.
    \end{enumerate}
\end{defn}

The following results are analogues of \cite[\aProp 6.5]{Kato:2021-lsfi-2}, the proof of which is completed in \cite[\aProp 3.27]{Niziol:2008-ktls-1}.

\begin{thm}\label{thm-ket-coh}
    Suppose that $X$ is an affinoid noetherian fs log adic space.  Then $H^i(X_\ket, \cF) = 0$, for all $i > 0$, in the following two situations:
    \begin{enumerate}
        \item\label{thm-ket-ana-coh} $\cF$ is an analytic coherent $\cO_{X_\ket}$-module.

        \item\label{thm-ket-gen-coh} $\cF$ is a coherent $\cO_{X_\ket}$-module, and $X$ is over an affinoid field $(k, k^+)$.
    \end{enumerate}

\end{thm}
\begin{proof}
    \begin{enumerate}
        \item As in the proof of Theorem \ref{thm-O-ket-sheaf}, by Lemma \ref{lem-Abhyankar-basic-cov} and Proposition \ref{prop-et-coh-descent}, it suffices to show the exactness of the \v{C}ech complex
            \[
                C^\bullet_\cF(Y / X): 0 \to \cF(X) \to \cF(Y) \to \cF(Y \times_X Y) \to \cdots,
            \]
            where $X$ is affinoid with a sharp fs chart $P$, and where $Y \to X$ is a standard Kummer cover.  By Proposition \ref{prop-ket-std}, the morphisms $Y \to X$, $Y \times_X Y \to X$, $Y \times_X Y \times_X Y \to X$, \etc are finite, and hence
            \[
                C^\bullet_\cF(Y / X) \cong C^\bullet(Y / X) \otimes_{\cO_X(X)} \cF(X),
            \]
            where $C^\bullet(Y / X)$ is as in Lemma \ref{lem-O-ket-Cech}.  Since the contracting homotopy used in the proof of Lemma \ref{lem-O-ket-Cech} \Pth{based on the proof of \cite[\aLem 3.28]{Niziol:2008-ktls-1}} is $\cO_X(X)$-linear, $C^\bullet_\cF(Y / X)$ is also exact, as desired.

        \item First assume that $X$ is modeled on a sharp fs monoid $P$.  By definition, there exists a Kummer \'etale covering $\{ U_i \to X \}_i$ such that each $\cF|_{U_i}$ is analytic coherent.  By Lemma \ref{lem-Abhyankar-basic-cov} and Proposition \ref{prop-et-coh-descent}, we may assume that $\cF|_U$ is analytic coherent, where $U = X^{\frac{1}{n}}$ for some $n$ invertible in $k$.  Let $G := (\frac{1}{n} P)^\gp / P^\gp$.  Since $H^j\bigl((U \times_X \cdots \times_X U)_\ket, \cF\bigr) = 0$, for all $j > 0$, by \Refenum{\ref{thm-ket-ana-coh}}, it suffices to show that $H^i\bigl(C^\bullet_\cF(U / X)\bigr) = 0$, for all $i > 0$.  As in the proof of Lemma \ref{lem-O-ket-Cech}, by Proposition \ref{prop-ket-std},
            \[
                \cO_{X_\ket}(U \times_X \cdots \times_X U) \cong \cO_{X_\ket}(U) \otimes_k k[G] \otimes_k \cdots \otimes_k k[G],
            \]
            and we can identify the complex
            \[
                \cF(U) \to \cF(U \times_X U) \to \cF(U \times_X U \times_X U) \to \cdots
            \]
            with the complex computing the group cohomology $H^i\bigl(G, \cF(U)\bigr)$.  Since $|G|$ is invertible in $k$ as $n$ is, and since $\cF(U)$ is a $k$-vector space, we have $H^i\bigl(G, \cF(U)\bigr) = 0$, for all $i > 0$.

            Let $\varepsilon_\et: X_\ket \to X_\et$ denote the natural projection of sites.  Then the argument above shows that $R^j \varepsilon_{\et, *}(\cF) = 0$, for all $j > 0$, and that $\varepsilon_{\et, *}(\cF)$ is a coherent sheaf on $X_\et$.  Since these statements are \'etale local in nature, they extend to all $X$ considered in the statement of the theorem, by Proposition \ref{prop-chart-stalk-chart-fs}.  Thus, we have $H^i(X_\ket, \cF) \cong H^i\bigl(X_\et, \varepsilon_{\et, *}(\cF)\bigr) = 0$, as desired, by Proposition \ref{prop-et-coh-descent}.  \qedhere
    \end{enumerate}
\end{proof}

Kummer \'etale descent of objects \Pth{coherent sheaves, log adic spaces, \etc} are usually not effective, mainly because fiber products of Kummer \'etale covers do not correspond to fiber products of structure rings.  Here is a standard counterexample.
\begin{exam}\label{ex-ket-descent-fails}
    Let $k$ be a nonarchimedean field.  As in Example \ref{ex-log-adic-sp-disc}, consider the unit disc $\bD = \Spa(k\Talg{T}, \cO_k\Talg{T})$ equipped with the log structure modeled on the chart $\bZ_{\geq 0} \to k\Talg{T}: \; 1 \mapsto T$.  By Proposition \ref{prop-ket-std}, we have a Galois standard Kummer \'etale cover $f_n: \bD \to \bD$ corresponding to the chart $\bZ_{\geq 0} \to \bZ_{\geq 0}: \; 1 \mapsto n$, where $n$ is invertible in $k$, with Galois group $\Grpmu_n$.  Then the ideal sheaf $\cI$ of the origin, a $\Grpmu_n$-invariant invertible sheaf on $\bD$, does not descend via $f_n$.
\end{exam}

Kummer \'etale descent of morphisms are more satisfactory.
\begin{prop}\label{prop-ket-descent-mor}
    Let $X$ be a locally noetherian fs log adic space, and let $f: Y \to X$ be a Kummer \'etale cover.  Let $\pr_1, \pr_2: Y \times_X Y \to Y$ denote the two projections.  Suppose that $\cE$ and $\cF$ are analytic coherent $\cO_{X_\ket}$-modules; and that $g': f^*(\cE) \to f^*(\cF)$ is a morphism on $Y$ such that $\pr_1^*(g') = \pr_2^*(g')$ on $Y \times_X Y$.  Then there exists a unique morphism $g: \cE \to \cF$ such that $f^*(g) = g'$.
\end{prop}
\begin{proof}
    By Lemma \ref{lem-Abhyankar-basic-cov} and Proposition \ref{prop-et-coh-descent}, we may assume that $X$ is affinoid and that $Y \to X$ is a standard Kummer cover.  Let $A := \cO_X(X)$, $B := \cO_{X_\ket}(Y)$, $C := \cO_{X_\ket}(Y \times_X Y)$, $M := \cE(X)$, and $N := \cF(X)$.  We need to show that
    \[
        0 \to \Hom_A(M, N) \to \Hom_B(B \otimes_A M, B\otimes_A N) \to \Hom_C(C \otimes_A M, C \otimes_A N)
    \]
    is exact, and where the third arrow is the difference between two pullbacks as usual. Equivalently, we need to show that
    \[
        0 \to \Hom_A(M, N) \to \Hom_A(M, B \otimes_A N) \to \Hom_A(M, C \otimes_A N)
    \]
    is exact.  By the left exactness of $\Hom_A(M, \,\cdot\,)$, we are reduced to showing that the sequence $0 \to N \to B \otimes_A N \to C \otimes_A N$ is exact.  But this is just the first three terms in the complex $C^\bullet_\cF(Y / X)$ in the proof of Theorem \ref{thm-ket-coh}\Refenum{\ref{thm-ket-ana-coh}}.
\end{proof}

To wrap up the subsection, let us introduce a convenient basis for the Kummer \'etale topology.
\begin{lem}\label{lem-basis-ket}
    Let $X$ be a locally noetherian fs log adic space.  Let $\cB$ be the full subcategory of $X_\ket$ consisting of affinoid adic spaces $V$ with fs global charts.  Then $\cB$ is a basis for $X_\ket$, and we have an isomorphism of topoi $X_\ket^\sim \Mi \cB^\sim$.
\end{lem}
\begin{proof}
    By \cite[III, 4.1]{SGA:4}, it suffices to show that every object in $X_\ket$ has a covering by objects in $\cB$.  But this is clear.
\end{proof}

\begin{lem}\label{lem-ana-coh-rule}
    Let $X$ be a locally noetherian fs log adic space, and let $\cB$ be as in Lemma \ref{lem-basis-ket}.  Suppose that $\cF$ is a rule that functorially assigns to each $V \in \cB$ a finite $\cO_{X_\ket}(V)$-module $\cF(V)$ such that $\cF(V) \otimes_{\cO_V(V)} \cO_{V'}(V') \Mi \cF(V')$ for all $V' \to V$ that are either \'etale morphisms or standard Kummer \'etale covers.  Then $\cF$ defines an analytic coherent sheaf on $X_\ket$.
\end{lem}
\begin{proof}
    This follows from Propositions \ref{prop-ket-descent-mor} and \ref{prop-et-coh-descent}, and Lemma \ref{lem-basis-ket}.
\end{proof}

\subsection{Descent of Kummer \'etale covers}\label{sec-ket-cov-descent}

\begin{defn}\label{def-fket}
    Let $X$ be a locally noetherian fs log adic space $X$.  Let $X_\fket$ denote the full subcategory of $X_\ket$ consisting of log adic spaces that are \emph{finite} Kummer \'etale over $X$.  Let $\underline{\Fket}$ denote the fibered category over the category of locally noetherian fs log adic spaces such that $\underline{\Fket}(X) = X_\fket$.
\end{defn}

The goal of this subsection is to show that Kummer \'etale covers satisfy effective descent in $\underline{\Fket}$.  We first study $X_\fket$ when $X$ is as in Examples \ref{ex-log-adic-sp-pt} and \ref{ex-log-adic-sp-pt-sep-cl}.

\begin{defn}\phantomsection\label{def-log-geom-pt}
    \begin{enumerate}
        \item A \emph{log geometric point} is a log point $\zeta = (\Spa(l, l^+), \cM, \alpha)$ \Pth{\Refcf{} Examples \ref{ex-log-adic-sp-pt} and \ref{ex-log-adic-sp-pt-sep-cl}} such that:
            \begin{enumerate}
                \item $l$ is a complete separably closed nonarchimedean field; and

                \item if $M := \Gamma(\Spa(l, l^+), \cM)$, then $\overline{M} = M / l^\times$ is uniquely $n$-divisible \Pth{see Definition \ref{def-monoid-n-div}} for all positive integers $n$ invertible in $l$.
            \end{enumerate}

        \item Let $X$ be a locally noetherian fs log adic space.  A \emph{log geometric point} of $X$ is a morphism of log adic spaces $\eta: \zeta \to X$ from a log geometric point $\zeta$.

        \item Let $X$ be a locally noetherian fs log adic space.  A \emph{Kummer \'etale neighborhood} of a log geometric point $\eta: \zeta \to X$ is a lifting of $\eta$ to some composition $\zeta \to U \Mapn{\phi} X$ in which $\phi$ is Kummer \'etale.
    \end{enumerate}
\end{defn}

\begin{constr}\label{constr-log-geom-pt}
    For each geometric point $\xi: \Spa(l, l^+) \to X$, let us construct some log geometric point $\widetilde{\xi}$ above it \Pth{\ie, the morphism $\widetilde{\xi} \to X$ of underlying adic spaces factors through $\xi \to X$} as follows.  By Proposition \ref{prop-chart-stalk-chart-fs}, up to \'etale localization on $X$, we may assume that $X$ admits a chart modeled on a sharp fs monoid $P$, so that we have a strict closed immersion $X \to X\Talg{P}$ as in Remark \ref{rem-def-chart-rel}.  We equip $\Spa(l, l^+)$ with the log structure $P^{\log}$ associated with the pre-log structure given by the composition of $P \to \cO_X(X) \to l$, so that $(\Spa(l, l^+), P^{\log})$ is an fs log point with a chart given by $P \to l$.  We shall still denote this fs log point by $\xi$.  For each positive integer $m$, let $P \Em \frac{1}{m} P$ be as in Definition \ref{def-Kummer-n-th}.  Consider
    \[
        \xi^{(\frac{1}{m})} := (\Spa(l, l^+) \times_{X\Talg{P}} X\Talg{\tfrac{1}{m} P})_\red,
    \]
    equipped with the natural log structure modeled on $\frac{1}{m} P$.  Note that $\xi^{(\frac{1}{m})}$ might differ from the $\xi^{\frac{1}{m}}$ in Definition \ref{def-Kummer-n-th}, because we are taking the reduced subspace, so that the underlying adic space of $\xi^{(\frac{1}{m})}$ is still isomorphic to $\Spa(l, l^+)$.  Then
    \[
        \widetilde{\xi} := \varprojlim_m \xi^{(\frac{1}{m})}
    \]
    where the inverse limit runs through all positive integers $m$ invertible in $l$, is a log geometric point above $\xi$.  The underlying adic space of $\widetilde{\xi}$ is isomorphic to $\Spa(l, l^+)$, endowed with the natural log structure modeled on
    \[
        P_{\bQ_{\geq 0}} := \varinjlim_m \tfrac{1}{m} P,
    \]
    where the direct limit runs through all positive integers $m$ invertible in $l$.
\end{constr}

\begin{lem}\label{lem-log-geom-pt}
    Let $\zeta \to X$ be a log geometric point of a locally noetherian fs log adic space.  Then the functor $\Sh(X_\ket) \to \underline{\Sets}: \, \cF \mapsto \cF_\zeta := \varinjlim \cF(U)$ from the category of sheaves on $X_\ket$ to the category of sets, where the direct limit is over Kummer \'etale neighborhoods $U$ of $\zeta$, is a fiber functor.  The fiber functors defined by log geometric points form a conservative system.
\end{lem}
\begin{proof}
    By Proposition \ref{prop-lem-four-pt} and Remark \ref{rem-Kummer-exact}, the category of Kummer \'etale neighborhood of $\zeta$ is filtered, and hence the first statement follows.  Since every point of $X$ admits some geometric point and hence some log geometric point above it \Pth{see Construction \ref{constr-log-geom-pt}}, and since every object $U$ in $X_\ket$ is covered by liftings of log geometric points of $X$, the second statement also follows.
\end{proof}

\begin{defn}\label{def-G-FSets}
    For each profinite group $G$, let $G\Utext{-}\underline{\FSets}$ denote the category of finite sets \Pth{with discrete topology} with continuous actions of $G$.
\end{defn}

\begin{defn}\label{def-mu}
    Let $l$ be a separably closed field.  For each positive integer $m$, let $\Grpmu_m(l)$ denote the group of $m$-th roots of unity in $l$.  Let $\Grpmu_\infty(l) := \varinjlim_m \Grpmu_m(l)$ and $\widehat{\bZ}'(1)(l) := \varprojlim_m \Grpmu_m(l)$, where the limits run through all positive integers $m$ invertible in $l$.  When $\chr(l) = 0$, we shall write $\widehat{\bZ}(1)(l)$ instead of $\widehat{\bZ}'(1)(l)$.  When there is no risk of confusion in the context, we shall simply write $\Grpmu_m$, $\Grpmu_\infty$, and $\widehat{\bZ}'(1)$, without the symbols $(l)$.
\end{defn}

\begin{prop}\label{prop-fket-log-pt}
    Let $\xi = (\Spa(l, l^+), \cM)$ be an fs log point with $l$ complete \Pth{by our convention} and separably closed.  Let $M := \cM_\xi$ and so $\overline{M} \cong M / l^\times$.  Let $\widetilde{\xi}$ be a log geometric point constructed as in Construction \ref{constr-log-geom-pt}.  Then the functor $F_{\widetilde{\xi}}: Y \mapsto \Hom_\xi(\widetilde{\xi}, Y)$ induces an equivalence of categories
    \[
        \xi_\fket \cong \Hom\bigl(\overline{M}^\gp, \widehat{\bZ}'(1)(l)\bigr)\Utext{-}\underline{\FSets}.
    \]
\end{prop}
\begin{proof}
    For simplicity, we shall omit the symbols $(l)$ as in Definition \ref{def-mu}.  Let $P := \overline{M}$, a sharp and fs monoid.  By Lemma \ref{lem-monoid-split}, we have some splitting $M \Mi l^\times \oplus P$ such that $P \Mapn{1 \oplus \Id} l^\times \oplus P \Mi M$ defines a chart for $\xi$.  For each $m$ invertible in $l$, the cover $\xi^{(\frac{1}{m})} \to \xi$ is given by $M \Mi l^\times \oplus P \Mapn{\Id \oplus [m]} l^\times \oplus P \Mi M$.  Note that any finite Kummer \'etale cover of $\xi$ is a finite disjoint union of fs log adic spaces of the form
    \[
        \xi_Q := \xi \times_{\xi\Talg{P}} \xi\Talg{Q},
    \]
    where $P \to Q$ is a Kummer homomorphism of sharp fs monoids whose cokernel is annihilated by some integer invertible in $l$.  We have
    \[
        F_{\widetilde{\xi}}(\xi_Q) = \Mor_\xi(\widetilde{\xi}, \xi_Q) \cong \Hom_{l^\times \oplus P}(l^\times \oplus Q, l^\times \oplus P_{\bQ_{\geq 0}}) \cong \Hom(Q^\gp / P^\gp, \Grpmu_\infty).
    \]
    The last group has a natural transitive action of
    \[
    \begin{split}
        & \Aut_\xi(\widetilde{\xi}) \cong \Hom_P(P_{\bQ_{\geq 0}}, l^\times) \cong \Hom\bigl((P_{\bQ_{\geq 0}})^\gp / P^\gp, \Grpmu_\infty\bigl) \\
        & \cong \varprojlim_m \Hom\bigl(P^\gp \otimes_\bZ (\tfrac{1}{m} \bZ / \bZ), \Grpmu_m\bigr) \cong \Hom\bigl(P^\gp, \widehat{\bZ}'(1)\bigr).
    \end{split}
    \]
    Hence, $F_{\widetilde{\xi}}$ is indeed a functor from $\xi_\fket$ to $\Hom\bigl(\overline{M}^\gp, \widehat{\bZ}'(1)\bigr)\Utext{-}\underline{\FSets}$.

    Let us verify that $F_{\widetilde{\xi}}$ is fully faithful.  By working with connected components, it suffices to show that, for any $Q_1$ and $Q_2$, the natural map
    \begin{equation}\label{eq-prop-fket-log-pt-ff}
        \Mor_\xi(\xi_{Q_1}, \xi_{Q_2}) \to \Hom\bigl(\Hom(Q_1^\gp / P^\gp, \Grpmu_\infty), \Hom(Q_2^\gp / P^\gp, \Grpmu_\infty)\bigr)
    \end{equation}
    is bijective.  Note that $\Mor_\xi(\xi_{Q_1}, \xi_{Q_2}) \cong \Hom_P(Q_2, l^\times \oplus Q_1)$.  Since $P_{\bQ_{\geq 0}}$ is uniquely divisible, the sharp fs $Q_1$ and $Q_2$ monoids can be viewed as submonoids of $P_{\bQ_{\geq 0}}$.  If $Q_2 \not\subset Q_1$, then both sides of \Refeq{\ref{eq-prop-fket-log-pt-ff}} are zero.  Otherwise, $Q_2 \subset Q_1$, and \Refeq{\ref{eq-prop-fket-log-pt-ff}} sends the morphism induced by $Q_2 \Em Q_1$ in $\Hom_P(Q_2, l^\times \oplus Q_1)$ to the homomorphism induced by restriction from $Q_1^\gp / P^\gp$ to $Q_2^\gp / P^\gp$.  Consequently, \Refeq{\ref{eq-prop-fket-log-pt-ff}} is bijective, because both sides of \Refeq{\ref{eq-prop-fket-log-pt-ff}} are principally homogeneous under compatible actions of $\Aut_\xi(\xi_{Q_2}) \cong \Hom\bigl(Q_2^\gp / P^\gp, \Grpmu_\infty\bigr)$.

    Finally, let us verify that $F_{\widetilde{\xi}}$ is essentially surjective.  Since any discrete finite set $S$ with a continuous action of $\Hom\bigl(P^\gp, \widehat{\bZ}'(1)\bigr) \cong \Hom\bigl((P_{\bQ_{\geq 0}})^\gp / P^\gp, \Grpmu_\infty\bigr)$ is a disjoint union of orbits, we may assume the action on $S$ is transitive.  Then $S$ is a quotient space of $\Hom\bigl((P_{\bQ_{\geq 0}})^\gp / P^\gp, \Grpmu_\infty\bigr)$, which corresponds by Pontryagin duality to a finite subgroup $G \subset (P_{\bQ_{\geq 0}})^\gp / P^\gp$.  Let $Q$ denote the preimage of $G$ in $P_{\bQ_{\geq 0}}$.  Then $Q^\gp / P^\gp \cong G$ and $F_{\widetilde{\xi}}(\xi_Q) \cong S$, as desired.
\end{proof}

\begin{prop}\label{prop-fket-str-loc}
    Let $(X, \cM_X)$ be a locally noetherian fs log adic space.  Let $\xi = \Spa(l, l^+)$ be a geometric point of $X$, and let $X(\xi)$ be the strict localization of $X$ at $\xi$, with its log structure pulled back from $X$.  Without loss of generality, let us assume that $l \cong \AC{\kappa}(x)$, the completion of a separable closure of the residue field $\kappa(x)$ of $\cO_{X, x}$, for some $x \in X$.  Let $M := \cM_{X, \xi}$ and so $\overline{M} \cong M / l^\times$.  Let us view $\xi$ and $X(\xi)$ as log adic spaces by equipping them with the log structures pulled back from $X$.  Let $\widetilde{\xi}$ be the log geometric point over $\xi$ constructed as in Construction \ref{constr-log-geom-pt}.  Then the functor $H_{\widetilde{\xi}}: Y \mapsto \Mor_X(\widetilde{\xi}, Y)$ induces an equivalence of categories
    \[
        X(\xi)_\fket \cong \Hom\bigl(\overline{M}^\gp, \widehat{\bZ}'(1)(l)\bigr)\Utext{-}\underline{\FSets}.
    \]
    In addition, we have $H_{\widetilde{\xi}} = F_{\widetilde{\xi}} \circ \iota^{-1}$, where $F_{\widetilde{\xi}}$ is as in Proposition \ref{prop-fket-log-pt}, and $\iota^{-1}: X(\xi)_\fket \to \xi_\fket$ is the natural pullback functor defined by $\iota: \xi \to X(\xi)$.
\end{prop}
Note that, if $x$ is an analytic point of $X$, then Proposition \ref{prop-fket-str-loc} follows immediately from Proposition \ref{prop-fket-log-pt}, because $\xi \cong X(\xi)$ in this case.  Nevertheless, the proof below works for non-analytic points as well.
\begin{proof}[Proof of Proposition \ref{prop-fket-str-loc}]
    It suffices to show that $\iota^{-1}$ is an equivalence of categories.  Write $P = \overline{M}$ and $X(\xi) = \Spa(R, R^+)$.  By Lemma \ref{lem-monoid-split}, we can choose a splitting $R^\times \oplus P \Mi M$ such that $P \Mapn{1 \oplus \Id} R^\times \oplus P \Mi M$ defines a chart for $X(\xi)$.  Note that objects in $X(\xi)_\fket$ \Pth{\resp $\xi_\fket$} are finite disjoint unions of fs log adic spaces of the form
    \begin{equation}\label{eq-prop-fket-str-loc-def-X-xi-Q}
        X(\xi)_Q := X(\xi) \times_{X(\xi)\Talg{P}} X(\xi)\Talg{Q}
    \end{equation}
    \Pth{\resp $\xi_Q$}, where $P \to Q$ is a Kummer homomorphism of sharp monoids.  Then $\iota^{-1}$ sends $X(\xi)_Q$ to $\xi_Q$, and hence is essentially surjective.  To see that $\iota^{-1}$ is fully faithful, it suffices to show that the canonical map
    \begin{equation}\label{eq-prop-fket-str-loc-ff}
        \Mor_{X(\xi)}(X(\xi)_{Q_1}, X(\xi)_{Q_2}) \to \Mor_\xi(\xi_{Q_1}, \xi_{Q_2})
    \end{equation}
    is bijective.  By definition, we have $X(\xi)_{Q_i} \cong \Spa(R_{Q_i}, R_{Q_i}^+)$, for $i = 1, 2$, where $R_{Q_i} = R \otimes_{R\Talg{P}} R\Talg{Q} \cong R \otimes_{R[P]} R[Q]$ are also strictly local rings with residue field $l$.  Therefore,
    \[
    \begin{split}
        & \Mor_{X(\xi)}(X(\xi)_{Q_1}, X(\xi)_{Q_2}) \cong \Hom_P(Q_2, R_{Q_1}^\times \oplus Q_1) \\
        & \cong \Hom_P(Q_2, l^\times \oplus Q_1) \cong \Mor_\xi(\xi_{Q_1}, \xi_{Q_2}),
    \end{split}
    \]
    and hence the map \Refeq{\ref{eq-prop-fket-str-loc-ff}} is bijective, as desired.
\end{proof}

Now, we are ready to prove the main result of this subsection; \ie, Kummer \'etale covers satisfy effective descent in the fibered category $\underline{\Fket}$.
\begin{thm}\label{thm-fket-descent}
    Let $X$ be a locally noetherian fs log adic space, and let $f: Y \to X$ be a Kummer \'etale cover.  Let $\pr_1, \pr_2: Y \times_X Y \to Y$ denote the two projections.  Suppose that $\breve{Y} \in Y_\fket$ and that there exists an isomorphism $\pr_1^{-1}(\breve{Y}) \Mi \pr_2^{-1}(\breve{Y})$ satisfying the usual cocycle condition.  Then there exists a unique $\breve{X} \in X_\fket$ \Pth{up to isomorphism} such that $\breve{Y} \cong \breve{X} \times_X Y$.
\end{thm}
\begin{proof}
    By \'etale descent \Pth{see Proposition \ref{prop-et-coh-descent}}, by \'etale localization on $X$, it suffices to prove the theorem in the case where $X$ is affinoid with a sharp fs chart $P$, and where $Y \to X$ is a standard Kummer \'etale cover induced by a Kummer homomorphism of sharp monoids $u: P \to Q$, with $G := Q^\gp / u^\gp(P^\gp)$ a finite group of order invertible in $\cO_X$.  By Proposition \ref{prop-ket-std}, up to further \'etale localization on $X$, we may assume that the morphism $Y \to X$ is a Galois cover with Galois group $\Gamma := \Hom(G, \cO_X(X)^\times)$; that $|G|$ is invertible in $\cO_X$, and $\cO_X(X)^\times$ contains all the $|G|$-th roots of unity; and that $Y \times_X Y \cong \Gamma_X \times_X Y$.  In this case, the descent datum is equivalent to an action of $\Gamma$ on $\breve{Y}$ over $X$ lifting the action of $\Gamma$ on $Y$ over $X$.  Let us write $X = (\Spa(R, R^+), \cM_X)$ and $\breve{Y} = (\Spa(\breve{S}, \breve{S}^+), \cM_{\breve{Y}})$.  By Lemma \ref{lem-quot}, $(\breve{R}, \breve{R}^+) := (\breve{S}^\Gamma, (\breve{S}^+)^\Gamma)$ is a Huber pair, and $\breve{X} := \Spa(\breve{R}, \breve{R}^+)$ is a noetherian adic space finite over $X$.  Moreover, the morphism $\breve{Y} \to \breve{X}$ is finite, open, surjective, and invariant under the $\Gamma$-action on $\breve{Y}$.  The \'etale sheaf of monoids $\cM_{\breve{X}}$ defined by $\cM_{\breve{X}}(U) := \bigl(\cM_{\breve{Y}}(\breve{Y} \times_{\breve{X}} U)\bigr)^\Gamma$, for each $U \in \breve{X}_\et$, is fine and saturated, and defines a log structure of $\breve{X}$.  We claim that the log adic space $\breve{X}$ thus obtained gives the desired descent.

    Let us first verify that the canonical morphism $\breve{Y} \to \breve{X} \times_X Y$ induced by the structure morphisms $\breve{Y} \to \breve{X}$ and $\breve{Y} \to Y$ is an isomorphism.  Since the morphism is between spaces that are finite over $X$, and since the formation of $\Gamma$-invariants is compatible with \Pth{strict} base change \Pth{as $|\Gamma|$ is invertible in $\cO_X$}, we may assume that $X = X(\xi)$ is strictly local, and so is $Y \cong X(\xi)_Q = X(\xi) \times_{X(\xi)\Talg{P}} X(\xi)\Talg{Q}$ \Pth{as in the proof of Proposition \ref{prop-fket-str-loc}}.  Without loss of generality, we may assume that $\breve{Y} \cong X(\xi)_{\breve{Q}} := X(\xi) \times_{X(\xi)\Talg{P}} X(\xi)\Talg{\breve{Q}}$ for some Kummer homomorphism of fs monoids $\breve{u}: P \to \breve{Q}$ \Pth{as in \Refeq{\ref{eq-prop-fket-str-loc-def-X-xi-Q}}, but without assuming that $\breve{Q}$ is sharp}, which is the composition of $u: P \to Q$ with some homomorphism $Q \to \breve{Q}$.  Under the equivalence of categories $H_{\widetilde{\xi}}: X(\xi)_\fket \cong \Hom\bigl(P^\gp, \widehat{\bZ}'(1)(l)\bigr)\Utext{-}\underline{\FSets}$ as in Proposition \ref{prop-fket-str-loc}, $\breve{Y} \to X$ corresponds to the set $\breve{\Gamma} := \Hom(\breve{Q}^\gp / \breve{u}^\gp(P^\gp), \Grpmu_\infty)$ with a $\Gamma$-action; $Y \to X$ corresponds to $\Gamma$ itself \Pth{with its canonical $\Gamma$-action}; and $\breve{Y} \to Y$ corresponds to a surjective $\Gamma$-equivariant map $\breve{\Gamma} \Surj \Gamma$.  Since $\breve{Y} \to X$ is Kummer, we have $(\cM_{\breve{Y}}^\Gamma)^\gp = (\cM_{\breve{Y}}^\gp)^\Gamma$, and hence $\breve{X} \cong X(\xi)_{\breve{P}}$ for the fs monoid $\breve{P} = \breve{Q} \cap \breve{P}^\gp$ such that $\Hom(\breve{P}^\gp / \breve{u}^\gp(P), \Grpmu_\infty) \cong \breve{\Gamma} / \Gamma$, by explicitly computing $\breve{R} = \breve{S}^\Gamma \cong (R \otimes_{R[P]} R[\breve{Q}])^\Gamma$ using the identifications in the proof of Proposition \ref{prop-fket-str-loc}.  In particular, $\breve{X} \to X$ is finite Kummer \'etale, and $\breve{Y} \to \breve{X}$ corresponds to the quotient $\breve{\Gamma} \to \breve{\Gamma} / \Gamma$ under $H_{\widetilde{\xi}}$.  Since $\breve{\Gamma}$ is an abelian group, the canonical map $\breve{\Gamma} \to (\breve{\Gamma} / \Gamma) \times \Gamma$ is bijective, and hence the corresponding canonical morphism $\breve{Y} \to \breve{X} \times_X Y$ is indeed an isomorphism.

    Consequently, $\breve{Y} \cong \breve{X} \times_X Y \cong \breve{X} \times_{\breve{X}\Talg{P}} \breve{X}\Talg{Q} \to \breve{X}$ is finite Kummer \'etale.  By construction, $\breve{X} \to X$ is also finite Kummer \Pth{firstly by assuming that $X$ is strictly local as above, and then by extending the identifications of charts over \'etale neighborhoods of $X$ in general}.  By Lemma \ref{lem-ket-log-et-ex}, it remains to show that $\breve{X} \to X$ is log \'etale.  Since $\breve{Y} \to Y$ is log \'etale, and since $Y \to X$ is a Kummer \'etale covering, this follows from Proposition \ref{prop-log-sm-et-ket-descent}, as desired.
\end{proof}

\begin{cor}\label{cor-fket-quot}
    Let $X$ be a locally noetherian fs log adic space, and let $f: Y \to X$ be a finite Kummer \'etale cover.  Let $\Gamma$ be a finite group which acts on $Y$ by morphisms over $X$.  Then the canonical morphisms $Y \to Z := Y / \Gamma \to X$ induced by $f$ \Pth{by Lemma \ref{lem-quot}} are both finite Kummer \'etale covers.
\end{cor}
\begin{proof}
    By Lemma \ref{lem-quot}, both morphisms $Y \to Z$ and $Z \to X$ are finite, and $Y \to Z$ is finite Kummer.  By Lemma \ref{lem-ket-log-et-ex} and Proposition \ref{prop-ket-mor}, it suffices to show that $Z \to X$ is finite Kummer \'etale.  Then the first projection $\breve{f}: \breve{Y} := Y \times_X Y \to Y$ is a pullback of $Y \to X$, which inherits an action of $\Gamma$.  By \cite[\aLem 1.7.6]{Huber:1996-ERA}, under the noetherian hypothesis, the formation of quotients by $\Gamma$ as in Lemma \ref{lem-quot} is compatible with base changes under \'etale morphisms of affinoid adic spaces.  By Proposition \ref{prop-ket-std} and Remark \ref{rem-ket-site-gen}, up to \'etale localization on $X$, we may assume that $Y \to X$ is a composition of a finite \'etale morphism $f_1$ and a standard Kummer \'etale cover $f_2$ as in Proposition \ref{prop-ket-std}, in which case $\breve{Y}$ is a disjoint union of sections of $\breve{f}$.  Then $\Gamma$ acts on $\breve{Y}$ by permuting such sections, and we have a quotient $\breve{Z} := \breve{Y} / \Gamma \to Y$, which is clearly finite Kummer \'etale.  Moreover, the pullbacks of $\breve{Z} \to Y$ to $\breve{Y}$ along the two projections are isomorphic to each other by interchanging the factors, and hence $\breve{Z} \to Y$ descends to a finite Kummer \'etale cover of $X$, by Theorem \ref{thm-fket-descent}.  We claim that this cover is canonically isomorphic to $Z \to X$.  Since the set of sections of $\breve{f}: \breve{Y} \to Y$ is a disjoint union of subsets formed by the sections of pullback of the finite \'etale morphism $f_1$, we can reduce the claim to the extremal cases where either $f = f_1$ is finite \'etale or $f = f_2$ is standard Kummer \'etale.  In the former case, the claim follows from the usual theory for finite \'etale covers of schemes, as in \cite[V]{SGA:1}.  In the latter case, the claim follows from Proposition \ref{prop-ket-std}\Refenum{\ref{prop-ket-std-4}}.
\end{proof}

\begin{defn}\label{def-ket-loc-const}
    Let $X$ be a locally noetherian fs log adic space, and let $\Lambda$ be a commutative ring.
    \begin{enumerate}
        \item A sheaf $\cF$ on $X_\ket$ is called a \emph{constant sheaf of sets} \Pth{\resp \emph{constant sheaf of $\Lambda$-modules}} if it is the sheafification of a constant presheaf given by some set $S$ \Pth{\resp some $\Lambda$-module $M$}.

        \item A sheaf $\cF$ on $X_\ket$ is called \emph{locally constant} if there exists a Kummer \'etale covering $\{ U_i \}_{i \in I} \to X$ such that all $\cF|_{U_i}$ are constant sheaves.  We denote by $\Loc(X_\ket)$ the category of locally constant sheaves of finite sets on $X_\ket$.
    \end{enumerate}
\end{defn}

\begin{thm}\label{thm-fket-loc-syst}
    Let $X$ be a locally noetherian fs log adic space.  The functor
    \[
        \phi: X_\fket \to \Loc(X_\ket): \; Y \mapsto \Mor_X(\,\cdot\,, Y)
    \]
    is an equivalence of categories.  Moreover:
    \begin{enumerate}
        \item\label{thm-fket-loc-syst-1} Fiber products exist in $X_\fket$ and $\Loc(X_\ket)$, and $\phi$ preserves fiber products.

        \item\label{thm-fket-loc-syst-2} Categorical quotients by finite groups exist in $X_\fket$ and $\Loc(X_\ket)$, and $\phi$ preserves such quotients.
    \end{enumerate}
\end{thm}
\begin{proof}
    By Proposition \ref{prop-ket-repr}, representable presheaves on $X_\ket$ are sheaves.  By Proposition \ref{prop-ket-std} and Remark \ref{rem-ket-site-gen}, any $Y \in X_\fket$ is Kummer \'etale locally \Pth{on $X$} a disjoint union of finitely many copies of $X$.  Hence, $\Mor_X(\,\cdot\,, Y)$ is indeed a locally constant sheaf of finite sets, and the functor $\phi$ is defined.  The functor $\phi$ is fully faithful for formal reasons.  Since any locally constant sheaf of finite set is Kummer \'etale locally represented by objects in $X_\fket$, these objects glue to a global object $Y$ by Theorem \ref{thm-fket-descent} and the full faithfulness of $\phi$.  This shows that $\phi$ is also essentially surjective, as desired.  As for the statements \Refenum{\ref{thm-fket-loc-syst-1}} and \Refenum{\ref{thm-fket-loc-syst-2}}, by Kummer \'etale localization, we just need to note that the statements become trivial after replacing the source and target of the functor $\phi$ with the categories of finite disjoint unions of copies of $X$ and of constant sheaves of finite sets, respectively.
\end{proof}

Next, let us define the Kummer \'etale fundamental groups.
\begin{lem}\label{lem-fket-Gal-cat}
    Let $X$ be a connected locally noetherian fs log adic space, and $\eta: \zeta \to X$ a log geometric point.  Let $\underline{\FSets}$ denote the category of finite sets.  Consider the fiber functor
    \begin{equation}\label{eq-lem-fket-fiber-functor}
        F: X_\fket \to \underline{\FSets}: \; Y \mapsto Y_\zeta := \Mor_X(\zeta, Y).
    \end{equation}
    Then $X_\fket$ together with the fiber functor $F$ is a Galois category.
\end{lem}
\begin{proof}
    We already know that the final object, fiber products \Pth{see Proposition \ref{prop-ket-compos-bc}}, categorical quotients by finite groups \Pth{see Corollary \ref{cor-fket-quot}}, and finite coproducts exist in $X_\fket$ \Pth{and $\underline{\FSets}$}.  It remains to verify the following conditions:
    \begin{enumerate}
        \item\label{lem-fket-Gal-cat-1}  $F$ preserves fiber products, finite coproducts, and quotients by finite groups.

        \item\label{lem-fket-Gal-cat-2}  $F$ reflexes isomorphisms \Pth{\ie, $F(f)$ being an isomorphism implies $f$ also being an isomorphism}.
    \end{enumerate}
    \Pth{We refer to \cite[V, 4]{SGA:1} for the basics on Galois categories.}

    As for condition \Refenum{\ref{lem-fket-Gal-cat-1}}, since $F$ is defined Kummer \'etale locally at the log geometric point $\zeta$, it suffices to verify the condition after restricting $F$ to the category of finite disjoint unions of $X$, in which case the condition clearly holds.

    As for condition \Refenum{\ref{lem-fket-Gal-cat-2}}, note that $F$ factors through the equivalence of categories $\phi$ in Theorem \ref{thm-fket-loc-syst} and induces the stalk functor
    \[
        \Loc(X_\ket) \to \underline{\FSets}: \; \cF \mapsto \cF_\zeta := \varinjlim \cF(U),
    \]
    where the direct limit is over Kummer \'etale neighborhoods $U$ of $\zeta$.  Since $X$ is connected, the stalk functors at any two log geometric points are isomorphic.  Thus, whether $f$ is an isomorphism can be checked at just one stalk.
\end{proof}

\begin{cor}\label{cor-fund-grp}
    Let $X$ and $\zeta \to X$ be as in Lemma \ref{lem-fket-Gal-cat}.  Then the fiber functor $F$ in \Refeq{\ref{eq-lem-fket-fiber-functor}} is prorepresentable.  Let $\pi_1^\ket(X, \zeta)$ be the automorphism group of $F$.  Then $F$ induces an equivalence of categories
    \begin{equation}\label{eq-fund-grp-fket}
        X_\fket \Mi \pi_1^\ket(X, \zeta)\Utext{-}\underline{\FSets},
    \end{equation}
    which is the composition of the equivalence of categories $\phi$ in Theorem \ref{thm-fket-loc-syst} with the equivalence of categories
    \begin{equation}\label{eq-fund-grp-loc-syst}
        \Loc(X_\ket) \Mi \pi_1^\ket(X, \zeta)\Utext{-}\underline{\FSets}
    \end{equation}
    induced by the stalk functor $\cF \mapsto \cF_\zeta$.
\end{cor}

\begin{rk}\label{rem-fund-grp}
    In Corollary \ref{cor-fund-grp}, since stalk functors at any two log geometric points $\zeta$ and $\zeta'$ are isomorphic, the fundamental groups $\pi_1^\ket(X, \zeta)$ and $\pi_1^\ket(X, \zeta')$ are isomorphic.  We shall omit $\zeta$ from the notation when the context is clear.
\end{rk}

\begin{cor}\label{cor-fund-grp-log-pt}
    Let $(X, \cM_X)$, $\xi = \Spa(l, l^+)$, $X(\xi)$, and $M := \cM_{X, \xi}$ be as in Proposition \ref{prop-fket-str-loc}.  In particular, the underlying adic spaces of $\xi$ \Pth{\resp $X(\xi)$} is a geometric point \Pth{\resp a strictly local adic space}.  Then we have
    \[
        \pi_1^\ket\bigl(X(\xi)\bigr) \cong \pi_1^\ket(\xi) \cong \Hom\bigl(\overline{M}^\gp, \widehat{\bZ}'(1)(l)\bigr).
    \]
    Since $\overline{M}$ is sharp and fs, we have $\overline{M}^\gp \cong \bZ^r$ for some $r$, and we obtain a noncanonical isomorphism $\pi_1^\ket(\xi) \cong \bigl(\widehat{\bZ}'(1)(l)\bigr)^r$.
\end{cor}

\begin{rk}\label{rem-fund-grp-comp}
    For any connected locally noetherian fs log adic space $X$ and any log geometric point $\xi$ of $X$, the natural inclusion from the category of finite \'etale covers to that of finite Kummer \'etale covers is fully faithful, and hence induces a canonical surjective homomorphism $\pi_1^\ket(X, \xi) \to \pi_1^\et(X, \xi)$ \Pth{see \cite[V, 6.9]{SGA:1}}.
\end{rk}

\begin{exam}\label{ex-fund-grp-log-pt}
    Let $(k, k^+)$ be an affinoid field, and let $s = (\Spa(k, k^+), \cM)$ be an fs log point as in Example \ref{ex-log-adic-sp-pt-fs}.  Let $\AC{s} = (\Spa(K, K^+), M)$ be a geometric point over $s$, where $K$ is the completion of a separable closure $k^\sep$ of $k$, and let $\widetilde{s}$ be a log geometric point over $\AC{s}$.  Then we have a canonical short exact sequence
    \[
        1 \to \pi_1^\ket(\AC{s}, \widetilde{s}) \to \pi_1^\ket(s, \widetilde{s}) \to \pi_1^\et(s, \AC{s}) \to 1,
    \]
    where $\pi_1^\ket(\AC{s}, \widetilde{s}) \cong \Hom\bigl(\overline{M}^\gp, \widehat{\bZ}'(1)(k^\sep)\bigr)$ \Pth{as we have seen in Corollary \ref{cor-fund-grp-log-pt}} and $\pi_1^\et(s, \AC{s}) \cong \Gal(k^\sep / k)$.  If $s$ is a split fs log point as in Example \ref{ex-log-adic-sp-pt-fs-split}, then any choice of a $\Gal(k^\sep / k)$-equivariant splitting of $M \to \overline{M} \cong M / K^\times$ also splits this exact sequence, inducing an isomorphism
    \[
        \pi_1^\ket(s, \widetilde{s}) \Mi \Hom\bigl(\overline{M}^\gp, \widehat{\bZ}'(1)(k^\sep)\bigr) \rtimes \Gal(k^\sep / k).
    \]
\end{exam}

\begin{exam}\label{ex-fund-grp-ket-0-partial}
    Let $(k, k^+)$ be an affinoid field.  Consider $0 = \Spa(k, k^+)$, the point of $\Spa(k\Talg{\bZ_{\geq 0}}, k^+\Talg{\bZ_{\geq 0}}) \cong \Spa(k\Talg{T}, k^+\Talg{T})$ defined by $T = 0$.  Let us denote by $0^\partial$ the log adic space with underlying adic space $0$ and with its log structure pulled back from $\Spa(k\Talg{\bZ_{\geq 0}}, k^+\Talg{\bZ_{\geq 0}})$, which is the split fs log point $(X, \cM_X) = (\Spa(k, k^+), \cO_{X_\et}^\times \oplus (\bZ_{\geq 0})_X)$ as in Example \ref{ex-log-adic-sp-pt-fs-split}.  Let $\AC{0}^\partial$ and $\widetilde{0}^\partial$ be defined over $0^\partial$ as in Example \ref{ex-fund-grp-log-pt} \Pth{with $s = 0^\partial$ there}.  Then
    \[
        \pi_1^\ket(0^\partial, \widetilde{0}^\partial) \cong \widehat{\bZ}'(1)(k^\sep) \rtimes \Gal(k^\sep / k).
    \]
    For each $n$ invertible in $k$, and each $r \geq 1$, we have a $\bZ / n$-local system $\bJ_{r, n}^\partial$ on $0^\partial_\ket$ defined by the representation of $\pi_1^\ket(0^\partial, \widetilde{0}^\partial)$ on $(\bZ / n)^r$ such that a topological generator of $\widehat{\bZ}'(1)(k^\sep)$ acts as the standard upper triangular principal unipotent matrix $J_r$ and $\Gal(k^\sep / k)$ acts diagonally on $(\bZ / n)^r$ and trivially on $\ker(J_r - 1)$. \Pth{The local system thus defined is independent of the choice of the generator of $\widehat{\bZ}'(1)(k^\sep)$ up to isomorphism.}  Moreover, for each $m \geq 1$ with $m$ invertible in $k$, we also have the $\bZ / n$-local system $\bK_{m, n}^\partial$ defined by the representation of $\pi_1^\ket(0^\partial, \widetilde{0}^\partial)$ induced from the trivial representation of $m \widehat{\bZ}'(1)(k^\sep) \rtimes \Gal(k^\sep / k)$ on $\bZ / n$.  \Pth{These local systems will be useful for defining quasi-unipotent nearby cycles in Section \ref{sec-nearby}.}
\end{exam}

\begin{exam}\label{ex-fund-grp-log-pt-mor}
    Let $s$, $\AC{s}$, and $\widetilde{s}$ be as in Example \ref{ex-fund-grp-log-pt}, and let $f: X \to s$ be any strict lft morphism of log adic spaces.  Let $\xi$ be a geometric point of $X$ above $\AC{s}$, and let $\widetilde{\xi}$ be a log geometric point above $\xi$ and $\widetilde{s}$.  Let $X(\xi)$ denote the strict localization of $X$ at $\xi$.  Then, by Proposition \ref{prop-fket-str-loc} and Corollary \ref{cor-fund-grp-log-pt}, we have $\pi_1^\ket(X(\xi), \widetilde{\xi}) \cong \pi_1^\ket(\AC{s}, \widetilde{s}) \cong \Hom\bigl(\overline{M}^\gp, \widehat{\bZ}'(1)(k^\sep)\bigr)$.
\end{exam}

\begin{lem}\label{lem-ket-to-et-stalk}
    Let $(X, \cM_X)$, $\xi = \Spa(l, l^+)$, $X(\xi)$, $x$, and $M := \cM_{X, \xi}$ be as in Proposition \ref{prop-fket-str-loc}.  Let $\varepsilon_\et: X_\ket \to X_\et$ be the natural projection of sites, as before.  Then, for each sheaf $\cF$ of finite abelian groups on $X_\ket$, we have
    \[
        \bigl(R^i\varepsilon_{\et, *}(\cF)\bigr)_\xi \cong H^i\bigl(\pi_1^\ket(\xi, \widetilde{\xi}), \cF_{\widetilde{\xi}}\bigr).
    \]
\end{lem}
\begin{proof}
    By definition, we have $\bigl(R^i\varepsilon_{\et, *}(\cF)\bigr)_\xi \cong \varinjlim H^i(U_\ket, \cF)$, where the direct limit is over the filtered category of \'etale neighborhoods $i_U: \xi \to U$ in $X$.  By Proposition \ref{prop-chart-stalk-chart-fs}, up to \'etale localization, we may assume that $X$ admits a chart modeled on $P := \overline{M}$.  Consider the morphism $i^{-1}: \varinjlim U_\ket \to x_\ket$, where the direct limit of sites $\varinjlim U_\ket$ is as in \cite[VI, 8.2.3]{SGA:4}, induced by the morphisms $i_U^{-1}: U_\ket \to x_\ket$.  Since each Kummer \'etale covering of $\xi$ can be further covered by some standard Kummer \'etale covers induced by $n$-th multiple maps $[n]: P \to P$, for some integers $n \geq 1$ invertible in $l$, and since coverings of the latter kind are in the essential image of $i^{-1}$, by Proposition \ref{prop-fket-log-pt}, we have an equivalence $\xi_\ket^\sim \cong \varprojlim U_\ket^\sim$ of the associated topoi.  Consequently, we also have $\varinjlim H^i(U_\ket, \cF) \cong H^i\bigl(\xi_\ket, \xi^{-1}(\cF)\bigr) \cong H^i\bigl(\pi_1^\ket(\xi, \widetilde{\xi}), \cF_{\widetilde{\xi}}\bigr)$, as desired.
\end{proof}

Let $(X, \cM_X)$ be a locally noetherian fs log adic space, and let $\cM_{X_\ket}$ be as in Proposition \ref{prop-M-ket-sheaf}.  For each positive integer $n$ invertible in $\cO_X$, let
\[
    \Grpmu_{n, \ket} := \ker(\cO_{X_\ket}^\times \Mapn{[n]} \cO_{X_\ket}^\times)
\]
and
\[
    \Grpmu_{n, \et} := \ker(\cO_{X_\et}^\times \Mapn{[n]} \cO_{X_\et}^\times).
\]
By Corollary \ref{cor-O-ket-et-an}, we have canonical isomorphisms
\[
    \varepsilon_\et^*(\Grpmu_{n, \et}) \Mi \Grpmu_{n, \ket}
\]
and
\[
    \Grpmu_{n, \et} \Mi \varepsilon_{\et, *}(\Grpmu_{n, \ket}).
\]
We shall omit the subscripts \Qtn{$\ket$} and \Qtn{$\et$}, and write simply $\Grpmu_n$, when there is no risk of confusion.  Consider the sequence
\[
    1 \to \Grpmu_n \to \cM_{X_\ket}^\gp \Mapn{[n]} \cM_{X_\ket}^\gp \to 1
\]
on $X_\ket$, which is exact by comparing stalks at log geometric points of $X$ \Pth{\Refcf{} Construction \ref{constr-log-geom-pt}}, whose pushforward under $\varepsilon_\et: X_\ket \to X_\et$ induces a long exact sequence
\[
    1 \to \Grpmu_n \to \cM_{X_\et}^\gp \Mapn{[n]} \cM_{X_\et}^\gp \to R^1\varepsilon_{\et, *}(\Grpmu_n),
\]
which is compatible with the Kummer exact sequence
\[
    1 \to \Grpmu_n \to \cO_{X_\et}^\times \Mapn{[n]} \cO_{X_\et}^\times \to 1
\]
and induces a canonical morphism
\begin{equation}\label{eq-Kummer-seq-conn}
    \overline{\cM}_X^\gp \big/ n \overline{\cM}_X^\gp \to R^1\varepsilon_{\et, *}(\Grpmu_n)
\end{equation}
on $X_\et$, which is nothing but the inverse of the isomorphism in Lemma \ref{lem-ket-to-et-stalk}, by comparing stalks at geometric points of $X$.  Therefore, we obtain the following:
\begin{lem}\label{lem-ket-to-et-Kummer}
    The above morphism \Refeq{\ref{eq-Kummer-seq-conn}} is an isomorphism and, for each $i$, the canonical morphism $\Ex^i\bigl(R^1\varepsilon_{\et, *}(\Grpmu_n)\bigr) \to R^i\varepsilon_{\et, *}(\Grpmu_n)$ is an isomorphism.
\end{lem}

\subsection{Localization and base change functors}\label{sec-ket-bc}

In this subsection, we study the behavior of sheaves on Kummer \'etale sites under certain direct image and inverse image functors.  \Pth{The readers are referred to \cite[IV]{SGA:4} for general notions concerning sites, topoi, and the functors and morphisms among them.}

For any morphism $f: Y \to X$ of locally noetherian fs log adic spaces, since pulling back by $f$ respects fiber products, we have a morphism of topoi
\[
    (f_\ket^{-1}, f_{\ket, *}): Y_\ket^\sim \to X_\ket^\sim.
\]
Concretely, we have the \emph{direct image} \Pth{or \emph{pushforward}} functor
\[
    f_{\ket, *}: \Sh(Y_\ket) \to \Sh(X_\ket): \; \cF \to \bigl(U \mapsto f_{\ket, *}(\cF)(U) := \cF(U \times_X Y)\bigr),
\]
and the \emph{inverse image} \Pth{or \emph{pullback}} functor
\[
    f_\ket^{-1}: \Sh(X_\ket) \to \Sh(Y_\ket)
\]
sending $\cG \in \Sh(X_\ket)$ to the sheafification of $V \mapsto \varinjlim_U \, \cG(U)$, where $U$ runs through the objects in $Y_\ket$ such that $V \to X$ factors through $f^{-1}(U) \to Y$.  It is formal that $f_{\ket, *}$ is the right adjoint of $f_\ket^{-1}$.  Moreover, $f_\ket^{-1}$ is exact, and $f_{\ket, *}$ is left exact.

For any \emph{Kummer \'etale} morphism $f: Y \to X$, the functor $f_{\ket, *}$ is also called the \emph{base change} functor, while the functor $f_\ket^{-1}$ is simply $f_\ket^{-1}(\cF)(U) := \cF(U)$, because any object $U$ of $Y_\ket$ gives an object in $X_\ket$ by composition with $f$.  Moreover, we have the \emph{localization} functor
\[
    f_{\ket, !}: \Sh(Y_\ket) \to \Sh(X_\ket).
\]
sending $\cF \in \Sh(Y_\ket)$ to the sheafification of the presheaf
\[
    f_{\ket, !}^p(\cF): U \mapsto \coprod_{h \in \Mor_X(U, Y)} \, \cF(U, h),
\]
where $\cF(U, h)$ means the value of $\cF$ on the object $U \Mapn{h} Y$ of $X_\ket$.  We shall also denote by $f_{\ket, !}: \Sh_\Ab(Y_\ket) \to \Sh_\Ab(X_\ket)$ the induced functor between the categories of abelian sheaves, in which case the above coproduct becomes a direct sum.  It is also formal that $f_{\ket, !}$ is left adjoint to $f_\ket^{-1}$, and that $f_{\ket, !}$ is right exact.

\begin{lem}\label{lem-ket-split}
    Let $f: V \to W$ be a finite Kummer \'etale morphism in $X_\ket$.  If $f$ has a section $g: W \to V$, then there exists a finite Kummer \'etale morphism $W' \to W$ and an isomorphism $h: V \Mi W \coprod W'$ such that the composition $h \circ g$ is the natural inclusion $W \Em W \coprod W'$.
\end{lem}
\begin{proof}
    We may assume that $W$ is connected.  Let $G := \pi_1^\ket(W)$ \Pth{see Remark \ref{rem-fund-grp}}.  Via the equivalence \Refeq{\ref{eq-fund-grp-fket}}, the finite Kummer \'etale cover $V \to W$ \Pth{\resp $W \Mi W$} corresponds to a finite set $S$ \Pth{\resp a singleton $S_0$} with a continuous $G$-action \Pth{\resp the trivial action}, and the section $g: W \to V$ corresponds to a $G$-equivariant map $g_*: S_0 \to S$.  This gives rise to a $G$-equivariant decomposition $S = g_*(S_0) \coprod S'$, and hence to the desired decomposition $h: V \Mi W \coprod W'$, by Corollary \ref{cor-fund-grp}.
\end{proof}

\begin{prop}\label{prop-ket-dir-im-ex}
    Given any finite Kummer \'etale morphism $f: Y \to X$ of locally noetherian fs log adic spaces, we have a natural isomorphism
    \[
        f_{\ket, !} \Mi f_{\ket, *}: \Sh_\Ab(Y_\ket) \to \Sh_\Ab(X_\ket).
    \]
    Consequently, both functors are exact.
\end{prop}
\begin{proof}
    Let $\cF$ be an abelian sheaf on $Y_\ket$.  For any $U \in X_\ket$, each morphism $h$ in $\Mor_X(U, Y)$ induces a section $U \to U \times_X Y$ of the natural projection $U \times_X Y \to U$.  By Lemma \ref{lem-ket-split}, we obtain a decomposition $U \times_X Y \cong U \coprod U'$ identifying $U \to U \times_X Y$ with $U \Em U \coprod U'$, which gives rise to a canonical map $\cF(U) \to \cF(U \times_X Y)$ because $\cF$ is a sheaf.  By combining such maps, we obtain a map of presheaves
    \[
        \bigl(f_{\ket, !}^p(\cF)\bigr)(U) = \oplus_{h \in \Mor_X(U, Y)} \, \cF(U, h) \to \bigl(f_{\ket, *}(\cF)\bigr)(U) = \cF(U \times_X Y ),
    \]
    which induces a canonical morphism $f_{\ket, !} \to f_{\ket, *}$ by sheafification.

    By the above construction, it remains to show that, \'etale locally on $U$, there exists a finite Kummer \'etale cover $V \to U$ such that $\Mor_X(V, Y)$ is a finite set and such that the sections $V \to V \times_X Y$ given by $h \in \Mor_X(V, Y)$ induces $\coprod_{h \in \Mor_X(V, Y)} V \Mi V \times_X Y$.  Note that this is true in the special case where $Y \to X$ is strictly finite \'etale, because $Y$ is \'etale locally on $X$ a finite disjoint union of copies of $X$.  In general, up to \'etale localization on $X$, we may assume that $X$ is affinoid and modeled on a sharp fs monoid $P$; and that $Y \to X$ factors as a composition $Y \to X_Q := X \times_{X\Talg{P}} X\Talg{Q} \to X$, where the first morphism is strictly finite \'etale, and where the second morphism is the standard Kummer \'etale cover induced by a Kummer homomorphism $u: P \to Q$ of sharp fs monoids such that the order of $G = Q^\gp / u^\gp(P^\gp)$ is invertible in $\cO_Y$.  By Proposition \ref{prop-ket-std}, $Y \times_X X_Q \cong Y \times_{X_Q} (X_Q \times_X X_Q) \cong Y \times_X X\Talg{G} \to X_Q$ is strictly finite \'etale.  Hence, as explained above, there exists a finite Kummer \'etale cover $V \to U \times_X X_Q$ such that $\coprod_{h \in \Mor_{X_Q}(V, Y \times_X X_Q)} \, V \cong V \times_{X_Q} (Y \times_X X_Q)$.  Since $\coprod_{h \in \Mor_X(V, Y)} \, V \cong \coprod_{h \in \Mor_{X_Q}(V, Y \times_X X_Q)} \, V \cong V \times_{X_Q} (Y \times_X X_Q) \cong V \times_X Y$, the composition of the finite Kummer \'etale covers $V \to U \times_X X_Q \to U$ gives the desired finite Kummer \'etale cover $V \to U$.
\end{proof}

\begin{lem}\label{lem-exc}
    Let $X$ be a locally noetherian fs log adic space.  Let $\imath: Z \to X$ be a strict closed immersion of log adic spaces, and $\jmath: W \to X$ an open immersion of log adic spaces, as in Definition \ref{def-imm}, such that $W = X - Z$.  For $? = \an$, $\et$, or $\ket$, let $(\imath_?^{-1}, \imath_{?, *})$ and $(\jmath_?^{-1}, \jmath_{?, *})$ denote the associated morphisms of topoi, and let $\jmath_{?, !}$ denote the left adjoint of $\jmath_?^{-1}$ \Pth{which is defined as explained above}.
    \begin{enumerate}
        \item\label{lem-exc-ex-seq} For each abelian sheaf $\cF$ on $X_?$, we have the excision short exact sequence $0 \to \jmath_{?, !} \, \jmath_{?}^{-1}(\cF) \to \cF \to \imath_{?, *} \, \imath_{?}^{-1}(\cF) \to 0$ in $\Sh_\Ab(X_?)$.

        \item\label{lem-exc-adj-cl} For each abelian sheaf $\cG$ on $Z_?$, the adjunction morphism $\imath_?^{-1} \, \imath_{?, *}(\cG) \to \cG$ is an isomorphism in $\Sh_\Ab(Z_?)$, and hence $\imath_{?, *}$ is exact and fully faithful.

        \item\label{lem-exc-adj-op} For each abelian sheaf $\cH$ on $W_?$, the adjunction morphism $\cH \Mi \jmath_?^{-1} \, \jmath_{?, !}(\cH)$ is an isomorphism in $\Sh_\Ab(W_?)$, and hence $\jmath_{!, *}$ is exact and fully faithful.
    \end{enumerate}
\end{lem}
\begin{proof}
    These follow easily from the definitions of the objects involved, by evaluating them at points \Pth{\resp geometric points, \resp log geometric points} when $? = \an$ \Pth{\resp $\et$, \resp $\ket$}.  \Pth{See \cite[\aProp 2.5.5]{Huber:1996-ERA} and Lemma \ref{lem-log-geom-pt}.}
\end{proof}

\begin{lem}\label{lem-cl-imm-ket-mor}
    Let $f: X \to \breve{X}$ and $g: Z \to \breve{Z}$ be morphisms of locally noetherian fs log adic spaces whose underlying morphisms of adic spaces are isomorphisms, and let $\imath: Z \to X$ and $\breve{\imath}: \breve{Z} \to \breve{X}$ be strict immersions compatible with $f$ and $g$.  Then, for any abelian sheaf $\cF$ on $X_\ket$, and for each $i \geq 0$, we have
    \begin{equation}\label{eq-lem-cl-imm-ket-mor}
        R^i g_{\ket, *} \, \imath_\ket^{-1}(\cF) \cong \breve{\imath}_\ket^{-1} \, R^i f_{\ket, *}(\cF).
    \end{equation}
    This applies, in particular, to the case where $\breve{X}$ and $\breve{Z}$ are the underlying adic spaces of $X$ and $Z$, respectively, equipped with their trivial log structures, in which case $\breve{X}_\ket \cong X_\et$ and $\breve{Z}_\ket \cong Z_\et$, and therefore $f_\ket: X_\ket \to \breve{X}_\ket$ and $g_\ket: Z_\ket \to \breve{Z}_\ket$ can be identified with the natural morphisms $X_\ket \to X_\et$ and $Z_\ket \to Z_\et$, respectively.
\end{lem}
\begin{proof}
    Up to compatibly replacing $X$ and $\breve{X}$ with open subspaces, we may assume that $\imath$ and $\breve{\imath}$ are compatible strict closed immersions.  By Lemma \ref{lem-exc}\Refenum{\ref{lem-exc-adj-cl}}, and by applying $\breve{\imath}_{\ket, *}$ to \Refeq{\ref{eq-lem-cl-imm-ket-mor}}, it suffices to show that we have
    \begin{equation}\label{eq-lem-cl-imm-ket-mor-pf}
        R^i f_{\ket, *} \, \imath_{\ket, *} \, \imath_{\ket}^{-1}(\cF) \cong \breve{\imath}_{\ket, *} \, \breve{\imath}_\ket^{-1} \, R^i f_{\ket, *}(\cF).
    \end{equation}
    Let $\jmath: W \to X$ and $\breve{\jmath}: \breve{W} \to \breve{X}$ denote the complementary open immersions.  By Lemma \ref{lem-exc}\Refenum{\ref{lem-exc-ex-seq}}, we have a long exact sequence
    \[
        \cdots \to R^i f_{\ket, *} \, \jmath_{\ket, !} \, \jmath_\ket^{-1}(\cF) \to R^i f_{\ket, *}(\cF) \to R^i f_{\ket, *} \, \imath_{\ket, *} \, \imath_\ket^{-1}(\cF) \to \cdots.
    \]
    By the definition of $\jmath_{\ket, !}$ and $\breve{\jmath}_{\ket, !}$, and by comparing stalks at log geometric points as in Lemma \ref{lem-log-geom-pt}, we obtain $R^i f_{\ket, *} \, \jmath_{\ket, !} \, \jmath_\ket^{-1}(\cF) \cong \breve{\jmath}_{\ket, !} \, \breve{\jmath}_\ket^{-1} \, R^i f_{\ket, *}(\cF)$, which induces the desired \Refeq{\ref{eq-lem-cl-imm-ket-mor-pf}}, by Lemma \ref{lem-exc}\Refenum{\ref{lem-exc-ex-seq}} again.
\end{proof}

\begin{lem}\label{lem-cl-imm-O+-p}
    Let $\imath: Z \to X$ be a strict closed immersion of locally noetherian fs log adic spaces over $\Spa(\bQ_p, \bZ_p)$.
    \begin{enumerate}
        \item For $? = \an$, $\et$, or $\ket$, the canonical morphism $\imath_{Z, ?}^{-1}(\cO_{X_?}^+ / p) \to \cO_{Z_?}^+ / p$ is an isomorphism.

        \item For any $\bF_p$-sheaf $\cF$ on $Z_\ket$, the canonical morphism
            \[
                \bigl(\imath_{Z, \ket, *}(\cF)\bigr) \otimes (\cO_{X_\ket}^+ / p) \to \imath_{Z, \ket, *}\bigl(\cF \otimes (\cO_{Z_\ket}^+ / p)\bigr)
            \]
            is an isomorphism.
    \end{enumerate}
\end{lem}
\begin{proof}
    The case where $? = \an$ or $\et$ is already in \cite[\aLem 3.14]{Scholze:2013-pss}.  As for $? = \ket$, the proof is similar, which we explain as follows.  At each log geometric point $\zeta = \bigl(\Spa(l, l^+), M\bigr) \to X$, the stalk of $\cO_{X_\ket}^+ / p$ at $\zeta$ is isomorphic to $l^+ / p$ by construction, because $\ker(\cO_{X_\ket, \zeta}^+ \to l^+)$, which is the same as $\ker(\cO_{X_\zeta, \zeta} \to l)$, is $p$-divisible \Pth{as $X$ is defined over $\Spa(\bQ_p, \bZ_p)$}.  The analogous statement for $\cO_{Z_\ket}^+ / p$ is true.  Thus, we can finish the proof by comparing stalks, by Lemma \ref{lem-log-geom-pt}.
\end{proof}

\begin{lem}\label{lem-ket-mor-O+-p}
    Let $g: Y \to Z$ be a morphism of locally noetherian fs log adic spaces over $\Spa(\bQ_p, \bZ_p)$ such that its underlying morphism of adic spaces is an isomorphism.  Suppose moreover that $g^\sharp_{\AC{y}}: \cM_{Z, g(\AC{y})} \to \cM_{Y, \AC{y}}$ is injective and splits, for each geometric point $\AC{y}$ of $Y$.  Then, for any $\bF_p$-sheaf $\cF$ on $Y_\ket$, the canonical morphism $R^i g_{\ket, *}(\cF) \otimes_{\bF_p} (\cO_{Z_\ket}^+ / p) \to R^i g_{\ket, *}\bigl(\cF \otimes_{\bF_p} (\cO_{Y_\ket}^+ / p)\bigr)$ is an isomorphism.
\end{lem}
\begin{proof}
    By Lemma \ref{lem-log-geom-pt}, it suffices to show that, for each log geometric point $\widetilde{z}: (\Spa(l, l^+), M) \to Z$ as in Construction \ref{constr-log-geom-pt}, the induced morphism
    \begin{equation}\label{eq-lem-ket-mor-O+-p-stalk}
        \bigl(R^i g_{\ket, *}(\bF_p)\bigr)_{\widetilde{z}} \otimes_{\bF_p} (\cO_{Z_\ket}^+ / p)_{\widetilde{z}} \to \bigl(R^i g_{\ket, *}(\cO_{Y_\ket}^+ / p)\bigr)_{\widetilde{z}}
    \end{equation}
    is an isomorphism.  Since $g$ induces an isomorphism of the underlying adic spaces, the underlying geometric point $\AC{z}$ of $\widetilde{z}$ uniquely lifts to a geometric point $\AC{y}$.  By assumption, $g^\sharp_{\AC{y}}: \cM_{Z, \AC{z}} \to \cM_{Y, \AC{y}}$ is injective and splits, in which case we have $\cM_{Y, \AC{y}} \cong \cM_{Z, \AC{z}} \oplus N$, for some fs monoid $N$.  Therefore, we can lift $\widetilde{z}$ to some \Pth{saturated by not necessarily fine} log point $\bigl(\Spa(l, l^+), M'\bigr) \to Y$ such that $M' \cong M \oplus N$.  By taking fiber products with $\Spa(l\Talg{\frac{1}{m} N}, l^+\Talg{\frac{1}{m} N})$ over $\Spa(l\Talg{N}, l^+\Talg{N})$, by taking reduced subspaces, and by taking the limit with respect to $m$ \Pth{\Refcf{} Construction \ref{constr-log-geom-pt}}, we can further lift this log point to a log geometric point $\widetilde{y}$ of $Y$ above $\AC{y}$.  Thus, by a limit argument similar to the one in the proof of Lemma \ref{lem-ket-to-et-stalk}, and by Proposition \ref{prop-fket-log-pt} and Lemma \ref{lem-cl-imm-O+-p}, we may identify \Refeq{\ref{eq-lem-ket-mor-O+-p-stalk}} with
    \[
        H^i(\Gamma, \cF) \otimes_{\bF_p} (l^+ / p) \to H^i(\Gamma, l^+ / p),
    \]
    where
    \[
        \Gamma := \ker\bigl(\pi_1^\ket(\AC{y}, \widetilde{y}) \to \pi_1^\ket(\AC{z}, \widetilde{z})\bigr) \cong \Hom\bigl(\overline{N}^\gp, \widehat{\bZ}(1)(l)\bigr).
    \]
    Since $H^i(\Gamma, \cF)$ is computed by some bounded complex of free $\bF_p$-modules, and since $H^i(\Gamma, l^+ / p)$ is computed by the tensor product of this complex with the flat $\bF_p$-module $l^+ / p$, we see that \Refeq{\ref{eq-lem-ket-mor-O+-p-stalk}} is an isomorphism, as desired.
\end{proof}

\begin{prop}\label{prop-ket-mor-O+-p}
    Let $f: Y \to X$ be a morphism of locally noetherian fs log adic spaces over $\Spa(\bQ_p, \bZ_p)$ such that its underlying morphism of adic spaces is a closed immersion.  Suppose moreover that $\bigl(f^*(\cM_X)\bigr)_{\AC{y}} \to \cM_{Y, \AC{y}}$ is injective and splits, for each geometric point $\AC{y}$ of $Y$.  Then, for any $\bF_p$-sheaf $\cF$ on $Y_\ket$, the canonical morphism $R^i f_{\ket, *}(\cF) \otimes_{\bF_p} (\cO_{X_\ket}^+ / p) \to R^i f_{\ket, *}\bigl(\cF \otimes_{\bF_p} (\cO_{Y_\ket}^+ / p)\bigr)$ is an isomorphism.
\end{prop}
\begin{proof}
    In this case, let $Z$ denote the underlying adic space of $Y$ equipped with the log structure pulled back from $X$.  Then $f: Y \to X$ factors as the composition of a morphism $g: Y \to Z$ as in Lemma \ref{lem-ket-mor-O+-p} and a strict closed immersion $\imath: Z \to X$ as in Lemma \ref{lem-cl-imm-O+-p}, and we can combine Lemmas \ref{lem-exc}, \ref{lem-cl-imm-O+-p}, and \ref{lem-ket-mor-O+-p}.
\end{proof}

\subsection{Purity of torsion local systems}\label{sec-purity}

We have the following \emph{purity} result for torsion Kummer \'etale local systems.
\begin{thm}\label{thm-purity}
    Let $X$, $D$, and $k$ be as in Example \ref{ex-log-adic-sp-ncd}.  Let $U := X - D$, and let $\jmath: U \to X$ denote the canonical open immersion.  Suppose moreover that $\chr(k) = 0$ and $k^+ = \cO_k$.  Let $\bL$ be a torsion local system on $U_\et$.  Then $\jmath_{\ket, *}(\bL)$ is a torsion local system on $X_\ket$, and $R^i\jmath_{\ket, *}(\bL) = 0$ for all $i > 0$.
\end{thm}
Let us start with some preparations.
\begin{lem}\label{lem-ket-to-et-const}
    In the setting of Theorem \ref{thm-purity}, consider the commutative diagram:
    \begin{equation}\label{eq-lem-ket-to-et-const-diag}
        \xymatrix{ {U_\ket} \ar[d]^-\cong \ar[r]^-{\jmath_\ket} & {X_\ket} \ar[d]^-{\varepsilon_\et} \\
        {U_\et} \ar[r]^-{\jmath_\et} & {X_\et.} }
    \end{equation}
    Then the canonical morphism
    \begin{equation}\label{eq-lem-ket-to-et-const}
        R\varepsilon_{\et, *}(\bZ / n) \to R\jmath_{\et, *}(\bZ / n)
    \end{equation}
    is an isomorphism; and $R^i\jmath_{\et, *}(\bZ / n) \cong \bigl(\Ex^i (\overline{\cM}_X^\gp / n \overline{\cM}_X^\gp)\bigr)(-i)$, for every $i \geq 0$.
\end{lem}
\begin{proof}
    By Lemma \ref{lem-ket-to-et-Kummer}, it suffices to show that the composition
    \[
        \bigl(\Ex^i(\overline{\cM}_X^\gp / n \overline{\cM}_X^\gp)\bigr)(-i) \to R^i\varepsilon_{\et, *}(\bZ / n) \to R^i\jmath_{\et, *}(\bZ / n)
    \]
    \Pth{induced by \Refeq{\ref{eq-Kummer-seq-conn}} and \Refeq{\ref{eq-lem-ket-to-et-const}}} is an isomorphism.  Since this assertion is \'etale local on $X$, we may assume that $D \subset X$ is the analytification of a normal crossings divisor on a smooth scheme over $k$, and further reduce the assertion to its classical analogue for schemes by \cite[\aProp 2.1.4 and \aThm 3.8.1]{Huber:1996-ERA}, which is known \Pth{see, for example, \cite[\aThm 7.2]{Illusie:2002-fknle}}.
\end{proof}

\begin{lem}\label{lem-ket-const}
    In the setting of Lemma \ref{lem-ket-to-et-const}, the canonical morphism
    \begin{equation}\label{eq-lem-ket-const}
        \bZ / n \to R\jmath_{\ket, *}(\bZ / n)
    \end{equation}
    is an isomorphism.
\end{lem}
\begin{proof}
    Let $C$ be the cone of \Refeq{\ref{eq-lem-ket-const}} \Pth{in the derived category}.  It suffices to show that $H^\bullet(W_\ket, C) = 0$, for each $W \to X$ that is the composition of an \'etale covering and a standard Kummer \'etale cover of $X$.  Note that the complement of $U \times_X W$ in $W$ is a normal crossings divisor, which induces the fs log structure of $W$ as in Example \ref{ex-log-adic-sp-ncd}.  Consider the diagram \Refeq{\ref{eq-lem-ket-to-et-const-diag}}, with $U \to X$ replaced with $U \times_X W \to W$.  Since the corresponding morphism \Refeq{\ref{eq-lem-ket-to-et-const}} for this new diagram is an isomorphism by Lemma \ref{lem-ket-to-et-const}, and since \Refeq{\ref{eq-lem-ket-to-et-const}} is obtained from \Refeq{\ref{eq-lem-ket-const}} by applying $\varepsilon_\et$ to both sides, we have $R\varepsilon_{\et, *}(C|_{W_\ket}) = 0$ on $W_\et$, and so $H^\bullet(W_\ket, C) = 0$, as desired.
\end{proof}

\begin{proof}[Proof of Theorem \ref{thm-purity}]
    Let $V \to U$ be a finite \'etale cover trivializing $\bL$.  By Proposition \ref{prop-Abhyankar}, it extends to a finite Kummer \'etale cover $f: Y \to X$, where $Y$ is a normal rigid analytic variety with its log structures defined by the preimage of $D$.  Moreover, if $Y' \to Y$ is Kummer \'etale, then $Y'$ is locally a normal rigid analytic variety, and any section of a finite torsion constant sheaf over the preimage of $U$ uniquely extends to a section of the constant sheaf with the same coefficients over $Y'$, by Example \ref{ex-log-adic-sp-toric}, Proposition \ref{prop-ket-std}, Corollary \ref{cor-ket-op}, and Proposition \ref{prop-Abhyankar}.  Thus, $\jmath_{\ket, *}(\bL)|_{Y_\ket}$ is constant, and $\jmath_{\ket, *}(\bL)$ is a torsion local system on $X_\ket$.  Given this $f: Y \to X$, up to \'etale localization on $X$, we have some $X^{\frac{1}{m}} \to X$ as in Lemma \ref{lem-Abhyankar-basic}.  Then the underlying adic space of $Z := Y \times_X X^{\frac{1}{m}}$ is a smooth rigid analytic variety, its fs log structure is defined by some normal crossings divisor as in Example \ref{ex-log-adic-sp-ncd}, and the induced morphism $Z \to X$ is Kummer \'etale, because these are true for $X^{\frac{1}{m}}$.  \Pth{Alternatively, we can construct $Z \to X$, as in the proof of Proposition \ref{prop-Abhyankar}, by using Lemma \ref{lem-Abhyankar-ess} and the last assertion of Lemma \ref{lem-Abhyankar-ref}.}  Thus, in order to show that $R^i\jmath_{\ket, *}(\bL) = 0$, for all $i > 0$, up to replacing $X$ with $Z$, we may assume that $\bL = \bZ / n$ is constant, in which case Lemma \ref{lem-ket-const} applies.
\end{proof}

\begin{cor}\label{cor-purity}
    Let $k$ and $\jmath: U \Em X$ be as in Theorem \ref{thm-purity}.  Let $\bL$ be an \'etale $\bF_p$-local system on $U_\et$.  Then $\overline{\bL} := \jmath_{\ket, *}(\bL)$ is a Kummer \'etale $\bF_p$-local system extending $\bL$.  Conversely, any \'etale $\bF_p$-local system $\overline{\bL}$ on $X_\ket$ is of this form.  In either case, $H^i(U_\et, \bL) \cong H^i(X_\ket, \overline{\bL})$, for all $i \geq 0$.
\end{cor}
\begin{proof}
    This follows from Theorem \ref{thm-purity} and the fact that, for any Kummer \'etale $\bF_p$-local system $\overline{\bL}$ on $X$, the canonical adjunction morphism $\overline{\bL} \to R \jmath_{\ket, *} \, \jmath_\ket^{-1}(\overline{\bL})$ is an isomorphism because it is a morphism between local systems whose restriction to the open dense $U$ is the identity morphism of $\bL$.
\end{proof}

\section{Pro-Kummer \'etale topology}\label{sec-proket}

\subsection{The pro-Kummer \'etale site}\label{sec-proket-site}

In this subsection, we define the pro-Kummer \'etale site on log adic spaces, a log analogue of Scholze's \emph{pro-\'etale site} in \cite{Scholze:2013-phtra}.

For any category $\cC$, by \cite[\aProp 3.2]{Scholze:2013-phtra}, the category $\pro\cC$ is equivalent to the category whose objects are functors $F: I \to \cC$ from small cofiltered index categories and whose morphisms are $\Mor(F, G) = \varprojlim_J \varinjlim_I \Mor\bigl(F(i), G(j)\bigr)$, for each $F: I \to \cC$ and $G: J \to \cC$.  We shall use this equivalent description in what follows.  For each $F: I \to \cC$ as above, we shall denote $F(i)$ by $F_i$, for each $i \in I$, and denote the corresponding object in $\pro\cC$ as $\varprojlim_{i \in I} F_i$.

Let $X$ be a locally noetherian fs log adic space, with the category $\pro{}X_\ket$ as above.  Then any object in \pro{}$X_\ket$ is of the form $U = \varprojlim_{i \in I} U_i$, where each $U_i \to X$ is Kummer \'etale, with underlying topological space $|U| := \varprojlim_i |U_i|$.

\begin{defn}\phantomsection\label{def-proket-mor}
    \begin{enumerate}
        \item\label{def-proket-mor-1}  We say that a morphism $U \to V$ in \pro$X_\ket$ is \emph{Kummer \'etale} \Pth{\resp \emph{finite Kummer \'etale}, \resp \emph{\'etale}, \resp \emph{finite \'etale}} if it is the pullback under some morphism $V \to V_0$ in \pro$X_\ket$ of some Kummer \'etale \Pth{\resp finite Kummer \'etale, \resp strictly \'etale, \resp strictly finite \'etale} morphism $U_0 \to V_0$ in $X_\ket$.

        \item\label{def-proket-mor-2}  We say that a morphism $U \to V$ in \pro$X_\ket$ is \emph{pro-Kummer \'etale} if it can be written as a cofiltered inverse limit $U = \varprojlim_{i \in I} U_i$ of objects $U_i \to V$ Kummer \'etale over $V$ such that $U_j \to U_i$ is finite Kummer \'etale and surjective for all sufficiently large $i$ \Pth{\ie, all $i \geq i_0$, for some $i_0 \in I$}.  Such a presentation $U = \varprojlim_i U_i \to V$ is called a \emph{pro-Kummer \'etale presentation}.

        \item\label{def-proket-mor-3}  We say that a morphism $U \to V$ as in \Refenum{\ref{def-proket-mor-2}} is \emph{pro-finite Kummer \'etale} if all $U_i \to V$ there are finite Kummer \'etale.
    \end{enumerate}
\end{defn}

\begin{defn}\label{def-proket-site}
    The \emph{pro-Kummer \'etale site} $X_\proket$ has as underlying category the full subcategory of $\pro{}X_\ket$ consisting of objects that are pro-Kummer \'etale over $X$, and each covering of an object $U \in X_\proket$ is given by a family of pro-Kummer \'etale morphisms $\{ f_i: U_i \to U \}_{i \in I}$ such that $|U| = \cup_{i \in I} f_i(|U_i|)$ and such that $f_i: U_i \to U$ can be written as an inverse limit $U_i = \varprojlim_{\mu < \lambda} U_\mu \to U$ satisfying the following conditions \Pth{\Refcf{} \cite{Scholze:2016-phtra-corr}}, for each $i \in I$:
    \begin{enumerate}
        \item\label{def-proket-site-1} Each $U_\mu \in X_\proket$, and $U = U_0$ is an initial object in the limit.

        \item\label{def-proket-site-2} The limit runs through the set of ordinals $\mu$ less than some ordinal $\lambda$.

        \item\label{def-proket-site-3} For each $\mu < \lambda$, the morphism $U_\mu \to U_{<\mu} := \varprojlim_{\mu' < \mu} U_{\mu'}$ is the pullback of a Kummer \'etale morphism in $X_\ket$, and is the pullback of a surjective finite Kummer \'etale morphism in $X_\ket$ for all sufficiently large $\mu$.
    \end{enumerate}
\end{defn}

\begin{rk}\label{rem-pro-et-Scholze-Weinstein}
    There is another version of pro-\'etale site introduced in \cite{Scholze/Weinstein:2020-BLG}.  But we will not try to introduce the corresponding version of pro-Kummer \'etale site in this paper.
\end{rk}

This definition is justified by the following analogue of \cite[\aLem 3.10]{Scholze:2013-phtra}:
\begin{lem}\phantomsection\label{lem-proket-site-check}
    \begin{enumerate}
        \item\label{lem-proket-site-check-1}  Let $U$, $V$, and $W$ be objets in $\pro{}X_\ket$.  Suppose that $U \to V$ is a Kummer \'etale \Pth{\resp finite Kummer \'etale, \resp \'etale, \resp finite \'etale, \resp pro-Kummer \'etale, \resp pro-finite Kummer \'etale} morphism and that $W \to V$ is any morphism.  Then the fiber product $U \times_V W$ exists in $\pro{}X_\ket$, and $U \times_V W \to W$ is Kummer \'etale \Pth{\resp finite Kummer \'etale, \resp \'etale, \resp finite \'etale, \resp pro-Kummer \'etale, \resp pro-finite Kummer \'etale}.  Moreover, the induced map $|U \times_V W| \to |U| \times_{|V|} |W|$ is surjective.

        \item\label{lem-proket-site-check-2}  A composition of two Kummer \'etale \Pth{\resp finite Kummer \'etale, \resp \'etale, \resp finite \'etale} morphisms in $\pro{}X_\ket$ is still Kummer \'etale \Pth{\resp finite Kummer \'etale, \resp \'etale, \resp finite \'etale}.

        \item\label{lem-proket-site-check-3}  Let $U$ be an object in $\pro{}X_\ket$, and let $W \subset |U|$ be a quasi-compact open subset.  Then there exists an object $V$ in $\pro{}X_\ket$ with an \'etale morphism $V \to U$ such that $|V| \to |U|$ induces a homeomorphism $|V| \Mi W$.  If, in addition, $U$ is an object in $X_\proket$, then there exists $V$ as above that, for any morphism $V' \to U$ in $X_\proket$ such that $|V'| \to |U|$ factors through $W$, the morphism $V' \to U$ also factors through $V$.

        \item\label{lem-proket-site-check-4}  Pro-Kummer \'etale morphisms in $\pro{}X_{ket}$ are open \Pth{\ie, they induce open maps between the underlying topological spaces}.

        \item\label{lem-proket-site-check-5}  Let $V$ be an object in $X_\proket$.  A surjective Kummer \'etale \Pth{\resp surjective finite Kummer \'etale} morphism $U \to V$ in $\pro{}X_\ket$ is the pullback under some morphism $V \to V_0$ in $\pro{}X_\ket$ of a surjective Kummer \'etale \Pth{\resp surjective finite Kummer \'etale} morphism $U_0 \to V_0$ in $X_\ket$.

        \item\label{lem-proket-site-check-6}  Let $W$ be an object in $X_\proket$, and let $U \to V \to W$ be pro-Kummer \'etale \Pth{\resp pro-finite Kummer \'etale} morphisms in $\pro{}X_\ket$.  Then $U$ and $V$ are also objects in $X_\proket$, and the composition $U \to W$ is pro-Kummer \'etale \Pth{\resp pro-finite Kummer \'etale}.

        \item\label{lem-proket-site-check-7}  Arbitrary finite inverse limits exist in $X_\proket$.

        \item\label{lem-proket-site-check-8}  Any base change of a covering in $X_\proket$ is also a covering.
    \end{enumerate}
\end{lem}
\begin{proof}
    The statements \Refenum{\ref{lem-proket-site-check-1}}--\Refenum{\ref{lem-proket-site-check-7}} follow from essentially the same arguments as in the proof of \cite[\aLem 3.10]{Scholze:2013-phtra}, with inputs from Propositions \ref{prop-lem-four-pt}, \ref{prop-adj-int-sat}, and \ref{prop-fiber-prod-log-adic}, Corollary \ref{cor-ket-op}, and Proposition \ref{prop-ket-compos-bc} here.

    As for the remaining statement \Refenum{\ref{lem-proket-site-check-8}}, suppose that $\{ U_i \to U \}_{i \in I}$ is a covering of $U \in X_\proket$, and that $V \to U$ is a morphism in $X_\proket$.  We need to show that $\{ U_i \times_U V \to V \}_{i \in I}$ is also a covering.  Firstly, if $U_i = \varprojlim_{\mu < \lambda} U_\mu \to U$ is an inverse limit satisfying the conditions in Definition \ref{def-proket-site}, then so is the pullback $U_i \times_U V \cong \varprojlim_{\mu < \lambda} (U_\mu \times_U V) \to V$.  As for the surjectivity, by working locally $U$, we are reduced to the case where $U$ is quasi-compact, in which case we may assume that $\{ U_i \to U \}_{i \in I}$ is a finite covering.  By taking the disjoint union of $U_i \to U$, we are further reduced the special case where $I = \{ i_0 \}$ is a singleton, in which case $|U_{i_0} \times_U V| \to |V|$ is surjective, by \Refenum{\ref{lem-proket-site-check-1}}.
\end{proof}

Let us also record the following analogue of \cite[\aProp 3.12]{Scholze:2013-phtra}.
\begin{prop}\label{prop-proket-site-qcqs}
    Let $X$ be a locally noetherian fs log adic space.
    \begin{enumerate}
        \item\label{prop-proket-site-qcqs-1}  Let $U = \varprojlim_i U_i \to X$ be a pro-Kummer \'etale presentation of $U \in X_\proket$ such that all $U_i$ are affinoid.  Then $U$ is quasi-compact and quasi-separated.

        \item\label{prop-proket-site-qcqs-2}  Objects $U$ as in \Refenum{\ref{prop-proket-site-qcqs-1}} generates $X_\proket$, and are stable under fiber products.

        \item\label{prop-proket-site-qcqs-3}  The topos $X_\proket^\sim$ is algebraic.

        \item\label{prop-proket-site-qcqs-4}  An object $U$ in $X_\proket$ is quasi-compact \Pth{\resp quasi-separated} if and only if $|U|$ is quasi-compact \Pth{\resp quasi-separated}.

        \item\label{prop-proket-site-qcqs-5}  Suppose that $U \to V$ is an inverse limit of finite Kummer \'etale surjective morphisms in $X_\proket$.  Then $U$ is quasi-compact \Pth{\resp quasi-separated} if and only if $V$ is.

        \item\label{prop-proket-site-qcqs-6}  A morphism $U \to V$ in $X_\proket$ is quasi-compact \Pth{\resp quasi-separated} if and only if $|U| \to |V|$ is quasi-compact \Pth{\resp quasi-separated}.

        \item\label{prop-proket-site-qcqs-7}  The site $X_\proket$ is quasi-separated \Pth{\resp coherent} if and only if $|X|$ is quasi-separated \Pth{\resp coherent}.
    \end{enumerate}
\end{prop}
\begin{proof}
    By Lemma \ref{lem-proket-site-check}\Refenum{\ref{lem-proket-site-check-4}}, pro-Kummer \'etale morphisms are open.  Hence, the same arguments as in the proof of \cite[\aProp 3.12]{Scholze:2013-phtra} also work here.
\end{proof}

Let $\upsilon: X_\proket \to X_\ket$ be the natural projection of sites.  We have induced functors $\upsilon^{-1}: \Sh(X_\ket) \to \Sh(X_\proket)$ and $\upsilon_*: \Sh(X_\proket) \to \Sh(X_\ket)$.

\begin{prop}\label{prop-proket-vs-ket}
    Let $\cF$ be any abelian sheaf on $X_\ket$, and let $U = \varprojlim_i U_i$ be any qcqs object in $X_\proket$.  Then $H^j\bigl(U_\proket, \upsilon^{-1}(\cF)\bigr) = \varinjlim_i H^j(U_{i, \ket}, \cF)$, for all $j \geq 0$.
\end{prop}
\begin{proof}
    This follows from essentially the same argument as in the proof of \cite[\aLem 3.16]{Scholze:2013-phtra}, by using Proposition \ref{prop-proket-site-qcqs} \Refenum{\ref{prop-proket-site-qcqs-1}} and \Refenum{\ref{prop-proket-site-qcqs-2}} here.
\end{proof}

\begin{prop}\label{prop-proket-vs-ket-adj}
    For any abelian sheaf $\cF$ on $X_\ket$, the canonical morphism $\cF \to R\upsilon_* \upsilon^{-1}(\cF)$ is an isomorphism.
\end{prop}
\begin{proof}
    For each $j \geq 0$, the sheaf $R^j \upsilon_* \upsilon^{-1}(\cF)$ on $X_\ket$ is associated with the presheaf $U \mapsto H^j\bigl(U_\proket, \upsilon^{-1}(\cF)\bigr)$.  If $j = 0$, then $\cF \Mi \upsilon_* \upsilon^{-1}(\cF)$, because $H^j(U_\ket, \cF) \Mi H^i\bigl(U_\proket, \upsilon^{-1}(\cF)\bigr)$, for all qcqs objects $U$ in $X_\ket$, by Proposition \ref{prop-proket-vs-ket}.  If $j > 0$, essentially by definition, the cohomology $H^j(U, \cF)$ vanishes locally in the Kummer \'etale topology, and hence the associated sheaf $R^j \upsilon_* \upsilon^{-1}(\cF)$ is zero, as desired.
\end{proof}

\begin{cor}\label{cor-proket-vs-ket}
    The functor $\upsilon^{-1}: \Sh_\Ab(X_\ket) \to \Sh_\Ab(X_\proket)$ is fully faithful.
\end{cor}

For technical purposes, let us also define the pro-finite Kummer \'etale site.
\begin{defn}\label{def-profket}
    The \emph{pro-finite Kummer \'etale site} $X_\profket$ has as underlying category the category $\pro{}X_\fket$, and each covering of $U \in X_\profket$ is given by a family of pro-finite Kummer \'etale morphisms $\{ f_i: U_i \to U \}_{i \in I}$ such that $|U| = \cup_{i \in I} \, f_i(|U_i|)$ and such that each $f_i: U_i \to U$ can be written as an inverse limit $U_i = \varprojlim_{\mu < \lambda} U_\mu \to U$ satisfying the following conditions:
    \begin{enumerate}
        \item Each $U_\mu \in X_\profket$, and $U = U_0$ is an initial object in the limit.

        \item The limit runs through the set of ordinals $\mu$ less than some ordinal $\lambda$.

        \item For each $\mu < \lambda$, the morphism $U_\mu \to U_{<\mu} := \varprojlim_{\mu' < \mu} U_{\mu'}$ is the pullback of a finite Kummer \'etale morphism in $X_\fket$, and is the pullback of a surjective finite Kummer \'etale morphism for all sufficiently large $\mu$.
    \end{enumerate}
\end{defn}

\begin{defn}\label{def-G-PFSets}
    For each profinite group $G$, the site $G\Utext{-}\underline{\PFSets}$ has as underlying category the category of profinite sets with continuous actions of $G$, and each covering of $S \in G\Utext{-}\underline{\PFSets}$ is given by a family of continuous $G$-equivariant maps $\{ f_i: S_i \to S \}_{i \in I}$ such that $S = \cup_{i \in I} f_i(S_i)$ and such that each $f_i: S_i \to S$ can be written as an inverse limit $S_i = \varprojlim_{\mu < \lambda} S_\mu \to S$ satisfying the following conditions:
    \begin{enumerate}
        \item Each $S_\mu \in G\Utext{-}\underline{\PFSets}$, and $S = S_0$ is an initial object in the limit.

        \item The limit runs through the set of ordinals $\mu$ less than some ordinal $\lambda$.

        \item For each $\mu < \lambda$, the map $S_\mu \to S_{<\mu} := \varprojlim_{\mu' < \mu} S_{\mu'}$ is the pullback of a surjective map of finite sets.
    \end{enumerate}
\end{defn}

\begin{rk}\label{rem-G-PFSets}
    Since a profinite set with a continuous action of a profinite group $G$ is equivalent to an inverse limit of finite sets with continuous $G$-actions, we have a canonical equivalence of categories $G\Utext{-}\underline{\PFSets} \cong \pro(G\Utext{-}\underline{\FSets})$.
\end{rk}

\begin{prop}\label{prop-fund-grp-profket}
    Let $X$ be a connected locally noetherian fs log adic space, and let $\zeta$ be a log geometric point of $X$.  Then there is an equivalence of categories
    \[
        X_\profket \cong \pi_1^\ket(X, \zeta)\Utext{-}\underline{\PFSets}
    \]
    sending $U = \varprojlim_i U_i \to X$ to $S(U) := \varprojlim_i \Mor_X(\zeta, U_i)$.
\end{prop}
\begin{proof}
    By Corollary \ref{cor-fund-grp}, $X_\fket \cong \pi_1^\ket(X, \zeta)\Utext{-}\underline{\FSets}$, and hence the composition of $X_\profket = \pro{}X_\fket \cong \pro(\pi_1^\ket(X, \zeta)\Utext{-}\underline{\FSets}) \cong \pi_1^\ket(X, \zeta)\Utext{-}\underline{\PFSets}$ sends $U$ to $S(U)$, and gives the desired equivalence of categories.  By comparing definitions, the equivalence thus obtained also matches the coverings.
\end{proof}

\subsection{Localization and base change functors}\label{sec-proket-bc}

For any morphism $f: Y \to X$ of locally noetherian fs log adic spaces, by the same explanations as in Section \ref{sec-ket-bc}, we have a morphism of topoi
\[
    (f_\proket^{-1}, f_{\proket, *}): Y_\proket^\sim \to X_\proket^\sim.
\]

\begin{prop}\label{prop-dir-im-ket-vs-proket}
    Let $f: Y \to X$ be a qcqs morphism of locally noetherian fs log adic spaces.  Consider the natural functors $\upsilon_X^{-1}: \Sh(X_\ket) \to \Sh(X_\proket)$ and $\upsilon_Y^{-1}: \Sh(Y_\ket) \to \Sh(Y_\proket)$.  Then, for any abelian sheaf $\cF$ on $Y_\ket$, we have a natural isomorphism $\upsilon_X^{-1} \, Rf_{\ket, *}(\cF) \Mi Rf_{\proket, *} \, \upsilon_Y^{-1}(\cF)$.
\end{prop}
\begin{proof}
    This is because, for each $i \geq 0$, the $i$-th cohomology of both sides of the morphism $\upsilon_X^{-1} \, Rf_{\ket, *}(\cF) \to Rf_{\proket, *} \, \upsilon_Y^{-1}(\cF)$ \Pth{canonically defined by adjunction} can be identified with the sheafification of the presheaf sending a qcqs object $U = \varprojlim_j U_j$ in $X_\proket$ to $\varinjlim_j \, H^i(U_j \times_X Y, \cF)$.
\end{proof}

When $Y \in X_\proket$, let ${X_\proket}_{/Y}$ denote the localized site.  Then we have the following natural functors:
\begin{enumerate}
    \item The \emph{inverse image} \Pth{or \emph{pullback}} functor
        \[
        \begin{split}
            f_\proket^{-1}: \Sh(X_\proket) & \to \Sh({X_\proket}_{/Y}): \\
            \cF & \mapsto \bigl(U \mapsto f_\proket^{-1}(\cF)(U) := \cF(U)\bigr).
        \end{split}
        \]

    \item The \emph{base change} functor
        \[
        \begin{split}
            f_{\proket, *}: \Sh({X_\proket}_{/Y}) & \to \Sh(X_\proket): \\
            \cF & \mapsto \bigl(U \mapsto f_{\proket, *}(\cF)(U) := \cF(U \times_X Y)\bigr).
        \end{split}
        \]

    \item The \emph{localization} functor
        \[
            f_{\proket, !}: \Sh({X_\proket}_{/Y}) \to \Sh(X_\proket)
        \]
        sending $\cF \in \Sh({X_\proket}_{/Y})$ to the sheafification of the presheaf
        \[
            f_{\proket, !}^p(\cF): U \mapsto \coprod_{h: U \to Y} \cF(U, h),
        \]
        where the coproduct is over all pro-Kummer \'etale morphisms $h: U \to Y$ over $X$.  We also denote by $f_{\proket, !}: \Sh_\Ab({X_\proket}_{/Y}) \to \Sh_\Ab(X_\proket)$ the induced functor between the categories of abelian sheaves, in which case the above coproduct becomes a direct sum.
\end{enumerate}
It is formal that $f_{\proket, !}$ is left adjoint to $f_\proket^{-1}$, and hence $f_{\proket, !}$ is right exact.

\begin{rk}\label{rem-ket-proket-loc}
    If $Y \to X$ is Kummer \'etale, then naturally ${X_\proket}_{/Y} \cong Y_\proket$.
\end{rk}

\begin{lem}\label{lem-ket-split-proket}
    Let $f: V \to W$ be a finite Kummer \'etale morphism between objects in $X_\proket$.  If $f$ has a section $g: W \to V$, then there exists a finite Kummer \'etale morphism $W' \to W$ between objects in $X_\proket$, and an isomorphism $h: V \Mi W \coprod W'$ such that the composition $h \circ g$ is the natural inclusion $W \Em W \coprod W'$.
\end{lem}
\begin{proof}
    By Lemma \ref{lem-proket-site-check}\Refenum{\ref{lem-proket-site-check-5}}, we may assume that $f: V \to W$ is the pullback of some finite Kummer \'etale morphism $V_0 \to W_0$ in $X_\ket$.  Let $W = \varprojlim_i W_i$ be a pro-Kummer \'etale presentation.  Without loss of generality, we may assume that all transition morphisms $W_j \to W_i$ are finite Kummer \'etale, and that $W \to W_0$ factors through $W_i \to W_0$ for all $i$, so that $V \cong \varprojlim_i (W_i \times_{W_0} V_0)$.  Then we may replace $W_0$ with some $W_i$ and assume that $W \to W_0$ and hence $V \to W_0$ are pro-finite Kummer \'etale.  We may also assume that $W_0$ is connected.  Let $G := \pi_1^\ket(W_0)$ \Pth{see Remark \ref{rem-fund-grp}}.  Via the equivalence in Proposition \ref{prop-fund-grp-profket}, $V \to W_0$ and $W \to W_0$ corresponds to profinite sets $S$ and $S_0$ with continuous $G$-actions, and the morphism $f: V \to W$ and the splitting $g: W \to V$ gives rise to a $G$-equivariant decomposition $S \Mi S_0 \coprod S'$ and hence to the desired decomposition $h: V \Mi W \coprod W'$ \Pth{\Refcf{} the proof of Lemma \ref{lem-ket-split}}.
\end{proof}

\begin{prop}\label{prop-proket-dir-im-ex}
    Given any finite Kummer \'etale morphism $f: Y \to X$ of locally noetherian fs log adic spaces, we have a natural isomorphism
    \[
        f_{\proket, !} \Mi f_{\proket, *}: \Sh_\Ab(Y_\proket) \to \Sh_\Ab(X_\proket).
    \]
    Consequently, both functors are exact.
\end{prop}
\begin{proof}
    For each $U \in X_\proket$, any pro-Kummer \'etale morphism $h: U \to Y$ over $X$ induces a splitting $U \to Y \times_X U$ of the finite Kummer \'etale morphism $Y \times_X U \to U$, and hence we have a decomposition $Y \times_X U \cong U \coprod U'$, by Lemma \ref{lem-ket-split-proket}.  Then we can finish the proof by the same arguments as in the proof of Proposition \ref{prop-ket-dir-im-ex}.
\end{proof}

\subsection{Log affinoid perfectoid objects}\label{sec-log-aff-perf}

Recall that affinoid perfectoid objects form a basis for the pro-\'etale topology of any locally noetherian adic space over $\Spa(\bQ_p, \bZ_p)$ \Pth{see \cite[\aDef 4.3 and \aProp 4.8]{Scholze:2013-phtra} and \cite[\aLem 15.3]{Scholze:2017-ecd}}.  We would like to establish a suitable \emph{log analogue} of this fact.

\begin{defn}\label{def-log-aff-perf}
    Let $X$ be an analytic locally noetherian fs log adic space over $\Spa(\bZ_p, \bZ_p)$.  An object $U$ in $X_\proket$ is called \emph{log affinoid perfectoid} if it admits a pro-Kummer \'etale presentation
    \[
        U = \varprojlim_{i \in I} U_i = \varprojlim_{i \in I} (\Spa(R_i, R_i^+), \cM_i, \alpha_i) \to X
    \]
    satisfying the following conditions:
    \begin{enumerate}
        \item\label{def-log-aff-perf-1}  There is an initial object $0 \in I$.

        \item\label{def-log-aff-perf-2}  Each $U_i$ admits a global sharp fs chart $P_i$ such that each transition morphism $U_j \to U_i$ is modeled on a Kummer chart $P_i \to P_j$.

        \item\label{def-log-aff-perf-3}  The affinoid algebra $(R, R^+) := \bigl(\varinjlim_{i \in I} \, (R_i^u, R_i^{u+})\bigr)^\wedge$, where each $(R_i^u, R_i^{u+})$ is the \emph{uniformization} of $(R_i, R_i^+)$ as in \cite[\aDef 2.8.13]{Kedlaya/Liu:2015-RPH} \Pth{\ie, $R_i^u$ is the completion of $R_i$ for the spectral seminorm and $R_i^{u+}$ is the completion of the image of $R_i^+$ in $R_i^u$} and where the completion is as in \cite[\aDef 2.6.1]{Kedlaya/Liu:2015-RPH}, is a \emph{perfectoid} affinoid algebra.

        \item\label{def-log-aff-perf-4}  The monoid $P := \varinjlim_{i \in I} P_i$ is \emph{$n$-divisible}, for all $n \geq 1$.
    \end{enumerate}
    In this situation, we say that $U = \varprojlim_{i\in I} U_i$ is a perfectoid presentation of $U$.
\end{defn}

The following remark provides an equivalent form of Definition \ref{def-log-aff-perf}\Refenum{\ref{def-log-aff-perf-3}} when $X$ is defined over $\Spa(\bQ_p, \bZ_p)$, and explains the compatibility of the definitions of affinoid perfectoid objects in \cite[\aDef 4.3]{Scholze:2013-phtra} and \cite[\aDef 9.2.4]{Kedlaya/Liu:2015-RPH}.
\begin{rk}\label{rem-unif-compl-lim}
    In Definition \ref{def-log-aff-perf}, suppose that $X$ is defined over $\Spa(\bQ_p, \bZ_p)$.  In this case, $p$ is invertible in $\ker(R_i^+ \to R_i^{u+})$.  It follows that $R_i^+ / p^n \cong R_i^{u+} / p^n$, for each $n \geq 1$.  Then the completion $(R, R^+) = \bigl(\varinjlim_{i \in I} \, (R_i^u, R_i^{u+})\bigr)^\wedge$ is simply the $p$-adic completion of $\varinjlim_{i \in I} \, (R_i, R_i^+)$.
\end{rk}

\begin{rk}\label{rem-def-log-aff-perf}
    By Proposition \ref{prop-proket-site-qcqs}, a log affinoid perfectoid object $U$ as in Definition \ref{def-log-aff-perf} is qcqs.  By abuse of language, we shall sometimes say that $U$ is \emph{modeled on $P$}.  Since the transition morphisms $P_i \to P_j$ are Kummer and hence injective \Pth{by Definition \ref{def-Kummer}}, $P$ is sharp and saturated because each $P_i$ is \Pth{by assumption}.  Therefore, the condition \Refenum{\ref{def-log-aff-perf-4}} in Definition \ref{def-log-aff-perf} is equivalent to the condition that $P$ is \emph{uniquely $n$-divisible}, for all $n \geq 1$.  \Pth{The condition \Refenum{\ref{def-log-aff-perf-4}} in Definition \ref{def-log-aff-perf} will be useful in the proof of Lemma \ref{lem-log-aff-perf-ket} below.}
\end{rk}

\begin{lem}\label{lem-log-aff-perf-toric}
    Let $P$ be a sharp fs monoid.  Suppose that $X$ is an analytic locally noetherian adic space over $\Spa(\bZ_p, \bZ_p)$ equipped with the trivial log structure as in Example \ref{ex-log-adic-sp-triv}, and that $Y = X\Talg{P}$ is as in Example \ref{ex-log-adic-sp-monoid}.  Suppose that $\varprojlim_{i \in I} U_i$ is an affinoid perfectoid object in $X_\proket$, which exists by \cite[\aLem 15.3]{Scholze:2017-ecd}, where all $U_i$ are equipped with the trivial log structures as well.  Let us equip $I \times \bZ_{\geq 1}$ with the partial ordering such that $(i, m) \geq (j, n)$ exactly when $i \geq j$ and $n | m$.  Then $\varprojlim_{(i, n) \in I \times \bZ_{\geq 1}} \, U_i\Talg{\frac{1}{n} P}$ is a log affinoid perfectoid object in $Y_\proket$, which gives a pro-Kummer \'etale \Pth{\resp pro-finite Kummer \'etale} cover of $Y$ when $\varprojlim_{i \in I} U_i$ is a pro-\'etale \Pth{\resp pro-finite \'etale} cover of $X$.
\end{lem}
\begin{proof}
    This follows from Lemmas \ref{lem-monoid-alg-perf} and \ref{lem-proket-site-check}, and Definition \ref{def-log-aff-perf}.
\end{proof}

\begin{rk}\label{rem-def-log-aff-perf-sp}
    Let $U \in X_\proket$ be a log affinoid perfectoid object as in Definition \ref{def-log-aff-perf}.  Then $\widehat{U} = \Spa(R, R^+)$ is an affinoid perfectoid space, called the \emph{associated affinoid perfectoid space}, or simply the \emph{associated perfectoid space}.  In this case, we write $\widehat{U} \sim \varprojlim_i U_i$.  The assignment $U \mapsto \widehat{U}$ defines a functor from the category of log affinoid perfectoid objects to the category of affinoid perfectoid spaces.  We emphasize that $\widehat{U}$ does not live in $X_\proket$ in general.  Thanks to the following Lemma \ref{lem-log-aff-perf-sp}, we can identify the underlying topological spaces of $\widehat{U}$ and $\varprojlim_i U_i$.
\end{rk}

\begin{lem}\label{lem-log-aff-perf-sp}
    Let $U = \varprojlim_i U_i \in X_\proket$ be a log affinoid perfectoid object, and let $\widehat{U}$ be the associated affinoid perfectoid space, as in Remark \ref{rem-def-log-aff-perf-sp}.  Then the natural map of topological spaces $|\widehat{U}| \to \varprojlim_i |U_i|$ is a homeomorphism.
\end{lem}
\begin{proof}
    The map is bijective because a continuous valuation on $R$ is equivalent to a compatible system of continuous valuations on $R_i$'s.  By \cite[\aLem 2.6.5]{Kedlaya/Liu:2015-RPH}, each rational subset of $\widehat{U}$ comes from the pullback of a rational subset of some $U_i$, and hence the topologies also agree, as desired.
\end{proof}

\begin{lem}\label{lem-log-aff-perf-cl-imm}
    Let $\imath: Z \to X$ be a strict closed immersion \Pth{see Definition \ref{def-imm}} of analytic locally noetherian fs log adic spaces over $\Spa(\bZ_p, \bZ_p)$.  Then, for each log affinoid perfectoid object $U = \varprojlim_{i \in I} U_i$ of $X_\proket$, the pullback $V := U \times_X Z := \varprojlim_{i \in I} (U_i \times_X Z)$ is a log affinoid perfectoid object of $Z_\proket$.  Moreover, the natural morphism $\widehat{V} \to \widehat{U}$ is a closed immersion of adic spaces.
\end{lem}
\begin{proof}
    By definition and by Proposition \ref{prop-ket-compos-bc}, the conditions \Refenum{\ref{def-log-aff-perf-1}}, \Refenum{\ref{def-log-aff-perf-2}}, and \Refenum{\ref{def-log-aff-perf-4}} in Definition \ref{def-log-aff-perf} are satisfied.  It remains to verify the condition \Refenum{\ref{def-log-aff-perf-3}}.  Let $(R, R^+)$ be the completion of $\varinjlim_{i \in I} \, (R_i^u, R_i^{u+})$, which is perfectoid by assumption.  For each $i \in I$, write $\Spa(R_i, R_i^+) \times_X Z = \Spa(S_i, S_i^+)$.  Then the induced homomorphism $R_i^u \to S_i^u$ is surjective, because $\imath: Z \to X$ is strict and hence the underlying adic space of the fiber product coincides with the fiber product of the underlying adic spaces.  Since $R_i^u$ and $S_i^u$ are uniform, the quotient norm on $S_i^u$ induced by the one on $R_i^u$ is just the \Pth{spectral} norm on $S_i^u$.  Let $(S, S^+)$ be the completion of $\varinjlim_{i \in I} \, (S_i^u, S_i^{u+})$.  Then $S$ is uniform and admits a surjective bounded homomorphism $R \to S$.  In this case, $S$ is perfectoid, by \cite[\aThm 3.6.17(b)]{Kedlaya/Liu:2015-RPH}, and the natural morphism $\Spa(S, S^+) \to \Spa(R, R^+)$ is a closed immersion, as desired.
\end{proof}

\begin{lem}\label{lem-log-aff-perf-ket}
    Let $U = \varprojlim_{i \in I} U_i \in X_\proket$ be a log affinoid perfectoid object, with associated perfectoid space $\widehat{U}$, as in Remark \ref{rem-def-log-aff-perf-sp}.  Suppose that $V \to U$ is a Kummer \'etale \Pth{\resp finite Kummer \'etale} morphism in $X_\proket$ that is the pullback of some Kummer \'etale \Pth{\resp finite Kummer \'etale} morphism $V_0 \to U_0$ between affinoid log adic spaces in $X_\ket$.  Then $V \to U$ is \'etale \Pth{\resp finite \'etale}, and $V$ is log affinoid perfectoid.  The induced morphism $\widehat{V} \to \widehat{U}$ is \'etale \Pth{\resp finite \'etale}.  The construction $V \mapsto \widehat{V}$, for each log affinoid perfectoid object $V$ of ${X_\proket}_{/U}$, induces an equivalence of topoi $\widehat{U}^\sim_\proet \cong {X^\sim_\proket}_{/U}$.
\end{lem}
\begin{proof}
    We may assume that $0 \in I$ and that $0$ is an initial object \Pth{up to replacing $I$ with a cofinal subcategory}.  Let $U_0^{\frac{1}{m}}$ be as in Lemma \ref{lem-Abhyankar-basic}, for some $m \geq 1$ such that $V_0 \times_{U_0} U_0^{\frac{1}{m}} \to U_0^{\frac{1}{m}}$ is strictly \'etale \Pth{\resp strictly finite \'etale}.  Since $P = \varinjlim_{i \in I} P_i$ is $m$-divisible, there is some $i \in I$ such that $P_0 \to P_i$ factors as $P_0 \to \frac{1}{m} P_0 \to P_i$.  Then $V_0 \times_{U_0} U_i \to U_i$ is strictly \'etale \Pth{\resp strictly finite \'etale}.  We may replace $I$ with the cofinal full subcategory of objects that receive morphisms from $i$.  Then $V := (V_0 \times_{U_0} U_i) \times_{U_i} U \to U$ is strictly \'etale \Pth{\resp strictly finite \'etale}, and hence so is $\widehat{V} \to \widehat{U}$.  This shows that we have a well-defined morphism of sites $\widehat{U}_\proet \to {X_\proket}_{/U}$.  This induces an equivalence of topoi, because every \'etale morphism $W \to \widehat{U}$ that is a composition of rational localizations and finite \'etale morphisms arises in the above way, by \cite[\aLem 2.6.5 and \aProp 2.6.8]{Kedlaya/Liu:2015-RPH}.
\end{proof}

\begin{cor}\label{cor-log-aff-perf-loc-fket}
    Let $U = \varprojlim_{i \in I} U_i$ be an object in $X_\proket$ as in Definition \ref{def-log-aff-perf} such that $U_i \to X$ is a composition of rational localizations and finite Kummer \'etale morphisms, for all sufficiently large $i$.  Then $U \times_X V$ is a log affinoid perfectoid of $X_\proket$, for each log affinoid perfectoid object $V$ of $X_\proket$.
\end{cor}
\begin{proof}
    By Lemma \ref{lem-log-aff-perf-ket}, $U_i \times_X V$ is log affinoid perfectoid, for all sufficiently large $i$.  Hence, $U \times_X V  \cong \varprojlim_{i \in I} \, (U_i \times_X V)$ is also log affinoid perfectoid, because the $p$-adic completion of a direct limit of perfectoid affinoid algebras is again perfectoid, and any direct limit of divisible monoids is still divisible.
\end{proof}

\begin{lem}\label{lem-log-add-perf-mor}
    Let $U$ and $V$ be log affinoid perfectoid objects as in Definition \ref{def-log-aff-perf}, with a morphism $V \to U$, in $X_\proket$.  Suppose that $U = \varprojlim_{i \in I} U_i$ is a pro-Kummer \'etale presentation with $U_i$ and $U$ modeled on $P_i$ and $P = \varinjlim_{i \in I} P_i$, respectively.  Then $V$ admits a pro-Kummer \'etale presentation $V = \varprojlim_{j \in J} V_j$ with each $V_j$ modeled on some $P_i$, so that $V$ is also modeled on $P$.
\end{lem}
\begin{proof}
    Let $V = \varprojlim_{h \in H} V_h$ be a pro-Kummer \'etale presentation.  For each $i$, the morphism $V \to U$ in $X_\proket$ factors through some morphism $V_h \to U_i$ in $X_\ket$, for all sufficiently large $h$.  For each such $(i, h)$, by the argument in the proof of Lemma \ref{lem-log-aff-perf-ket}, there is some $t_{i, h} \in I$ such that $V_h \times_{U_i} U_t \to U_t$ is \'etale for all $t \geq t_{i, h}$, in which case $V_h \times_{U_i} U_t$ is modeled on $P_t$.  Hence, we obtain the desired presentation $V = \varprojlim_{j \in J} V_j$ by considering the index category $J$ formed by $j = (i, h, t)$ such that $t \geq t_{i, h}$, with the partial ordering such that $(i, h, t) \geq (i', h', t')$ in $J$ exactly when $i \geq i'$, $h \geq h'$, and $t \geq t'$; and by taking $V_j := V_h \times_{U_i} U_t$ for each $j = (i, h, t) \in J$.
\end{proof}

\begin{prop}\label{prop-log-aff-perf-fiber-prod}
    Suppose that $X$ is an analytic locally noetherian fs log adic space over $\Spa(\bZ_p, \bZ_p)$.  Then the subcategory of log affinoid perfectoid objects in $X_\proket$ is stable under fiber products.
\end{prop}
\begin{proof}
    Let $U$, $V$, and $W$ be log affinoid perfectoid objects in $X_\proket$, with morphisms $V \to U$ and $W \to U$.  Let $U = \varprojlim_i U_i$ be as in Lemma \ref{lem-log-add-perf-mor}.  By Lemma \ref{lem-log-add-perf-mor}, $V$ admits a pro-Kummer \'etale presentation $V = \varprojlim_j V_j$ such that each $V_j$ is modeled on some $P_i$, and the same is true for $W$.  Consequently, $V \times_U W$ also admits a pro-Kummer \'etale presentation of this kind, and hence is log affinoid perfectoid, with associated perfectoid space $\widehat{V \times_U W} \cong \widehat{V} \times_{\widehat{U}} \widehat{W}$.  \Pth{The last fiber product is indeed a perfectoid space, by \cite[\aProp 6.18]{Scholze:2012-ps}.}
\end{proof}

\begin{prop}\label{prop-log-aff-perf-basis}
    Suppose that $X$ is an analytic locally noetherian fs log adic space over $\Spa(\bZ_p, \bZ_p)$.  Then the log affinoid perfectoid objects in $X_\proket$ form a basis.
\end{prop}
\begin{proof}
    We need to show that, for each $U = \varprojlim_{i \in I} U_i \in X_\proket$, \'etale locally on $U$ and $X$, there exists a pro-Kummer \'etale cover of $U$ by log affinoid perfectoid objects; and we may assume that it has a final object $U_0$ such that $U_i \to U_0$ is finite Kummer \'etale, for all $i \in I$, and that $U_0 \to X$ is a composition of rational localizations and finite Kummer \'etale morphisms.  By Lemma \ref{lem-proket-site-check} and \cite[\aProp 4.8]{Scholze:2013-phtra}, we may assume that $X = \Spa(R, R^+)$ is affinoid, and that its underlying adic space admits a pro-\'etale cover by an affinoid perfectoid object $\varprojlim_{j \in J} U_j$ in $X_\proet$.  Also, we may assume that $X$ admits a sharp fs chart $P_X \to \cM$, which induces a \emph{strict} closed immersion $X \Em Y := X\Talg{P}$ as in Remark \ref{rem-def-chart-rel}.  Consider the pro-Kummer \'etale cover of $Y$, as in Lemma \ref{lem-log-aff-perf-toric}, given by the log affinoid perfectoid object $\varprojlim_{(j, n) \in J \times \bZ_{\geq 1}} U_j\Talg{\frac{1}{n} P}$ in $Y_\proket$, whose pullback to $X$ gives a pro-Kummer \'etale cover of $X$ by a log affinoid perfectoid object $V$ of $X_\proket$, by Lemma \ref{lem-log-aff-perf-cl-imm}.  Thus, $U \times_X V \to U$ is a desired pro-Kummer \'etale cover of $U$ by a log affinoid perfectoid object in $X_\proket$, by Lemma \ref{lem-proket-site-check} and Corollary \ref{cor-log-aff-perf-loc-fket}.
\end{proof}

\begin{prop}\label{prop-proket-K-pi-1}
    Suppose that $X$ is an analytic locally noetherian fs log adic space over $\Spa(\bZ_p, \bZ_p)$.  Then $X_\proket$ has a basis $\cB$ such that
    \[
        H^i\bigl({X_\proket}_{/V}, \upsilon^{-1}(\bL)\bigr) = 0,
    \]
    for all $V \in \cB$, all $p$-torsion locally constant sheaf $\bL$ on $X_\ket$, and all $i > 0$.
\end{prop}
\begin{proof}
    Let $U$ be a log affinoid perfectoid object of $X_\proket$, with $\widehat{U} = \Spa(A, A^+)$ the associated affinoid perfectoid space.  By passing to a covering, we may assume that $A$ is integral.  Let $(A_\infty, A_\infty^+)$ be a universal cover of $A$ \Pth{\ie, $A_\infty$ is the union of all finite \'etale extensions $A_j$ of $A$ in a fixed algebraic closure of the fractional field of $A$, and $A_\infty^+$ is the integral closure of $A^+$ in $A_\infty$}.  Let $(\widehat{A}_\infty, \widehat{A}_\infty^+) := \bigl(\varinjlim_j (A_j, A_j^+)\bigr)^\wedge$.  Then $\widehat{U}_\infty := \Spa(\widehat{A}_\infty, \widehat{A}_\infty^+)$ is affinoid perfectoid.  By the argument in the proof of Lemma \ref{lem-log-aff-perf-ket}, there is some $V \to U$ in $X_\proket$, with $V = \varprojlim_j V_j$ log affinoid perfectoid, such that $\widehat{V} \cong \widehat{U}_\infty$ over $\widehat{U}$.  Note that $\bL|_V$ is a trivial local system because, for any finite Kummer \'etale cover $Y \to X$ trivializing $\bL$, the pullback $W := Y \times_X V \to V$ and the induced morphism $\widehat{W} \to \widehat{V}$ are strictly finite \'etale by Lemma \ref{lem-log-aff-perf-ket}, and therefore $\widehat{W} \to \widehat{V}$ has a section, by the assumption on $\widehat{V} \cong \widehat{U}_\infty$.  Consequently, we have $H^i\bigl({X_\proket}_{/V}, \upsilon^{-1}(\bL)\bigr) \cong H^i\bigl(\widehat{V}_\proet, \upsilon^{-1}(\bL)\bigr) \cong H^i\bigl(\widehat{V}_\et, \bL\bigr) = 0$, for all $i > 0$, where the first and second isomorphisms follow from Lemma \ref{lem-log-aff-perf-ket} and \cite[\aCor 3.17(i)]{Scholze:2013-phtra} \Pth{note that the locally noetherian assumption there on $X$ is not needed}, respectively, and the last equality follows essentially verbatim from the last paragraph of the proof of \cite[\aThm 4.9]{Scholze:2013-phtra}.
\end{proof}

\subsection{Structure sheaves}\label{sec-proket-sheaves}

\begin{defn}\label{def-proket-sheaves}
    Suppose that $X$ is a locally noetherian fs log adic space over $\Spa(\bQ_p, \bZ_p)$.  We define the following sheaves on $X_\proket$.
    \begin{enumerate}
        \item The \emph{integral structure sheaf} is $\cO_{X_\proket}^+ := \upsilon^{-1}(\cO_{X_\ket}^+)$, and the \emph{structure sheaf} is $\cO_{X_\proket} := \upsilon^{-1}(\cO_{X_\ket})$.

        \item The \emph{completed integral structure sheaf} is $\widehat{\cO}_{X_\proket}^+ := \varprojlim_n \bigl(\cO_{X_\proket}^+ / p^n\bigr)$, and the \emph{completed structure sheaf} is $\widehat{\cO}_{X_\proket} := \widehat{\cO}_{X_\proket}^+[\frac{1}{p}]$.

        \item The \emph{tilted structure sheaves} are $\widehat{\cO}_{X_\proket}^{\flat+} := \varprojlim_\Phi \widehat{\cO}_{X_\proket}^+ \cong \varprojlim_\Phi \bigl(\cO_{X_\proket}^+ / p\bigr)$ and $\widehat{\cO}_{X_\proket}^\flat := \varprojlim_\Phi \widehat{\cO}_{X_\proket}$, where the transition morphisms $\Phi$ are given by $x \mapsto x^p$.  When the context is clear, we shall simply write $(\widehat{\cO}^\flat, \widehat{\cO}^{\flat+})$ instead of $(\widehat{\cO}_{X_\proket}^{\flat+}, \widehat{\cO}_{X_\proket}^\flat)$.

        \item We have canonical morphisms $\alpha: \cM_{X_\proket} := \upsilon^{-1}(\cM_{X_\ket}) \to \cO_{X_\proket}$ and $\alpha^\flat: \cM_{X_\proket}^\flat := \varprojlim_{a \mapsto a^p} \cM_{X_\proket} \to \widehat{\cO}_{X_\proket}^\flat$ induced by the constructions.  When the context is clear, we shall simply write $\cM$ and $\cM^\flat$ instead of $\cM_{X_\proket}$ and $\cM_{X_\proket}^\flat$, respectively.
    \end{enumerate}
\end{defn}

\begin{prop}\label{prop-proket-log-str}
    In Definition \ref{def-proket-sheaves}, we have $\cM_{X_\proket}(U) = \varinjlim_i \cM_{U_i}(U_i)$, for any pro-Kummer \'etale presentation $U = \varprojlim_i U_i \in X_\proket$.
\end{prop}
\begin{proof}
    The proof is similar to the one of Proposition \ref{prop-proket-vs-ket}.  Note that, by definition, $\upsilon^{-1}(\cM_{X_\proket})$ is the sheaf associated with the presheaf $\widetilde{\cM}$ sending $U = \varprojlim_j U_j$ to $\varinjlim_j \cM_{U_j}(U_j)$.  Also, quasi-compact objects form a basis of $X_\proket$.  Hence, it suffices to show that exactness of
    \[
        0 \to \widetilde{\cM}(U) \to \prod_h \widetilde{\cM}(V_h) \rightrightarrows \prod_{h, h'}\widetilde{\cM}(V_h \times_U V_{h'}),
    \]
    for any quasi-compact $U$ and any finite covering $\{ V_h \to U \}_h$ by quasi-compact objects in $X_\proket$.  By Proposition \ref{prop-proket-site-qcqs} and the same argument as in the proof of \cite[\aLem 3.16]{Scholze:2013-phtra}, this is reduced to the case of a single Kummer \'etale cover $V \to U$, and therefore to the known exactness of
    \[
        0 \to \cM_{X_\ket}(U_0) \to \cM_{X_\ket}(V_0) \rightrightarrows \cM_{X_\ket}(V_0 \times_{U_0} V_0),
    \]
    for some Kummer \'etale cover $V_0 \to U_0$ in $X_\ket$, by Proposition \ref{prop-M-ket-sheaf}.
\end{proof}

The following result is an analogue of \cite[\aThm 4.10]{Scholze:2013-phtra}.
\begin{thm}\label{thm-almost-van-hat}
    Suppose that $X$ is a locally noetherian fs log adic space over $\Spa(\bQ_p, \bZ_p)$.  Let $U \in X_\proket$ be a log affinoid perfectoid object, with associated perfectoid space $\widehat{U} = \Spa(R, R^+)$.  Let $(R^\flat, R^{\flat+})$ be its tilt.
    \begin{enumerate}
        \item\label{thm-almost-van-hat-1}  For each $n > 0$, we have $\cO_{X_\proket}^+(U) / p^n \cong R^+ / p^n$, and it is canonically almost isomorphic to $(\cO_{X_\proket}^+ / p^n)(U)$.

        \item\label{thm-almost-van-hat-2}  For each $n > 0$, we have $H^i(U, \cO_{X_\proket}^+ / p^n)^a = 0$, for all $i > 0$.  Consequently, $H^i(U, \widehat{\cO}_{X_\proket}^+)^a = 0$, for all $i > 0$.

        \item\label{thm-almost-van-hat-3}  We have $\widehat{\cO}_{X_\proket}^+(U) \cong R^+$ and $\widehat{\cO}_{X_\proket}(U) \cong R$, and the ring $\widehat{\cO}_{X_\proket}^+(U)$ is canonically isomorphic to the $p$-adic completion of $\cO_{X_\proket}^+(U)$.

        \item\label{thm-almost-van-hat-4}  We have $\widehat{\cO}_{X_\proket}^{\flat+}(U) \cong R^{\flat+}$ and $\widehat{\cO}_{X_\proket}^\flat(U) \cong R^\flat$.

        \item\label{thm-almost-van-hat-5}  We have $H^i(U, \widehat{\cO}_{X_\proket}^{\flat+})^a = 0$, for all $i > 0$.
    \end{enumerate}
\end{thm}
\begin{proof}
    Let us temporarily omit the subscripts \Qtn{$\proket$} \etc from $\cO_{X_\proket}^+$ \etc.

    Let us first prove \Refenum{\ref{thm-almost-van-hat-1}} and \Refenum{\ref{thm-almost-van-hat-2}}.  By definition, $\cO_X^+(U) / p^n \cong R^+ / p^n$.  By Proposition \ref{prop-log-aff-perf-basis}, giving a sheaf on $X_\proket$ is equivalent to giving a presheaf on the full subcategory of log affinoid perfectoid objects $U$ in $X_\proket$, satisfying the sheaf property for pro-Kummer \'etale coverings by such objects.  Consider such a presheaf of almost $R^+$-algebras $\cF$ given by assigning $\cF(U) := (\cO_X^+(U) / p^n)^a$ to each such $U$.  We claim that $\cF$ is a sheaf with cohomology vanishing above degree zero.  By Proposition \ref{prop-proket-site-qcqs} and the same argument as in the proof of \cite[\aLem 3.16]{Scholze:2013-phtra}, it suffices to verify the exactness of the \v{C}ech complex
    \[
        0 \to \cF(U) \to \cF(V) \to \cF(V \times_U V) \to \cF(V \times_U V\times_U V) \to \cdots
    \]
    for some Kummer \'etale cover $V \to U$ in $X_\proket$ that is the pullback of a Kummer \'etale cover $V_0 \to U_0$ in $X_\ket$.  Furthermore, we may assume that $V_0 \to U_0$ is a composition of finite Kummer \'etale morphisms and rational localizations.  By Lemma \ref{lem-log-aff-perf-ket}, $V$ is log affinoid perfectoid, and $\widehat{V}$ is \'etale over $\widehat{U}$.  Moreover,
    \[
        \cF(U) = (\cO_X^+(U) / p^n)^a \cong (\cO_{\widehat{U}}^+(\widehat{U}) / p^n)^a
    \]
    and
    \[
        \cF(V \times_U \cdots \times_U V) = (\cO_X^+(V \times_U \cdots \times_U V) / p^n)^a \cong(\cO_{\widehat{U}}^+(\widehat{V} \times_{\widehat{U}} \cdots \times_{\widehat{U}} \widehat{V}) / p^n)^a.
    \]
    Hence, \Refenum{\ref{thm-almost-van-hat-1}} and the first statement of \Refenum{\ref{thm-almost-van-hat-2}} follows from the almost exactness of
    \[
        0 \to \cO_{\widehat{U}}^+(\widehat{U}) / p^n \to \cO_{\widehat{U}}^+(\widehat{V}) / p^n \to \cO_{\widehat{U}}^+(\widehat{V} \times_{\widehat{U}} \widehat{V}) / p^n \to \cdots,
    \]
    by \cite[\aThm 7.13]{Scholze:2012-ps} and the $p$-torsionfreeness of $\cO_{\widehat{U}}^+$.  From these, by \cite[\aLem 3.18]{Scholze:2013-phtra}, the remaining statement of \Refenum{\ref{thm-almost-van-hat-2}} also follows.

    As for \Refenum{\ref{thm-almost-van-hat-3}}, we first show that the image of $(\cO_X^+ / p^n)(U) \to (\cO_X^+ / p^m)(U)$ is equal to $R^+ / p^m$ \Pth{under the canonical isomorphisms}, for any $n > m$.  For each $f \in (\cO_X^+ / p^n)(U)$, by \Refenum{\ref{thm-almost-van-hat-1}}, there exists some $g \in R^+$ such that $p^{n - m} f = g$ in $(\cO_X^+ / p^n)(U)$.  Let $h = g / p^{n - m} \in R^+$.  Since the multiplication by $p^{n - m}$ induces an injection $(\cO_X^+ / p^m)(U) \to (\cO_X^+ / p^n)(U)$, it follows that $f = h$ in $(\cO_X^+ / p^m)(U)$.  Therefore, $(\cO_X^+ / p^n)(U) \to (\cO_X^+ / p^m)(U)$ maps $f$ into $R^+ / p^m$, and the assertion follows.  Consequently, $\widehat{\cO}_{X_\proket}^+(U) \cong \varprojlim_n (\cO_X^+ / p^n)(U) = \varprojlim_n R^+ / p^n \cong R^+$, and hence $\widehat{\cO}_{X_\proket}(U) = R$.

    Next, let us prove \Refenum{\ref{thm-almost-van-hat-5}} and an almost version of \Refenum{\ref{thm-almost-van-hat-4}}.  Let $\cG := \cO_{X_\proket}^+ / p$.  By Proposition \ref{thm-almost-van-hat}, $H^i(U_\proket, \cG)^a = 0$, for all log affinoid perfectoid $U \in X_\proket$ and $i > 0$.  Moreover, $\cG(U)^a \cong (\cO_{X_\proket}^+(U) / p)^a$, for any such $U$.  By definition, $\widehat{\cO}_{X_\proket}^{\flat+} \cong \varprojlim_\Phi \cG$.  Let $\cB$ be the basis of $X_\proket$ formed by log affinoid perfectoid objects.  By applying \cite[\aLem 3.18]{Scholze:2013-phtra} to the sheaf $\cG$ and the basis $\cB$, we know that $R^j\varprojlim_\Phi \cG$ is almost zero, for all $j > 0$, and there are almost isomorphisms $\widehat{\cO}_{X_\proket}^{\flat+}(U) \cong (\varprojlim_\Phi \cG)(U) \cong \varprojlim_\Phi \bigl(\cG(U)\bigr) \cong \varprojlim_\Phi (R^+ / p) \cong R^{\flat+}$.  By \cite[\aLem 3.18]{Scholze:2013-phtra} again, $H^i(U_\proket, \widehat{\cO}_{X_\proket}^{\flat+})^a \cong H^i(U_\proket, \varprojlim_\Phi \cG)^a = 0$, for all $i > 0$.

    Finally, let us prove \Refenum{\ref{thm-almost-van-hat-4}}.  Consider the sheaf associated with the presheaf $\cH$ on $X_\proket$ determined by $\cH(U) = \cO_{\widehat{U}^\flat}^+(\widehat{U}^\flat)$, for each $U \in \cB$.  It suffices to show that $\cH$ satisfies the sheaf property for coverings by objects in $\cB$.  Let $U$ and $V$ be log affinoid perfectoid objects in $X_\proket$, and let $V \to U$ be a pro-Kummer \'etale cover.  Let $R$, $S$, and $T$ be the perfectoid algebras associated with $U$, $V$ and $U\times_V U$, respectively.  Then it suffices to show the exactness of $0 \to R^{\flat+} \to S^{\flat+} \to T^{\flat+}$.  Note that this is the inverse limit \Pth{along Frobenius} of $0 \to R^+ / p \to S^+ / p \to T^+ / p$, and this last sequence is exact by the fact that $\cO_{\widehat{U}_\ket}^+$ is a sheaf and $p$-torsion free.  Thus, the desired exactness follows.
\end{proof}

The following proposition is an analogue of \cite[\aThm 9.2.15]{Kedlaya/Liu:2015-RPH}.
\begin{thm}\label{thm-loc-syst-vs-proj}
    Suppose that $X$ is a locally noetherian fs log adic space over $\Spa(\bQ_p, \bZ_p)$.  Let $U$ be a log affinoid perfectoid object of $X_\proket$.  The functor
    \[
        \cH \mapsto H := \cH(U)
    \]
    is an equivalence from the category of finite locally free $\widehat{\cO}_{X_\proket}|_U$-modules on ${X^\sim_\proket}_{/U}$ to the category of finite projective $\widehat{\cO}_{X_{\proket}}(U)$-modules, with a quasi-inverse given by
    \[
        H \mapsto \cH(V) := H \otimes_{\widehat{\cO}_{X_\proket}(U)} \widehat{\cO}_{X_\proket}(V),
    \]
    for each log affinoid perfectoid object $V$ in $X_\proket$ over $U$.  Moreover, for each finite locally free $\widehat{\cO}_{X_\proket}|_U$-module $\cH$, for all $i > 0$, we have
    \[
        H^i({X_\proket}_{/U}, \cH) = 0.
    \]
\end{thm}
\begin{proof}
    By Lemma \ref{lem-log-aff-perf-ket}, the first statement follows from \cite[\aThm 9.2.15]{Kedlaya/Liu:2015-RPH}.  The proof of the second statement is similar to that of \cite[\aProp 2.3]{Liu/Zhu:2017-rrhpl}, with the input of \cite[\aLem 2.6.5(a)]{Kedlaya/Liu:2015-RPH} replaced with Lemma \ref{lem-proket-site-check}\Refenum{\ref{lem-proket-site-check-3}} here.
\end{proof}

By combining Lemma \ref{lem-log-aff-perf-cl-imm} and Theorem \ref{thm-almost-van-hat}, we obtain
the following:
\begin{prop}\label{prop-surj-O-hat-cl-imm}
    Let $\imath: Z \to X$ be a strict closed immersion of locally noetherian fs log adic spaces over $\Spa(\bQ_p, \bZ_p)$.  Then the natural morphism $\widehat{\cO}_{X_\proket} \to \imath_{\proket, *}(\widehat{\cO}_{Z_\proket})$ is surjective.  More precisely, its evaluation at every log affinoid perfectoid object $U$ in $X_\proket$ is surjective.
\end{prop}

\section{Kummer \'etale cohomology}\label{sec-loc-syst}

\subsection{Toric charts revisited}\label{sec-toric-chart}

Let $V = \Spa(S_1, S_1^+)$ be a log smooth affinoid fs log adic space over $\Spa(k, k^+)$, where $(k, k^+)$ is as in Definition \ref{def-log-sm-base-field} and where $k^+ = \cO_k$, with a toric chart
\[
    V \to \bE = \Spa(k\Talg{P}, k^+\Talg{P}) = \Spa(R_1, R_1^+),
\]
as in Proposition \ref{prop-toric-chart} and Definition \ref{def-toric-chart}, where $P$ is a sharp fs monoid.  The goal of this subsection is to prove the following:
\begin{prop}\label{prop-coh-tor-chart}
    In the above setting, assume moreover that $k$ is characteristic zero and contains all roots of unity.  Let $V \to \bE$ be a toric chart as above, and let $\bL$ be an $\bF_p$-local system on $V_\ket$.  Then we have the following:
    \begin{enumerate}
        \item\label{prop-coh-tor-chart-1}  $H^i\bigl(V_\ket, \bL \otimes_{\bF_p} (\cO_V^+ / p)\bigr)$ is almost zero, for all $i > n = \dim(V)$.

        \item\label{prop-coh-tor-chart-2}  Let $V' \subset V$ be a rational subset such that $V'$ is strictly contained in $V$ \Pth{\ie, the closure $\overline{V}'$ of $V'$ is contained in $V$}.  Then the image of the canonical morphism $H^i\bigl(V_\ket, \bL \otimes_{\bF_p} (\cO_V^+ / p)\bigr) \to H^i\bigl(V'_\ket, \bL \otimes_{\bF_p} (\cO_V^+ / p)\bigr)$ is an almost finitely generated $k^+$-module, for each $i\geq 0$.
    \end{enumerate}
\end{prop}

In order to prove Proposition \ref{prop-coh-tor-chart}, we need some preparations.  Let us first introduce an explicit pro-finite Kummer \'etale cover of $\bE$.  For each $m \geq 1$, consider
\[
    \bE_m := \Spa(k\Talg{\tfrac{1}{m} P}, k^+\Talg{\tfrac{1}{m} P}) = \Spa(R_m, R_m^+)
\]
and the log affinoid perfectoid object
\[
    \widetilde{\bE} := \varprojlim_m \bE_m \in \bE_\proket,
\]
where the transition maps $\bE_{m'} \to \bE_m$ \Pth{for $m | m'$} are induced by the natural inclusions $\frac{1}{m} P \Em \frac{1}{m'} P$.  Let $P_{\bQ_{\geq 0}} := \varinjlim_m \bigl(\frac{1}{m} P\bigr)$ as before.  Then the associated perfectoid space is
\[
    \widehat{\widetilde{\bE}} := \Spa(k\Talg{P_{\bQ_{\geq 0}}}, k^+\Talg{P_{\bQ_{\geq 0}}}) = \Spa(R, R^+).
\]
For each $m \geq 1$, let us write
\[
    V_m := V \times_\bE \bE_m = \Spa(S_m, S_m^+)
\]
and
\[
    \widetilde{V} := V \times_\bE \widetilde{\bE} \in V_\proket.
\]
Then $\widetilde{V} \cong \varprojlim_m V_m$ is also a log affinoid perfectoid object in $V_\proket$, with associated perfectoid space $\widehat{\widetilde{V}} \cong \Spa(S, S^+)$, where $(S, S^+) = \bigl(\varinjlim_m \, (S_m, S_m^+)\bigr)$.

\begin{defn}\label{def-prof-Gal}
    Suppose that $X$ is a locally noetherian fs log adic space over $\Spa(k, k^+)$, where $(k, k^+)$ is an affinoid field, and where $k$ is of characteristic zero and contains all roots of unity.  Let $G$ be a profinite group.  A pro-Kummer \'etale cover $U \to X$ is a \emph{Galois cover with \Pth{profinite} Galois group $G$} if there exists a pro-Kummer \'etale presentation $U = \varprojlim_i U_i \to X$ such that each $U_i \to X$ is a Galois finite Kummer \'etale cover with Galois group $G_i$ \Pth{as in Proposition \ref{prop-ket-std}, where $G_i$ is a constant group object because contains all roots of unity}, and such that $G \cong \varprojlim_i G_i$, in which case the group action and the second projection induces a canonical isomorphism $G \times U \cong U \times_X U$ over $X$.
\end{defn}

Since $P$ is a sharp fs monoid, $P^\gp$ is a finitely generated free abelian group.  Let $P^\gp_\bQ := (P_{\bQ_{\geq 0}})^\gp \cong P^\gp \otimes_\bZ \bQ$.  Then $\bE_m \to \bE$ and therefore $V_m \to V$ are finite Kummer \'etale covers with Galois group
\begin{equation}\label{eq-geom-tower-Gal-mod-m}
\begin{split}
    \Gamma_{/m} := \Hom\bigl((\tfrac{1}{m} P)^\gp / P^\gp, \Grpmu_\infty\bigr) & \cong \Hom(P^\gp / m P^\gp, \Grpmu_m) \\
    & \cong \Hom(P^\gp, \Grpmu_m),
\end{split}
\end{equation}
and $\widetilde{V} \to V$ is a Galois pro-finite Kummer \'etale cover with Galois group
\begin{equation}\label{eq-geom-tower-Gal}
\begin{split}
    \Gamma := \varprojlim_m \Gamma_{/m} & \cong \Hom\bigl(P^\gp, \varprojlim_m \Grpmu_m\bigr) \cong \Hom\bigl(P^\gp, \widehat{\bZ}(1)\bigr) \\
    & \cong \Hom\bigl(\varinjlim_m (\tfrac{1}{m} P)^\gp / P^\gp, \Grpmu_\infty\bigr) \cong \Hom(P^\gp_\bQ / P^\gp, \Grpmu_\infty),
\end{split}
\end{equation}
where $\Grpmu_m$, $\Grpmu_\infty$, and $\widehat{\bZ}(1)$ are as in Definition \ref{def-mu} \Pth{with the symbols $(k)$ omitted}.

Consider the $k^+[P]$-module decomposition
\begin{equation}\label{eq-k-P-decomp}
    k^+[P_{\bQ_{\geq 0}}] = \oplus_\chi \bigl( k^+[P_{\bQ_{\geq 0}}]_\chi \bigr)
\end{equation}
according to the action of $\Gamma$, where the direct sum is over all finite-order characters $\chi$ of $\Gamma$.  Note that the set of finite-order characters of $\Gamma$ can be naturally identified with $P^\gp_\bQ / P^\gp$, via \Refeq{\ref{eq-geom-tower-Gal}}.  If we denote by $\pi$ the natural map $\pi: P_{\bQ_{\geq 0}} \to P^\gp_\bQ / P^\gp$, then we have the $k^+$-module decomposition
\[
    k^+[P_{\bQ_{\geq 0}}]_\chi = \oplus_{a \in P_{\bQ_{\geq 0}}, \, \pi(a) = \chi} \, \bigl( k^+ \mono{a} \bigr).
\]

\begin{lem}\phantomsection\label{lem-k-P-decomp}
    \begin{enumerate}
        \item\label{lem-k-P-decomp-1}  $k^+[P_{\bQ_{\geq 0}}]_1 = k^+[P]$ for the trivial character $\chi = 1$.

        \item\label{lem-k-P-decomp-2}  Each direct summand $k^+[P_{\bQ_{\geq 0}}]_\chi$ is a finite $k^+[P]$-module.
    \end{enumerate}
\end{lem}
\begin{proof}
    The assertion \Refenum{\ref{lem-k-P-decomp-1}} follows from the observation that $P_{\bQ_{\geq 0}} \cap P^\gp = P$ as subsets of $P^\gp_\bQ$.  As for the assertion \Refenum{\ref{lem-k-P-decomp-2}}, it suffices to show that, for each $\chi$ in $P^\gp_\bQ / P^\gp$, if $\chi \in (\frac{1}{m} P)^\gp / P^\gp$ for some $m \geq 1$, and if $P$ is generated as a monoid by some finite subset $\{ a_1, \ldots, a_r \}$, then there exists some integer $m' \geq m$ \Pth{depending on $m$} such that $\pi^{-1}(\chi) \subset \frac{1}{m'} P$, so that $\pi^{-1}(\chi) = S_\chi + P$ for the finite subset $S_\chi := \{ \sum_{i = 1}^r \frac{c_i}{m'} a_i \in \pi^{-1}(\chi) : 0 \leq c_i < m' \}$ of $\pi^{-1}(\chi)$.  Concretely, since $P$ is sharp by assumption, $\sigma := \bR_{\geq 0} \, a_1 + \cdots + \bR_{\geq 0} \, a_r$ is a convex subset of $P^\gp \otimes_\bZ \bR$ of the form $\sigma = \{ a \in P^\gp \otimes_\bZ \bR : \Utext{$b_j(a) \geq 0$, for all $j = 1, \ldots, s$} \}$ for some homomorphisms $b_j: P^\gp \to \bZ$ \Pth{\Refcf{} \cite[\aCh I, \aSec 1, \apages 6--7]{Kempf/Knudsen/Mumford/Saint-Donat:1973-TE-1}} such that $\Hom_\bZ(P^\gp, \bQ) = \sum_{j = 1}^s \bQ \, b_j$.  It follows that $\Hom_\bZ(P^\gp, \bZ) \subset \sum_{j = 1}^s \frac{1}{N} \bZ \, b_j$, for some $N \geq 1$, and hence $\{ a \in P^\gp \otimes_\bZ \bQ : \Utext{$b_j(a) \in \bZ$, for all $j = 1, \ldots, s$} \} \subset \frac{1}{N} P^\gp$ by duality.  Thus, $\pi^{-1}(\chi) \subset P_{\bQ_{\geq 0}} \cap \frac{1}{m'} P^\gp = \frac{1}{m'} P$, for $m' = m N$, as desired.
\end{proof}

\begin{lem}\label{lem-Gal-coh-afg}
    Fix $n \geq 0$.  Let $M$ be any $k^+ / p^n$-module on which $\Gamma$ acts via a primitive character $\chi: \Gamma \to \Grpmu_m$.  Then $H^i(\Gamma, M)$ is annihilated by $\zeta_m - 1$, where $\zeta_m \in \Grpmu_m$ is any primitive $m$-th root of unity, for each $i \geq 0$.  Moreover, if we have a finite extension $k_0$ of $\bQ_p(\Grpmu_m)$ in $k$ with ring of integers $k_0^+$, a finitely generated $k_0^+ / p^n$-algebra $T_0$, and a finite \Pth{and therefore finitely presented} $T_0$-module $M_0$ such that $M \cong M_0 \otimes_{k_0^+ / p^n} (k^+ / p^n)$ as $\Gamma$-modules over $T := T_0 \otimes_{k_0^+ / p^n} (k^+ / p^n)$, then $H^i(\Gamma, M)$ is a finitely presented $T$-module, for each $i \geq 0$.
\end{lem}
\begin{proof}
    By choosing a $\bZ$-basis of $P^\gp$, we have $\Gamma \cong \widehat{\bZ}(1)^n$, where $n = \rank_\bZ(P^\gp)$ \Pth{see \Refeq{\ref{eq-geom-tower-Gal}}}.  Then the lemma follows from a direct computation using the Koszul complex of $\Gamma$ \Pth{as in the proof of \cite[\aLem 5.5]{Scholze:2013-phtra}} and \Pth{for the last assertion of the lemma} using the flatness of $k^+ / p^n$ over $k_0^+ / p^n$ \Pth{and the compatibility with flat base change in the formation of Koszul complexes}.
\end{proof}

\begin{rk}\label{rem-Gal-coh}
    Since $R_1^+ / p \cong (k^+ / p)[P]$ and $R^+ / p \cong (k^+ / p)[P_{\bQ_{\geq 0}}]$, by Lemmas \ref{lem-k-P-decomp} and \ref{lem-Gal-coh-afg}, the natural injection $R_1^+ / p \Em R^+ / p$ induces an injection $H^i\bigl(\Gamma, (R_1^+ / p)\bigr) \Em H^i\bigl(\Gamma, (R^+ / p)\bigr)$, with cokernel annihilated by $\zeta_p - 1$, for each $i \geq 0$.  Moreover, the $R_1^+ / p$-module $H^i\bigl(\Gamma, (R^+ / p)\bigr)$ is almost finitely presented, because, for each $\epsilon > 0$ such that $p^\epsilon$-torsion makes sense, there are only finitely many $\chi$ such that the finitely presented $R^+ / p$-module direct summand $H^i\bigl(\Gamma, (k^+ / p)[P_{\bQ_{\geq 0}}]_\chi\bigr)$ is nonzero and not $p^\epsilon$-torsion.  By the use of Koszul complexes as in the proof of Lemma \ref{lem-Gal-coh-afg}, for any composition of rational localizations and finite \'etale morphisms $\Spa(S_1, S_1^+) \to \Spa(R_1, R_1^+)$, we have $H^i\bigl(\Gamma, (S_1^+ / p) \otimes_{R_1^+ / p} (R^+ / p)\bigr) \cong (S_1^+ / p) \otimes_{R_1^+ / p} H^i(\Gamma, R^+ / p)$.
\end{rk}

By the same argument as in the proof of \cite[\aLem 4.5]{Scholze:2013-phtra}, we obtain the following:
\begin{lem}\label{lem-Scholze-ref}
    Let $X$ be a locally noetherian fs log adic space over $\Spa(k, k^+)$.  Let
    \[
        U = \varprojlim_{i \in I} U_i = \varprojlim_{i \in I} \, (\Spa(R_i, R_i^+), \cM_i)
    \]
    be a log affinoid perfectoid object in $X_\proket$, and let $(R, R^+) := \bigl(\varinjlim_i \, (R_i, R_i^+)\bigr)^\wedge$, so that $\widehat{U} = \Spa(R, R^+)$ is the associated affinoid perfectoid space.

    Suppose that, for some $i \in I$, there exists a strictly \'etale morphism
    \[
        V_i = \Spa(S_i, S_i^+) \to U_i
    \]
    that is a composition of rational localizations and strictly finite \'etale morphisms.  For each $j \geq i$, let $V_j := V_i \times_{U_i} U_j = \Spa(S_j, S_j^+)$, and let
    \[
        V := V_i \times_{U_i} U \cong \varprojlim_j V_j \in X_\proket.
    \]
    Let $(S, S^+) := \bigl(\varinjlim_j \, (S_j, S_j^+)\bigr)^\wedge$.  Let $T_j$ be the $p$-adic completion of the $p$-torsion free quotient of $S_j^+ \otimes_{R_j^+} R^+$.  Then we have the following:
    \begin{enumerate}
        \item $(S, S^+)$ is a perfectoid affinoid $(k, k^+)$-algebra, and $V$ is a log affinoid perfectoid object in $X_\proket$ with associated perfectoid space $\widehat{V} = \Spa(S, S^+)$.  Moreover, $\widehat{V} = V_j \times_{U_j} \widehat{U}$ in the category of adic spaces.

        \item For each $j \geq i$, we have $S \cong T_j[\frac{1}{p}]$, and the cokernel of $T_j \to S^+$ is annihilated by some power of $p$.

        \item For each $\varepsilon \in \bQ_{> 0}$, there exists some $j \geq i$ such that the cokernel of $T_j \to S^+$ is annihilated by $p^\varepsilon$.
    \end{enumerate}
\end{lem}

\begin{rk}\label{rem-lem-Scholze-ref}
    Lemma \ref{lem-Scholze-ref} is applicable, in particular, to the log affinoid perfectoid object $U = \varprojlim_{m \geq 1} \bE_m$ in $\bE_\proket$ and any strictly \'etale morphism $V = \Spa(S_1, S_1^+) \to \bE$ \Pth{for $m = 1$} giving a toric chart.
\end{rk}

\begin{lem}\label{lem-coh-tor-log-aff-perf}
    Let $X$ be a locally noetherian fs log adic space over $\Spa(k, k^+)$.  Suppose that $U$ is a log affinoid perfectoid object of $X_\proket$, with associated perfectoid space $\widehat{U} = \Spa(R, R^+)$.  Let $\bL$ be an $\bF_p$-local system on $U_\ket$.  Then:
    \begin{enumerate}
        \item $H^i\bigl(U_\ket, \bL \otimes_{\bF_p} (\cO_X^+ / p)\bigr)$ is almost zero, for all $i > 0$.

        \item $L(U) := H^0\bigl(U_\ket, \bL \otimes_{\bF_p} (\cO_X^+ / p)\bigr)$ is an almost finitely generated projective $R^+ / p$-module \Pth{see \cite[\aDef 2.4.4]{Gabber/Ramero:2003-ART}}.  In addition, for any morphism $U' \to U$ in $X_\proket$, where $U'$ is a log affinoid perfectoid object in $X_\proket$, with associated perfectoid space $\widehat{U}' = \Spa(R', R'^+)$, we have a canonical almost isomorphism $L(U') \cong L(U) \otimes_{R^+ / p} (R'^+ / p)$.
    \end{enumerate}
\end{lem}
\begin{proof}
    By replacing $X$ with its connected components, we may assume that $X$ is connected.  Choose any Galois finite Kummer \'etale cover $Y \to X$ trivializing $\bL \Mi \bF_p^r$.  By Lemma \ref{lem-log-aff-perf-ket}, $W := U \times_X Y \to U$ is finite \'etale, and $W$ is log affinoid perfectoid, with associated perfectoid space $\widehat{W} = \Spa(T, T^+)$.  For each $j \geq 1$, let $W^{j / U}$ denote the $j$-fold fiber product of $W$ over $U$.  By Proposition \ref{prop-proket-vs-ket-adj} and Theorem \ref{thm-almost-van-hat}, $H^i\bigl(W^{j / U}_\ket, \bL \otimes_{\bF_p} (\cO_W^+ / p)\bigr)$ is almost zero, for all $i > 0$ and $j$, and $H^0\bigl(W^{j / U}_\ket, \bL \otimes_{\bF_p} (\cO_V^+ / p)\bigr)$ is canonically almost isomorphic to $(\cO_{W^{j / U}}^+(W^{j / U}) / p)^r$.  By the faithful flatness of $T^{+a} / p \to R^{+a} / p$, the desired results follow from almost faithfully flat descent \Pth{see \cite[\aSec 3.4]{Gabber/Ramero:2003-ART}}.
\end{proof}

Now we are ready for the following:
\begin{proof}[Proof of Proposition \ref{prop-coh-tor-chart}]
    Consider the Galois cover $\widetilde{V} \to V = \Spa(S_1, S_1^+)$ with Galois group $\Gamma$, and with $\widehat{\widetilde{V}} = \Spa(S, S^+)$, as above.  Since $\widetilde{V}^{j / V} \cong \widetilde{V} \times \Gamma^{j - 1}$ is a log affinoid perfectoid object in $V_\proket$, for each $j \geq 1$, we have
    \[
        H^i\bigl(\widetilde{V}^{j / V}_\ket, \bL \otimes_{\bF_p} (\cO_V^+ / p)\bigr) \cong \Hom_\cont\bigl(\Gamma^{j - 1}, L\bigr),
    \]
    where $L := H^0\bigl(\widetilde{V}_\ket, \bL \otimes_{\bF_p} (\cO_V^+ / p)\bigr)^a$ is an almost finitely generated projective $S^{+a} / p$-module, equipped with the discrete topology, by Propositions \ref{prop-proket-vs-ket} and \ref{prop-proket-vs-ket-adj}.  By Proposition \ref{prop-proket-vs-ket-adj} again, and by Lemma \ref{lem-coh-tor-log-aff-perf} and the Cartan--Leray spectral sequence \Pth{see \cite[V, 3.3]{SGA:4}}, we have an almost isomorphism
    \[
        H^i\bigl(V_\ket, \bL \otimes_{\bF_p} (\cO_V^+ / p)\bigr) \cong \check{H}^i\bigr(\{ \widetilde{V} \to V \}, \bL \otimes_{\bF_p} (\cO_V^+ / p)\bigr) \cong H^i(\Gamma, L),
    \]
    where the last isomorphism follows from Proposition \ref{prop-fund-grp-profket} and \cite[\aProp 3.7(iii)]{Scholze:2013-phtra} \Pth{and the correction in \cite{Scholze:2016-phtra-corr}}.  Hence, the statement \Refenum{\ref{prop-coh-tor-chart-1}} of Proposition \ref{prop-coh-tor-chart} follows from the fact that $\Gamma \cong \widehat{\bZ}(1)^n$ has cohomological dimension $n$.

    As for the statement \Refenum{\ref{prop-coh-tor-chart-2}}, write $V' = \Spa(S'_1, S'^+_1)$ and $\widehat{\widetilde{V}}{}' = \Spa(S', S'^+)$.  We need to show that the image of $H^i(\Gamma, L) \to H^i\bigl(\Gamma, L \otimes_{S^+ / p} (S'^+/p)\bigr)$ is an almost finitely generated $k^+$-module.  Since $L$ is an almost finitely generated projective $S^{+a} / p$-module, it suffices to show that the image of
    \[
        H^i(\Gamma, S^+ / p) \to H^i(\Gamma, S'^+ / p)
    \]
    is an almost finitely generated $k^+$-module.  Choose rational subsets $\{ V^{(j)} \}_{1 \leq j \leq n + 2}$ such that $V^{(n + 2)} = V'$, $V^{(1)} = V$, and $V^{(j + 1)}$ is strictly contained in $V^{(j)}$, for $1 \leq j \leq n + 1$.  Write $V_m^{(j)} := V^{(j)} \times_\bE \bE_m = \Spa(S_m^{(j)}, S_m^{(j)+})$, for all $1 \leq j \leq n + 2$ and $m \geq 1$.  Then $\widetilde{V}^{(j)} := \varprojlim_m V_m^{(j)}$ is a log affinoid perfectoid object in $V_\proket$, with associated perfectoid space $\widehat{\widetilde{V}}{}^{(j)} = \Spa(S^{(j)}, S^{(j)+})$.  By Lemma \ref{lem-Scholze-ref} and Remark \ref{rem-lem-Scholze-ref}, it suffices to show that the image of
    \[
        H^i(\Gamma, (S_m^{(1)+} \otimes_{R_m^+} R^+) / p) \to H^i(\Gamma, (S_m^{(n + 2)+} \otimes_{R_m^+} R^+) / p)
    \]
    is almost finitely generated, for all $m \in \bZ_{\geq 1}$.  Note that $m\Gamma$ acts trivially on $S_m^{(j)+}$, and we have the Hochschild--Serre spectral sequence
    \[
        H^{i_1}\bigl(\Gamma_{/m}, H^{i_2}(m\Gamma, (S_m^{(j)+} \otimes_{R_m^+} R^+) / p)\bigr) \Rightarrow H^{i_1 + i_2}(\Gamma, (S_m^{(j)+} \otimes_{R_m^+} R^+) / p).
    \]
    By \cite[\aLem 5.4]{Scholze:2013-phtra} and Remark \ref{rem-Gal-coh}, it suffices to show that the image of
    \[
        (S_m^{(j)+} / p) \otimes_{R_m^+ / p} H^i(m\Gamma, R^+ / p) \to (S_m^{(j + 1)+} / p) \otimes_{R_m^+ / p} H^i(m\Gamma, R^+ / p)
    \]
    is almost finitely generated, for all $j = 1, \ldots, n + 1$ and $m \geq 1$.  Since the image of $S_m^{(j)+} / p \to S_m^{(j + 1)+} / p$ is an almost finitely generated $k^+$-module, it suffices to note that $H^i(m\Gamma, R^+ / p)$ is almost finitely generated over $R_m^+ / p$, by Remark \ref{rem-Gal-coh} \Pth{up to replacing $(R_1, R_1^+)$, $\Gamma$, \etc with $(R_m, R_m^+)$, $\Gamma_m$, etc}.
\end{proof}

\subsection{Primitive comparison theorem}\label{sec-thm-prim-comp}

The main goal of this subsection is to prove the following \emph{primitive comparison theorem}, with the finiteness of cohomology as a byproduct, generalizing the strategy in \cite[\aSec 5]{Scholze:2013-phtra}:
\begin{thm}\label{thm-prim-comp}
    Let $(k, k^+)$ be an affinoid field, where $k$ is algebraically closed and of characteristic zero, and let $X$ be a proper log smooth fs log adic space over $\Spa(k, k^+)$ \Pth{see Definitions \ref{def-log-str} and \ref{def-log-sm}}.  Let $\bL$ be an $\bF_p$-local system on $X_\ket$.  Then we have the following:
    \begin{enumerate}
        \item\label{thm-prim-comp-1} $H^i\bigl(X_\ket, \bL \otimes_{\bF_p} (\cO_X^+ / p)\bigr)$ is an almost finitely generated $k^+$-module \Pth{see \cite[\aDef 2.3.8]{Gabber/Ramero:2003-ART}} for each $i \geq 0$, and is almost zero for $i \gg 0$.

        \item\label{thm-prim-comp-2} There is a canonical almost isomorphism
            \[
                H^i(X_\ket, \bL) \otimes_{\bF_p} (k^+ / p) \Mi H^i\bigl(X_\ket, \bL \otimes_{\bF_p} (\cO_X^+ / p)\bigr)
            \]
            of $k^+$-modules, for each $i \geq 0$.
    \end{enumerate}
    Consequently, $H^i(X_\ket, \bL)$ is a finite-dimensional $\bF_p$-vector space for each $i \geq 0$, and $H^i(X_\ket, \bL) = 0$ for $i \gg 0$.  In addition, if $X$ is as in Example \ref{ex-log-adic-sp-ncd}, then $H^i(X_\ket, \bL) = 0$ for $i > 2 \dim(X)$.
\end{thm}

\begin{rk}\label{rem-coh-fin}
    Recall that there is no general finiteness results for the \'etale cohomology of $\bF_p$-local systems on non-proper rigid analytic varieties over $k$, as is well known \Pth{via Artin--Schreier theory} that $H^1(\bD, \bF_p)$ is infinite.
\end{rk}
Nevertheless, we have the following:
\begin{cor}\label{cor-coh-fin}
    Let $U$ be a smooth rigid analytic variety that is Zariski open in a proper rigid analytic variety over $k$.  Then $H^i(U_\et, \bL)$ is a finite-dimensional $\bF_p$-vector space, for each $\bF_p$-local system $\bL$ on $U_\et$ and each $i \geq 0$.  Moreover, $H^i(U_\et, \bL) = 0$ for $i > 2 \dim(U)$.
\end{cor}
\begin{proof}
    By resolution of singularities \Pth{as in \cite{Bierstone/Milman:1997-cdbml}}, we may assume that we have a smooth compactification $U \Em X$ such that $U = X - D$ for some normal crossings divisor $D$ of $X$.  Now apply Theorems \ref{thm-purity} and \ref{thm-prim-comp}.
\end{proof}

\begin{lem}\label{lem-Scholze-cov}
    Let $X$ be a proper log smooth fs log adic space over $\Spa(k, \cO_k)$.  For each integer $N \geq 2$, we can find $N$ affinoid \'etale coverings of $X$
    \[
        \{ V_h^{(N)} \}_{h = 1}^m, \ldots, \{ V_h^{(1)} \}_{h = 1}^m
    \]
    satisfying the following properties:
    \begin{itemize}
        \item $V_h^{(N)} \subset \cdots \subset V_h^{(1)}$ is a chain of rational subsets, for each $h = 1, \ldots, m$.

        \item $V_h^{(j + 1)} \subset \overline{V}_h^{(j + 1)} \subset V_h^{(j)}$, for all $h = 1, \ldots, m$ and $j = 1, \ldots, N - 1$.

        \item $V_{h_1}^{(1)} \times_X V_{h_2}^{(1)} \to V_{h_1}^{(1)}$ is a composition of rational localizations and finite \'etale morphisms, for $1 \leq h_1, h_2 \leq m$.

        \item Each $V_h^{(1)}$ admits a toric chart $V_h^{(1)} \to \Spa(k\Talg{P_h}, \cO_k\Talg{P_h})$, for some sharp fs monoid $P_h$.
    \end{itemize}
\end{lem}
\begin{proof}
    By Proposition \ref{prop-toric-chart} and the same argument as in the proof of \cite[\aLem 5.3]{Scholze:2013-phtra}, there exist $N$ affinoid analytic open coverings of $X$
    \[
        \{ U_h^{(N)} \}_{h = 1}^m, \ldots, \{ U_h^{(1)} \}_{h = 1}^m
    \]
    satisfying the following properties:
    \begin{itemize}
        \item $U_h^{(N)} \subset \cdots \subset U_h^{(1)}$ is a chain of rational subsets, for each $h = 1, \ldots, m$.

        \item $U_h^{(j + 1)} \subset \overline{U}_h^{(j + 1)} \subset U_h^{(j)}$, for all $h = 1, \ldots m$ and $j = 1, \ldots, N - 1$.

        \item $U_{h_1}^{(1)} \cap U_{h_2}^{(1)} \subset U_{h_1}^{(1)}$ is a rational subset, for $1 \leq h_1, h_2 \leq m$.

        \item There exist finite \'etale covers $V_h^{(1)} \to U_h^{(1)}$ such that each $V_h^{(1)}$ admits a toric chart $V_h^{(1)} \to \bE_j = \Spa(k\Talg{P_h}, \cO_k\Talg{P_h})$ \Pth{which is, in particular, a composition of rational localizations and finite \'etale morphisms} for some sharp fs monoid $P_h$.
    \end{itemize}
    Then it suffices to take $V_h^{(j)} := V_h^{(1)} \times_{U_h^{(1)}} U_h^{(j)}$, for all $h$ and $j$.
\end{proof}

\begin{proof}[Proof of Theorem \ref{thm-prim-comp}\Refenum{\ref{thm-prim-comp-1}}]
    Consider $X' := X \times_{\Spa(k, k^+)} \, \Spa(k, \cO_k) \subset X$.  Consider any covering $\{ U_h \}_h$ of $X$ by log affinoid perfectoid objects in $X_\proket$, whose pullback $\{ U_h \times_X X' \}_h$ is a covering of $X'$ by log affinoid perfectoid objects in $X'_\proket$.  By Lemma \ref{lem-coh-tor-log-aff-perf}, we have a canonical almost isomorphism
    \[
        H^i\bigl(U_{h, \ket}, \bL \otimes_{\bF_p} (\cO_X^+ / p)\bigr) \Mi H^i\bigl(U_{k, \ket} \times_X X', \bL \otimes_{\bF_p} (\cO_X^+ / p)\bigr),
    \]
    for all $i \geq 0$ and all $h$.  By Proposition \ref{prop-proket-vs-ket-adj} and by comparing the spectral sequences associated with the coverings, we obtain an almost isomorphism
    \[
        H^i\bigl(X_\ket, \bL \otimes_{\bF_p} (\cO_X^+ / p)\bigr) \Mi H^i\bigl(X'_\ket, \bL \otimes_{\bF_p} (\cO_X^+ / p)\bigr),
    \]
    for each $i \geq 0$.  Hence, for the purpose of this proof, up to replacing $X$ with $X'$, we may assume that $k^+ = \cO_k$ in what follows.

    Let $\{ V_h^{(N)} \}_{h = 1}^m, \ldots, \{ V_h^{(1)} \}_{h = 1}^m$ be affinoid \'etale coverings of $X$ satisfying the same properties as in Lemma \ref{lem-Scholze-cov}.  For each subset $H = \{ h_1, \ldots, h_s \}$ of $\{ 1, \ldots, m \}$, let $V_H^{(j)} := V_{h_1}^{(j)} \times_X \cdots \times_X V_{h_s}^{(j)}$.  For each $j = 1, \ldots, N$, we have a spectral sequence
    \[
        E_{1, (j)}^{i_1, i_2} = \oplus_{|H| = i_1 + 1} \, H^{i_2}\bigl(V_{H, \ket}^{(j)}, \bL \otimes_{\bF_p} (\cO_X^+ / p)\bigr) \Rightarrow H^{i_1 + i_2}\bigl(X_\ket, \bL \otimes_{\bF_p} (\cO_X^+ / p)\bigr).
    \]
    For $j = 1, \ldots, N - 1$, we also have natural morphisms between spectral sequences $E_{*, (j)}^{i_1, i_2} \to E_{*, (j + 1)}^{i_1, i_2}$.  Then the desired finiteness result follows from Proposition \ref{prop-coh-tor-chart} and \cite[\aLem 5.4]{Scholze:2013-phtra}.  Moreover, by Proposition \ref{prop-coh-tor-chart} and the spectral sequence for $j = 1$, we have $H^i\bigl(X_\ket, \bL \otimes_{\bF_p} (\cO_X^+ / p)\bigr)^a = 0$ for $i \gg 0$.
\end{proof}

\begin{proof}[Proof of Theorem \ref{thm-prim-comp}\Refenum{\ref{thm-prim-comp-2}}]
    Consider the Artin--Schreier sequence
    \begin{equation}\label{eq-thm-prim-comp-pf-AS}
        0 \to \bL \to \bL \otimes_{\bF_p} \widehat{\cO}_{X_\proket}^\flat \Mapn{\sigma} \bL \otimes_{\bF_p} \widehat{\cO}_{X_\proket}^\flat \to 0,
    \end{equation}
    where $\sigma = \Id \otimes (\Phi - \Id)$ and $\Phi$ is the Frobenius morphism \Pth{induced by $x \mapsto x^p$}.  The exactness of \Refeq{\ref{eq-thm-prim-comp-pf-AS}} can be checked locally on log affinoid perfectoid objects $U \in X_\proket$ over which $\bL$ is trivial, which then follows \Pth{by using Lemma \ref{lem-log-aff-perf-ket}} from the same argument in the proof of \cite[\aThm 5.1]{Scholze:2013-phtra}.

    Choose any $\varpi \in k^\flat$ such that $\varpi^\sharp = p$.  By Theorem \ref{thm-prim-comp}\Refenum{\ref{thm-prim-comp-1}} and \cite[\aLem 2.12]{Scholze:2013-phtra}, there exists some $r \geq 0$ such that we have
    \[
        H^i\bigl(X_\proket, \bL \otimes_{\bF_p} (\widehat{\cO}_{X_\proket}^{\flat+} / \varpi^m)\bigr)^a \cong (\cO_{k^\flat}^a / \varpi^m)^r,
    \]
    for all $m$, which are compatible with each other and with the Frobenius morphism.  By \cite[\aLem 3.18]{Scholze:2013-phtra}, we have
    \[
        R\varprojlim_m \bigl(\bL \otimes_{\bF_p} (\widehat{\cO}_X^{\flat+} / \varpi^m)\bigr)^a \cong (\bL \otimes_{\bF_p} \widehat{\cO}_X^{\flat+})^a,
    \]
    and so
    \[
        H^i(X_\proket, \bL \otimes_{\bF_p} \widehat{\cO}_{X_\proket}^{\flat+})^a \cong (\cO_{k^\flat}^a)^r
    \]
    and
    \[
        H^i(X_\proket, \bL \otimes_{\bF_p} \widehat{\cO}_{X_\proket}^\flat) \cong (k^\flat)^r
    \]
    \Pth{by inverting $\varpi$}, which are still compatible with the Frobenius morphisms.

    Thus, by considering the long exact sequence associated with \Refeq{\ref{eq-thm-prim-comp-pf-AS}}, and by Proposition \ref{prop-proket-vs-ket-adj}, we see that
    \[
        H^i(X_\ket, \bL) \cong H^i(X_\proket, \bL \otimes_{\bF_p} \widehat{\cO}_{X_\proket}^\flat)^{\Phi - \Id} \cong \bF_p^r
    \]
    and
    \[
        H^i(X_\ket, \bL) \otimes_{\bF_p} (k^{+a} / p) \cong H^i\bigl(X_\ket, \bL \otimes_{\bF_p} (\cO_{X_\ket}^+ / p)\bigr)^a,
    \]
    as desired.
\end{proof}

\begin{proof}[Proof of the remaining statements of Theorem \ref{thm-prim-comp}]
    It remains to show that, if $X$ is as in Example \ref{ex-log-adic-sp-ncd}, then $H^i(X_\ket, \bL) = 0$ for $i > 2\dim(X)$.  By Theorem \ref{thm-prim-comp}\Refenum{\ref{thm-prim-comp-2}}, it suffices to show that $H^i\bigl(X_\ket, \bL \otimes_{\bF_p} (\cO_{X_\ket}^+ / p)\bigr)^a = 0$, for $i > 2\dim(X)$.  Note that, in Example \ref{ex-log-adic-sp-ncd-chart}, since $k$ is algebraically closed, $X$ \emph{analytic locally} admits smooth toric charts $X \to \bD^n$.  Hence, by the same argument as in the proof of \cite[\aLem 5.3]{Scholze:2013-phtra}, all the \'etale coverings $\{ V_j^{(1)} \}_{j = 1}^m$ in Lemma \ref{lem-Scholze-cov} can be chosen to be analytic coverings.  Let $\lambda: X_\ket \to X_\an$ denote the natural projection of sites.  By Proposition \ref{prop-coh-tor-chart}, $R^j\lambda_*\bigl(\bL \otimes_{\bF_p} (\cO_{X_\ket}^+ / p)\bigr)^a = 0$, for all $j > \dim(X)$.  Since the cohomological dimension of $X_\an$ is bounded by $\dim(X)$, by \cite[\aProp 2.5.8]{deJong/vanderPut:1996-ecras}, the desired vanishing follows.  \Pth{This is essentially the same argument as in the proof of \cite[\aLem 5.9]{Scholze:2013-phtra}.}
\end{proof}

\subsection{{$p$}-adic local systems}\label{sec-lisse}

\begin{defn}\label{def-ket-lisse}
    Let $X$ be a locally noetherian fs log adic space.
    \begin{enumerate}
        \item A \emph{$\bZ_p$-local system} on $X_\ket$, also called a \emph{lisse $\bZ_p$-sheaf} on $X_\ket$, is an inverse system of $\bZ / p^n$-modules $\bL = (\bL_n)_{n \geq 1}$ on $X_\ket$ such that each $\bL_n$ is a locally constant sheaf which are locally \Pth{on $X_\ket$} associated with finitely generated $\bZ / p^n$-modules, and such that the inverse system is isomorphic in the pro-category to an inverse system in which $\bL_{n + 1} / p^n \cong \bL_n$.

        \item A \emph{$\bQ_p$-local system} \Pth{or \emph{lisse $\bQ_p$-sheaf}} on $X_\ket$ is an object of the stack associated with the fibered category of isogeny lisse $\bZ_p$-sheaves.
    \end{enumerate}
\end{defn}

\begin{defn}\label{def-proket-loc-syst}
    Let $X$ be a locally noetherian fs log adic space.  Let
    \[
        \widehat{\bZ}_p := \varprojlim_n (\bZ/p^n)
    \]
    as a sheaf of rings on $X_\proket$, and let
    \[
        \widehat{\bQ}_p := \widehat{\bZ}_p[\tfrac{1}{p}].
    \]
    A \emph{$\widehat{\bZ}_p$-local system} on $X_\proket$ is a sheaf of $\widehat{\bZ}_p$-modules on $X_\proket$ that is locally \Pth{on $X_\proket$} isomorphic to $L \otimes_{\bZ_p} \widehat{\bZ}_p$ for some finitely generated $\bZ_p$-modules $L$.  The notion of \emph{$\widehat{\bQ}_p$-local system} on $X_\proket$ is defined similarly.
\end{defn}

\begin{lem}\label{lem-proket-lisse}
    Let $X$ be a locally noetherian fs log adic space over $\Spa(\bQ_p, \bZ_p)$.  Let $\upsilon: X_\proket \to X_\ket$ denote the natural projection of sites.
    \begin{enumerate}
        \item\label{lem-proket-lisse-1}  The functor
            \[
                \bL = (\bL_n)_{n \geq 1} \mapsto \widehat{\bL} := \varprojlim_n \upsilon^{-1}(\bL_n)
            \]
            is an equivalence of categories from the category of $\bZ_p$-local systems on $X_\ket$ to the category of $\widehat{\bZ}_p$-local systems on $X_\proket$.  Moreover, $\widehat{\bL} \otimes_{\widehat{\bZ}_p} \widehat{\bQ}_p$ is a $\widehat{\bQ}_p$-local system.

        \item\label{lem-proket-lisse-2}  For all $i > 0$, we have $R^i\varprojlim_n \upsilon^{-1}(\bL_n) = 0$.
    \end{enumerate}
\end{lem}
\begin{proof}
    Apply Proposition \ref{prop-proket-K-pi-1} and \cite[\aLem 3.18]{Scholze:2013-phtra}.
\end{proof}

\begin{cor}\label{cor-purity-lisse}
    Let $k$, $X$, and $U$ be as in Theorem \ref{thm-purity}.  Let $\bL$ be an \'etale $\bZ_p$-local system on $U_\et$.  Then $\overline{\bL} := \jmath_{\ket, *}(\bL)$ is a Kummer \'etale $\bZ_p$-local system extending $\bL$.  Conversely, any \'etale $\bZ_p$-local system $\overline{\bL}$ on $X_\ket$ is of this form.  In either case, there are canonical isomorphisms
    \[
        H^i(U_{\AC{k}, \et}, \bL) \cong H^i(X_{\AC{k}, \ket}, \overline{\bL}) \cong H^i(X_{\AC{k}, \proket}, \widehat{\overline{\bL}})
    \]
    of finite $\bZ_p$-modules, for each $i \geq 0$, where $\AC{k}$ denotes any algebraic closure of $k$.
\end{cor}
\begin{proof}
    The assertions on $\bL$ and $\overline{\bL}$ in the first four sentences, together with the first isomorphism \Pth{displayed above}, follow from Corollary \ref{cor-purity} by taking limits of $\bZ_p / p^m$-local systems over $m \in \bZ_{\geq 1}$, which is justified by the finite-dimensionality of the cohomology of $\bF_p$-local systems on $X_{\AC{k}, \ket}$ shown in Theorem \ref{thm-prim-comp}.  The second isomorphism follows from Proposition \ref{prop-proket-vs-ket-adj} and Lemma \ref{lem-proket-lisse}\Refenum{\ref{lem-proket-lisse-2}}.  The finiteness of these isomorphic $\bZ_p$-modules follows, again, from Theorem \ref{thm-prim-comp}.
\end{proof}

\begin{cor}\label{cor-lisse-pr}
    Let $f: X \to Y$ be a log smooth morphism of log adic spaces whose log structures are defined by normal crossings divisors $D$ and $E$ of smooth rigid analytic varieties $X$ and $Y$, respectively, as in Example \ref{ex-log-adic-sp-ncd}.  Assume that the underlying morphisms of adic spaces of $f$ and $f|_{X - D}: X - D \to Y - E$ are both proper.  Let $\bL$ be any $\bZ_p$-local system $\bL$ on $X_\ket$.  Then $R^i f_{\ket, *}(\bL)$ is a $\bZ_p$-local system on $Y_\ket$, for each $i$.
\end{cor}
\begin{proof}
    This follows from \cite[\aThm 10.5.1]{Scholze/Weinstein:2020-BLG} and Corollary \ref{cor-purity-lisse}.
\end{proof}

The combination of pullbacks of $\widehat{\bQ}_p$-local systems and completed structure sheaves under strict closed immersions can be described as follows:
\begin{lem}\label{lem-Q-p-loc-cl-imm}
    Let $\imath: Z \to X$ be a strict closed immersion of locally noetherian fs log adic spaces over $\Spa(\bQ_p, \bZ_p)$.  Let $\widehat{\bL}$ be a $\widehat{\bQ}_p$-local system on $X_\proket$.  Then we have a canonical isomorphism
    \[
    \begin{split}
        & (\widehat{\bL} \otimes_{\widehat{\bQ}_p} \widehat{\cO}_{X_\proket})(U) \otimes_{\widehat{\cO}_{X_\proket}(U)} \widehat{\cO}_{Z_\proket}(U \times_X Z) \\
        & \Mi \bigl(\imath_\proket^{-1}(\widehat{\bL}) \otimes_{\widehat{\bQ}_p} \widehat{\cO}_{Z_\proket}\bigr)(U \times_X Z),
    \end{split}
    \]
    for each log affinoid perfectoid object $U$ of $X_\proket$.
\end{lem}
\begin{proof}
    By Lemma \ref{lem-log-aff-perf-cl-imm}, $U \times_X Z$ is a log affinoid perfectoid object of $Z_\proket$, and the natural morphism $\widehat{\cO}_{X_\proket} \to \imath_{\proket, *}(\widehat{\cO}_{Z_\proket})$ induces a surjective homomorphism $\widehat{\cO}_{X_\proket}(U) \to \widehat{\cO}_{Z_\proket}(U \times_X Z)$.  By Theorem \ref{thm-loc-syst-vs-proj}, it suffices to prove the lemma by replacing $U$ with some log affinoid perfectoid object $V$ of $X_\proket$ over $U$ such that $\widehat{\bL}|_V$ is trivial, in which case the assertion is clear.
\end{proof}

Finally, let us define and study the notion of unipotent and quasi-unipotent geometric monodromy actions along a normal crossings divisor.  Let $\imath: D \to X$ and $k$ be as in Example \ref{ex-log-adic-sp-ncd}, with $\jmath: U := X - D \to X$ the complementary open immersion.  Let $\bL$ be a $\bQ_p$-local system on $X_\ket$.

\begin{defn}\label{def-unip-qunip-monod}
    With $k$, $X$, $D$, and $\bL$ as above, we say that $\bL|_{U_\et}$ has \emph{unipotent} \Pth{\resp \emph{quasi-unipotent}} \emph{geometric monodromy along $D$} if $\pi_1^\ket\bigl(X(\xi), \widetilde{\xi}\bigr)$ acts unipotently \Pth{\resp quasi-unipotently} on the stalk $\bL_{\widetilde{\xi}}$, for each log geometric points $\widetilde{\xi}$ of $X$ lying above each geometric point $\xi$ of $D$, where the log structure of the strict localization $X(\xi)$ is pulled back from $X$, as in Proposition \ref{prop-fket-str-loc}.  By abuse of language, when there is no risk of confusion, we shall also say that $\bL$ has unipotent \Pth{\resp quasi-unipotent} geometric monodromy along $D$, without writing $\bL|_{U_\et}$.
\end{defn}

\begin{exam}\label{ex-ncd-monod}
    Suppose that $\{ D_j \}_{j \in I}$ is the set of irreducible components of $D$ \Pth{see \cite{Conrad:1999-icrs}}.  For each $J \subset I$, suppose moreover that $X_J := X \cap \bigl(\cap_{j \in J} \, D_j\bigr)$ is smooth and geometrically connected, and consider the fs log adic spaces $U_J$ and $U_J^\partial$ introduced in Example \ref{ex-log-adic-sp-ncd-strict-cl-imm}, together with a canonical morphism $\varepsilon_J^\partial: U_J^\partial \to U_J$ \Pth{whose underlying morphism of adic spaces is a canonical isomorphism} and a strict immersion $\imath_J^\partial: U_J^\partial \to X$.  Note that the log structure of $U_J$ is trivial, while the one of $U_J^\partial$ is pulled back from $X_J$.  We shall simply denote the underlying adic space of $U_J^\partial$ by $U_J$.  By construction, $X$ is set-theoretically the disjoint union of such locally closed subspaces $U_J$.  At each geometric point $\xi = \Spa(l, l^+)$ of $U_J$ \Pth{and hence also of $U_J^\partial$}, by projection to factors of polydiscs as in Examples \ref{ex-log-adic-sp-ncd} and \ref{ex-log-adic-sp-ncd-strict-cl-imm}, we have locally a strict morphism from $U_J^\partial$ to $s = (\Spa(k, \cO_k), \bZ_{\geq 0}^J)$ as in Example \ref{ex-fund-grp-log-pt-mor}, which is the restriction of a strict morphism from a neighborhood of $\xi$ in $X$ to a neighborhood of $s$ in $\bD^{|J|}$ \Pth{with its canonical log structure defined as in Example \ref{ex-log-adic-sp-disc}}.  As result, by Corollary \ref{cor-fund-grp-log-pt}, we have compatible isomorphisms
    \begin{equation}\label{eq-ex-ncd-monod-char}
        \bZ_{\geq 0}^J \Mi \overline{\cM}_{X, \xi} \Mi \overline{\cM}_{U_J^\partial, \xi}
    \end{equation}
    and
    \begin{equation}\label{eq-ex-ncd-monod-fund-grp}
    \begin{split}
        & \pi_1^\ket\bigl(U_J^\partial(\xi)\bigr) \cong \Hom\bigl(\overline{\cM}_{U_J^\partial, \xi}^\gp, \widehat{\bZ}'(1)\bigr) \\
        & \Mi \pi_1^\ket\bigl(X(\xi)\bigr) \cong \Hom\bigl(\overline{\cM}_{X, \xi}^\gp, \widehat{\bZ}'(1)\bigr) \Mi \Gamma^J := \bigl(\widehat{\bZ}'(1)\big)^J
    \end{split}
    \end{equation}
    \Pth{with $(l)$ omitted from the notation of $\widehat{\bZ}'(1)(l)$, whose operations will be denoted multiplicatively}.  Therefore, any $\bZ_p$-local system on $X(\xi)_\ket$ is equivalent to a $\bZ_p$-local system on $U_J^\partial(\xi)_\ket$, which is in turn equivalent to a \Pth{trivial} $\bZ_p$-local system on $U_J(\xi)_\et$ with $\Gamma^J$-action.  \Pth{The analogous statement for $\bQ_p$-local systems follows.}  Thus, in Definition \ref{def-unip-qunip-monod}, the local system $\bL$ on $X_\ket$ has unipotent \Pth{\resp quasi-unipotent} geometric monodromy along $D$ if and only if, for each $J \subset I$ and each geometric point $\xi$ of $U_J$, the action of $\pi_1^\ket\bigl(X(\xi)\bigr) \cong \Gamma^J$ on $\bL_\xi$ is unipotent \Pth{\resp quasi-unipotent}, and this property depends only the pullback of $\bL$ to $U_J^\partial(\xi)_\ket$.
\end{exam}

\begin{lem}\label{lem-unip-qunip-monod-irred}
    In Definition \ref{def-unip-qunip-monod}, it suffices to verify the condition for geometric points $\xi$ of $X$ lying above the smooth locus of $D$.  \Pth{That is, $\xi$ does not lie on the intersections, including self-intersections, of irreducible components of $D$.}
\end{lem}
\begin{proof}
    Since Definition \ref{def-unip-qunip-monod} requires only strict localizations of $X$, we may replace $k$ with a complete algebraic closed extension.  Moreover, up to \'etale localization, we may assume that $X$ is affinoid and admits a smooth toric chart $X \to \bD^n$ as in Example \ref{ex-log-adic-sp-ncd-chart}, with the log structure induced by maps $\bZ_{\geq 0}^n \to \cM_X(X) \to \cO_X(X)$ sending the $i$-th standard basis element $e_i$ to the images of the $i$-th coordinate $T_i$ of $\bD^n$.  Consider the tower $\cdots \to X_m \to \cdots \to X$ defined by the toric chart $X \to \bD^n$ as in Section \ref{sec-toric-chart} \Pth{with $P = \bZ_{\geq 0}^n$}, with Galois group $\Gamma \cong \bigl(\widehat{\bZ}'(1)\big)^n$.  Up to further \'etale localization, we may assume that the subspace of $X$ defined by $T_i = 0$ is either empty or irreducible.  Then, in the setting of Example \ref{ex-ncd-monod}, we may identify $I$ with a subset of $\{ 1, \ldots, n \}$, with irreducible components $D_j$ of $D$ defined by $T_j = 0$, for $j \in I$.  In this case, if $J' \subset J \subset I$, then we have canonical projections $\bZ_{\geq 0}^n \Surj \bZ_{\geq 0}^J \Surj \bZ_{\geq 0}^{J'}$ which induce inclusions $\Gamma^{J'} \Em \Gamma^J \Em \Gamma$, by \Refeq{\ref{eq-ex-ncd-monod-char}} and \Refeq{\ref{eq-ex-ncd-monod-fund-grp}}.  Let $\xi$ and $\xi'$ be any geometric points of $U_J$ and $U_{J'}$, respectively.  By pulling back the tower above to $X(\xi)$ and $X(\xi')$, respectively, and by Proposition \ref{prop-fket-str-loc}, we can identify the above inclusions $\Gamma^J \Em \Gamma$ and $\Gamma^{J'} \Em \Gamma$ with homomorphisms $\pi_1^\ket\bigl(X(\xi)\bigr) \to \Gamma$ and $\pi_1^\ket\bigl(X(\xi')\bigr) \to \Gamma$.  As a result, we obtain an inclusion $\pi_1^\ket\bigl(X(\xi')\bigr) \Em \pi_1^\ket\bigl(X(\xi)\bigr)$, which can be canonically identified with the above inclusion $\Gamma^{J'} \Em \Gamma^J$ above inclusion inside $\Gamma$, for any $\xi$ and $\xi'$ as above.  Since $U_J$ is contained in the closure $X_{J'}$ of $U_{J'}$, every Kummer \'etale neighborhood of $\xi$ admits the lifting of \emph{some} geometric point $\xi'$ of $X_{J'}$.  Thus, since each $\bZ_p$-local system is trivialized by some inverse system of such neighborhoods, and since each $\bQ_p$-local system is \Pth{by definition} an isogeny class of $\bZ_p$-local systems, if $\pi_1^\ket\bigl(X(\xi')\bigr)$ acts unipotently \Pth{\resp quasi-unipotently} on $\bL|_{\xi'}$, for \emph{all} geometric points $\xi'$ of $U_{J'}$, then the subgroup $\Gamma^{J'}$ of $\Gamma^J \cong \pi_1^\ket\bigl(X(\xi)\bigr)$ acts unipotently \Pth{\resp quasi-unipotently} on $\bL_\xi$.  Since $\Gamma^J \cong \prod_{j \in J} \Gamma^{\{ j \}}$ is generated by $\Gamma^{\{ j \}}$, for $j \in J$; and since the smooth locus of $D$ is \Pth{set-theoretically} $\cup_{j \in I} \, U_{\{ j \}}$, the lemma follows.
\end{proof}

\begin{rk}\label{rk-unip-qunip-monod-alg}
    Suppose that $X$, $D$, and $\bL$ is the analytification of $X_0$, $D_0$, and $\bL_0$, respectively, where $D_0$ is a normal crossings divisor on a smooth algebraic variety $X_0$ over $k$, and where $\bL_0$ is an \'etale $\bQ_p$-local system on $X_0$.  Since the construction of standard Kummer \'etale covers are compatible with analytification, by comparing the constructions in Proposition \ref{prop-ket-std} and \cite[\aProp 3.2]{Illusie:2002-fknle}, we have obvious analogues of Definition \ref{def-unip-qunip-monod}, Example \ref{ex-ncd-monod}, and Lemma \ref{lem-unip-qunip-monod-irred} in the algebraic setting, which are all compatible with analytification.
\end{rk}

\begin{rk}\label{rk-unip-qunip-monod-alg-cl}
    In Remark \ref{rk-unip-qunip-monod-alg}, for each irreducible component of $D_0$, its generic point is a point of codimension one, and hence the strict localization $X_0(\xi_0)$ at any geometric point $\xi_0 = \Spec(l)$ above such a generic point is the spectrum of a strictly local ring $R$ with residue field $l$.  Let $K := \Frac(R)$, let $\AC{K}$ be any algebraic closure of $K$, and let $K^\tr$ be the maximal tamely ramified extension of $K$ in $\AC{K}$.  Let $\eta_0 := \Spec(\AC{K})$.  Then $\bL_0|_{\eta_0}$ is naturally a representation of $\Gal(\AC{K} / K)$.  As explained in \cite[\aEx 4.7(b)]{Illusie:2002-fknle}, $\pi_1^\ket\bigl(X_0(\xi_0)\bigr)$ is canonically isomorphic to the tame inertia group, which is $\Gal(K^\tr / K) \cong \widehat{\bZ}'(1)$ in this case; and the induced isomorphism $\pi_1^\ket\bigl(X_0(\xi_0)\bigr) \cong \widehat{\bZ}'(1)$ can be canonically identified with $\pi_1^\ket\bigl(X_0(\xi_0)\bigr) \cong \Hom\bigl(\overline{\cM}_{X_0, \xi_0}^\gp, \widehat{\bZ}'(1)\bigr) \cong \widehat{\bZ}'(1)$, with the last isomorphism induced by $\overline{\cM}_{X_0, \xi_0} \cong \bZ_{\geq 0}$.  Since $\bL_0|_{U_{0, \et}}$ extends over $X_{0, \ket}$, the action of $\Gal(\AC{K} / K)$ on $\bL_0|_{\eta_0}$ factors through $\Gal(K^\tr / K)$.  Note that, in the algebraic setting, if $\xi_0$ specializes to some geometric point $\xi_0'$ of $X_0$, then we have a canonical morphism $X_0(\xi_0) \to X_0(\xi_0')$, and hence a canonical homomorphism $\pi_1^\ket\bigl(X_0(\xi_0)\bigr) \to \pi_1^\ket\bigl(X_0(\xi_0')\bigr)$.  When $\xi_0'$ does not lie on any other irreducible component of $D_0$, this last homomorphism can be canonically identified with the identity homomorphism of $\widehat{\bZ}'(1)$.  Since every geometric point of the smooth locus of $D_0$ is some such $\xi_0'$, by the algebraic analogue of Lemma \ref{lem-unip-qunip-monod-irred}, we see that $\bL_0$ has unipotent \Pth{\resp quasi-unipotent} monodromy along $D_0$ \Pth{by the algebraic analogue of Definition \ref{def-unip-qunip-monod}} if and only if, for each $\xi_0$ as above, the induced action of $\Gal(K^\tr / K) \cong \widehat{\bZ}'(1)$ is unipotent \Pth{\resp quasi-unipotent}.  In fact, this last condition is a more classical definition for schemes, whose formulation does not rely on log geometry at all.  Nevertheless, our Definition \ref{def-unip-qunip-monod} has the advantage of not relying on the notion of generic points \Pth{or specializations}.
\end{rk}

\subsection{Quasi-unipotent nearby cycles}\label{sec-nearby}

In this subsection, as an application of our results, we reformulate Beilinson's ideas \Pth{see \cite{Beilinson:1987-hgps}; \Refcf{} \cite{Reich:2010-nbhgp}} and define the unipotent and quasi-unipotent nearby cycles in the rigid analytic setting.

Let $k$ be any field of characteristic zero, and let $\AC{k}$ be any algebraic closure of $k$.  Let $\bG_m := \Spec(k[z, z^{-1}])$ be the multiplicative group scheme over $k$.  Let $\AC{k}$ be any fixed algebraic closure of $k$.  Then $\pi_1(\bG_m, 1) \cong \pi_1(\bG_{m, \AC{k}}, 1) \rtimes \Gal(\AC{k} / k)$, and $\pi_1(\bG_{m, \AC{k}}, 1) \cong \widehat{\bZ}(1)$ as $\Gal(\AC{k} / k)$-modules.  For each $r \geq 1$, let $\bJ_r$ denote the rank $r$ unipotent \'etale $\bZ_p$-local system on $\bG_m$ defined by the representation of $\pi_1(\bG_{m, \AC{k}}, 1)$ on $\bZ_p^r$ such that a topological generator $\gamma \in \pi_1(\bG_{m, \AC{k}}, 1)$ acts as a principal unipotent matrix $J_r$ and such that $\Gal(\AC{k} / k)$ acts diagonally on $\bZ_p^r$ and trivially on $\ker(J_r - 1)$.  \Pth{As in Example \ref{ex-fund-grp-ket-0-partial}, the local system thus defined is independent of the choice of $\gamma$ up to isomorphism.}  There is a natural inclusion $\bJ_r \Em \bJ_{r + 1}$, together with a projection $\bJ_{r + 1} \to \bJ_r(-1)$ such that the composition $\bJ_r \to \bJ_r(-1)$ is given by the monodromy action.  For each $m \geq 1$, let $[m]$ denote the $m$-th power homomorphism $[m]$ of $\bG_m$, and let $\bK_m := [m]_*(\bZ_p)$.  When $m \mid m'$, there is a natural inclusion $\bK_m \Em \bK_{m'}$ \Pth{defined by adjunction}.

Now let $k$ be a nontrivial nonarchimedean field, and let $k^+ = \cO_k$.  We shall denote the analytifications of the above objects and morphisms to $\bG_m^\an$, and their further pullbacks to $\bD^\times = \bD - \{ 0 \}$, by the same symbols.

Let $X$ be a rigid analytic variety over $k$.  Let $f: X \to \bD$ be a morphism over $k$ that induces an open immersion $\jmath: U := f^{-1}(\bD^\times) \to X$ and a closed immersion $\imath: f^{-1}(0) \to X$ such that $D := f^{-1}(0)_\red$ \Pth{the reduced subspace} is a \emph{normal crossings divisor}, so that $X$ is equipped with the fs log structure defined by $D \Em X$, as in Example \ref{ex-log-adic-sp-ncd}.  Note that $\bigr(f^{-1}(0)\bigr)_\et \cong D_\et$.  Let $U$ be equipped with the trivial log structure, with an open immersion $\jmath: U \to X$.  Let $D^\partial$ be the adic space $D$ equipped with the log structure pulled back from $X$, with a canonical morphism $\varepsilon^\partial: D^\partial \to D$ and a strict closed immersion $\imath^\partial: D^\partial \to X$.

\begin{defn}\label{def-nearby}
    In the above setting, for any given $\bQ_p$-local system $\bL$ on $U_\et \cong U_\ket$, its sheaf of \emph{unipotent nearby cycles} \Pth{with respect to $f$} is
    \[
          R\Psi_f^\unip(\bL) := R\varepsilon^\partial_{\et, *} \, \Bigl(\varinjlim_r \, \imath^{\partial, -1}_\ket \, \jmath_{\ket, *}\bigl(\bL \otimes_{\bZ_p} f_\et^{-1}(\bJ_r)\bigr)\Bigr),
    \]
    and its sheaf of \emph{quasi-unipotent nearby cycles} is
    \[
          R\Psi^\qunip_f(\bL) := R\varepsilon^\partial_{\et, *} \, \Bigl(\varinjlim_{m, r} \, \imath^{\partial, -1}_\ket \, \jmath_{\ket, *}\bigl(\bL \otimes_{\bZ_p} f_\et^{-1}(\bK_m) \otimes_{\bZ_p} f_\et^{-1}(\bJ_r)\bigr)\Bigr).
    \]
\end{defn}

Suppose that $\{ D_j \}_{j \in I}$ is the set of irreducible components of $D$ \Pth{see \cite{Conrad:1999-icrs}}, so that $f^{-1}(0) = \sum_{j \in I} \, n_j D_j$ \Pth{as Cartier divisors on $X$; see \cite[\aLec 5.3, especially \aProp 5.3.4]{Scholze/Weinstein:2020-BLG}}, for some integers $n_j \geq 1$ giving the multiplicities of $D_j$.  For each $J \subset I$, let $X_J$, $\varepsilon_J^\partial: U_J^\partial \to U_J$, and $\imath_J^\partial: U_J^\partial \to X$ be as in Example \ref{ex-ncd-monod}.  Given any geometric point $\xi = \Spa(l, l^+)$ of $U_J$ \Pth{and hence also of $U_J^\partial$}, let $\varepsilon^\partial_{J, \xi}: U_J^\partial(\xi) \to U_J(\xi)$ denote the pullback of $\varepsilon_J^\partial$ to $U_J(\xi)$.  Let $\Gamma^J \cong \bigl(\widehat{\bZ}(1)\big)^J$ be as in \Refeq{\ref{eq-ex-ncd-monod-fund-grp}}.  Then, by Lemma \ref{lem-ket-to-et-stalk} and the explanations in Example \ref{ex-ncd-monod}, we have a canonical isomorphism $R^i\varepsilon^\partial_{J, \xi, \et, *}(\bM) \cong H^i(\Gamma^J, \bM)$, for each $i \geq 0$.

Let $0^\partial$ and $\widetilde{0}^\partial$, and the $(\bZ / n)$-local systems $\bJ_{r, n}^\partial$ and $\bK_{m, n}^\partial$ on $0^\partial_\ket$ defined by representations of $\pi_1^\ket(0^\partial, \widetilde{0}^\partial) \cong \widehat{\bZ}(1) \rtimes \Gal(\AC{k} / k)$, be as in Example \ref{ex-fund-grp-ket-0-partial}.  By taking limits over $n \in p^{\bZ_{\geq 1}}$, we obtain $\bZ_p$-local systems $\bJ_r^\partial$ and $\bK_m^\partial$ on $0^\partial_\ket$, which can be identified with the pullbacks of $\overline{\bJ}_r := \jmath_{\ket, *}(\bJ_r)$ and $\overline{\bK}_m := \jmath_{\ket, *}(\bK_m)$, respectively.  By pulling back $f: X \to \bD$ \Pth{as a morphism of fs log adic spaces}, we obtain a canonical morphism $f_\xi: U_J^\partial(\xi) \to 0^\partial$ for any $\xi$ and $\widetilde{\xi}$ as in the last paragraph, and the induced homomorphism $\pi_1^\ket(U_J^\partial(\xi), \widetilde{\xi}) \to \pi_1^\ket(0^\partial, \widetilde{0}^\partial)$ can be identified with the composition of $\Gamma^J \cong \bigl(\widehat{\bZ}(1)\bigr)^J \to \widehat{\bZ}(1): (x_j)_{j \in J} \mapsto \sum_{j \in J} \, n_j x_j$ with the canonical homomorphism $\widehat{\bZ}(1) \Em \widehat{\bZ}(1) \rtimes \Gal(\AC{k} / k)$.  Let $\gamma_j$ be any element of the $j$-th factor of $\Gamma^J \cong \bigl(\widehat{\bZ}(1)\bigr)^J$ that is mapped to $n_j \gamma$ in $\widehat{\bZ}(1)$.  Then $\gamma_j$ acts by $J_r^{n_j}$ on the rank $r$ local system $\bigl(f^{-1}(\overline{\bJ}_r)\bigr)|_{U_J^\partial(\xi)} \cong f_\xi^{-1}(\bJ_r^\partial)$.

For each $\bQ_p$-local system $\bM$ on $U_J^\partial(\xi)_\ket$, let us denote by $W$ a formal variable on which $\gamma_j^{-1}$ acts by $W \mapsto W + n_j$, and write $\bM[W] = \varinjlim_r \, \bigl(\bM[W]^{\leq r - 1}\bigr)$, where the superscript \Qtn{$\leq r - 1$} means \Qtn{up to degree $r - 1$}.  Note that, by matching a standard basis of $\bJ_r^\partial$ with binomial monomials up to degree $r - 1$ in $W$, as in the proof of \cite[\aLem 2.10]{Liu/Zhu:2017-rrhpl}, we have $\bM[W]^{\leq r - 1} \cong \bM \otimes_{\bZ_p} f_\xi^{-1}(\bJ_r^\partial)$.
\begin{lem}\label{lem-nearby-coh-stab}
    Suppose there exists some $j_0 \in J$ such that $\gamma_{j_0}$ acts quasi-unipotently \Pth{\ie, a positive power of $\gamma_{j_0}$ acts unipotently} on $\bM$.  Then the local systems $H^i\bigl(\Gamma^J, \bM[W]^{\leq r - 1}\bigr)$ stabilize as $r \rightarrow \infty$, and hence the direct limit $H^i\bigl(\Gamma^J, \bM[W])$ exists as a $\bQ_p$-local system, for each $i \geq 0$.  When $J = \{ j_0 \}$ is a singleton, $H^i\bigl(\Gamma^{\{ j_0 \}}, \bM[W])$ vanishes when $i \neq 0, 1$; is canonically isomorphic to the maximal subsheaf of $\bM$ on which $\Gamma^{\{ j_0 \}}$ acts unipotently, when $i = 0$; is finite-dimensional when $i = 1$; and is zero when $i = 1$ and $\gamma_j$ acts unipotently.
\end{lem}
\begin{proof}
    By the Hochschild--Serre spectral sequence, by first considering the action of $[m](\Gamma^{\{ j_0 \}})$ for some $m$, and then the induced action of the finite quotient $\Gamma^{\{ j_0 \}} / [m](\Gamma^{\{ j_0 \}}) \cong \bZ / m \bZ$, and then the induced action of $\Gamma^J / \Gamma^{\{ j_0 \}} \cong \Gamma^{J - \{ j_0 \}}$, it suffices to treat the special case where $J = \{ j_0 \}$ is a singleton and $\gamma_{j_0}$ acts unipotently.  Then the lemma is reduced to its last statement, which follows from the same argument as in the proof of \cite[\aLem 2.10]{Liu/Zhu:2017-rrhpl}, by matching a basis of $\bJ_r$ with binomial monomials up to degree $r - 1$ in $W$.
\end{proof}

\begin{lem}\label{lem-nearby-ket-cov}
    Let $\bL$ be a $\bQ_p$-local system on $X_\ket$ such that $\bL|_{U_\et}$ has quasi-unipotent geometric monodromy along $D$ \Pth{as in Definition \ref{def-unip-qunip-monod}}.  For each integer $m \geq 1$, consider the canonical morphism $[m]: \bD \to \bD$ induced by sending the standard coordinate of $\bD$ to its $m$-th power, whose pullback under $f: X \to \bD$ is a finite Kummer \'etale cover $g_m: X_m \to X$, which induces $f_m: X_m \Mapn{g_m} X \Mapn{f} \bD$ by composition.  Let $D_m$ denote the reduced subspace of $X_m \times_X D$ \Pth{in the category of adic spaces}, which is canonical isomorphic to $D$ via the second projection, and let $U_m := f_m^{-1}(\bD^\times) = g_m^{-1}(U) = X_m - D_m$.  Then there exists $m_0 \geq 1$ such that $R\Psi^\qunip_f(\bL|_U) \cong R\Psi^\unip_{f_m}\bigl(g_m^{-1}(\bL)|_{U_m}\bigr)$ over $D_\et$, whenever $m_0 | m$.
\end{lem}
\begin{proof}
    For each $m \geq 1$, let $D_m^\partial$ denote the adic space $D_m$ equipped with the log structure pulled back from $X_m$.  Let $\jmath_m: U_m \to X_m$, $\imath_m^\partial: D_m^\partial \to X_m$, and $\varepsilon^\partial_m: D_m^\partial \to D_m$ denote the canonical morphisms.  Then
    \[
        \bL|_U \otimes_{\bZ_p} f_\et^{-1}(\bK_m) \cong (g_m|_{U_m})_{\et, *}(\bL|_{U_m})
    \]
    over $U_\et$, and
    \[
    \begin{split}
        & R\varepsilon^\partial_{\et, *} \, \imath_\ket^{\partial, -1} \, \jmath_{\ket, *}\bigl(\bL|_U \otimes_{\bZ_p} f_\et^{-1}(\bK_m) \otimes_{\bZ_p} f_\et^{-1}(\bJ_r)\bigr) \\
        & \cong R\varepsilon^\partial_{m, \et, *} \, \imath_{m, \ket}^{\partial, -1}\bigl((\bL|_{X_m}) \otimes_{\bZ_p} f_\ket^{-1}(\overline{\bJ}_r)\bigr)
    \end{split}
    \]
    over $D_\et$, by Proposition \ref{prop-ket-dir-im-ex} and Lemma \ref{lem-cl-imm-ket-mor}.  Since $\bL$ has quasi-unipotent geometric monodromy along $D$, there exists some $m_0 \geq 1$ such that $\bL|_{U_m}$ has unipotent geometric monodromy along $D_m$, whenever $m_0 | m$.

    We claim that, when $m_0 | m$, the canonical morphism
    \[
        R\varepsilon^\partial_{m_0, \et, *} \, \imath_{m_0, \ket}^{\partial, -1}\bigl(\bL|_{X_{m_0}}) \otimes_{\bZ_p} f_\ket^{-1}(\overline{\bJ}_r)\bigr) \to R\varepsilon^\partial_{m, \et, *} \, \imath_{m, \ket}^{\partial, -1}\bigl(\bL|_{X_m}) \otimes_{\bZ_p} f_\ket^{-1}(\overline{\bJ}_r)\bigr)
    \]
    induced by $\bK_{m_0} \Em \bK_m$ is an isomorphism for all sufficiently large $r$ \Pth{depending on $m_0$ and $m$}.  Given this claim, for all $m$ divisible by $m_0$, we have
    \[
    \begin{split}
        & R\Psi^\qunip_f(\bL) \cong R\varepsilon^\partial_{\et, *} \, \imath^{\partial, -1}_\ket \, \jmath_{\ket, *}\bigl((\bL|_U) \otimes_{\bZ_p} f_\et^{-1}(\bK_m) \otimes_{\bZ_p} f_\et^{-1}(\bJ_r)\bigr) \\
        & \cong R\varepsilon^\partial_{m, \et, *} \, \imath^{\partial, -1}_{m, \ket} \, \jmath_{m, \ket, *}\bigl((\bL|_{U_m}) \otimes_{\bZ_p} f_{m, \et}^{-1}(\bJ_r)\bigr) \cong R\Psi^\unip_{f_m}\bigl(g_m^{-1}(\bL)\bigr)
    \end{split}
    \]
    for all sufficiently large $r$, and the lemma follows.

    It remains to verify the claim.  For this purpose, by \Refeq{\ref{eq-ex-ncd-monod-fund-grp}}, we may pullback to $U_J^\partial(\xi)$, for all nonempty $J \subset I$ and all geometric point $\xi$ of $U_J^\partial$.  By Lemmas \ref{lem-cl-imm-ket-mor} and \ref{lem-nearby-coh-stab}, and by \Refeq{\ref{eq-ex-ncd-monod-fund-grp}} again, it suffices to show that the canonical morphism
    \[
        H^i\bigl([m_0](\Gamma^J), \bL|_{U_J^\partial(\xi)}[W]\bigr) \to H^i\bigl([m](\Gamma^J), \bL|_{U_J^\partial(\xi)}[W]\bigr)
    \]
    is an isomorphism.  By the Hochschild--Serre spectral sequence, we may first compare the cohomology of $[m_0](\Gamma^{\{ j_0 \}})$ and $[m](\Gamma^{\{ j_0 \}})$, for some $j_0 \in J$, which is concentrated in degree zero and gives the full $\bL|_{U_J^\partial(\xi)}$ in both cases, because $\gamma_j^{m_0}$ and $\gamma_j^m$ act unipotently, by assumption.  Then we compare the cohomology groups of $[m_0](\Gamma^{J - \{ j_0 \}})$ and $[m](\Gamma^{J - \{ j_0 \}})$, which coincide as they are related by a Hochschild--Serre spectral sequence in terms of the cohomology of $\bQ_p$-modules with unipotent actions of the finite group $[m_0](\Gamma^{J - \{ j_0 \}}) / [m](\Gamma^{J - \{ j_0 \}})$, and the claim follows.
\end{proof}

\begin{prop}\label{prop-nearby}
    Let $\bL$ be a $\bQ_p$-local system on $X_\ket$ such that $\bL|_{U_\et}$ has quasi-unipotent geometric monodromy along $D$ \Pth{as in Definition \ref{def-unip-qunip-monod}}.  Consider any integer $m \geq 1$ such that $\bL|_{U_m}$ has unipotent geometric monodromy along $D_m$, where $U_m$ and $D_m$ are as in Lemma \ref{lem-nearby-ket-cov}.  Then, for each nonempty $J \subset I$ and each $i \geq 0$, and for each geometric point $\xi$ of $U_J^\partial$, we have
    \[
        R^i\Psi_f^\unip(\bL|_U)|_{U_J^\partial(\xi)} \cong H^i\bigl(\Gamma^J, \bL|_{U_J^\partial(\xi)}[W]\bigr)
    \]
    and
    \[
        R^i\Psi_f^\qunip(\bL|_U)|_{U_J^\partial(\xi)} \cong H^i\bigl([m](\Gamma^J), \bL|_{U_J^\partial(\xi)}[W]\bigr)
    \]
    as $\bQ_p$-local systems on $U_J(\xi)_\et$.

    Consequently, if $D = \bigl(f^{-1}(0)\bigr)_\red$ is smooth over $k$ and if $\bL|_U$ has quasi-unipotent monodromy along $D$, then $R\Psi_f^\unip(\bL|_U)$ is concentrated in degree zero and can be identified with the subsheaf $\bL|_{D^\partial}^{\Utext{unip}}$ of $\bL|_{D^\partial}$ whose pullback to $D^\partial(\xi)$ is the maximal subsheaf on which $\pi_1^\ket(D^\partial(\xi), \widetilde{\xi}) \cong \widehat{\bZ}(1)$ \Pth{as in Example \ref{ex-fund-grp-log-pt-mor}} acts unipotently, for each log geometric point $\widetilde{\xi}$ of $D^\partial$ above each geometric point $\xi$ of $D$; and $R\Psi_f^\qunip(\bL|_U)$ \Pth{which is the same $R\Psi_f^\unip(\bL|_U)$ as above when $\bL|_U$ has unipotent monodromy along $D$} is also concentrated in degree zero and can be identified with the whole $\bL|_{D^\partial}$.
\end{prop}
\begin{proof}
    Combine Lemmas \ref{lem-cl-imm-ket-mor}, \ref{lem-nearby-coh-stab}, and \ref{lem-nearby-ket-cov}.
\end{proof}

\numberwithin{equation}{section}

\appendix

\section{Kiehl's property for coherent sheaves}\label{app-Kiehl}

In this appendix, by adapting the gluing argument in \cite[\aSec 2.7]{Kedlaya/Liu:2015-RPH} and by using \cite[\aThm 2.5]{Huber:1994-gfsra}, we establish Kiehl's property for coherent sheaves on \Pth{possibly nonanalytic} noetherian adic spaces.  By combining this with results in \cite[\aSec 8.2]{Kedlaya/Liu:2015-RPH} and \cite[\aSecs 1.3--1.4]{Kedlaya:2019-AWS}, we also state some versions of Tate's sheaf property and Kiehl's gluing property for adic spaces that are either locally noetherian, or analytic and \emph{stably adic}.  \Pth{We will review the definition below.}

Recall the following definition from \cite[\aDef 1.3.7]{Kedlaya/Liu:2015-RPH}:
\begin{defn}\label{def-gluing-dat}
    By a \emph{gluing diagram}, we will mean a commuting diagram of ring homomorphisms
    \[
        \xymatrix{ {R} \ar[r] \ar[d] & {R_1} \ar[d] \\
        {R_2} \ar[r] & {R_{12}} }
    \]
    such that the $R$-module sequence
    \[
        0 \to R \to R_1 \oplus R_2 \to R_{12} \to 0,
    \]
    in which the last nontrivial arrow is the difference between the given homomorphisms, is exact.  By a \emph{gluing datum} over this diagram, we mean a datum consisting of modules $M_1$, $M_2$, and $M_{12}$ over $R_1$, $R_2$, and $R_{12}$, respectively, equipped with isomorphisms $\psi_1: M_1 \otimes_{R_1} R_{12} \Mi M_{12}$ and $\psi_2: M_2 \otimes_{R_2} R_{12} \Mi M_{12}$.  We say such a gluing datum is \emph{finite} if the modules are finite over the respective rings.
\end{defn}
Given a gluing datum as above, let $M := \ker(\psi_1 - \psi_2: M_1 \oplus M_2 \to M_{12})$.  There are natural morphisms $M \to M_1$ and $M \to M_2$ of $R$-modules, which induce maps $M \otimes_R R_1 \to M_1$ and $M \otimes_R R_2 \to M_2$, respectively.

The following is \cite[\aLem 1.3.8]{Kedlaya/Liu:2015-RPH}:
\begin{lem}\label{lem-Kiehl-surj}
    Consider a finite gluing datum for which $M \otimes_R R_1 \to M_1$ is surjective.  Then the following are true.
    \begin{enumerate}
        \item The morphism $\psi_1 - \psi_2: M_1 \oplus M_2 \to M_{12}$ is surjective.

        \item The morphism $M \otimes_R R_2 \to M_2$ is also surjective.

        \item There exists a finitely generated $R$-submodule $M_0$ of $M$ such that, for $i = 1, 2$, the morphism $M_0 \otimes_R R_i \to M_i$ is surjective.
    \end{enumerate}
\end{lem}

\begin{lem}\label{lem-Kiehl-descent}
    In the above setting, suppose in addition that $R_i$ is noetherian and that $R_i \to R_{12}$ is flat, for $i = 1, 2$.  Suppose that, for every finite gluing datum, the map $M \otimes_R R_1 \to M_1$ is surjective.  Then, for any finite gluing datum, $M$ is a finitely presented $R$-module, and $M \otimes_R R_1 \to M_1$ and $M \otimes_R R_2 \to M_2$ are bijective.
\end{lem}
\begin{proof}
    Let $M_0$ be as in Lemma \ref{lem-Kiehl-surj}.  Choose a surjection $F \to M_0$ of $R$-modules, with $F$ finite free.  Let $F_1 := F \otimes_R R_1$, $F_2 := F \otimes_R R_2$, and $F_{12} := F \otimes_R R_{12}$.  Let $N := \ker(F \to M)$, $N_1 := \ker(F_1 \to M_1)$, $N_2 := \ker(F_2 \to M_2)$, and $N_{12} := \ker(F_{12} \to M_{12})$.  By Lemma \ref{lem-Kiehl-surj}, we have a commutative diagram
    \begin{equation}\label{eq-lem-Kiehl-descent}
        \xymatrix{ & {0} \ar[d] & {0} \ar[d] & {0} \ar[d] & \\
        {0} \ar[r] & {N} \ar[r] \ar[d] & {N_1 \oplus N_2} \ar[r] \ar[d] & {N_{12}} \ar@{..>}[r] \ar[d] & {0} \\
        {0} \ar[r] & {F} \ar[r] \ar[d] & {F_1 \oplus F_2} \ar[r] \ar[d] & {F_{12}} \ar[r] \ar[d] & {0} \\
        {0} \ar[r] & {M} \ar[r] \ar@{..>}[d] & {M_1 \oplus M_2} \ar[r] \ar[d] & {M_{12}} \ar[d] \ar[r] & {0} \\
        & {0} & {0} & {0} }
    \end{equation}
    with exact rows and columns, excluding the dotted arrows.  Since $R_{12}$ is flat over $R_i$, the sequence
    \[
        0 \to N_i \otimes_{R_i} R_{12} \to F_{12} \to M_{12} \to 0
    \]
    is exact, and hence $N_i \otimes_{R_i} R_{12} \cong N_{12}$.  By hypothesis, $R_i$ is noetherian, and so $N_i$ is finite over $R_i$.  Consequently, $N_1$, $N_2$, and $N_{12}$ form a finite gluing datum as well.  By Lemma \ref{lem-Kiehl-surj} again, the dotted horizontal arrow in \Refeq{\ref{eq-lem-Kiehl-descent}} is surjective.  By diagram chasing, the dotted vertical arrow in \Refeq{\ref{eq-lem-Kiehl-descent}} is also surjective; that is, we may add the dotted arrows to \Refeq{\ref{eq-lem-Kiehl-descent}} while preserving exactness of the rows and columns.  In particular, $M$ is a finitely generated $R$-module.  It follows that $N$ is finitely generated.  This implies that $M$ is finitely presented.

    For $i = 1, 2$, we obtain a commutative diagram
    \[
        \xymatrix{ & {N \otimes_R R_i} \ar[r] \ar[d] & {F_i} \ar[r] \ar@{=}[d] & {M \otimes_R R_i} \ar[r] \ar[d] & {0} \\
        {0} \ar[r] & {N_i} \ar[r] & {F_i} \ar[r] & {M_i} \ar[r] & {0} }
    \]
    with exact rows---the first one is derived from the left column of \Refeq{\ref{eq-lem-Kiehl-descent}} by tensoring with $R_i$ over $R$, while the second one is derived from the middle column of \Refeq{\ref{eq-lem-Kiehl-descent}}.  Since the left vertical arrow is surjective, by the five lemma, the right vertical arrow is injective.  It follows that the map $M \otimes_R R_i \to M_i$ is a bijection, as desired.
\end{proof}

\begin{defn}\label{def-Huber-mor-strict-adic}
    We call a homomorphism of Huber rings $f: A \to B$ \emph{strict adic} if, for one \Pth{and hence every} choice of an ideal of definition $I \subset A$, the image $f(I)$ is an ideal of definition of $B$.  It is clear that a strict adic morphism is adic.
\end{defn}

The following is modeled on \cite[\aLem 2.7.2]{Kedlaya/Liu:2015-RPH}.
\begin{lem}\label{lem-Cartan-fac}
    Let $R_1 \to S$ and $R_2 \to S$ be homomorphisms of complete Huber rings such that their sum $\psi: R_1 \oplus R_2 \to S$ is strict adic.  Then, for any ideal of definition $I_S$ of $S$, there exists some integer $l \geq 1$ such that, for every $n > 0$, every $U \in \GL_n(S)$ with $U - 1 \in \M_n(I^l_S)$ is of the form $\psi(U_1) \, \psi(U_2)$ for some $U_i \in \GL_n(R_i)$, for $i = 1, 2$.
\end{lem}
\begin{proof}
    Since $\psi$ is strict adic, for any ideals of definition $I_1 \subset R_1$ and $I_2 \subset R_2$, we have an ideal of definition $I'_S := \psi(I_1 \oplus I_2) \subset S$.  Choose $l > 0$ such that $I_S^l \subset I'_S$.  Then it suffices to show that every $U \in \GL_n(S)$ with $U - 1 \in \M_n(I'_S)$ is of the form $\psi(U_1) \, \psi(U_2)$ for some $U_i \in \GL_n(R_i)$, for $i = 1, 2$.  Given $U \in \GL_n(S)$ with $U - 1\in \M_n(I'^m_S)$ for some $m > 0$, put $V = U - 1$.  By assumption, we may lift $V$ to a pair $(X, Y) \in \M_n(I^{m}_1) \times \M_n(I_2^{m})$.  Then it is straightforward that the matrix $U' = \psi(1 - X) \, U \, \psi(1 - Y)$ satisfies $U' - 1 \in \M_n(I'^{2m}_S)$.  Hence, we may construct the desired matrices by iterating this construction.
\end{proof}

The following is modeled on \cite[\aLem 2.7.4]{Kedlaya/Liu:2015-RPH}.
\begin{lem}\label{lem-Kiehl-surj-str-adic}
    In the context of Definition \ref{def-gluing-dat} and the paragraph following it, suppose in addition that
    \begin{enumerate}
        \item the Huber rings $R_1$, $R_2$ and $R_{12}$ are complete;

        \item $R_1 \oplus R_2 \to R_{12}$ is strict adic; and

        \item the map $R_2 \to R_{12}$ has a dense image.
    \end{enumerate}
    Then, for $i = 1, 2$, the natural map $M \otimes_R R_i \to M_i$ is surjective.
\end{lem}
\begin{proof}
    Choose sets of generators $\{ m_{1, 1}, \ldots, m_{n, 1} \}$ and $\{ m_{1, 2}, \dots, m_{n, 2} \}$ of $M_1$ and $M_2$, respectively, of the same cardinality.  Then there exist $A, B \in \M_n(R_{12})$ such that $\psi_2(m_{j, 1}) = \sum_i \, A_{ij} \, \psi_1(m_{i, 2})$ and $\psi_1(m_{j, 2}) = \sum_i \, B_{ij} \, \psi_2(m_{i, 1})$, for all $j$.  Since $R_2 \to R_{12}$ has a dense image, by Lemma \ref{lem-Cartan-fac}, there exists $B' \in \M_n(R_2)$ such that $1 + A (B' - B) = C_1 C_2^{-1}$ for some $C_i \in \GL_n(R_i)$, for $i = 1, 2$.  For $j = 1, \ldots, n$, let
    \[
        x_j := (x_{j, 1}, x_{j, 2}) = \Bigl(\sum_i \, (C_1)_{ij} \, m_{i, 1}, \sum_i \, (B' C_2)_{ij} \, m_{i, 2}\Bigr) \in M_1 \times M_2.
    \]
    Then $x_j \in M$, because
    \[
    \begin{split}
        \psi_1(x_{j, 1}) - \psi_2(x_{j, 2}) & = \sum_i \, (C_1 - A B' C_2)_{ij} \, \psi_1(m_{i, 1}) \\ & = \sum_i \, \bigl((1 - A B) C_2\bigr){}_{ij} \, \psi_1(m_{i, 2}) = 0.
    \end{split}
    \]
    For $i = 1, 2$, since $C_i \in \GL_n(R_i)$, we see that $\{ x_{j, i} \}_{j = 1}^n$ generates $M_i$ over $R_i$ as well.  Thus, $M \otimes_R R_i \to M_i$ is surjective, as desired.
\end{proof}

\begin{thm}\label{thm-Kiehl}
    Let $X = \Spa(R, R^+)$ be a noetherian affinoid adic space.  The categories of coherent sheaves on $X$ and finitely generated $R$-modules are equivalent via the global sections functor.
\end{thm}
\begin{proof}
    By \cite[\aLem 2.4.20]{Kedlaya/Liu:2015-RPH}, it suffices to verify Kiehl's gluing property for any simple Laurent covering $\{ \Spa(R_i, R_i^+) \to X \}_{i = 1, 2}$.  In this case, let us write $\Spa(R_{12}, R_{12}^+) = \Spa(R_1, R_1^+) \times_X \Spa(R_2, R_2^+)$, with all Huber pairs completed by our convention.  By the noetherian hypothesis, and by \cite[\aThm 2.5]{Huber:1994-gfsra}, $R$, $R_i$, and $R_{12}$ form a gluing diagram.  Also, $R_i \to R_{12}$ is flat with dense image, for $i = 1, 2$.  Hence, we can finish the proof by applying Lemmas \ref{lem-Kiehl-descent} and \ref{lem-Kiehl-surj-str-adic}.
\end{proof}

Thus, we have the following version of Tate's sheaf property and Kiehl's gluing property \Pth{see \cite[\aDef 2.7.6]{Kedlaya/Liu:2015-RPH}} over certain affinoid adic spaces:
\begin{prop}\label{prop-coh-descent}
    Let $X = \Spa(R, R^+)$ be a noetherian \Pth{\resp analytic} affinoid adic space, and let $M$ be a finite \Pth{\resp finite projective} $R$-module.  Let $\widetilde{M}$ denote the presheaf on $X$ defined by setting $\widetilde{M}(U) = M \otimes_R \cO_X(U)$, for each open subset $U \subset X$.  Then the following are true:
    \begin{enumerate}
        \item\label{prop-coh-descent-1}  The presheaf $\widetilde{M}$ is a sheaf.  Moreover, the sheaf $\widetilde{M}$ is \emph{acyclic} in the sense that $H^i(U, \widetilde{M}) = 0$ for every rational subset $U\subset X$ and every $i > 0$.

        \item\label{prop-coh-descent-2}  The functor $M \mapsto \widetilde{M}$ defines an equivalence of categories from the category of finite \Pth{\resp finite projective} $R$-modules to the category of coherent sheaves \Pth{\resp vector bundles} on $X$, with a quasi-inverse given by $\cF \mapsto \cF(X)$.
    \end{enumerate}
\end{prop}
\begin{proof}
    When $X$ is noetherian, \Refenum{\ref{prop-coh-descent-1}} is \cite[\aThm 2.5]{Huber:1994-gfsra}, while \Refenum{\ref{prop-coh-descent-2}} is Theorem \ref{thm-Kiehl}.  When $X$ is analytic, these follow from \cite[\aThms 1.4.2 and 1.3.4]{Kedlaya:2019-AWS}.
\end{proof}

Recall that an adic space $X$ is called \emph{stably adic} \Pth{as in \cite[\aDef 8.2.19]{Kedlaya/Liu:2015-RPH}} if $X_\et$ is a site with a \emph{stable basis} $\cB$; i.e., a basis stable under fiber products such that, for any morphism $Y' \to Y$ in $X_\et$ that is either finite \'etale or a rational localization, if $Y \in \cB$, then $Y' \in \cB$ as well.  We know $X$ is stably adic if $X$ is locally noetherian \Pth{see \cite[(1.1.1) and \aSec 1.7]{Huber:1996-ERA}} or a perfectoid space \Pth{see \cite[\aLec 7]{Scholze/Weinstein:2020-BLG}}.

By \cite[\aProp 8.2.20]{Kedlaya/Liu:2015-RPH}, we have the following analogue of Proposition \ref{prop-coh-descent} for the \'etale topology:
\begin{prop}\label{prop-et-coh-descent}
    Let $X = \Spa(R, R^+)$ be a noetherian \Pth{\resp analytic stably adic} affinoid adic space.  Let $\cB$ be a stable basis of $X_\et$ as above, which exists because $X$ is stably adic.  Let $M$ be a finite \Pth{\resp finite projective} $R$-module.  Let $\widetilde{M}$ denote the presheaf on $X_\et$ defined by setting $\widetilde{M}(U) = M \otimes_R \cO_X(U)$, for each $U \in X_\et$.  Then the following are true:
    \begin{enumerate}
        \item The presheaf $\widetilde{M}$ is a sheaf.  Moreover, $\widetilde{M}$ is acyclic on $\cB$; \ie, for every $Y \in \cB$, we have $H^0(Y, \widetilde{M}) = \widetilde{M}(U)$ and $H^i(Y, \widetilde{M}) = 0$, for all $i > 0$.

        \item The functor $M \mapsto \widetilde{M}$ defines an equivalence of categories from the category of finite \Pth{\resp finite projective} $R$-modules to the category of coherent sheaves \Pth{\resp vector bundles} on $X_\et$, with a quasi-inverse given by $\cF \mapsto \cF(X)$.
    \end{enumerate}
\end{prop}

\begin{cor}\label{cor-et-sheafy}
    For any $X$ in Proposition \ref{prop-et-coh-descent}, the presheaf $\cO_{X_\et}$ is a sheaf.  Therefore, $X$ is \'etale sheafy.
\end{cor}

%\bibliographystyle{amsalpha}
%\bibliography{log-adic}

\providecommand{\bysame}{\leavevmode\hbox to3em{\hrulefill}\thinspace}
\providecommand{\MR}{\relax\ifhmode\unskip\space\fi MR }
% \MRhref is called by the amsart/book/proc definition of \MR.
\providecommand{\MRhref}[2]{%
  \href{http://www.ams.org/mathscinet-getitem?mr=#1}{#2}
}
\providecommand{\href}[2]{#2}

\end{document}